\documentclass[10pt,a4paper]{article}
\usepackage[latin1]{inputenc}
\usepackage[english]{babel}
\usepackage[bottom]{footmisc}
\usepackage{amsmath}
\usepackage{amsfonts}
\usepackage{amssymb}
\usepackage{amsthm}
\usepackage{cases}
\usepackage{graphicx}
\usepackage{color}
\usepackage{float}
\usepackage{subfigure}
\usepackage{caption}
\usepackage{hyperref}
\usepackage{geometry}
\numberwithin{equation}{section}
\newcommand\numberthis{\addtocounter{equation}{1}\tag{\theequation}}
\author{Andreas Petrak}
\DeclareMathOperator{\im}{im}
\DeclareMathOperator{\Ind}{Ind}
\DeclareMathOperator{\Hess}{Hess}
\DeclareMathOperator{\codim}{codim}
\DeclareMathOperator{\Cord}{Cord}
\DeclareMathOperator{\grad}{grad}
\DeclareMathOperator{\sign}{sign}
\DeclareMathOperator{\spann}{span}
\DeclareMathOperator{\supp}{supp}
\DeclareMathOperator{\lk}{lk}
\DeclareMathOperator{\B}{B}

\newcommand{\ntransv}{\mathrel{\text{\npitchfork}}}
\makeatletter
\newcommand{\npitchfork}{%
\vbox{
\baselineskip\z@skip
\lineskip-1.5ex
\lineskiplimit\maxdimen
\m@th
\ialign{##\crcr\hidewidth\smash{$/$}\hidewidth\crcr$\pitchfork$\crcr}
}%
}
\makeatother

\newtheorem{lem}{Lemma}[section]
\newtheorem{thm}[lem]{Theorem}
\newtheorem{cor}[lem]{Corollary}
\theoremstyle{definition}
\newtheorem{defi}[lem]{Definition}
\newtheorem{ex}[lem]{Example}
\newtheorem{rem}[lem]{Remark}

\begin{document}
\begin{center}
\vspace*{1cm}
{\huge Definition of the cord algebra of knots\\ using Morse Theory}\\
\vspace*{1cm}
Andreas Petrak\\
\end{center}
\section{Introduction}
The cord algebra is a knot invariant first introduced in 2005 by Lenhard Ng in \cite{Ng2}. The topological definition of the cord algebra given there is inspired by the Legendre contact homology (see \cite{Ng1}). Lenhard Ng extends the definition 2008 in \cite{Ng3} and 2014 in \cite{Ng4}. In 2017 Kai Cieliebak, Tobias Ekholm, Janko Latschev and Lenhard Ng develop in \cite{Cie3} a noncommutative refinement of the original definition, including a base point of the knot and a framing. This results in four additional generators in the cord algebra. In \cite{Cie3} it is also shown that the cord algebra is isomorphic to the Legendre contact homology. This is done with the help of string homology by first showing the isomorphism of string homology and the cord algebra and then the isomorphism of Legendre contact homology and string homology. To define the string homology (in degree zero), a chain complex is defined which contains broken strings or 1-parameter families of broken strings as generators in degree 0 or 1, respectively. Broken strings are curves in space that intersect the knot at several points. In the proof of the isomorphism of string homology and cord algebra, a retraction of the broken strings on words in (linear) cords is used, where (linear) cords are linear curves in space with start and end points on the knot. Therefore, it makes sense to describe the cord algebra with the help of a suitable complex of (linear) cords. We will do this in this paper. For this purpose we will define a suitable chain complex which contains the critical points of index 0 and 1 of a Morse function in degree 0 and 1, respectively, and the four additional generators mentioned above in degree 0. Further degrees of the chain complex are not needed. These critical points correspond to binormal (linear) cords on the knot. In order to construct a suitable differential, we let the binormal (linear) cords of index 1 flow along their unstable manifolds until they, taking into account four relations, reach (linear) cords of index 0. These four relations are used in \cite{Cie3} to define the cord algebra, where two of them are already defined by Lenhard Ng in \cite{Ng2} and a third is introduced in \cite{Ng3}. We will adapt them here only slightly to the changed concept.\\[.5em]
We will need several results from differential topology and Morse theory. Therefore, in Appendix~\ref{Backgroundmaterial} some statements of the mentioned fields are presented without proofs. The proofs can be found in the literature referred to.\\
Important statements from differential topology which we will use are the transversality theorem and the jet transversality theorem as well as the relative versions of these two theorems.\\
To define the cord algebra, we will consider a gradient vector field. This vector field has to satisfy the Smale condition. For this it may be necessary to change the gradient to a pseudo gradient. From Morse theory we also need the statement that in a generic 1-parameter family of vector fields without nonconstant periodic orbits only birth-death type degeneracies occur.\\[.5em]
At the beginning of Section \ref{Corddef} we will first explain the noncommutative definition of the cord algebra from \cite{Cie3}. A cord is a continuous path in space that has a start and end point on the knot $K$ and does not intersect the knot in any other point. We also need a base point on $K$ and a framing which is a slightly shifted copy of $K$. Four relations are defined which specify relations between different cords, partly taking into account the base point and the framing. These relations generate an ideal~$\mathcal{I}$ in a noncommutative unital ring $\mathcal{A}$, where $\mathcal{A}$ is generated by homotopy classes of cords and four other generators $\lambda^{\pm1},\mu^{\pm1}$ modulo the relations $\lambda\cdot\lambda^{-1}=\lambda^{-1}\cdot\lambda=\mu\cdot\mu^{-1}=\mu^{-1}\cdot\mu=1$ and $\lambda\cdot\mu=\mu\cdot\lambda$. The cord algebra of $K$, $\Cord(K)$, is then defined as the quotient ring~$\mathcal{A}/\mathcal{I}$.\\[.5em]
In section \ref{enfct} we will redefine the terms cord and framing to be able to define the cord algebra using Morse theory. A (linear) cord of $K$ is now a straight line in space with start and end points on the knot. A framing is now a smooth map that assigns a unit normal vector to each point of the knot. The four relations from the original definition of the cord algebra are adapted accordingly. We'll look at the energy function $E:K\times K\to\mathbb{R},(x,y)\mapsto\frac{1}{2}\vert x-y\vert^{2}$ and change it to a Morse function. The critical points of Morse index~1 of this Morse function are binormal (linear) cords whose unstable manifolds are one-dimensional. These (linear) cords are moved along their unstable manifolds and, taking into account the relations, generate an ideal $I$ in an $R$-algebra $C_{0}$ which is generated by critical points of index 0. Here $R=\mathbb{Z}[\lambda^{\pm1},\mu^{\pm1}]$ is the commutative ring over $\mathbb{Z}$ which is generated by $\lambda,\lambda^{-1},\mu$ and $\mu^{-1}$ modulo the relations $\lambda\cdot\lambda^{-1}=\mu\cdot\mu^{-1}=1$. We then define in Section \ref{Cord} the cord algebra as $\Cord(K):=C_{0}/I$.\\[.5em]
Before that, however, we will look at the following three subsets of $K\times K$: 
\begin{itemize}
\item[B]is the set of (linear) cords that have the base point $\ast$ as start or end point. If we choose a parametrization $\gamma$ of $K$ with $\gamma(0)=\ast$, then $B=(K\times\lbrace0\rbrace)\cup(\lbrace0\rbrace\times K)$.
\item[S]is the set of (linear) cords that intersect the knot $K$ in their interior. Generically, $S\subset K\times K$ is an immersed curve with boundary. This and other properties of $S$ are shown in \cite{Cie3}.
\item[F]is the set of (linear) cords that intersect the framing. Since $F$ is symmetric, $F$ can be split into two subsets $F^{s}$ and $F^{e}$ of (linear) cords that intersect the framing at their start and end points, respectively. These subsets are one-dimensional submanifolds with boundary of $K\times K$. In addition, we have $\partial F=\partial S$. We will formulate this lemma in Section \ref{enfct}. The proof of this lemma is given in Appendix \ref{framinglemmaproof}.
\end{itemize}
In Section \ref{enfct} we will also show some generic properties of the function $E_{g}=E+g$, where $g$ is a small perturbation of the function $E$.\\[.5em]
In the definition of the cord algebra we use the Seifert framing as a canonical framing. To determine the cord algebra of a knot, however, it is easier to use the blackboard framing. In order to obtain the cord algebra with respect to the Seifert framing, certain transformations must be applied. We will discuss these transformations in Section \ref{Framing}.\\[.5em]
In Section \ref{Examples} we will determine the cord algebras for the unknot and the right-handed trefoil knot.\\[.5em]
The proof that the cord algebra according to our definition is a knot invariant will be given in Section \ref{Knotinv}. We consider two knots $K_{0}$ and $K_{1}$ which are connected by a smooth isotopy of knots. So it must be shown that the cord algebras of the two knots are the same. First of all, it is necessary to note that in the course of the isotopy there are only a finite number of knots that are not generic. Therefore, we can assume that only one non-generic knot occurs during the isotopy and show that the cord algebra does not change in the course of the isotopy. Since there are several cases of degeneracies, we have to go through all these cases. Furthermore, we will also show that the cord algebra does not change in the course of the isotopy if no non-generic knot occurs.
\subsection*{Acknowledgements} A thousand thanks to Kai Cieliebak for many inspiring conversations.
\section{Definition of the Cord Algebra}\label{Corddef}
The cord algebra is a knot invariant developed only a few years ago \cite{Ng2,Ng3,Ng4}. In \cite{Cie3} a non-commutative refinement of the original definition is presented. This refined version is briefly explained in the following. First, we will define the terms framing and cord. However, in Section~\ref{enfct} we will modify the definitions of these terms to define the cord algebra using Morse theory.
\begin{defi}
Let $K\subset\mathbb{R}^{3}$ be a knot of length $L$ and $\gamma:[0,L]\to\mathbb{R}^{3}$ be an arclength parametrization of $K$. Let $\nu:[0,L]\to S^{2}$ be a smooth map, where $\nu(t)$ is a unit normal vector to $K$ at the point $\gamma(t)$ for all $t\in[0,L]$. Let $\varepsilon>0$ be small enough such that the strip $\lbrace\gamma(t)+\alpha\nu(t):t\in[0,L],\alpha\in[0,\varepsilon]\rbrace$ has no self-intersections.\\
A \textit{framing} of $K$ is the set $K^{\prime}:=\lbrace\gamma(t)+\varepsilon\nu(t):t\in[0,L]\rbrace$.
\end{defi}
Let $K\subset\mathbb{R}^{3}$ be an oriented knot equipped with a framing $K^{\prime}$. Choose a base point $\ast$ on $K$ and a corresponding base point $\ast$ on $K^{\prime}$ (in fact only the base point on $K^{\prime}$ will be needed).
\begin{defi}
A \textit{cord} of $K$ is a continuous map $\alpha:[0,1]\to\mathbb{R}^{3}$ such that $\alpha([0,1])\cap K=\emptyset$ and $\alpha(0),\alpha(1)\in K^{\prime}\setminus\lbrace\ast\rbrace$. Two cords are \textit{homotopic} if they are homotopic through cords.
\end{defi}
We now construct a noncommutative unital ring $\mathcal{A}$ as follows: as a ring, $\mathcal{A}$ is freely generated by homotopy classes of cords and four extra generators $\lambda^{\pm1},\mu^{\pm1}$, modulo the relations
\[\lambda\cdot\lambda^{-1}=\lambda^{-1}\cdot\lambda=\mu\cdot\mu^{-1}=\mu^{-1}\cdot\mu=1,\phantom{blabla}\lambda\cdot\mu=\mu\cdot\lambda.\]
Thus, $\mathcal{A}$ is generated as a $\mathbb{Z}$-module by (noncommutative) words in homotopy classes of cords and powers of $\lambda$ and $\mu$ (and the powers of $\lambda$ and $\mu$ commute with each other, but not with any cords).
\begin{defi}\label{ngcordalgebradef}
The \textit{cord algebra} of $K$ is the quotient ring
\[\Cord(K)=\mathcal{A}/\mathcal{I},\]
where $\mathcal{I}$ is the two-sided ideal of $\mathcal{A}$ generated by the relations shown in Figure \ref{ngrelations}.
\begin{figure}[ht]
\hspace*{3.5cm}\subfigure{\includegraphics[scale=0.35]{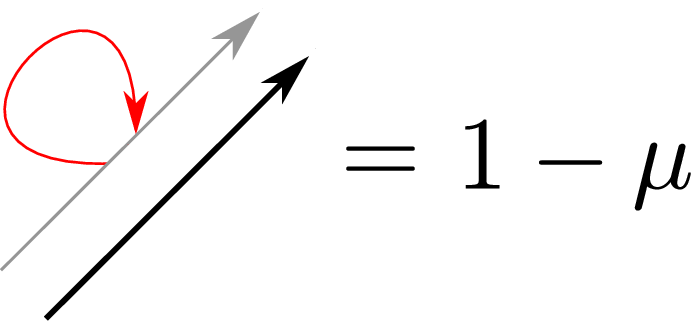}}\\
\hspace*{3.5cm}\subfigure{\includegraphics[scale=0.35]{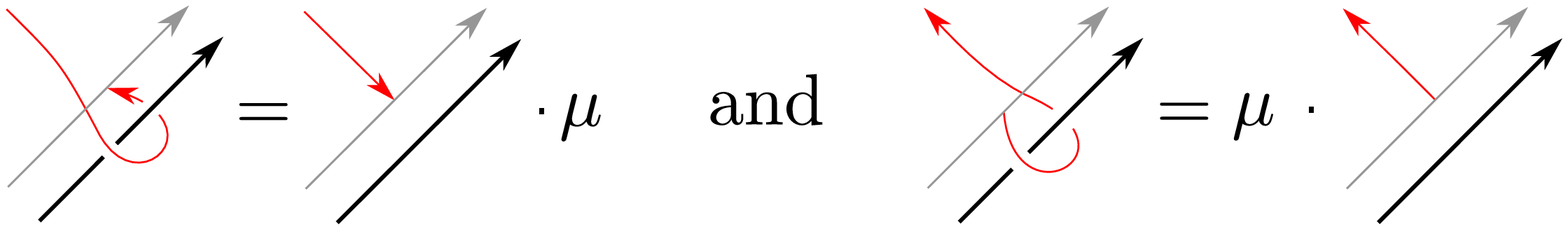}}\\
\hspace*{3.5cm}\subfigure{\includegraphics[scale=0.35]{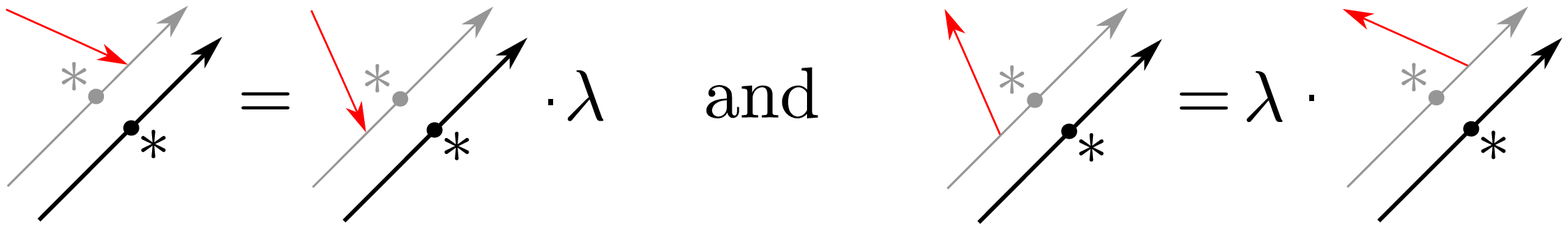}}\\
\hspace*{3.5cm}\subfigure{\includegraphics[scale=0.35]{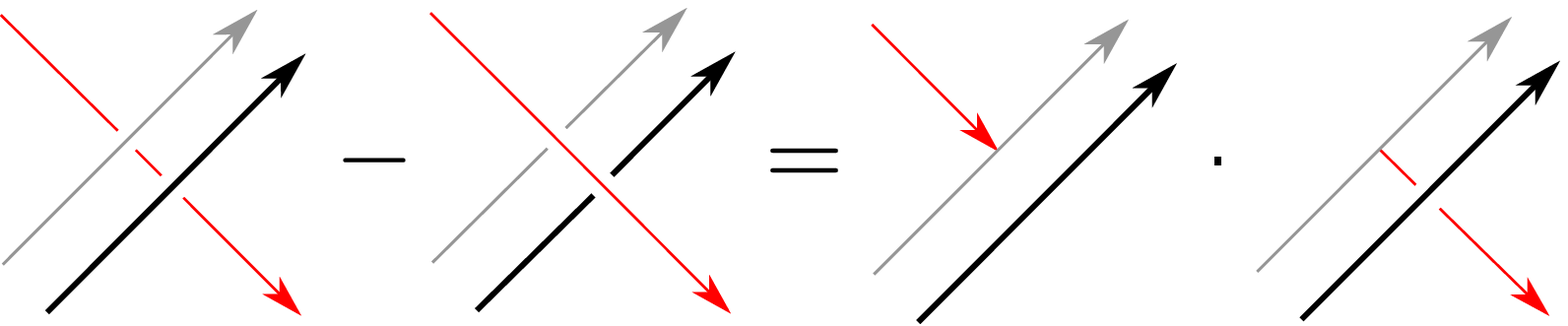}}
\caption{Relations for cords}\label{ngrelations}
\end{figure}
\end{defi}
Here $K$ is depicted in black and $K^{\prime}$ in gray, and cords are drawn in red.
\begin{rem}[\cite{Cie3}]\label{remng}
The relations in Definition \ref{ngcordalgebradef} depict cords in space that agree
outside of the drawn region (except in (iv), where either of the two cords on the left hand side of the equation splits into the two on the right). Thus, (ii) states that appending a meridian to the beginning or end of a cord multiplies that cord by $\mu$ on the left or right. $\mu$ is also called \textit{meridian}. In the third relation it is shown that crossing the base point multiplies the cord by $\lambda$ from the left or right. $\lambda$ is also called \textit{longitude}. Relation (iv) is equivalent to the relation shown in Figure \ref{ngrelationivequiv}.
\begin{figure}[ht]\centering
\includegraphics[scale=0.35]{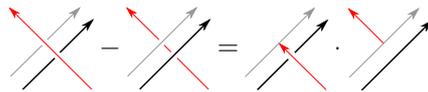}
\caption{Equivalent representation of the fourth relation}\label{ngrelationivequiv}
\end{figure}%

Applying the second and fourth relation in a suitable way to a contractible cord $c$, we get the equation
\[(c-(1-\mu))c=0.\]
The first relation means that the expression in the brackets already vanishes.
\end{rem}
\subsection{The energy function for generic knots}\label{enfct}
In this paper we want to define the cord algebra using Morse theory. For this we need a framing and cords as above. However, as already mentioned, we will redefine these two terms.
\begin{defi}\label{framingdef}
Let $K\subset\mathbb{R}^{3}$ be a knot of length $L$ and $\gamma:[0,L]\to\mathbb{R}^{3}$ be an arclength parametrization of $K$.
\begin{itemize}\itemsep0pt
\item[(i)]A \textit{framing} of $K$ is a smooth map $\nu:[0,L]\to S^{2}$, where $\nu(t)$ is a unit normal vector to $K$ at the point $\gamma(t)$ for all $t\in[0,L]$.
\item[(ii)]Given a Seifert surface for a knot $K$, a \textit{Seifert framing} is a framing such that the normal vector $\nu(t)$ is tangent to the Seifert surface and pointing inwards for all $t\in[0,L]$.
\end{itemize}
\end{defi}
\begin{rem}
Assume that $\ddot{\gamma}$ vanishes nowhere. A framing can also be understood as a map $\nu:S^{1}\to S^{1}$ by representing $\nu(t)$ for each $t\in S^{1}$ in local coordinates of the normal plane $N(t)=\spann(\ddot{\gamma}(t),\dot{\gamma}(t)\times\ddot{\gamma}(t))$ at the point $\gamma(t)$.\\
Given a knot $K$, we take a look at the set of framings of $K$ modulo homotopy:
\begin{align*}
\lbrace\text{framings of }K\rbrace/\sim&=C(S^{1},S^{1})/\sim\\
&=\pi_{1}(S^{1})\\
&\cong\mathbb{Z}.
\end{align*}
As a consequence, the Seifert framing is unique up to homotopy, since the linking number of $K$ and a copy of $K$ which is shifted slightly in the direction of the Seifert framing is always zero.
\end{rem}
Let $K\subset\mathbb{R}^{3}$ be a generic oriented knot of length $L$ and $\gamma:[0,L]\to K$ be an arclength parametrization of $K$. Also, let $K$ be equipped with a framing. For this we use the Seifert framing (obtained by the Seifert algorithm) as a canonical framing. To facilitate the determination of the cord algebra of a knot, however, we will use a different framing. This and the necessary transformations to obtain the cord algebra with respect to the Seifert framing will be discussed in Section \ref{Framing}. Furthermore, we choose a base point $\ast$ on $K$.
\begin{defi}
Let $K\subset\mathbb{R}^{3}$ be a knot (or a link). A \textit{cord} of $K$ is a curve $\alpha\in C^{2}([0,1],\mathbb{R}^{3})$ such that $\alpha(0),\alpha(1)\in K$ and $\ddot{\alpha}\equiv0$, i.e. $\alpha$ is a straight line in $\mathbb{R}^{3}$ starting and ending on the knot.
\end{defi}
The space of these cords can be canonically identified with $K\times K$ by associating to each cord its endpoints on $K$: For $s,t\in[0,L]$,
\[c=(\gamma(s),\gamma(t))\in K\times K\]
is a cord with startpoint $\gamma(s)$ and endpoint $\gamma(t)$. The space $K\times K$ can be canonically identified with the torus $T^{2}$ by using the identification $S^{1}\cong\mathbb{R}/L\mathbb{Z}$. The startpoint resp.~endpoint of a cord $c=(\gamma(s),\gamma(t))$ can be identified with $s\in S^{1}$ resp.~$t\in S^{1}$, and thus we can simply write 
\[c=(s,t).\]
We also assign an orientation to each cord: A cord $c=(\gamma(s),\gamma(t))$ is oriented from its startpoint $\gamma(s)$ to its endpoint $\gamma(t)$. 
\begin{defi}\label{framingintersectiondef}
Let $K\subset\mathbb{R}^{3}$ be an oriented knot of length $L$ equipped with a framing $\nu$. Let $\gamma:~[0,L]\to\mathbb{R}^{3}$ be an arclength parametrization of $K$. Let $N(t)\subset\mathbb{R}^{3}$ be the normal plane to $K$ at the point $\gamma(t)$ and $\pi_{t}:\mathbb{R}^{3}\to N(t)$ be the orthogonal projection onto $N(t)$. Let $c=(s,t)\in K\times K$ be a cord of $K$.\\
We say \textit{c intersects the framing} if one of the following conditions is satisfied (see Figure \ref{framingintersection}):
\begin{itemize}\itemsep-4pt
\item $\pi_{s}(\gamma(t)-\gamma(s))=\alpha\nu(s)$ for an $\alpha>0$ (\textit{$c$ intersects the framing at its startpoint}), or
\item $\pi_{t}(\gamma(s)-\gamma(t))=\alpha\nu(t)$ for an $\alpha>0$ (\textit{$c$ intersects the framing at its endpoint}).
\end{itemize}
\begin{figure}[!ht]\centering
\begin{minipage}{0.42\textwidth}
\includegraphics[width=1.0\textwidth]{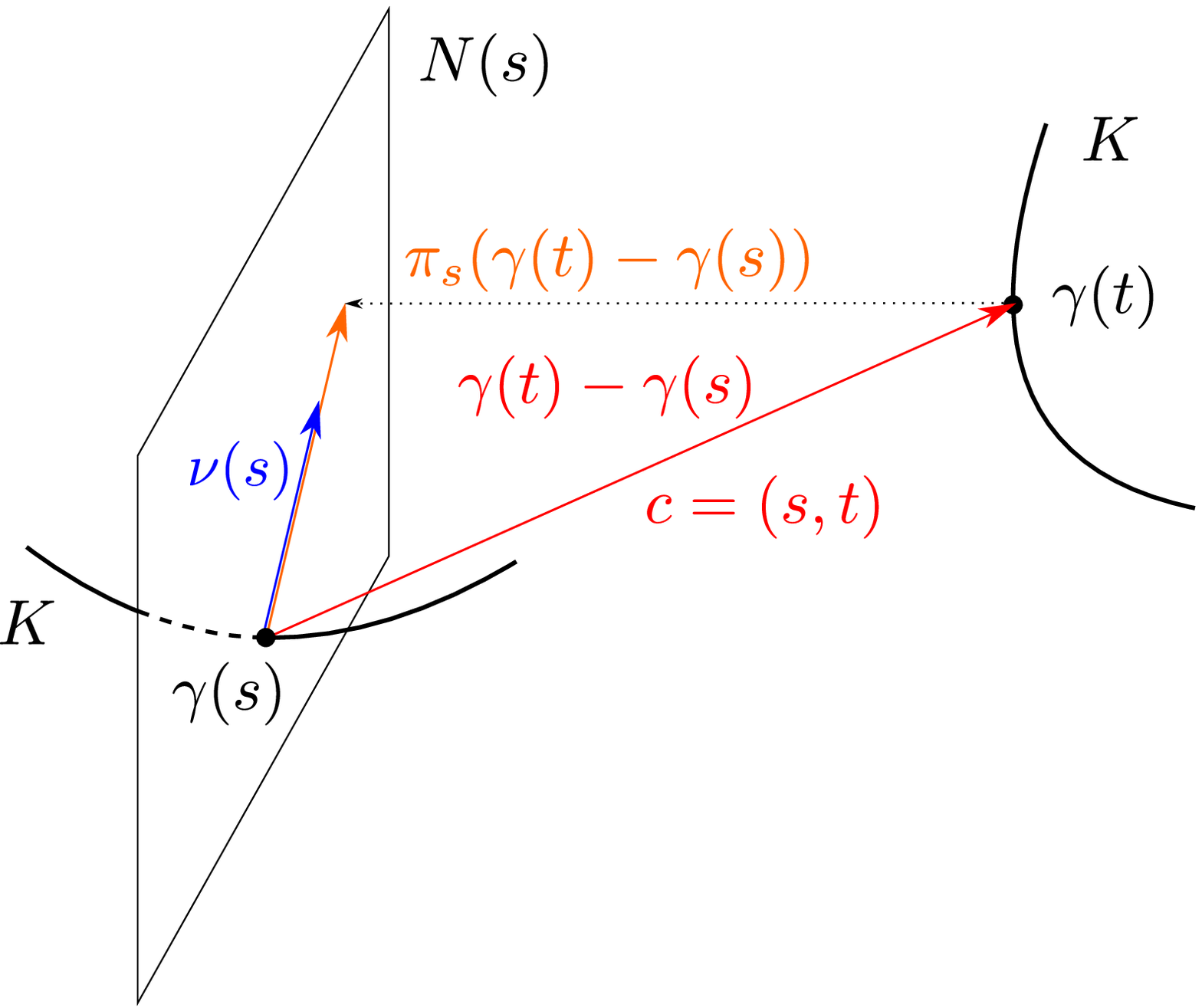}
\caption*{\hspace*{-.2cm}(a) $c$ intersects the framing at its startpoint}
\end{minipage}\hspace*{1cm}
\begin{minipage}{0.42\textwidth}
\vspace*{.7cm}
\includegraphics[width=1.0\textwidth]{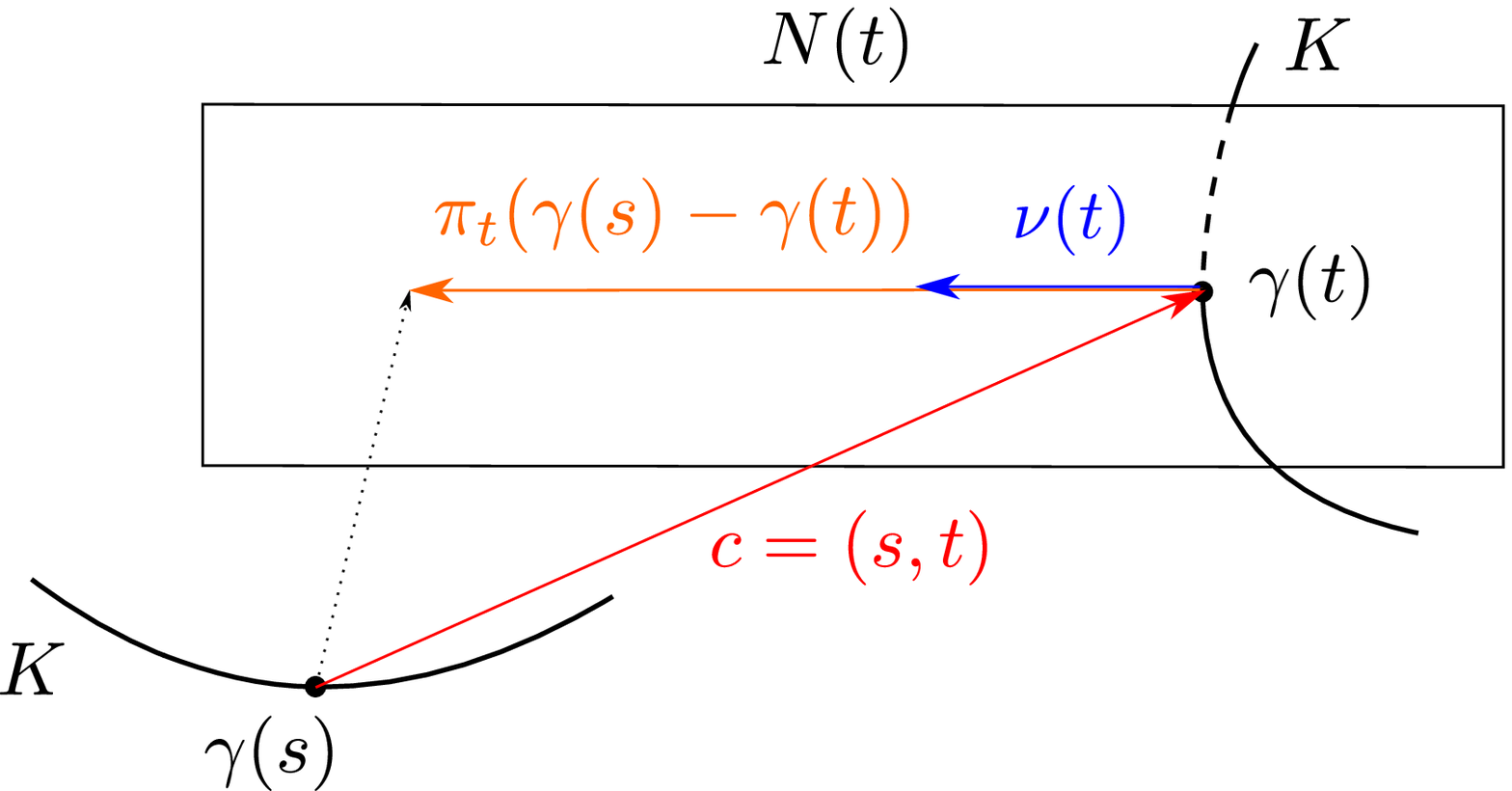}\vspace*{1.3cm}
\caption*{\hspace*{-.2cm}(b) $c$ intersects the framing at its endpoint}
\end{minipage}\\
\caption{A cord $c$ intersects the framing}\label{framingintersection}
\end{figure}
\end{defi}
\begin{rem}\label{symmetryofF}
From this definition follows: If a cord $(s,t)\in K\times K$ intersects the framing, so does the reverse oriented cord $(t,s)$. Thus, the set of cords intersecting the framing is symmetric with respect to the diagonal in $K\times K$.
\end{rem}
\begin{rem}
Let $\varepsilon>0$ such that the strip $\lbrace\gamma(t)+\alpha\nu(t):t\in[0,L],\alpha\in[0,\varepsilon]\rbrace$ has no self-intersections, i.e. the map
\begin{align*}
[0,L]\times[0,\varepsilon]&\to\mathbb{R}^{3}\\
(t,\alpha)&\mapsto\gamma(t)+\alpha\nu(t)
\end{align*}
is injective. In diagrams we will draw the set $K^{\prime}:=\lbrace\gamma(t)+\varepsilon\nu(t):t\in[0,L]\rbrace$ to visualize the framing. $K^{\prime}$ is also called \textit{framing}. An intersection of a cord with the framing according to Definition \ref{framingintersectiondef} corresponds approximately to an intersection of the cord with the set $K^{\prime}$ if the curvature of $K^{\prime}$ is not too strong in a neighborhood of this intersection. Also, there can be very short cords that do not intersect $K^{\prime}$, but intersect the framing according to definition~\ref{framingintersectiondef}. Therefore, $K^{\prime}$ is to be understood only as visualization of the framing. In order to determine intersections of cords with the framing, it may be necessary to use the above definition.
\end{rem}
Let
\begin{align*}
E:K\times K&\to\mathbb{R}\\
(x,y)&\mapsto\frac{1}{2}\vert x-y\vert^{2}
\end{align*}
be the \textit{energy function} on the space of cords, where $x$ resp.~$y$, as described above, is the startpoint resp.~endpoint of a cord (cf. \cite{Cie3}). Using the parametrization $\gamma$ we can write
\[E(s,t)=\frac{1}{2}\vert\gamma(s)-\gamma(t)\vert^{2}.\]
Furthermore, we will need the following subsets of $K\times K$: 
\begin{itemize}\itemsep0pt
\item Let $S\subset K\times K$ be the set of cords that intersect the knot $K$ in their interior.
\item Let $F\subset K\times K$ be the set of cords that intersect the framing. Since $F$  is symmetric with respect to the diagonal in $K\times K$ according to Remark \ref{symmetryofF}, $F$ can be split into
\[F=F^{s}\cup F^{e}\]
where $F^{s}$ resp.~$F^{e}$ is the set of cords that intersect the framing at their startpoint resp. endpoint. The following holds:
\[F^{e}=\lbrace(s,t)\in K\times K:(t,s)\in F^{s}\rbrace.\]
\item Let $B\subset K\times K$ be the set of cords that begin or end at the base point $\ast$. If we choose a parametrization with $\gamma(0)=\ast$, then $B=(K\times\lbrace0\rbrace)\cup(\lbrace0\rbrace\times K)$. 
\end{itemize}
\begin{lem}[\cite{Cie3}, Lemma 7.10]\label{cordlemma}
For a generic knot $K\subset\mathbb{R}^{3}$ the following holds for the space $K\times K$ of cords (see Figure \ref{cordintersect}):
\begin{itemize}\itemsep0pt
\item[(i)]$E$ attains its minimum 0 along the diagonal, which is a Bott nondegenerate
critical manifold; the other critical points are nondegenerate binormal cords of index
0, 1, 2.
\item[(ii)]The subset $S\subset K\times K$ of cords meeting $K$ in their interior is an immersed curve with boundary consisting of finitely many cords tangent to $K$ at one endpoint, and with finitely many transverse self-intersections consisting of finitely many cords meeting $K$ twice in their interior.
\item[(iii)]The negative gradient $-\nabla E$ is not pointing into $S$ at the boundary points.
\end{itemize}
\begin{figure}[H]\centering
\includegraphics[scale=0.3]{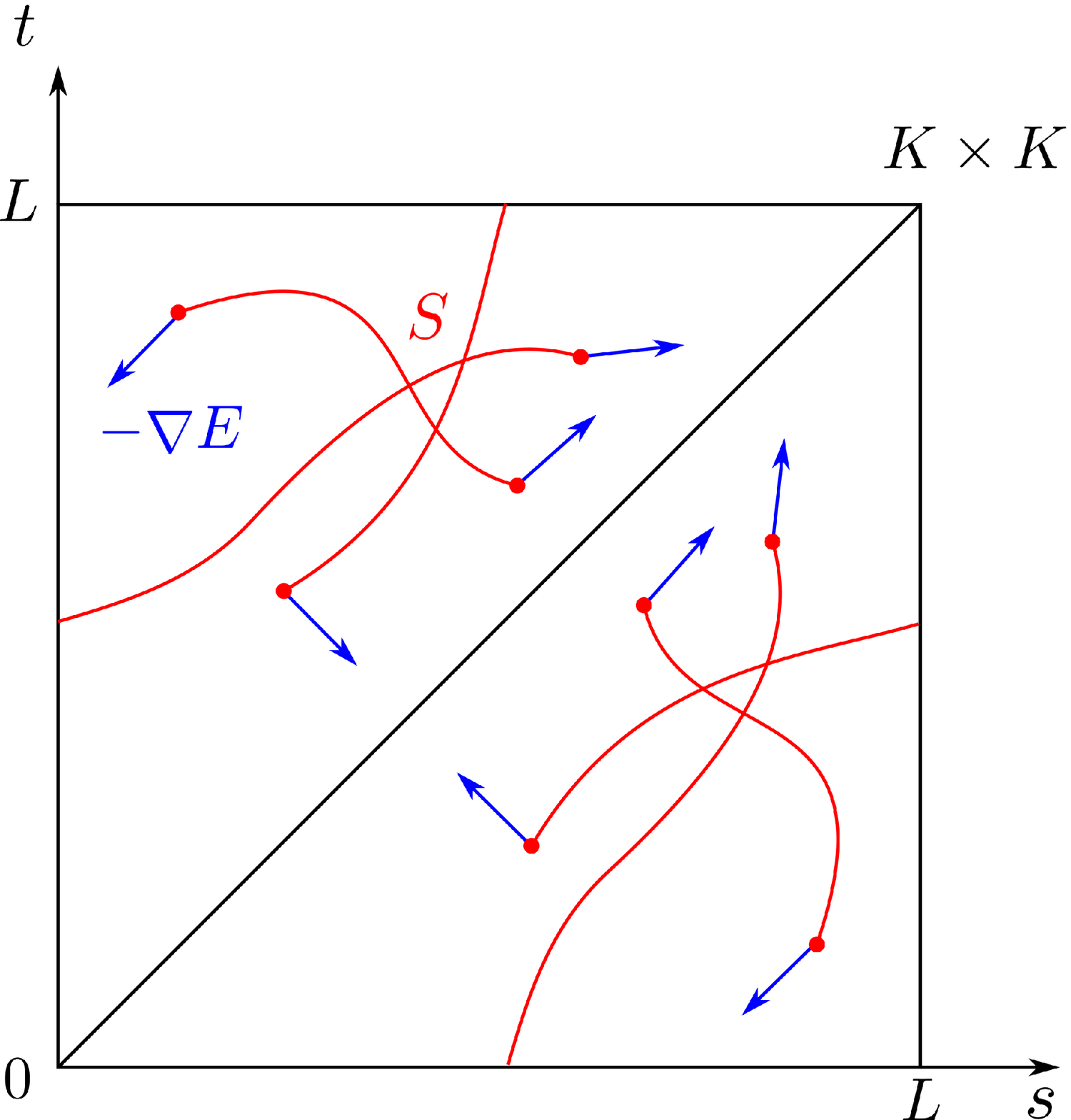}
\caption{The space $K\times K$ of cords}\label{cordintersect}
\end{figure}%
\end{lem}
The proof of this lemma is given in \cite{Cie3}.\\[1em]
From now on, the knot $K$ is as in Lemma \ref{cordlemma}.
\begin{rem}(\cite{Pet})\label{Morsegeneration}
The Morse Bott function $E$ can be converted into a Morse function by adding a smooth function $f:K\times K\to\mathbb{R}$ to $E$ that satisfies the following properties:
\begin{itemize}\itemsep0pt
\item[(i)]$f$ has exactly two critical points along the diagonal of $K\times K$, a minimum $m$ and a maximum~$M$. 
\item[(ii)]Outside the diagonal, the function $f$ is smoothly extended in such a way that it vanishes outside a small neighborhood of the diagonal, especially at all other critical points.
\end{itemize}
In the following, we assume that $E$ is a Morse function since this can be achieved by adding an arbitrarily small perturbation $f$. 
\end{rem}
\begin{rem}
The set $S$ is obviously symmetric with respect to the diagonal in $K\times K$: If the cord $(s,t)$ intersects the knot in its interior, so does the reverse oriented cord $(t,s)$. The same also holds for the boundary of $S$. So $\partial S$ can be split into
\[\partial S=\partial^{s}S\cup\partial^{e}S\]
where $\partial^{s}S$ resp.~$\partial^{e}S$ is the set of cords that are tangent to $K$ at their startpoint resp.~endpoint.
\end{rem}
\begin{lem}\label{framinglemma}
For a generic knot $K$ and a generic framing $\nu:S^{1}\to S^{2}$ the following holds for the space $K\times K$ of cords:\\
The set $F^{s}\subset K\times K$ of cords that intersect the framing at their startpoint is a one-dimensional submanifold with boundary and $\partial F^{s}=\partial^{s}S$.
\end{lem}
The proof of this lemma is given in Appendix \ref{framinglemmaproof}.
\begin{rem}
Lemma \ref{framinglemma} implies that $F^{s}$ has no self-intersections. Since $F$ is symmetric, an analogous statement holds for $F^{e}$. So the self-intersections of $F$ are the intersections of $F^{s}$ and $F^{e}$. These are either located on the diagonal of $K\times K$ or occur in pairs symmetrically to the diagonal and contain the cords that intersect the framing at their startpoint and endpoint. 
\end{rem}
\begin{defi}(\cite{Pet})
Let $K$ be a knot and $E:K\times K\to~\mathbb{R}$ be the energy function that has been perturbed to a Morse function as described above. Denote by $Crit_{k}(E)$ the set of critical points of Morse index $k$ of the function $E$. 
\end{defi}
\begin{lem}\label{cordlemmaadd1}
For a generic knot $K$, a generic base point and a generic framing the following holds for the space $K\times K$ of cords:
\[Crit_{0,1}\cap(B\cup F\cup S)=\emptyset\]
where $Crit_{0,1}:=Crit_{0}\cup Crit_{1}$.
\end{lem}
\begin{proof}
According to Lemma \ref{cordlemma}(i), $Crit_{0,1}$ is a finite set. The sets $B$, $F$, and $S$ can be considered separately from each other:\\[.5em]
If $Crit_{0,1}\cap B\neq\emptyset$:\\
An arbitrarily small shift of the base point is sufficient to achieve $Crit_{0,1}\cap B=\emptyset$.\\[.5em]
If $Crit_{0,1}\cap F^{s}\neq\emptyset$:\\
Let $k\in Crit_{0,1}\cap F^{s}$. If the framing is changed in an arbitrarily small neighborhood of the startpoint of the cord $k$ by an arbitrarily small perturbation, $k$ no longer intersects the framing at its startpoint. Thus, also the reverse oriented cord, which was an element of $F^{e}$, does not intersect the framing at its endpoint anymore. This procedure is necessary only finitely many times because $Crit_{0,1}\cap F$ is finite.\\[.5em]
If $Crit_{0,1}\cap S\neq\emptyset$:\\
Let $k\in Crit_{0,1}\cap S$ and $p$ be the intersection point of $k$ with the knot in the interior of $k$. If the knot is changed in an arbitrarily small neighborhood of $p$ by an arbitrarily small perturbation in a suitable direction (see Figure \ref{Crit01inS}), $k$ no longer intersects the knot in its interior. Since $Crit_{0,1}\cap S$ is finite, only finitely many such perturbations have to be made.
\begin{figure}[ht]\centering
\subfigure{\includegraphics[scale=0.35]{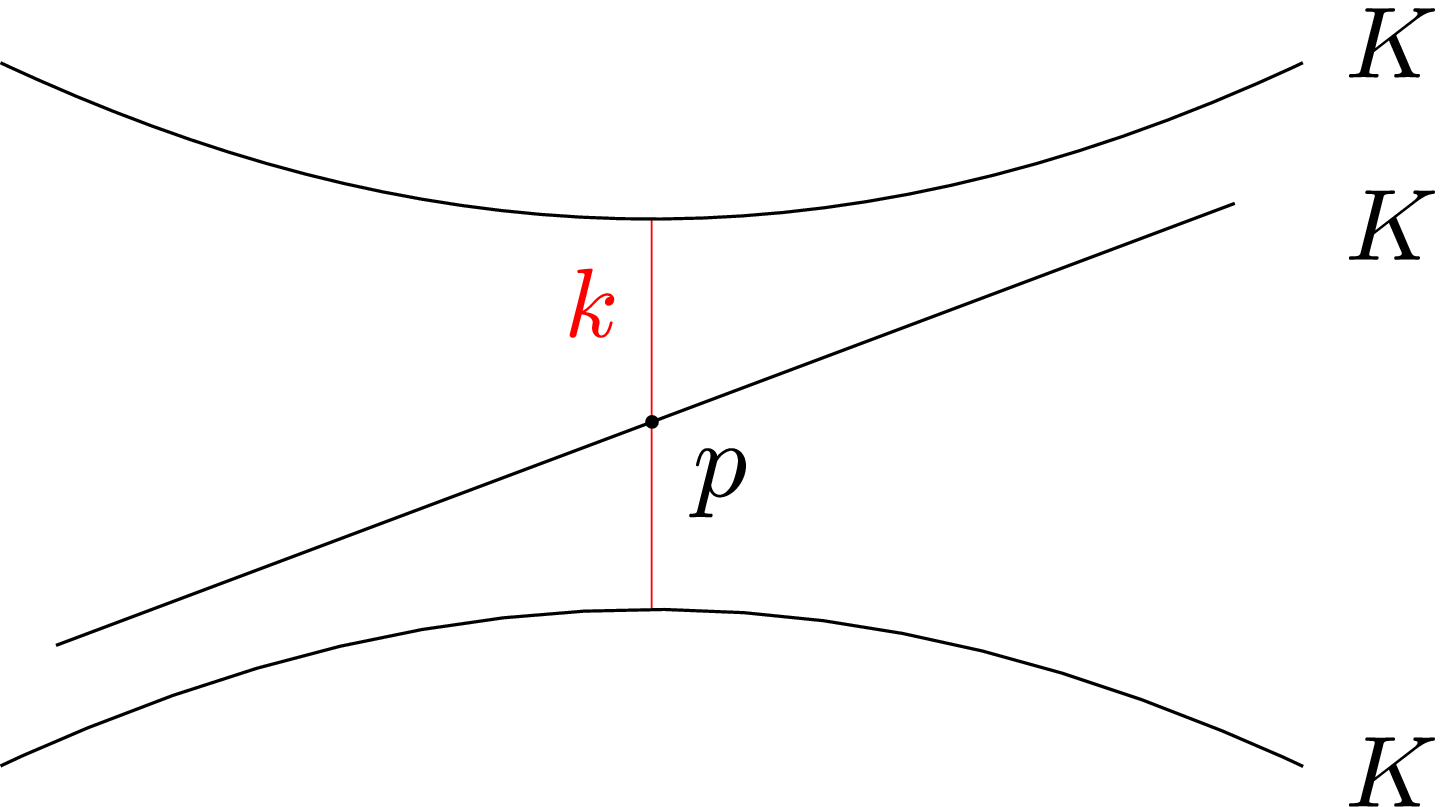}}\hspace*{2cm}
\subfigure{\includegraphics[scale=0.35]{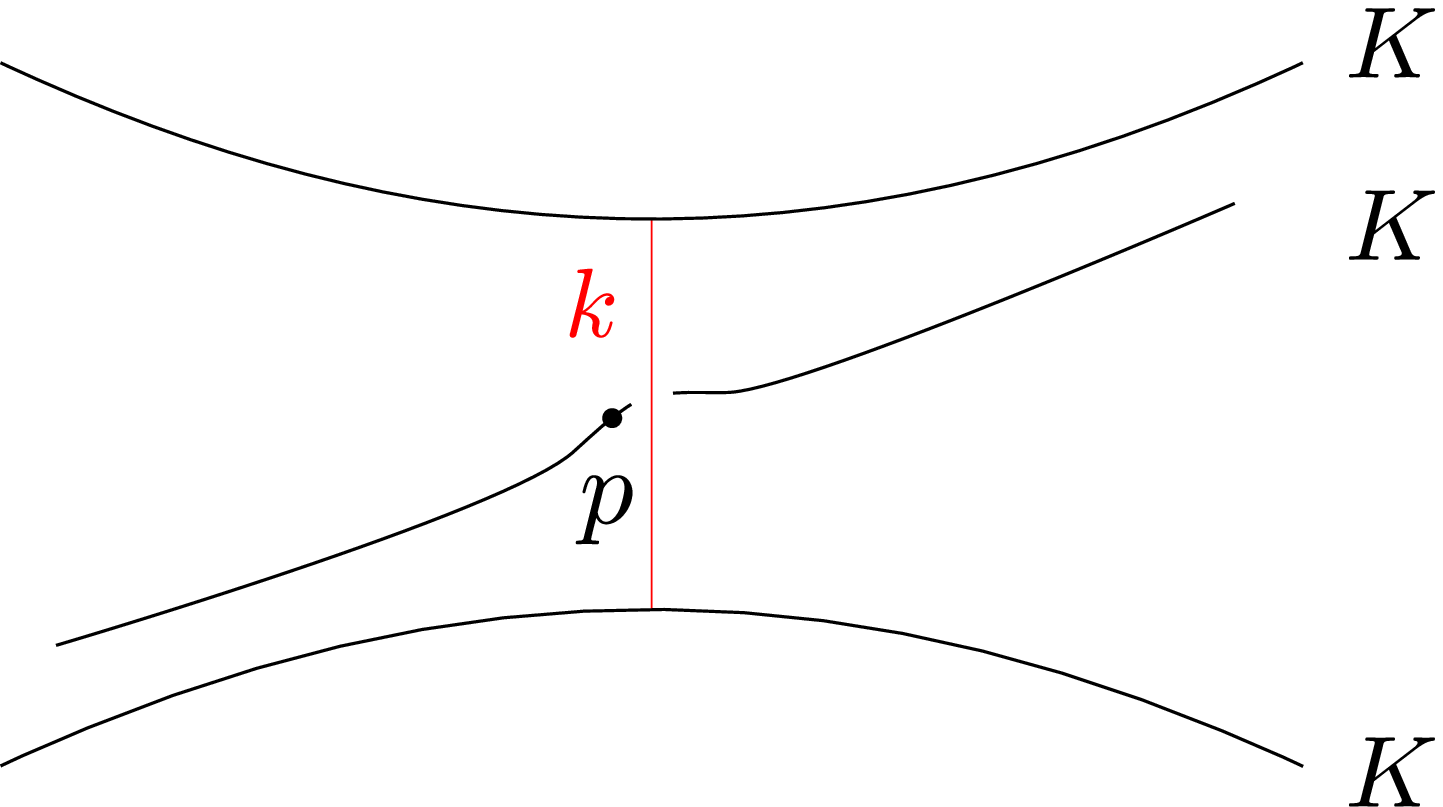}}
\caption{Perturbation of the knot in a neighborhood of $p$ such that the cord $k$ does not intersect the knot in its interior anymore}\label{Crit01inS}
\end{figure}
\vspace*{-.5cm}

\end{proof}
From now on the knot $K$ is as in Lemma \ref{cordlemmaadd1}, the energy function $E$ as in Remark \ref{Morsegeneration} and the framing as in Lemma \ref{framinglemma}.\\[.5em]
We will show more generic properties of the function $E$. For this it may be necessary to perturb the function $E$ with the help of the (jet) transversality theorem, i.e. to vary $E$ in the function space $C^{n}(T^{2},\mathbb{R})$ for an $n>0$. Such a variation can be considered as the addition of another smooth function $g:K\times K\to\mathbb{R}$ to $E$. Denote by
\[E_{g}:=E+g\]
the perturbed function. The function $g$ can always be chosen so that it is arbitrarily small (in the $C^{n}$ sense) (according to the (jet) transversality theorem) and vanishes in a small neighborhood of the diagonal in $K\times K$ and small neighborhoods of all critical points of $E$ (since $E$ does not have to be perturbed in these neighborhoods, because the properties shown in Lemma~\ref{cordlemmaadd2} and Lemma~\ref{cordlemmaadd3} for generic knots are already satisfied on the diagonal and because, according to Lemma~\ref{cordlemmaadd1}, the following holds: $Crit_{0. 1}\cap(B\cup F\cup S)=\emptyset$). Thus, all critical points remain unchanged, i.e. the following holds for $k=0,1,2$: 
\[Crit_{k}(E_{g})=Crit_{k}(E).\]
Likewise, the sets $S$, $F$, and $B$ do not change, since the knot itself is not changed.\\
We will consider the flow $\varphi^{s}$ of the negative gradient $-\nabla E$, in the perturbed case the flow $\varphi_{g}^{s}$ of $-\nabla E_{g}$. The stable resp.~unstable manifold of a critical point $k$ of $E$ is 
\begin{align*}
W^{s}(k)&=\lbrace x\in T^{2}:\lim_{s\to+\infty}\varphi^{s}(x)=k\rbrace\text{ resp.}\\
W^{u}(k)&=\lbrace x\in T^{2}:\lim_{s\to-\infty}\varphi^{s}(x)=k\rbrace.
\end{align*}
Similarly, the stable and unstable manifolds of critical points $k$ of $E_{g}$ are denoted by $W^{s}_{g}(k)$ and $W^{u}_{g}(k)$, respectively.\\
Let $k$ be a critical point of $E$ of index 1. Then $k$ is also a critical point of $E_{g}$ since $g$ vanishes in a neighborhood of $k$. So $W^{u}(k)$ and $W^{u}_{g}(k)$ coincide in this neighborhood. Choose a point $x\in W^{u}(k)\cap W^{u}_{g}(k)$ close to $k$, but $x\neq k$. Since the solution of a differential equation depends continuously on the function and the function $g$ is arbitrarily small, it can be guaranteed that $\varphi^{T}(x)$ and $\varphi^{T}_{g}(x)$, for a finite time $T>0$, lie in a small neighborhood of a critical point $c$ of index~0. Since $g$ also vanishes in this neighborhood, the following holds:
\[\lim_{s\to+\infty}\varphi^{s}(x)=\lim_{s\to+\infty}\varphi^{s}_{g}(x)=c.\]
A similar consideration results in: If $W^{u}(k)$ intersects one of the sets $S$, $F$, or $B$ (for $S$ and $F$: in the interior), so does $W^{u}_{g}(k)$ if $g$ is chosen small enough.\\[.5em]
To be able to show some generic properties of the function $E_{g}$, we need the following statement: If we perturb the unstable manifold of a critical point of index 1 a little bit, we can find a function~$\tilde{g}$ such that the perturbed unstable manifold is realized by the flow of the negative gradient of the function $E_{\tilde{g}}$. First, we show the following lemma.
\begin{lem}(\cite{Cie2})\label{cutterfunction}
For all $\delta,\varepsilon>0$ there exists a smooth function $\chi:\mathbb{R}\to[0,1]$ with the following properties:
\begin{itemize}\itemsep0pt
\item[(i)]$\chi$ is non-decreasing on $\mathbb{R}_{-}$ and non-increasing on $\mathbb{R}_{+}$.
\item[(ii)]$\chi$ is constant 1 in a neighborhood of 0.
\item[(iii)]$\chi$ is constant 0 on $\mathbb{R}\setminus(-\delta,\delta)$.
\item[(iv)]For all $x\in\mathbb{R}$ the following holds:
\[\vert x\chi^{\prime}(x)\vert<\varepsilon.\]
\end{itemize}
\end{lem}
\begin{proof}(\cite{Cie2})
We consider the function $f:[0,\delta]\to\mathbb{R},x\mapsto\varepsilon\log(\frac{\delta}{x})$. $f$ satisfies $xf^{\prime}(x)=-\varepsilon,f(\delta)=0$ and $f(\delta e^{-\frac{1}{\varepsilon}})=1$. Now shift the function $\max(f,1)$ slightly to the left, extend by 0 to the right, smoothen it and mirror it on the y-axis. This gives us the function $\chi$ we are looking for.
\end{proof}
\begin{lem}\label{perturbationofWug}
Let $a,b,c,d\in\mathbb{R}$ with $a<0<b$ and $c<d$. Let $f\in C^{\infty}([c,d],(a,b))$ be a function with the following properties:
\begin{itemize}\itemsep0pt
\item[(i)]$f(c)=0,f(d)=0$
\item[(ii)]for $1\leq k\leq\infty$ the following holds: $f^{(k)}(c)=f^{(k)}(d)=0$.
\end{itemize}
$\tilde{f}:[c,d]\to(a,b)\times[c,d],y\mapsto(f(y),y)$ is the graph of $f$ over the $y$-axis. Let $g:\mathbb{R}^{2}\to\mathbb{R},(x,y)\mapsto y$, be the projection to the second coordinate.\\
Then there exists a function $\tilde{g}:\mathbb{R}^{2}\to\mathbb{R}$ such that the following holds:
\begin{align*}
\tilde{g}&=g\text{ on }\mathbb{R}^{2}\setminus((a,b)\times(c,d))\\ 
\im(\tilde{f})&=\lbrace\varphi_{\tilde{g}}^{s}((0,d)):0\leq s\leq \tilde{t}\rbrace,
\end{align*}
where $\varphi_{\tilde{g}}^{s}$ is the flow along the vector field $-\nabla\tilde{g}$ and $\varphi^{\tilde{t}}_{\tilde{g}}((0,d))=(0,c)$.
\end{lem}
\begin{proof}
Figure \ref{Niveaug}(a) shows $\im(\tilde{f})$ and the level sets of $g$ in $R:=[a,b]\times[c,d]$. We now want to change $g$ on $R$ to $\tilde{g}$ so that the level sets of $\tilde{g}$ are perpendicular to $\im(\tilde{f})$ (see Figure \ref{Niveaug}(b)), but outside of $R$ coincide with the level sets of $g$. To do this we use a diffeomorphism 
\begin{align*}
\Phi:R&\to R\\
(x,y)&\mapsto(x,y+\Psi(x,y)),
\end{align*}
where $\Psi\in C^{\infty}(R,\mathbb{R})$ satisfies the following properties: 
\begin{itemize}\itemsep0pt
\item[(i)]$\Psi(x,y)=0$ for all $(x,y)\in\partial R$.
\item[(ii)]$\Psi(f(y),y)=0$.
\item[(iii)]$\frac{\partial\Psi}{\partial x}(f(y),y)=-f^{\prime}(y)$.
\end{itemize}
\begin{figure}[ht]\centering
\subfigure[$\im(\tilde{f})$ as a graph over the $y$-axis and the level sets of $g$]{\includegraphics[scale=0.5]{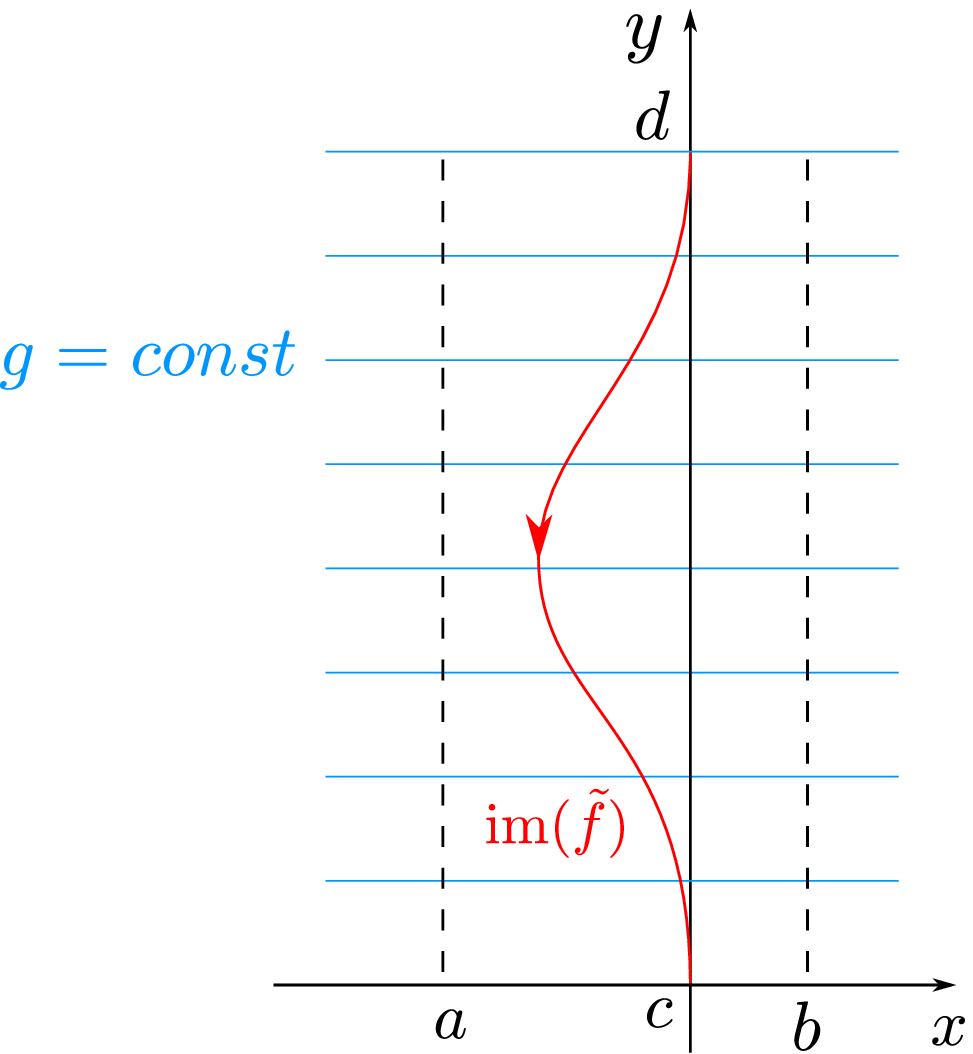}}\hspace*{1cm}
\subfigure[Changed level sets]{\includegraphics[scale=0.5]{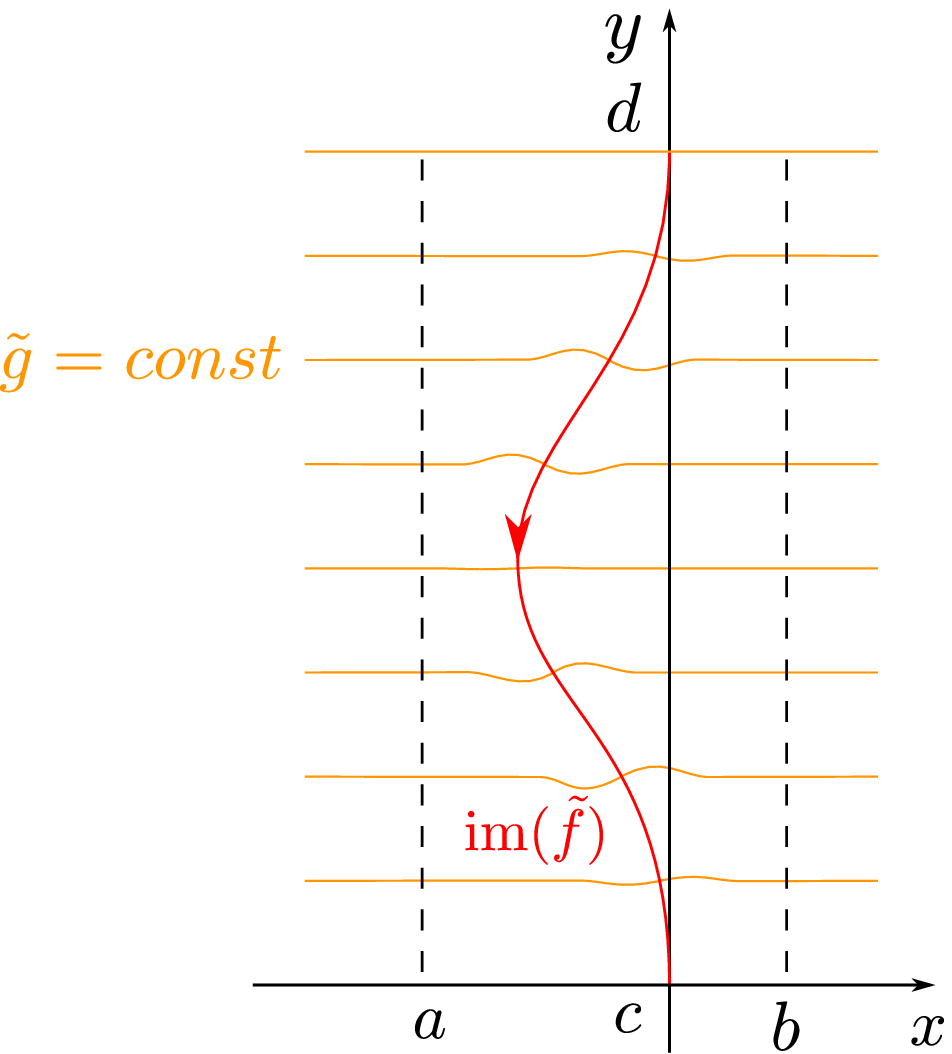}}
\caption{Adjustment of the level sets}\label{Niveaug}
\end{figure}%
With this function $\Phi$ we set
\[\tilde{g}:=g\circ\Phi^{-1}.\]
Thus, the level sets of $\tilde{g}$ are as required. The flow of the negative gradient runs perpendicular to the level sets and is unique. So we have
\[\im(\tilde{f})=\lbrace\varphi_{\tilde{g}}^{s}((0,d)):0\leq s\leq \tilde{t}\rbrace\]
with $\tilde{t}>0$ where $\varphi^{\tilde{t}}_{\tilde{g}}((0,d))=(0,c)$.\\[.5em]
So it remains to show that such a function $\Psi$ exists so that $\Phi$ is a diffeomorphism: We choose
\begin{align*}
\Psi:R&\to\mathbb{R}\\
(x,y)&\mapsto-f^{\prime}(y)(x-f(y))\chi(x-f(y)),
\end{align*}
where $\chi:\mathbb{R}\to\mathbb{R}$ is a smooth cutter function with the following properties:
\begin{itemize}\itemsep0pt
\item[(i)]$\chi(0)=1$.
\item[(ii)]$\supp\chi\subset[-\delta,\delta]$, where $\delta>0$ must be chosen so that $a+\delta<f(y)<b-\delta$ for all $y\in[c,d]$.
\item[(iii)]$\vert z\chi^{\prime}(z)\vert<\varepsilon$ for all $z\in\mathbb{R}$ and an $\varepsilon>0$, which will be determined more precisely later.
\end{itemize}
According to Lemma \ref{cutterfunction}, such a function $\chi$ exists for any $\delta,\varepsilon>0$. All properties required for $\Psi$ are satisfied. The following also holds:
\begin{align*}
\frac{\partial\Psi}{\partial y}(x,y)&=-f^{\prime\prime}(y)(x-f(y))\chi(x-f(y))+\underbrace{(f^{\prime}(y))^{2}\chi(x-f(y))}_{\geq0}+(f^{\prime}(y))^{2}(x-f(y))\chi^{\prime}(x-f(y))\\
&\geq-\vert f^{\prime\prime}(y)(x-f(y))\chi(x-f(y))\vert-\vert(f^{\prime}(y))^{2}(x-f(y))\chi^{\prime}(x-f(y))\vert.
\end{align*}
The first term in the last line does not vanish only if $\vert x-f(y)\vert<\delta$. After possibly decreasing $\delta$, $\vert f^{\prime\prime}(y)(x-f(y))\chi(x-f(y))\vert<\frac{1}{2}$ can be achieved. Also $\vert(x-f(y))\chi^{\prime}(x-f(y))\vert<\varepsilon$. So by choosing $\varepsilon$ small enough, we can achieve $\vert(f^{\prime}(y))^{2}(x-f(y))\chi^{\prime}(x-f(y))\vert<\frac{1}{2}$. Altogether, $\frac{\partial\Psi}{\partial y}(x,y)>-1$ can be guaranteed for all $(x,y)\in R$. \\
Now we can show that the above defined function $\Phi$ is a diffeomorphism: First, we consider 
\[D_{(x,y)}\Phi=
\begin{pmatrix}
1&0\\
\frac{\partial\Psi}{\partial x}(x,y)&1+\frac{\partial\Psi}{\partial y}(x,y)
\end{pmatrix}.\]
Since $1+\frac{\partial\Psi}{\partial y}(x,y)>0$ for all $(x,y)\in R$, $D_{(x,y)}\Phi$ is an isomorphism for all $(x,y)\in R$.\\
From $1+\frac{\partial\Psi}{\partial y}(x,y)>0$ also follows that the map $y\mapsto y+\Psi(x,y)$ is strictly increasing. Thus, $\Phi$ is injective.\\
Since $\Phi$ is continuous and $\Phi\vert_{\partial R}=id$, the surjectivity of $\Phi$ follows.\\
Thus, $\Phi$ is a diffeomorphism and this proves the lemma.
\end{proof}
Now the following generic properties of the function $E_{g}$ can be shown:
\begin{lem}\label{cordlemmaadd2}
For a generic function $g$ the following holds for the space $K\times K$ of cords:
\begin{itemize}\itemsep0pt
\item[(i)]For all critical points $k$ of the function $E_{g}$ of index 1 the following holds: $W_{g}^{u}(k)\pitchfork B$, $W_{g}^{u}(k)\pitchfork S$, and $W_{g}^{u}(k)\pitchfork F$.
\item[(ii)]For all critical points $k$ of the function $E_{g}$ of index 1 the following holds: $W_{g}^{u}(k)\cap\partial S=\emptyset$, $W_{g}^{u}(k)\cap\partial F=\emptyset$, and $W_{g}^{u}(k)\cap S_{2}=\emptyset$ where $S_{2}\subset S$ is the set of self-intersections of $S$.
\item[(iii)]For all critical points $k$ of the function $E_{g}$ of index 1 the following holds: $W_{g}^{u}(k)\cap B\cap S=\emptyset$ and $W_{g}^{u}(k)\cap F\cap S=\emptyset$.
\item[(iv)]There is no trajectory along the vector field $-\nabla E_{g}$ between two critical points of index 1.
\end{itemize}
\end{lem}
\begin{rem}
The set of self-intersections of $F$ does not have to be considered, because at these finitely many points a cord is multiplied by $\mu$ or $\mu^{-1}$ from the left and from the right. This is not a problem. The set $S_{2}$ must be avoided because at these points a cord would be divided into three parts. However, this situation is not covered by any relation. 
\end{rem}
\begin{proof}
According to Lemma \ref{cordlemmaadd1}, we have $Crit_{0,1}\cap(B\cup S\cup F)=\emptyset$. Thus, there exist open neighborhoods~$U_{k}$ of all $k\in Crit_{0,1}$ such that $U_{k}\cap(B\cup S\cup F)=\emptyset$. Consequently, for all $k\in Crit_{1}$ the following holds:
\[(W^{u}_{g}(k)\cap\widetilde{U})\pitchfork(B\cup S\cup F),\]
where $\widetilde{U}:=\bigcup\limits_{k^{\prime}\in Crit_{0,1}}\hspace*{-8pt}U_{k^{\prime}}$. We shrink $\widetilde{U}$ to $U$ such that $\bar{U}\subset\widetilde{U}$ holds for the closure $\bar{U}$. Then in the following it is sufficient to perturb $W^{u}_{g}(k)$ to $W^{u}_{\tilde{g}}(k)$ for all $k\in Crit_{1}$ on the subset $W^{u}_{g}(k)\setminus U$ with the help of the relative version of the jet transversality theorem (Theorem \ref{JetTransvthmrel}) so that the following holds:
\[W^{u}_{\tilde{g}}(k)\cap(\widetilde{U}\setminus U)=W^{u}_{g}(k)\cap(\widetilde{U}\setminus U)\]
and
\[(W^{u}_{\tilde{g}}(k)\setminus U)\pitchfork(B\cup S\cup F).\]
(i), (ii) Let $k$ be a critical point of index 1 and $f\in C^{\infty}(\mathbb{R},T^{2})$ such that $\im(f)=W_{g}^{u}(k)$. Let $I\subset\mathbb{R}$ be the compact interval for which the following holds (see Figure \ref{nbhdcritpts}):
\[\im(f\vert_{I})=W^{u}_{g}(k)\setminus U.\]
\begin{figure}[ht]\centering
\includegraphics[scale=0.35]{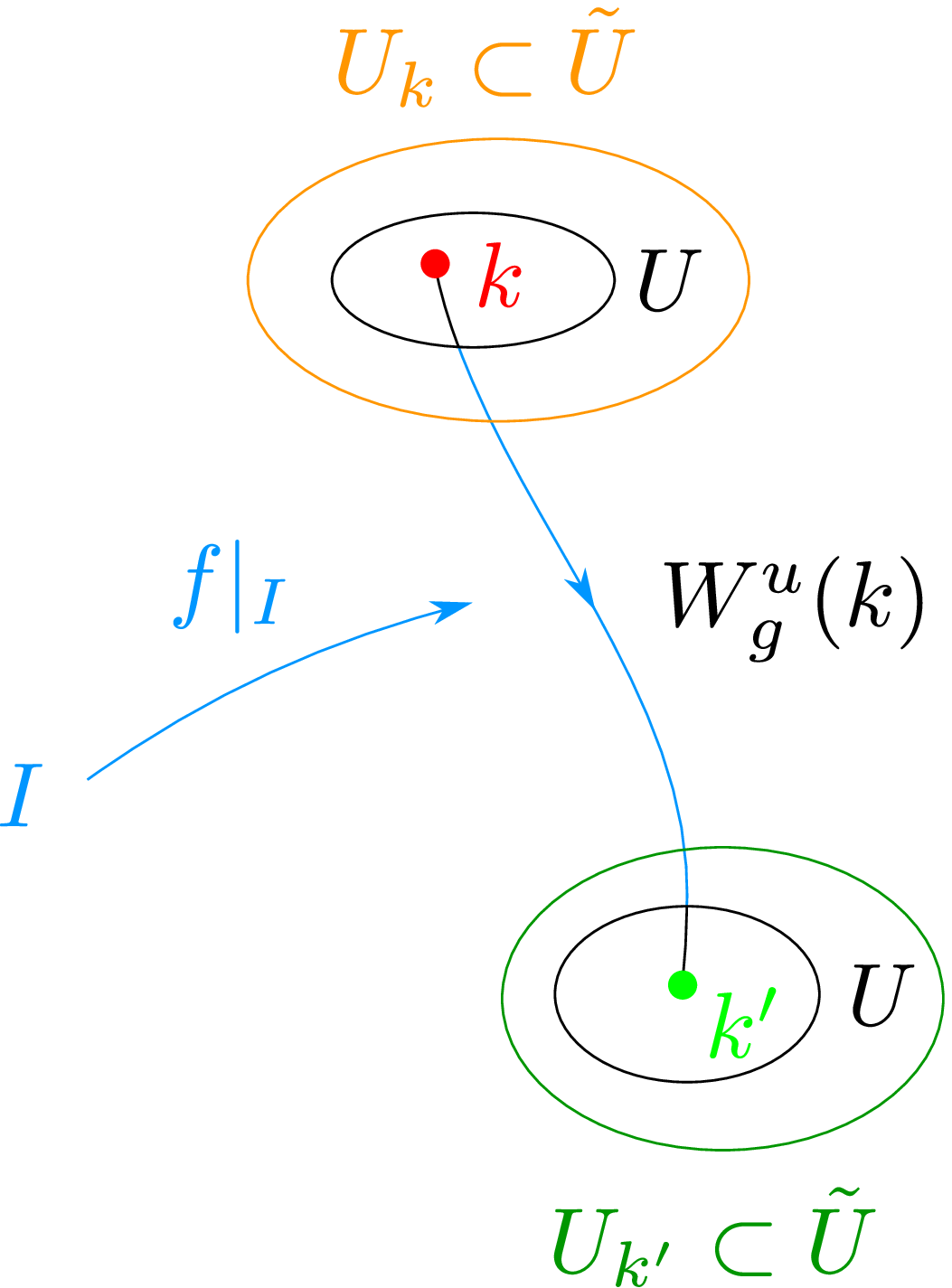}
\caption{Neighborhoods of critical points $k$ and $k^{\prime}$ of index 1 and 0, respectively, and the unstable manifold of $k$}\label{nbhdcritpts}
\end{figure}%
Let $W^{u}_{g}(k)\cap S\neq\emptyset$ and $p\in W^{u}_{g}(k)\cap S$.\\
\textbf{Case 1}: $p\in\partial S$. We construct a map $\tilde{f}:\mathbb{R}\to T^{2}$ as follows:
\begin{itemize}
\item $\tilde{f}\vert_{\mathbb{R}\setminus I}=f\vert_{\mathbb{R}\setminus I}$
\item $\partial S$ is a submanifold of $T^{2}$. Using the relative version of the transversality theorem (Theorem \ref{Transvthmrel}) we can perturb $f\vert_{I}$ in the space $C^{n}(I,T^{2})$, for $n$ big enough, to a map $\tilde{f}\vert_{I}$ such that $\tilde{f}\vert_{I}\pitchfork\partial S$ and $\im(\tilde{f}\vert_{I})\cap(\widetilde{U}\setminus U)=\im(f\vert_{I})\cap(\widetilde{U}\setminus U)$.
\end{itemize}
Thus, we get $\tilde{f}\pitchfork\partial S$. It follows that $\tilde{f}^{-1}(\partial S)=\emptyset$ since $\codim(\partial S\subset T^{2})=2$. In suitable local coordinates $\im(f)$ can be represented as a subset of the $y$-axis in $\mathbb{R}^{2}$. Since $\tilde{f}$ can be chosen arbitrarily close to $f$, we can guarantee that in these coordinates $\tilde{f}$ is a graph over the $y$-axis. According to Lemma \ref{perturbationofWug}, there exists a function $\tilde{g}$ for which $W^{u}_{\tilde{g}}(k)\cap U=W^{u}_{g}(k)\cap U$ and $W^{u}_{\tilde{g}}(k)\setminus U=\im(\tilde{f}\vert_{I})$. Thus, we get $W^{u}_{\tilde{g}}(k)\cap\partial S=\emptyset$.\\
\textbf{Case 2}: Since the set $S_{2}\subset S$ of self-intersections of $S$ is finite according to Lemma \ref{cordlemma}(ii), it follows analogously to case 1 that $W^{u}_{\tilde{g}}(k)\cap S_{2}=\emptyset$.\\
\textbf{Case 3}: $p\in\widehat{S}:=S\setminus(\partial S\cup S_{2})$.\\
The set of all 1-jets from $\mathbb{R}$ to $T^{2}$ is
\[J^{1}(\mathbb{R},T^{2})=\mathbb{R}\times T^{2}\times\mathbb{R}^{2}.\]
Consider the map
\begin{align*}
h:\mathbb{R}&\to J^{1}(\mathbb{R},T^{2})\\
t&\mapsto(t,f(t),\dot{f}(t)).
\end{align*}
We construct a map $\tilde{f}:\mathbb{R}\to T^{2}$ as follows:
\begin{itemize}
\item $\tilde{f}\vert_{\mathbb{R}\setminus I}=f\vert_{\mathbb{R}\setminus I}$
\item The tangent bundle $T\widehat{S}$ is a submanifold of $T(T^{2})=T^{2}\times\mathbb{R}^{2}$. So $\mathbb{R}\times T\widehat{S}$ is a submanifold of $\mathbb{R}\times T(T^{2})=\mathbb{R}\times T^{2}\times\mathbb{R}^{2}$. Using the relative version of the jet transversality theorem (Theorem~\ref{JetTransvthmrel}) we can perturb $f\vert _{I}$ in the space $C^{n}(I,T^{2})$, for $n$ big enough, to a map~$\tilde{f}\vert_{I}$ such that $\im(\tilde{f}\vert_{I})\cap(\widetilde{U}\setminus U)=\im(f\vert_{I})\cap(\widetilde{U}\setminus U)$ and $\tilde{h}\vert_{I}\pitchfork T\widehat{S}$ with the map $\tilde{h}\vert_{I}, t\mapsto(t,\tilde{f}(t),\dot{\tilde{f}}(t))$ which is also perturbed.
\end{itemize}
Thus, we get $\tilde{h}\pitchfork\mathbb{R}\times T\widehat{S}$. It follows that $\tilde{h}^{-1}(\mathbb{R}\times T\widehat{S})=\emptyset$ since $\codim(\mathbb{R}\times T\widehat{S}\subset\mathbb{R}\times T(T^{2}))=2$. So $\im(\tilde{f})\pitchfork\widehat{S}$ since for all $t\in\mathbb{R}$ with $\tilde{f}(t)\in\widehat{S}$ we have $\dot{\tilde{f}}(t)\notin T_{\tilde{f}(t)}\widehat{S}$. In suitable local coordinates $\im(f)$ can be represented as a subset of the $y$-axis in $\mathbb{R}^{2}$. Since $\tilde{f}$ can be chosen arbitrarily close to $f$, we can guarantee that in these coordinates $\tilde{f}$ is a graph over the $y$-axis. According to Lemma~\ref{perturbationofWug}, there exists a function~$\tilde{g}$ for which $W^{u}_{\tilde{g}}(k)\setminus U=\im(\tilde{f}\vert_{I})$ and $W^{u}_{\tilde{g}}(k)\cap U=W^{u}_{g}(k)\cap U$.\\[.5em]
Altogether, we get $W_{\tilde{g}}^{u}(k)\pitchfork S$.\\
Since $Crit_{1}$ is a finite set, finitely many (arbitrarily small) perturbations suffice to guarantee transversality to $S$ for all unstable manifolds of critical points of index 1.\\[.5em]
The proof of the statements $W_{g}^{u}(k)\pitchfork B, W_{g}^{u}(k)\pitchfork F$ and $W_{g}^{u}(k)\cap\partial F=\emptyset$ is analogous to that of the statements $W_{g}^{u}(k)\pitchfork S$ and $W_{g}^{u}(k)\cap\partial S=\emptyset$.\\[.5em]
(iii) These two statements can be shown by similar transversality arguments as before.\\[.5em]
(iv) Let $k$ be a critical point of index 1 and $f\in C^{\infty}(\mathbb{R},T^{2})$ with $\im(f)=W_{g}^{u}(k)$.\\
According to Lemma \ref{cordlemma}(i), $Crit_{1}$ is a submanifold of $T^{2}$ with $\codim(Crit_{1}\subset T^{2})=2$. Therefore, the statement follows analogously to (ii).
\end{proof}
Let $R:=\mathbb{Z}[\lambda^{\pm1},\mu^{\pm1}]$ be the commutative ring over $\mathbb{Z}$, generated by $\lambda,\lambda^{-1},\mu$ and $\mu^{-1}$ modulo the relations $\lambda\cdot\lambda^{-1}=\mu\cdot\mu^{-1}=1$. Let $C_{1}$ be the $\mathbb{Z}$-vector space generated by $Crit_{1}(E)$ and $C_{0}$ the non-commutative $R$-algebra generated by $Crit_{0}(E)$, i.e. $\lambda$ and $\mu$ commute with each other, but not with any elements of $Crit_{0}(E)$. These in turn do not commute at all. In the following we will let flow cords along the negative gradient $-\nabla E$. It can happen that a cord converges to a point, intersects the base point or the framing or intersects the knot in its interior. For these cases we define the relations in Figure \ref{relations}.
\begin{figure}[ht]
\hspace*{3cm}\subfigure{\includegraphics[scale=0.35]{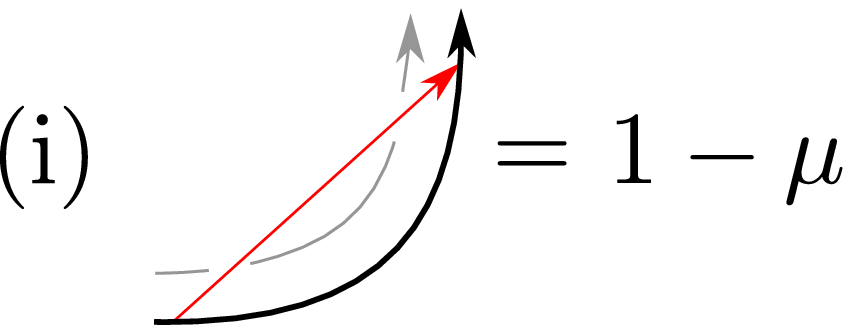}}\\
\hspace*{3cm}\subfigure{\includegraphics[scale=0.35]{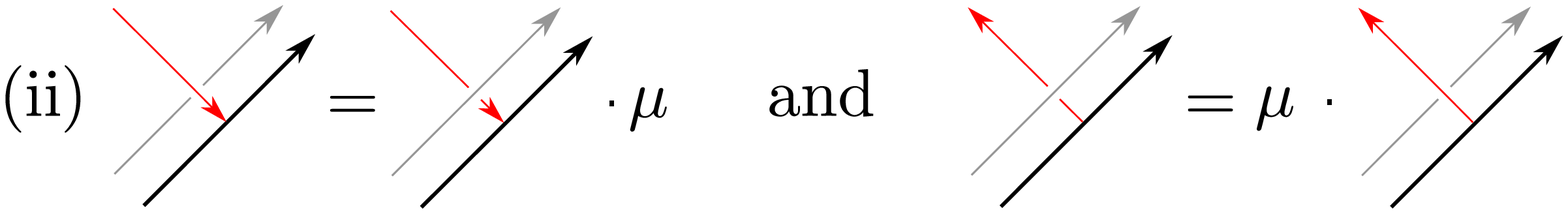}}\\
\hspace*{3cm}\subfigure{\includegraphics[scale=0.35]{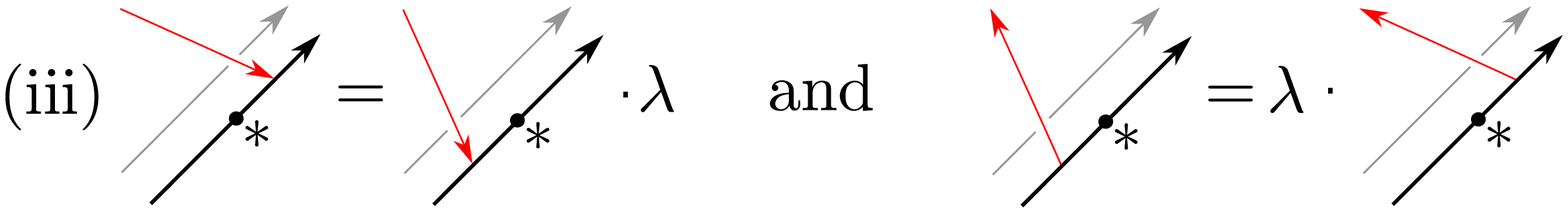}}\\
\hspace*{3cm}\subfigure{\includegraphics[scale=0.35]{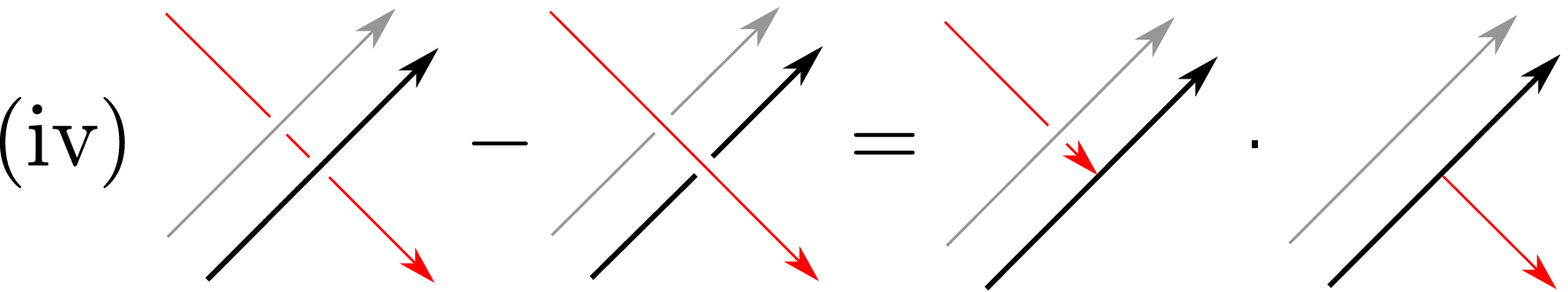}}
\caption{Relations for cords moving along the vector field $-\nabla E$ }\label{relations}
\end{figure}%
There the knot $K$ is drawn in black, the framing $K^{\prime}$ in grey and the cords in red. Analogous to Remark \ref{remng}, only the relevant parts of the knot are drawn in the pictures. The diagrams each show a small region of the knot. The first line refers to any contractible cord. In the second, third and fourth relation, the diagrams agree outside the drawn region. In the fourth line, the first two diagrams show a cord that runs once over and once under a strand of the knot. At the transition from the first to the second diagram, or vice versa, with the help of a homotopy, the cord intersects the knot. At this point, the cord is split into two parts, which is shown in the other two diagrams. All diagrams are to be understood as objects in three-dimensional space. Therefore, relation (iv) is for example equivalent to the relation shown in Figure \ref{relationivequiv}.
\begin{figure}[ht]\centering
\includegraphics[scale=0.35]{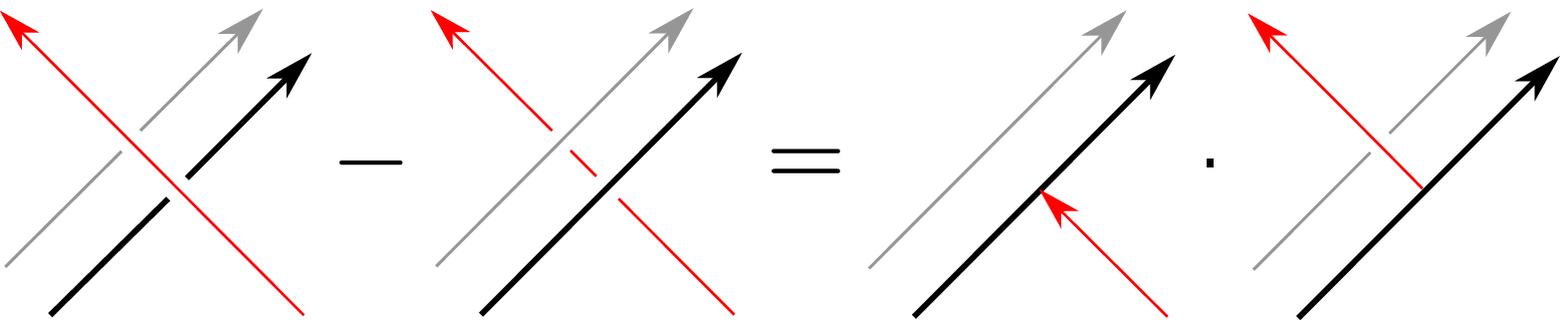}
\caption{Equivalent representation of relation (iv)}\label{relationivequiv}
\end{figure}%
\begin{lem}\label{finitelymanyintersections}
For a generic function $g$ the following holds for the perturbed function $E_{g}$:\\
Let $c\in W_{g}:=\hspace*{-.1cm}\bigcup\limits_{k\in Crit_{1}}\hspace*{-.2cm}W^{u}_{g}(k)\subset K\times K$. If $c$ is split into two parts according to relation~(iv) during the movement along the negative gradient $-\nabla E_{g}$, both parts and the original cord flow further along $-\nabla E_{g}$ and can be split again according to relation (iv). The following holds: These splits can occur only finitely many times, i.e. relation (iv) is applied only finitely many times. 
\end{lem}
\begin{proof}
1) Lemma \ref{cordlemma}(iii) implies that there exists a neighborhood $U\subset K\times K$ of the finitely many cords $\partial S$ that are tangent to $K$ at one endpoint and an $\varepsilon>0$ with the following property: Each cord in $U\cap S$ decreases in length by more than $\varepsilon$ under the flow of $-\nabla E_{g}$ before it meets $S$ again, and the same holds for the longer cord resulting from the splitting according to relation (iv). On the other hand, if a string $s\in S\setminus U$ is split at its intersection with the knot, both pieces are shorter than $s$ by at least a fixed length $\delta>0$ since $s$ can be neither tangent nor ``almost'' tangent to $K$. In total, each piece will be at least $\min(\varepsilon,\delta)$ shorter than the original cord. Since its length is finite, however, only finitely many intersections can happen. (cf. \cite{Cie3}, proof of Proposition 7.14)\\[.5em]
2) According to Lemma \ref{cordlemmaadd2}(i), we have $W_{g}\pitchfork S$ for a generic function $g$. It follows that there exist only finitely many intersections of $W_{g}$ with $S$ since $W_{g}$ is compact. The cords arising from splitting at these intersections according to relation~(iv) are, according to 1), at least $\min(\varepsilon,\delta)$ shorter than the original cords. Let $k\in Crit_{1}$. At the first intersection of $W^{u}_{g}(k)$ with $S$, starting from $k$ in one of the two possible directions, the cord is split into two cords $k_{1}$ and $k_{2}$. If $k_{1}$ lies now on $W^{u}_{g}(k)$, nothing more has to be done, because $W^{u}_{g}(k)$ is already transverse to $S$. Otherwise the trajectory of $k_{1}$ along the negative gradient $-\nabla E_{g}$ can be perturbed outside of $W^{u}_{g}(k)$ according to Lemma~\ref{perturbationofWug} such that this trajectory becomes transverse to $S$. This is possible because $W^{u}_{g}(k)$ and the trajectory of $k_{1}$ are disjoint outside the critical points.\\
The same procedure is used for $k_{2}$: If $k_{2}$ is neither on $W^{u}_{g}(k)$ nor on the trajectory starting at $k_{1}$ that may have been perturbed, we perturb the trajectory starting at $k_{2}$ with the help of Lemma~\ref{perturbationofWug} outside of $W^{u}_{g}(k)$ and the trajectory starting at $k_{1}$ so that it becomes transverse to~$S$.\\
All other intersections of $W^{u}_{g}(k)$ with $S$ are handled in the same way. Likewise the unstable manifolds of all other cords from $Crit_{1}$.\\
Since $W_{g}$ has only finitely many intersections with $S$, the use of relation~(iv) also results in a finite number of cords. The $\varepsilon$ from 1) may change due to the necessary perturbations of $g$. But since the perturbed function is arbitrarily close to $g$, the new $\varepsilon$ is arbitrarily close to the original one. Therefore, it can be achieved that all resulting cords are still at least $\min(\varepsilon,\delta)$ shorter than the original cords.\\
The newly created cords are now handled in the same way as the unstable manifolds of the critical points, and the function $g$ may be perturbed accordingly. Since the trajectories of these cords are transverse to $S$, they intersect $S$ only finitely many times. This process is continued as long as further intersections of a trajectory with $S$ occur and relation (iv) is applied. Since each cord has finite length and becomes shorter at a splitting of at least $\min(\varepsilon,\delta)$, this process ends after finitely many steps and the function $g$ has to be adapted only to finitely many trajectories, which are pairwise disjoint outside the critical points.\\[.5em]
All in all: Relation (iv) is used only finitely many times.
\end{proof}
Now we can show further properties of a generic function $E_{g}$: 
\begin{lem}\label{cordlemmaadd3}
For a generic function $g$ the following holds for the function $E_{g}$ and the space $K\times K$ of cords:\\
Let $k\in Crit_{1}$ with $W_{g}^{u}(k)\cap S\neq\emptyset$. If $k$ is moved along its unstable manifold and split according to relation (iv), denote by $c$ one of the resulting cords or a cord that results from further applying relation (iv) to an already split cord. Then the following statements hold for all such cords $c$:
\begin{itemize}\itemsep0pt
\item[(i)]$\varphi_{g}^{s}(c)\pitchfork B,\varphi_{g}^{s}(c)\pitchfork S$, and $\varphi_{g}^{s}(c)\pitchfork F$ for $s\geq0$.
\item[(ii)]The following holds for $s\geq0$: $\varphi_{g}^{s}(c)\cap\partial S=\emptyset, \varphi_{g}^{s}(c)\cap\partial F=\emptyset$, and $\varphi_{g}^{s}(c)\cap S_{2}=\emptyset$ where $S_{2}$ is the set of self-intersections of $S$.
\item[(iii)]The following holds for $s\geq0$: $\varphi_{g}^{s}(c)\cap B\cap S=\emptyset$ and $\varphi_{g}^{s}(c)\cap F\cap S=\emptyset$.
\item[(iv)]$c\notin(B\cup F\cup S)$.
\item[(v)]$\lim\limits_{s\to\infty}\varphi_{g}^{s}(c)\notin Crit_{1}$, i.e. no such cord runs along $-\nabla E_{g}$ to a critical point of index 1.
\item[(vi)]$c\notin Crit_{k}$ for $k=0,1,2$.
\end{itemize} 
\end{lem}
\begin{proof}
The statements (i) to (v) follow analogously to the proof of Lemma \ref{cordlemmaadd1} and Lemma~\ref{cordlemmaadd2}, since $Crit_{1}$ is a finite set and relation (iv) is applied only finitely many times according to Lemma~\ref{finitelymanyintersections}.\\
(vi) According to Lemma \ref{finitelymanyintersections}, the set of all such cords $c$ is a finite set, likewise $Crit_{k}$ is finite for $k=0,1,2$. Therefore, the statement can be reached by a small perturbation $g$, analogous to the proof of Lemma \ref{cordlemmaadd2}(ii). 
\end{proof}
\subsection{Representation of knots as closed braids}
Every knot can be represented as a closed braid \cite{Liv}. To define the cord algebra of a knot, we draw the knot as a closed braid along an ellipse in the following way (cf. \cite{Pet}):\\
The braid should be positioned in such a way that the binormal cords do not intersect the knot in their interior. This can be achieved by arranging the strands of the braid one above the other, i.e. from the drawing plane, with each strand being placed above a slightly larger ellipse than the one below. In addition, the strands should have a very small distance from each other. This ensures that no unwanted binormal cords are created that run from the area of the crossings of the knot to ``opposite'' strands. In order to be able to distinguish the strands in the diagram well, however, they are drawn with larger distance. They are also numbered from the outside to the inside.\\
By an arbitrarily small perturbation it can be achieved that the torsion vanishes only at finitely many points.\\
All crossings of the knot are drawn in a quarter of the ellipse so that they lie between two endpoints of the main axes. In the remaining three quarters the strands run parallel to each other. In Figure~\ref{index012} this is shown by the example of the right-handed trefoil knot, equipped with the Seifert framing. The knot is drawn in black and the framing in grey.\\
\begin{figure}[htbp]\centering 
\includegraphics[scale=0.6]{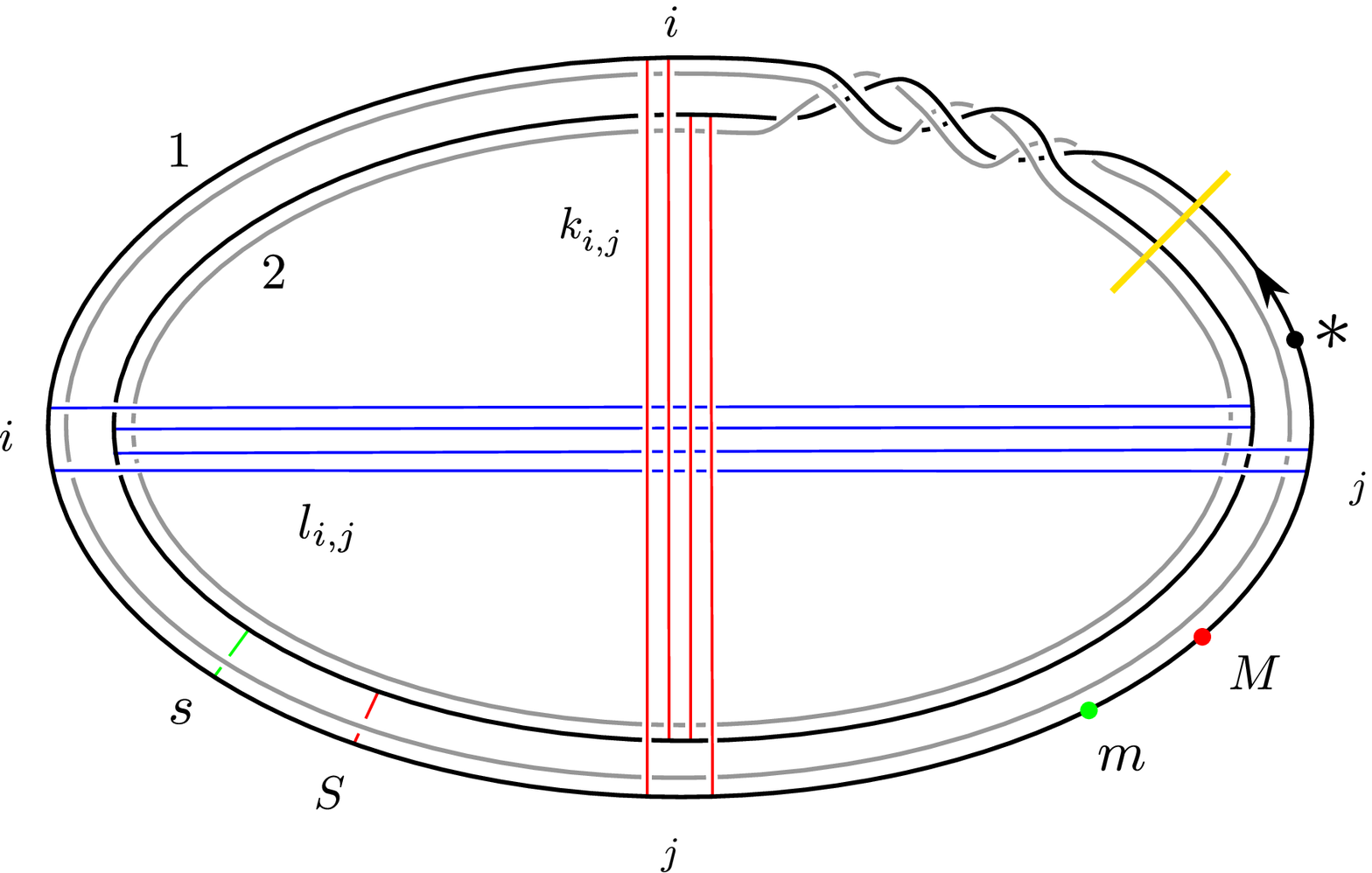}
\caption{Critical points of the function $E$ on the right-handed trefoil}\label{index012}
\end{figure}%
\begin{figure}[htbp]\centering 
\includegraphics[scale=0.3]{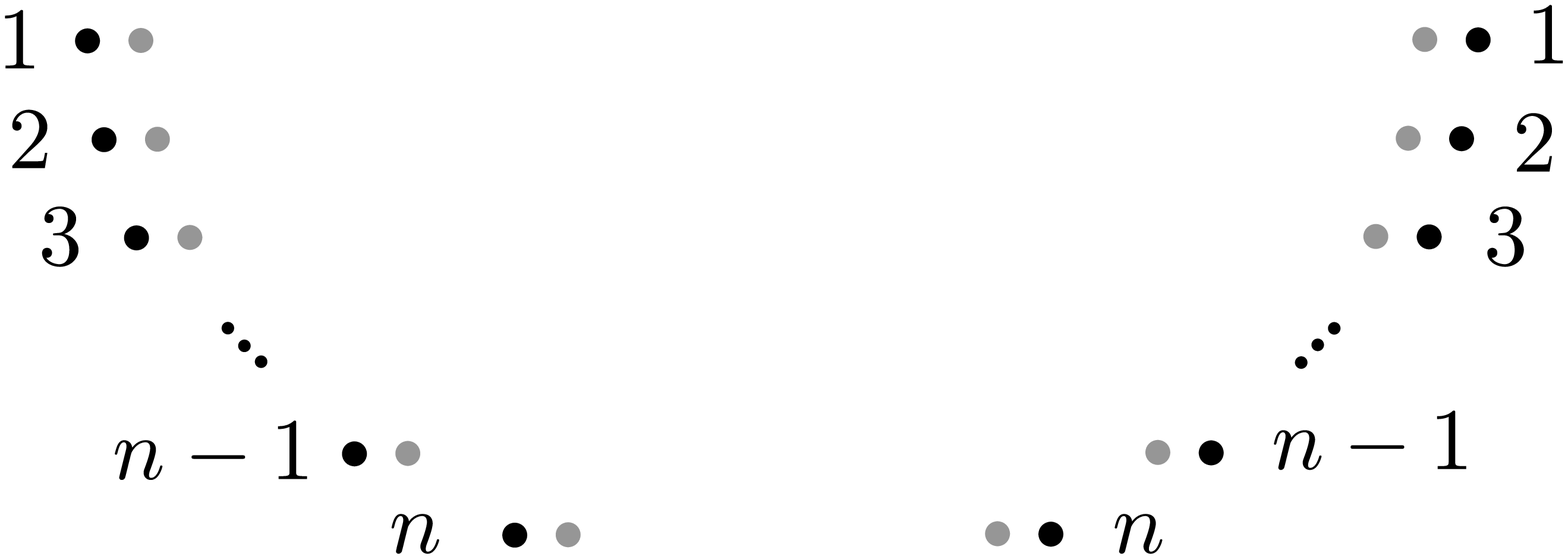}
\caption{Arrangement of $n$ strands in the cross section}\label{morethantwostrands}
\end{figure}%
The non-trivial binormal cords are, as is easy to see, the cords that run parallel to the main axes of the ellipse, as well as the very short cords that connect the strands together. However, the way of drawing the knot produces a one-dimensional critical submanifold for the latter. By a small perturbation of the knot this can be cleared up in such a way that exactly two critical points arise, one of the index 0, marked with $s$, and one of the index 1, marked with $S$. In addition, if there are more than two strands, there may be other short binormal cords between the strands in the area of the crossings. In the case of more than two strands, they are placed in such a way that the arrangement as shown in Figure \ref{morethantwostrands} is achieved. This guarantees that the very short cords connecting the strands do not intersect the knot in their interior. To ensure that no binormal cord intersects the framing, a small perturbation of the framing may be necessary. Another small perturbation of the framing may be necessary to satisfy the conditions listed in Remark \ref{framingassumption} (these conditions are necessary for the proof of Lemma \ref{framinglemma}).\\
In Figure \ref{index012} the critical points, i.e. the binormal cords, are distinguished by color: 
\begin{itemize}
\item Critical points of index 0 are marked green. In this example $m$ and $s$.
\item Critical points of index 1 are marked red. Here $M, S$ and the cords labeled with $k_{i,j}$. The indexing means that the cord runs from strand $i$ to strand $j$. In order to determine on which strand an endpoint of the cord lies, one moves, starting from this endpoint, along the knot until one reaches the numbering of the strands. The yellow mark must not be exceeded. The numbering is therefore unique.
\item Critical points of index 2 are marked blue, here $l_{i,j}$. These are only mentioned for the sake of completeness and will not be required further.
\end{itemize}
Since the knot $K$ is parametrized by $\gamma$, we get a canonical orientation of the knot, which is marked by an arrow. The position of the base point is chosen as shown in Figure \ref{index012}. This choice ensures that no binormal cord has the base point as its startpoint or endpoint.\\
Each cord $c=(s,t)\in K\times K$ can also be assigned an orientation as described above: $c$ is oriented from $s$ to $t$. If we consider the orientation, all cords that are not on the diagonal of $K\times K$ occur in pairs: Once with orientation in one direction and once with reverse orientation. In order to indicate the orientation of a cord in its labeling, another index is introduced, which is noted in the upper right of the name. Denote by $c^{s}=(s^{\prime},t)$ the cord $c$ for which $s^{\prime}<t$ holds, and by $c^{t}=(s,t^{\prime})$ the cord $c$ with $t^{\prime}<s$. For example, in Figure \ref{index012} $k_{1,1}^{s}$ is the longest cord of index 1 with orientation from top to bottom. This indexing is not necessary for points on the diagonal of $K\times K$, because then we have $s=t$, the corresponding cords have vanishing length and occur only once.\\
\begin{figure}[ht]\centering 
\subfigure[Labeling of the critical points]{\includegraphics[scale=0.4]{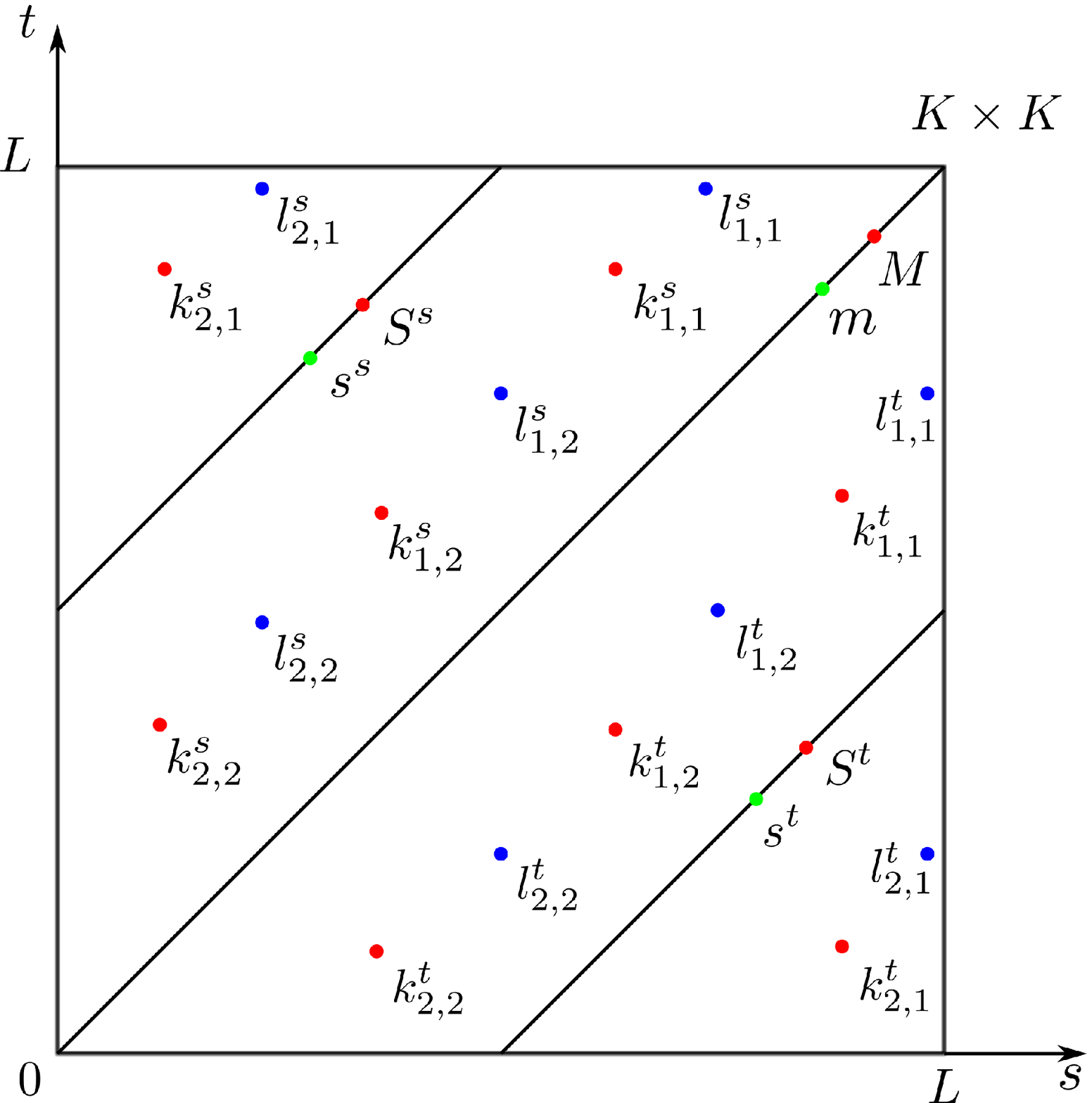}}
\subfigure[Gradient flow $-\nabla E$]{\includegraphics[scale=0.4]{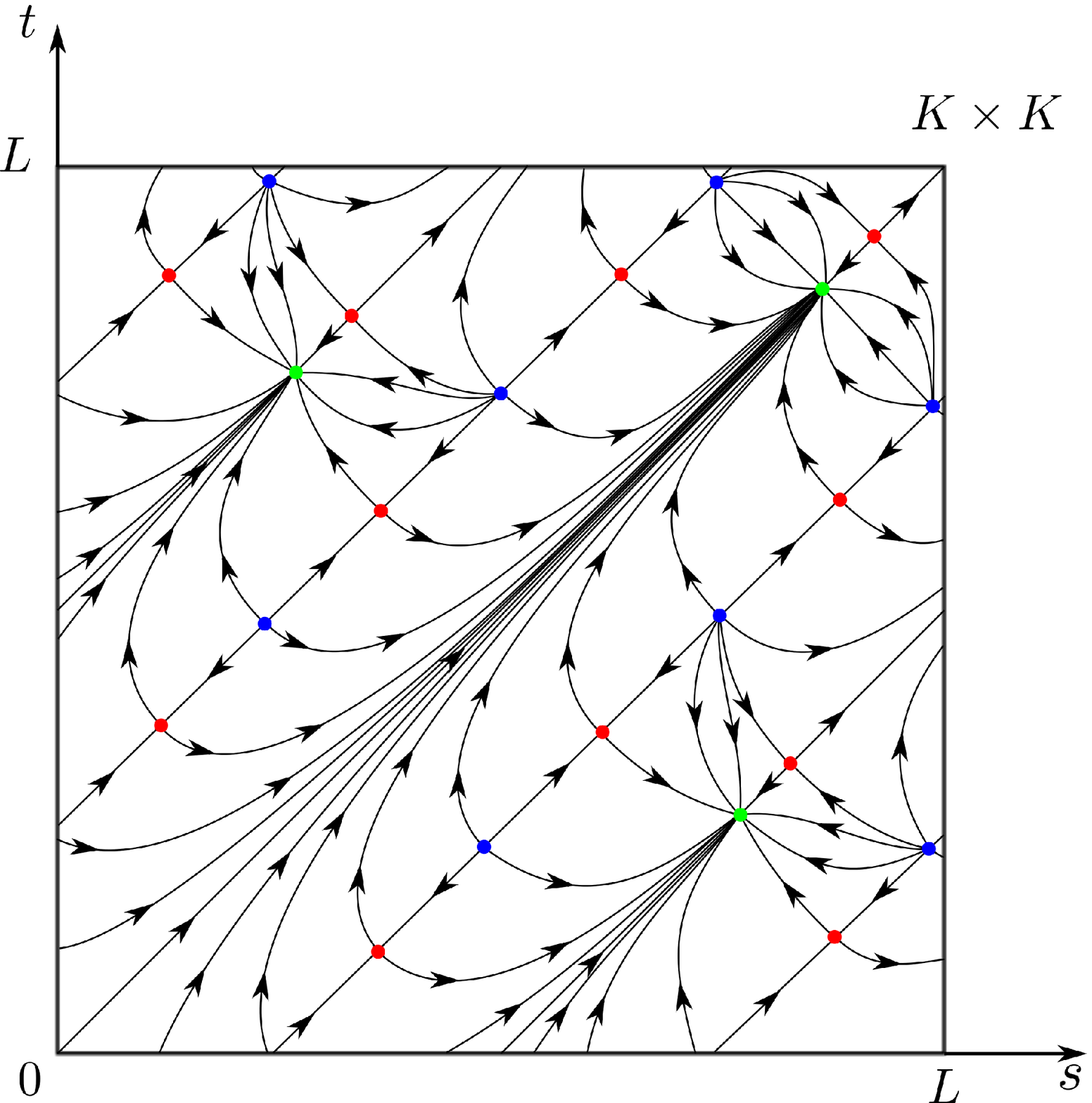}}
\caption{Critical points in $K\times K$ for the knot $K$ in Figure \ref{index012}}\label{KKcrit}
\end{figure}%
Figure \ref{KKcrit} shows the same critical points as Figure \ref{index012}, but here their positions are shown on $K\times K\cong T^{2}$. On the left side the critical points are labeled with their names. The arrows on the right indicate the direction of the gradient flow of the function $E$.\\
The figure shows that the position of $M$ on the diagonal can be chosen almost arbitrarily: $M$ must not be chosen in such a way that the gradient flow of a critical point of index 1 runs in the direction of $M$, as shown in Figure \ref{Mwrong}. 
\begin{figure}[ht]\centering
\includegraphics[scale=0.55]{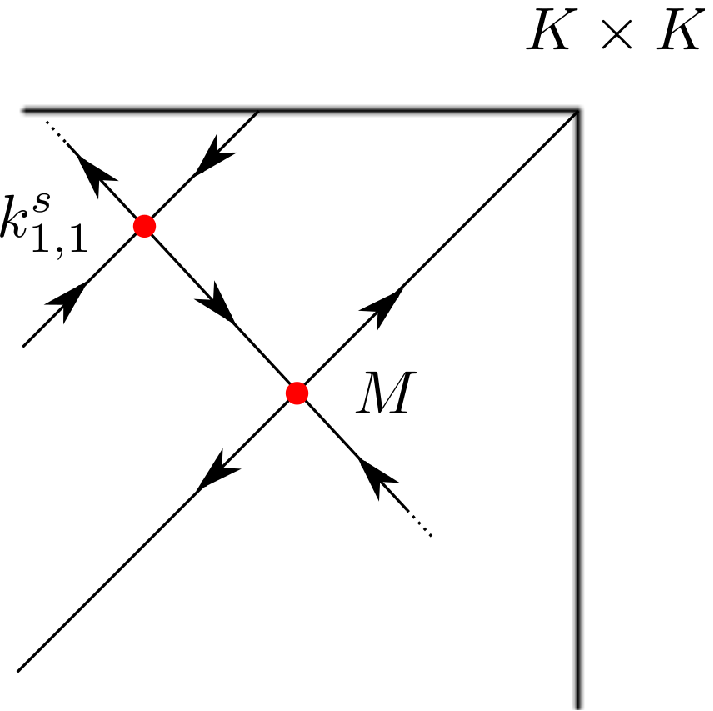}
\caption{Wrong position of $M$}\label{Mwrong}
\end{figure}
If $M$ is placed as shown in Figure \ref{Mwrong}, the Smale condition is not satisfied (see Corollary \ref{smale}).\\
$S$ is positioned analogously.\\
For all other critical points the Smale condition is satisfied because the unstable manifolds of points of index 2 and the stable manifolds of points of index 0 are already two-dimensional. So the tangent spaces at these manifolds are two-dimensional and the transversality is ensured.\\
Thus, $-\nabla E$ represents a gradient field that satisfies the Smale-condition.
\subsection{Definition of the cord algebra}\label{Cord}
Before defining the cord algebra of a knot $K$, we will first define a map $\widehat{D}:(K\times K)\setminus A\to C_{0}$, where the exceptional set $A\subset K\times K$ contains all points $c\in K\times K$ that satisfy at least one of the following properties: 
\begin{itemize}
\item[(1)] $c\in Crit_{2}(E)\cup\left(\bigcup\limits_{k\in Crit_{1}(E)}\hspace{-.35cm}W^{s}(k)\right)\cup Crit_{0}(E)$,
\item[(2)] $c\in(S\cup F\cup B)$,
\item[(3)] $\varphi^{s}(c)\ntransv B, \varphi^{s}(c)\ntransv F$, or $\varphi^{s}(c)\ntransv S$ for $s\geq0$,
\item[(4)] $\varphi^{s}(c)\cap\partial S\neq\emptyset, \varphi^{s}(c)\cap\partial F\neq\emptyset$, or $\varphi^{s}(c)\cap S_{2}\neq\emptyset$ for $s\geq0$,
\item[(5)] $\varphi^{s}(c)\in S$ for an $s>0$ and one of the properties (1) to (4) is true for one of the cords resulting from (possibly multiple) splitting according to relation (iv).
\end{itemize}
So now we can define the map $\widehat{D}$ as
\begin{align*}
\widehat{D}:(K\times K)\setminus A&\to C_{0}\\
c&\mapsto\partial(c)+\delta(c).
\end{align*}
The two maps $\partial$ and $\delta$ are described below:\\
First, let $\partial$ be the map
\begin{align*}
\partial:(K\times K)\setminus A&\to C_{0}\\
c&\mapsto\sum_{d\in Crit_{0}(E)}\hspace{-.3cm}n(c,d)\lambda^{\alpha_{1}}\mu^{\beta_{1}}d\lambda^{\alpha_{2}}\mu^{\beta_{2}},
\end{align*}
where
\begin{itemize}
\item $n(c,d)\in\lbrace0,1\rbrace$ is the number of trajectories along the vector field $-\nabla E$ from $c$ to $d$.
\item $\alpha_{1},\alpha_{2},\beta_{1},\beta_{2}\in\mathbb{Z}$ are such that the intersections with the framing or the base point occurring during the movement of the cord $c$ along $-\nabla E$ are taken into account according to relations (ii) or (iii), respectively.
\end{itemize}
$\partial$ is well defined since each point $c\in(K\times K)\setminus A$ lies on exactly one trajectory along the vector field $-\nabla E$. This trajectory ends at a critical point of index 0, so there is exactly one $d\in Crit_{0}(E)$ with $n(c,d)=1$. and we have $n(c,d)=0$ for all other $d\in Crit_{0}(E)$. The exponents $\alpha_{1},\alpha_{2},\beta_{1},\beta_{2}$ are uniquely determined by the relations (ii) and (iii).\\[1em]
To illustrate the procedure for determining $\partial(c)$ for a cord $c\in(K\times K)\setminus A$, we will consider two examples.%
\begin{ex}
We want to determine $\partial(c_{1})$ with $c_{1}$ as in Figure \ref{exdel1}. 
\begin{figure}[H]\centering
\includegraphics[scale=0.5]{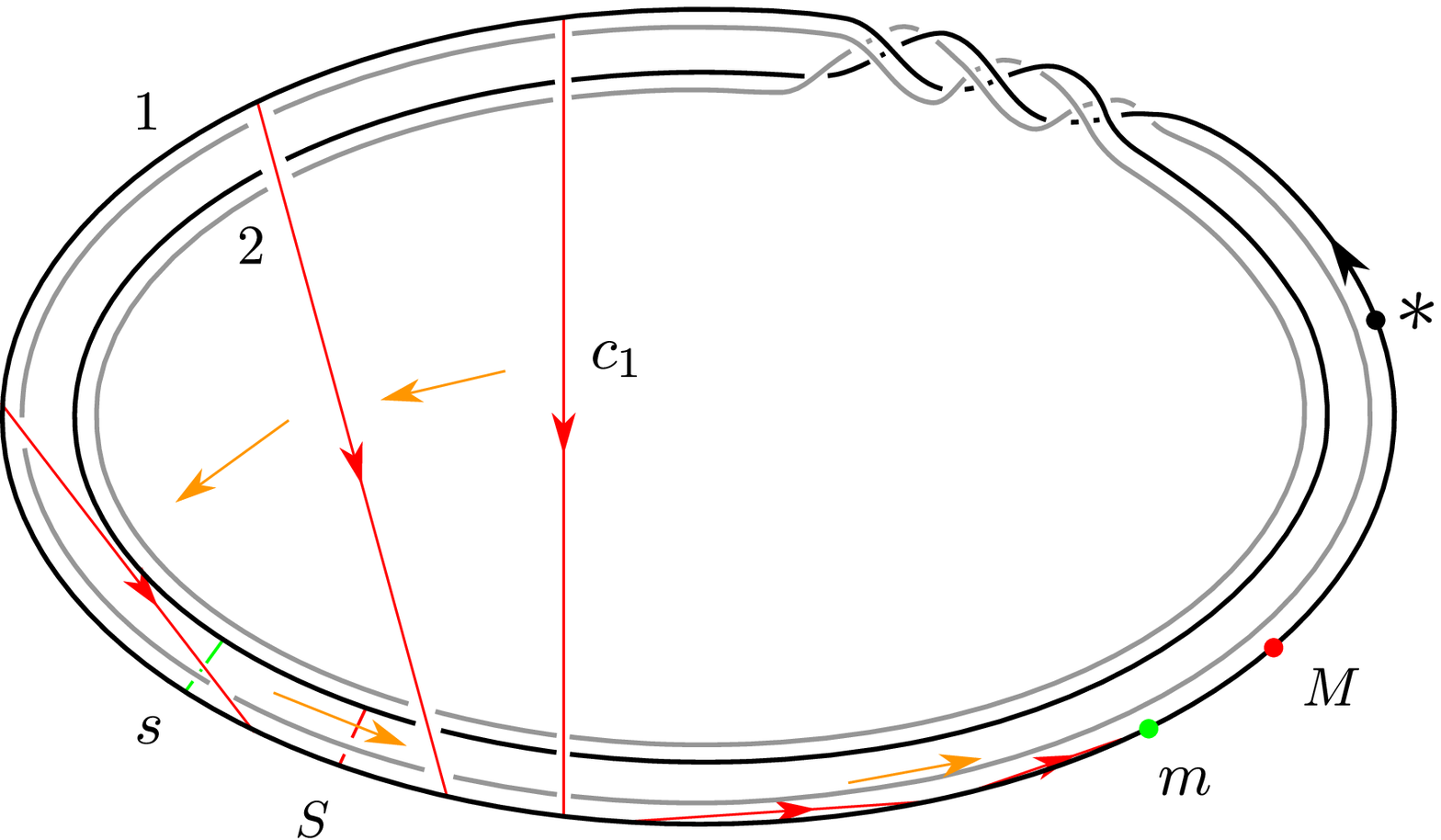}
\caption{Movement of $c_{1}$ along the gradient flow}\label{exdel1}
\end{figure}%
The movement of the cord along the gradient flow is indicated by the orange arrows in the figure. During its movement, the cord neither intersects the base point nor the framing and ends at $m$, so it is contractible. According to relation (i), we get
\[\partial(c_{1})=1-\mu.\]
\end{ex}
\begin{ex}\label{Exdel2}
In the second example we want to determine $\partial(c_{2})$ with $c_{2}$ as in Figure \ref{exdel2}. 
\begin{figure}[ht]\centering 
\includegraphics[scale=0.5]{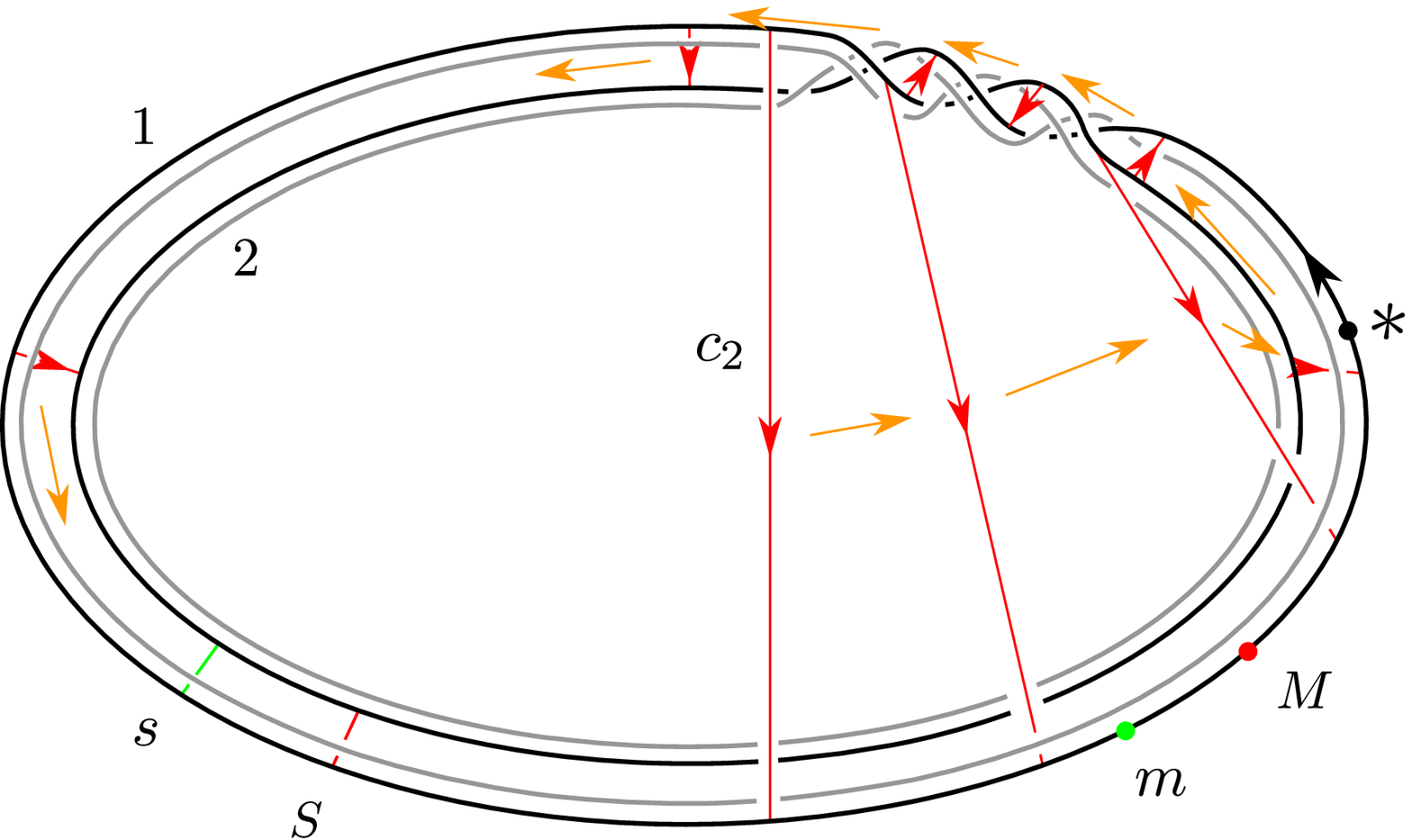}
\caption{Movement of $c_{2}$ along the gradient flow}\label{exdel2}
\end{figure}%
The movement of the cord is again indicated by orange arrows. As can be seen from the figure, the cord first intersects the framing at its endpoint, which results in a multiplication by $\mu$ from the right according to relation (ii), and then at its startpoint (results in multiplication by $\mu^{-1}$ from the left). Then the endpoint of the cord crosses the base point in the direction of the orientation of the knot, which, according to relation (iii), yields a contribution of $\cdot\lambda^{-1}$. At the crossings of the knot, the cord first intersects the framing twice at its endpoint (yields $\cdot\mu^{-2}$), then twice at its startpoint (yields $\mu^{2}\cdot$), and finally twice at its endpoint (yields $\cdot\mu^{-2}$). Then the cord runs without further intersections to the cord $s^{t}$ of index 0. So in total the following results:
\begin{align*}
\partial(c_{2})&=\mu^{2}\mu^{-1}s^{t}\mu\lambda^{-1}\mu^{-2}\mu^{-2}=\\
&=\mu s^{t}\lambda^{-1}\mu^{-3}
\end{align*}
because $\lambda$ and $\mu$ commute with each other, but not with $s^{t}$.
\end{ex}
Let's look at $\delta$ now: When determining $\partial(c)$ for a cord $c\in (K\times K)\setminus A$ it may happen that the cord intersects the knot $K$ in its interior during the movement along the gradient flow. Let $P$ be the set of all these intersections. According to relation (iv), at each $p\in P$ the cord is split into two cords, $c_{p,1}$ and $c_{p,2}$, which are multiplied by each other. Since the algebra $C_{0}$ is not commutative, the order of the factors must be chosen according to the orientation of $c$. Let $c_{p,1}$ be the first and $c_{p,2}$ be the second part in the direction of the orientation of $c$. Then we determine $\widehat{D}(c_{p,1})$ and $\widehat{D}(c_{p,2})$. Now we can define the map $\delta$ recursively: 
\begin{align*}
\delta:(K\times K)\setminus A&\to C_{0}\\
c&\mapsto\sum_{p\in P}\sign(c,p)\lambda^{\alpha_{1}(p)}\mu^{\beta_{1}(p)}\widehat{D}(c_{p,1})\widehat{D}(c_{p,2})\lambda^{\alpha_{2}(p)}\mu^{\beta_{2}(p)},
\end{align*}
where
\begin{itemize}
\item $\alpha_{1}(p),\alpha_{2}(p),\beta_{1}(p),\beta_{2}(p)\in\mathbb{Z}$ are analogous to the above definition such that the intersections with the framing or the base point occurring during the movement of the cord $c$ along $-\nabla E$ up to the point $p$ are taken into account according to the relations (ii) and (iii), respectively, and 
\item $\sign(c,p)\in\lbrace-1,1\rbrace$ is the sign in front of the product of the resulting cords obtained by applying relation (iv) to the cord $c$ at the intersection point $p$. 
\end{itemize}
\begin{rem}
When splitting $c$ into the cords $c_{p,1}$ and $c_{p,2}$ the exact position of the framing relative to the two cords, as shown in relation (iv), must be taken into account. For this it may be necessary to bend the framing a little bit in one direction. Then the framing is brought back to its original position, and one of the parts $c_{p,1}$ or $c_{p,2}$ is intersected and multiplied by $\mu$ or $\mu^{-1}$ from left or right according to relation (ii).
\end{rem}
$\delta$ is well defined:
\begin{itemize}
\item When determining $\widehat{D}(c)$ for a cord $c\in(K\times K)\setminus A$ relation (iv) is applied only finitely many times according to Lemma \ref{finitelymanyintersections}. Thus, both the sum and the recursion depth are finite.
\item The way in which the exceptional set $A$ is constructed prevents the resulting cords from being in $A$ when a cord is split according to relation (iv).
\end{itemize}%
\begin{ex}\label{exdelta}
We look again at $c_{2}$ from Figure \ref{exdel2} to determine $\delta(c_{2})$. Together with the result from Example \ref{Exdel2} we get $\widehat{D}(c_{2})$. If $c_{2}$ is moved along the negative gradient, there is a single intersection with the knot in the interior of the cord at point~$p$. There relation (iv) is applied and the cord is split into $c_{2,p,1}$ and $c_{2,p,2}$, as shown in Figure~\ref{exdelta1}. 
\begin{figure}[ht]\centering 
\subfigure{\includegraphics[scale=0.42]{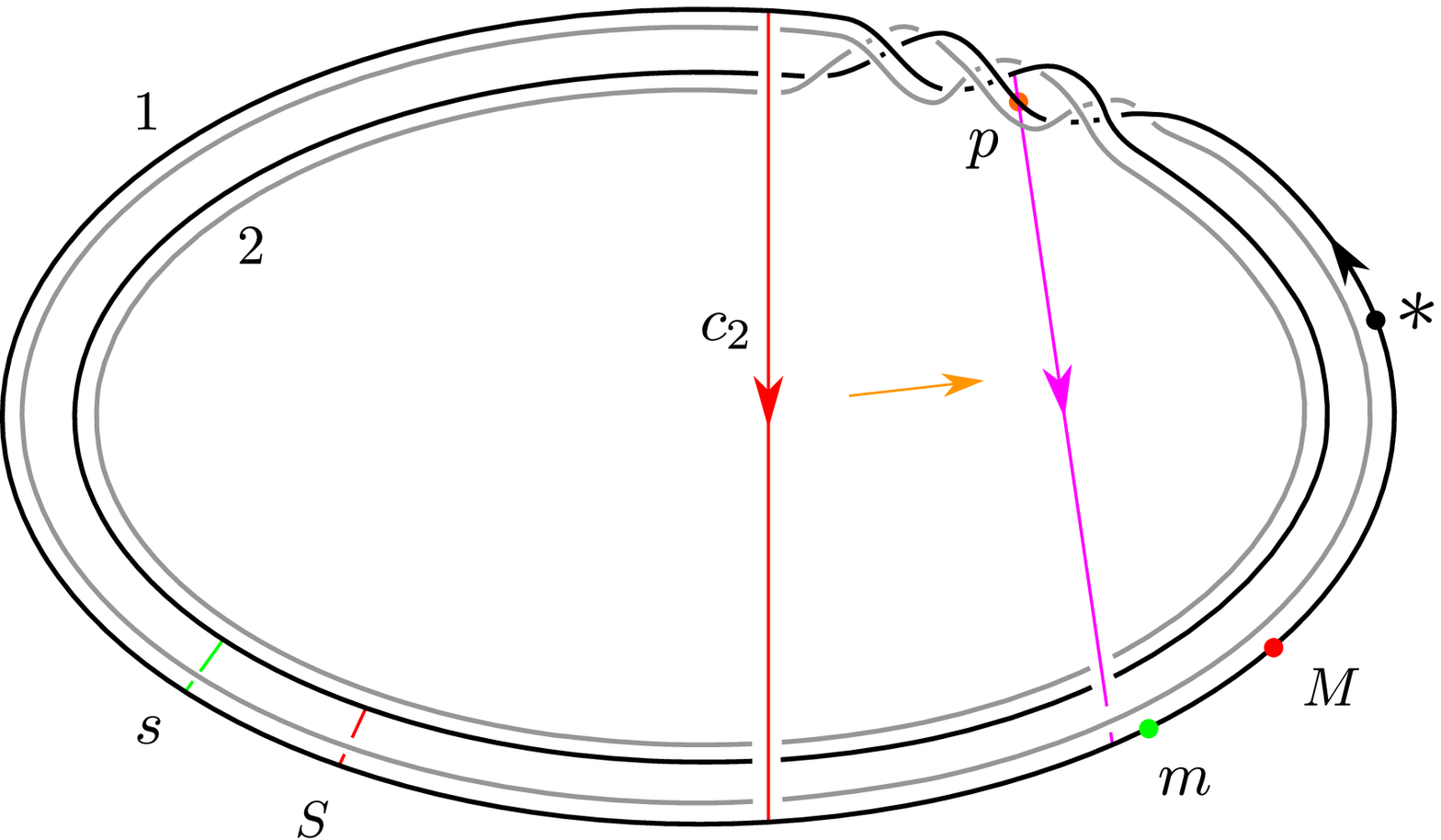}}
\hspace*{.4cm}\subfigure{\includegraphics[scale=0.47]{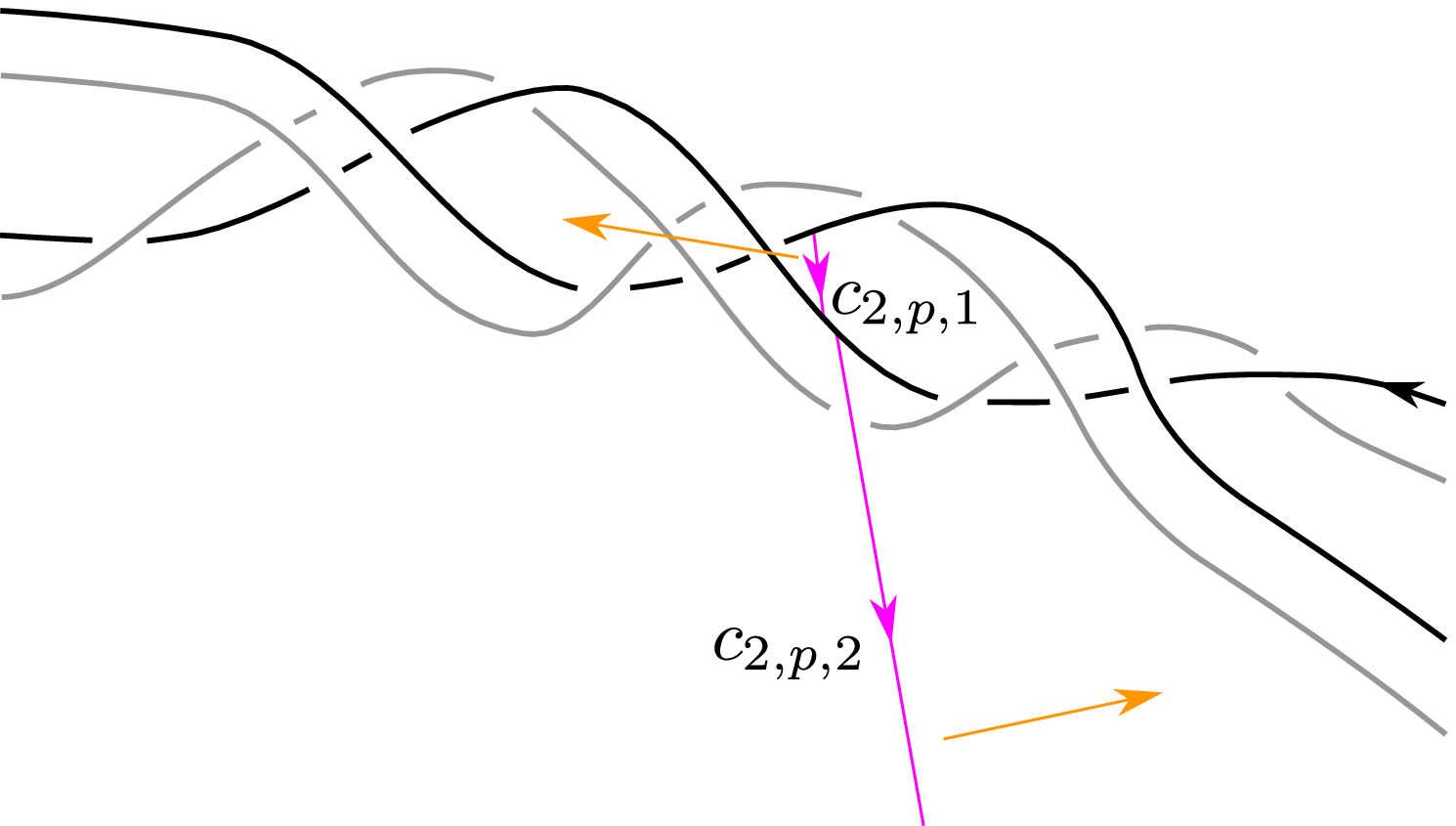}}
\caption{Determination of $\delta(c_{2})$}\label{exdelta1}
\end{figure}%
The position of the framing relative to the cord $c_{2,p,2}$ results from the representation of relation (iv) in Figure \ref{relationivequiv}. This figure also shows $\sign(c_{2},p)=-1$. According to Example \ref{Exdel2}, we get $\alpha_{1}(p)=\alpha_{2}(p)=0$ because the cord $c_{2}$ does not cross the base point before reaching the point $p$, and with relation (ii) we get $\beta_{1}(p)=-1,\beta_{2}(p)=1$ because the cord intersects the framing once at its startpoint and once at its endpoint. Thus, we get
\[\delta(c_{2})=-\mu^{-1}\widehat{D}(c_{2,p,1})\widehat{D}(c_{2,p,2})\mu\]
and we have to determine $\widehat{D}(c_{2,p,1})$ and $\widehat{D}(c_{2,p,2})$:\\
The cord $c_{2,p,1}$ intersects the framing once at its startpoint and twice at its endpoint, but not the base point or the knot in its interior, and then runs to the cord $s^{t}$. Thus, the application of relation~(ii) results in
\[\widehat{D}(c_{2,p,1})=\mu s^{t}\mu^{-2}.\]
For the cord $c_{2,p,2}$ we get according to the relations (i), (ii) and (iii)
\begin{align*}
\partial(c_{2,p,2})&=\mu^{-1}(1-\mu)\lambda^{-1}\mu^{-1}=\\
&=\lambda^{-1}\mu^{-2}-\lambda^{-1}\mu^{-1}
\end{align*}
and we get an intersection of the cord with the knot in its interior at the point $q$. There the cord is split into the parts $c_{2,p,2,q,1}$ and $c_{2,p,2,q,2}$ as shown in Figure \ref{exdelta2}. 
\begin{figure}[ht]\centering 
\subfigure{\includegraphics[scale=0.42]{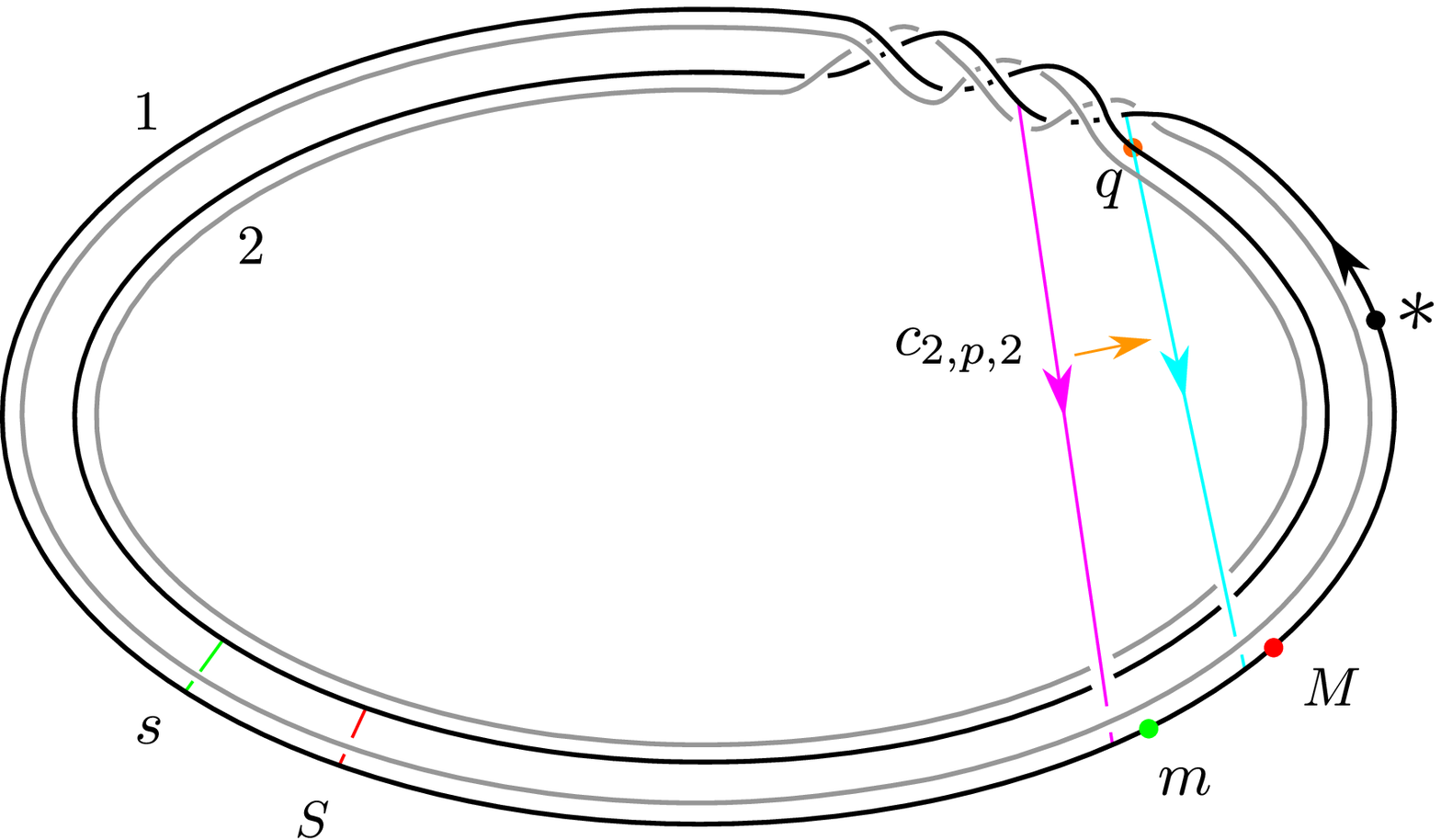}}
\hspace*{.4cm}\subfigure{\includegraphics[scale=0.43]{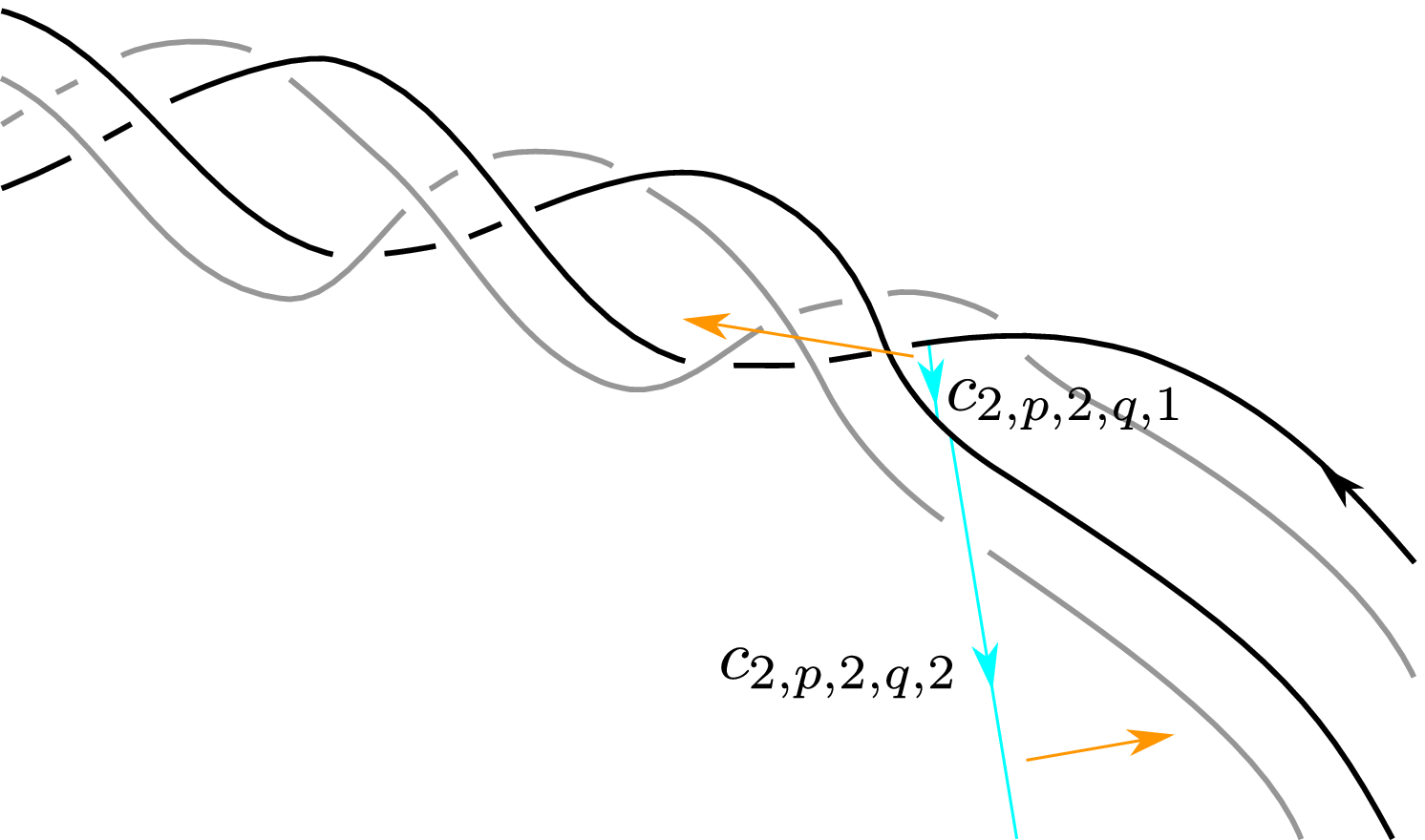}}
\caption{Determination of $\delta(c_{2,p,2})$}\label{exdelta2}
\end{figure}%
Since these two cords have no intersections with the knot in their interior during their movement along the gradient flow, we get
\begin{align*}
\delta(c_{2,p,2})&=-\mu^{-1}\widehat{D}(c_{2,p,2,q,1})\widehat{D}(c_{2,p,2,q,2})=\\
&=-\mu^{-1}(\mu\mu^{2}s^{s}\mu^{-2})(\mu^{2}s^{t}\mu^{-2}\mu^{-2}\lambda^{-1})=\\
&=-\mu^{2}s^{s}s^{t}\lambda^{-1}\mu^{-4}.
\end{align*}
Now we can determine $\delta(c_{2})$:
\begin{align*}
\delta(c_{2})&=-\mu^{-1}\widehat{D}(c_{2,p,1})\widehat{D}(c_{2,p,2})\mu=\\
&=-\mu^{-1}(\mu s^{t}\mu^{-2})(\lambda^{-1}\mu^{-2}-\lambda^{-1}\mu^{-1}-\mu^{2}s^{s}s^{t}\lambda^{-1}\mu^{-4})\mu=\\
&=-s^{t}\lambda^{-1}\mu^{-3}+s^{t}\lambda^{-1}\mu^{-2}+s^{t}s^{s}s^{t}\lambda^{-1}\mu^{-3}.
\end{align*}
Together with the result from Example \ref{Exdel2}, we get
\begin{align*}
\widehat{D}(c_{2})&=\partial(c_{2})+\delta(c_{2})=\\
&=\mu s^{t}\lambda^{-1}\mu^{-3}-s^{t}\lambda^{-1}\mu^{-3}+s^{t}\lambda^{-1}\mu^{-2}+s^{t}s^{s}s^{t}\lambda^{-1}\mu^{-3}.
\end{align*}
This example clearly shows that even for simple knots $K$ the determination of $\widehat{D}(c)$ for a cord $c\in(K\times K)\setminus A$ can be laborious and has to be done very carefully.
\end{ex}
To define the cord algebra of a knot $K$, we first choose two points $k_{+},k_{-}\in W^{u}(k)$ for each critical point $k\in Crit_{1}(E)$ as follows:\\
We move $k$ a little bit in the direction of its unstable manifold such that the startpoint of $k$ moves in the direction of the orientation of the knot. Choose a point $k_{+}$ near $k$ such that none of the sets $S,F$, and $B$ intersects the unstable manifold of $k$ between $k$ and $k_{+}$. Analogously choose $k_{-}$ on the other side of $k$. Examples to illustrate the choice of $k_{+}$ and $k_{-}$ are shown in Figure \ref{kplusminus}: On the left handside as cords on the knot and on the right handside as points in $K\times K$.\\
\begin{figure}[ht]\centering 
\subfigure{\includegraphics[scale=0.45]{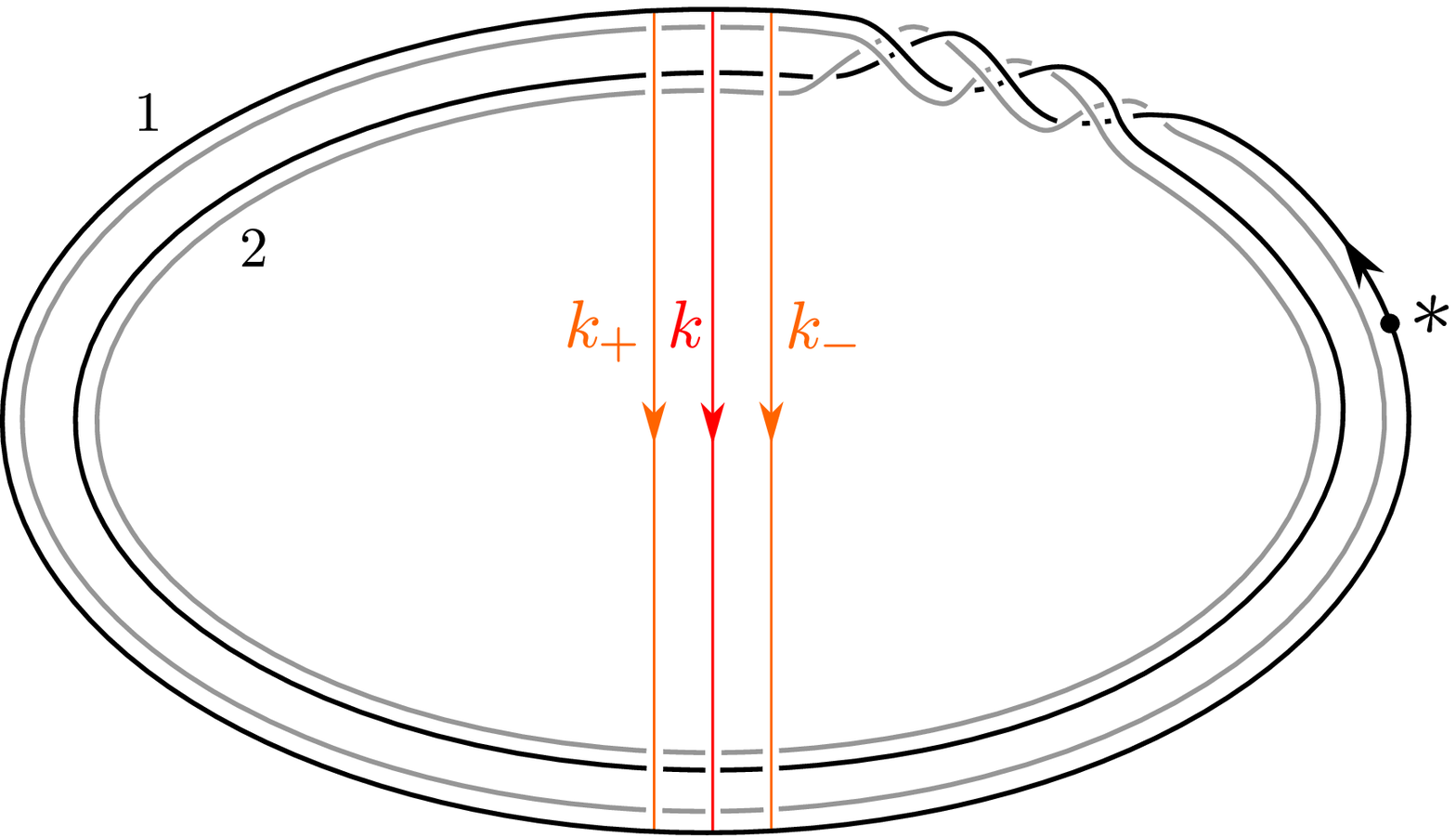}}
\hspace*{1.2cm}\subfigure{\includegraphics[scale=0.5]{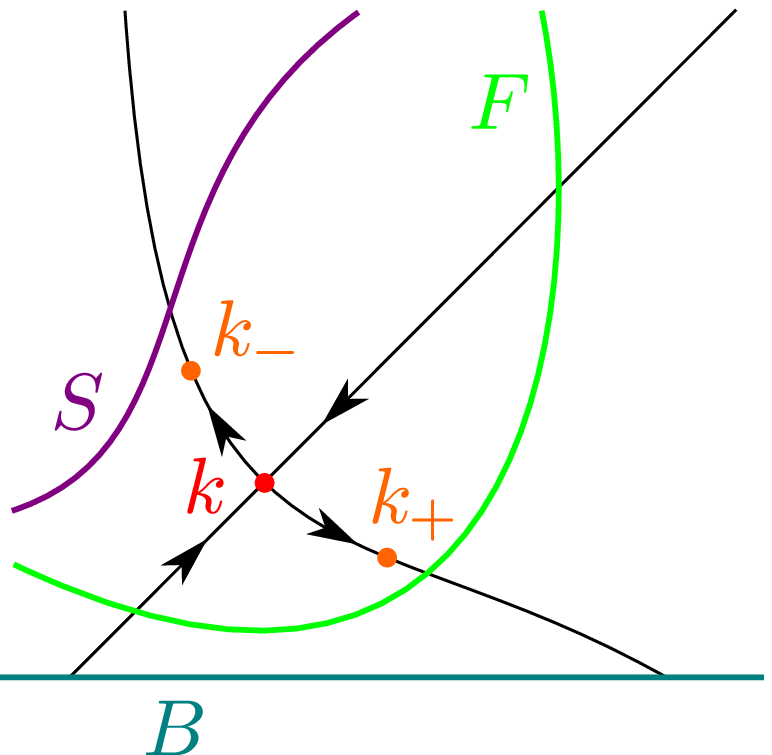}}
\caption{Choice of $k_{+}$ and $k_{-}$}\label{kplusminus}
\end{figure}%

Now we define the linear map $D$ on generators:
\begin{align*}
D:C_{1}&\to C_{0}\\
k&\mapsto\widehat{D}(k_{+})-\widehat{D}(k_{-}).
\end{align*}
\begin{ex}
We want to determine $D(k_{1,1}^{s})$ for the cord $k_{1,1}^{s}$ of the right-handed trefoil knot as in Figure \ref{index012}. As can easily be seen, $(k_{1,1}^{s})_{+}$ can be chosen as the cord $c_{1}$ in Figure \ref{exdel1} and $(k_{1,1}^{s})_{-}$ can be chosen as the cord $c_{2}$ in Figure \ref{exdel2}. With the above results we get
\begin{align*}
D(k_{1,1}^{s})&=\widehat{D}(c_{1})-\widehat{D}(c_{2})=\\
&=1-\mu-\mu s^{t}\lambda^{-1}\mu^{-3}+s^{t}\lambda^{-1}\mu^{-3}-s^{t}\lambda^{-1}\mu^{-2}-s^{t}s^{s}s^{t}\lambda^{-1}\mu^{-3}
\end{align*}
because the cord~$c_{1}$ does not intersect the knot in its interior during its movement along the negative gradient.
\end{ex}
\begin{rem}\label{RemDMequalzero}
On the diagonal of $K\times K$ there are exactly two critical points, $m$ of index 0 and $M$ of index 1 (see Remark \ref{Morsegeneration}). Both sides of the unstable manifold of $M$, which corresponds to the diagonal $\Delta$ in $K\times K$, end at $m$. According to relation~(i), we have $m=1-\mu$. To determine $D(M)$, we select $M_{+}$ and $M_{-}$ as shown in Figure \ref{DMequalzero}. 
\begin{figure}[ht]\centering
\includegraphics[scale=0.45]{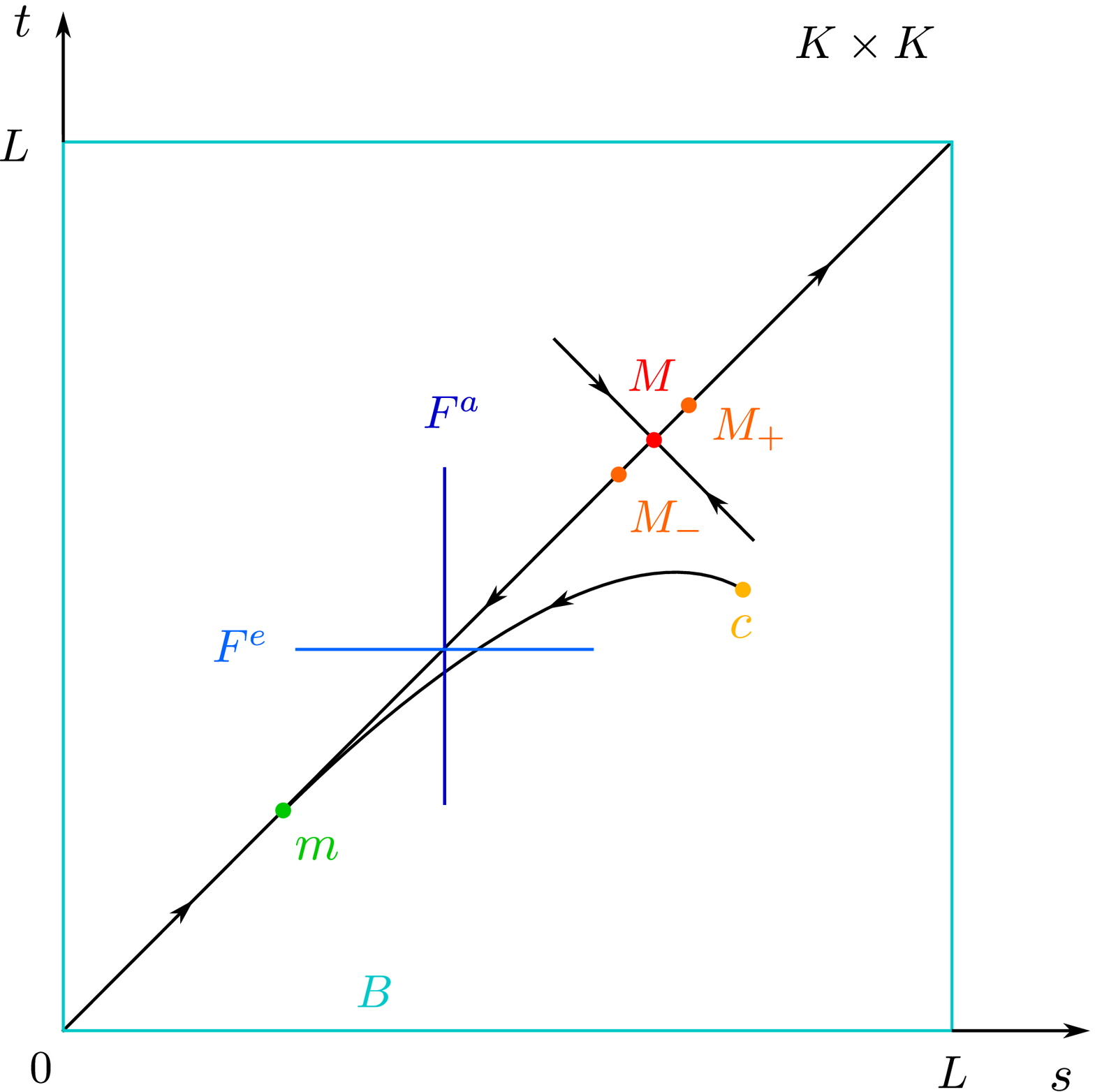}
\caption{Determination of $D(M)$}\label{DMequalzero}
\end{figure}

Since the set $S$ does not intersect the diagonal, we get
\[\delta(M_{+})=\delta(M_{-})=0.\]
But the diagonal intersects $B$ and possibly $F$. When determining $\partial(M_{\pm})$ the relations (ii) and (iii) must therefore be taken into account. Since the startpoint and the endpoint of the cord $M$ coincide, both intersect the base point and the framing. In the example of Figure \ref{DMequalzero} the following therefore holds:
\[\partial(M_{+})=\lambda(1-\mu)\lambda^{-1}=1-\mu\]
and depending on the exact course of the framing (i.e. depending on how relation (ii) is applied):
\begin{align*}
\partial(M_{-})&=\mu(1-\mu)\mu^{-1}=1-\mu\text{ or}\\
\partial(M_{-})&=\mu^{-1}(1-\mu)\mu=1-\mu.
\end{align*}
The same result is obtained if the diagonal intersects the set $F$ several times.\\
These considerations are independent of the concrete knot. Therefore, we get that in total the following holds for all knots: 
\[D(M)=\widehat{D}(M_{+})-\widehat{D}(M_{-})=(1-\mu)-(1-\mu)=0.\]
By an analogous consideration we get for a cord $c$, whose trajectory runs close enough to the diagonal and ends at $m$ (see Figure \ref{DMequalzero}): Also here both startpoint and endpoint of the cord intersect the base point or the framing if at all such an intersection takes place. Therefore, $\partial(c)=1-\mu$ follows here as well.\\
If the trajectory of a cord runs close enough to the diagonal in $K\times K$, the further intersections of this trajectory with $B$ and $F$ can be ignored. 
\end{rem}
Now we can define the cord algebra of a knot:
\begin{defi}\label{cordalgebradef}
Let $K\subset\mathbb{R}^{3}$ be a generic oriented knot, equipped with a Seifert framing and a base point. The \textit{cord algebra} of $K$ is
\[\Cord(K):=C_{0}(K)/I_{K},\]
where $I_{K}=\langle D(C_{1}(K))\rangle\subset C_{0}(K)$ is the twosided ideal generated by the image of $C_{1}(K)$ under the map $D$.
\end{defi}
\begin{rem}
It may happen that the unstable manifold of a critical point $k$ of index 1 is (in a neighborhood of $k$) parallel to the $t$-axis of $K\times K$. In this case the startpoint of $k$ does not move if $k$ is moved along its unstable manifold. Therefore, we can't determine $k_{+}$ and $k_{-}$ as described above and the map $D$ is not welldefined. But we can choose two points near $k$ on both sides of its unstable manifold and call them $k_{+}$ and $k_{-}$ as we like. Then $D$ is welldefined up to a sign. As can be seen in Definition \ref{cordalgebradef}, we get the same result in the quotient, regardless of which cord we designate with $k_{+}$ and $k_{-}$.\\
Therefore, the cord algebra of a knot is welldefined.
\end{rem}

\subsection{Change of framing}\label{Framing}
In the definition of the cord algebra the knot $K$ is equipped with the Seifert framing. However, the determination of $D(C_{1}(K))$ is more convenient if a framing~$\nu$ is used instead of the Seifert framing such that the associated set $K^{\prime}$ is a vertically shifted copy of $K$. For this purpose the strands of the knot are arranged one above the other and each strand is drawn over a slightly larger ellipse than the one below as described above. Then $K^{\prime}$ as a copy of $K$ can be moved vertically upwards by an $\varepsilon>0$ which is small enough. For each $s\in S^{1}$ then $\nu(s)$ is the unit normal vector pointing from the origin of the normal plane~$N(s)$ in the direction of the intersection of $N(s)$ with $K^{\prime}$. In the diagram, however, the framing is drawn slightly outside the corresponding strand in order to clearly distinguish it from the strand. Such a framing is called a \textit{blackboard framing}.
\begin{figure}[!ht]\centering
\begin{minipage}{0.55\textwidth}
\includegraphics[width=1.0\textwidth]{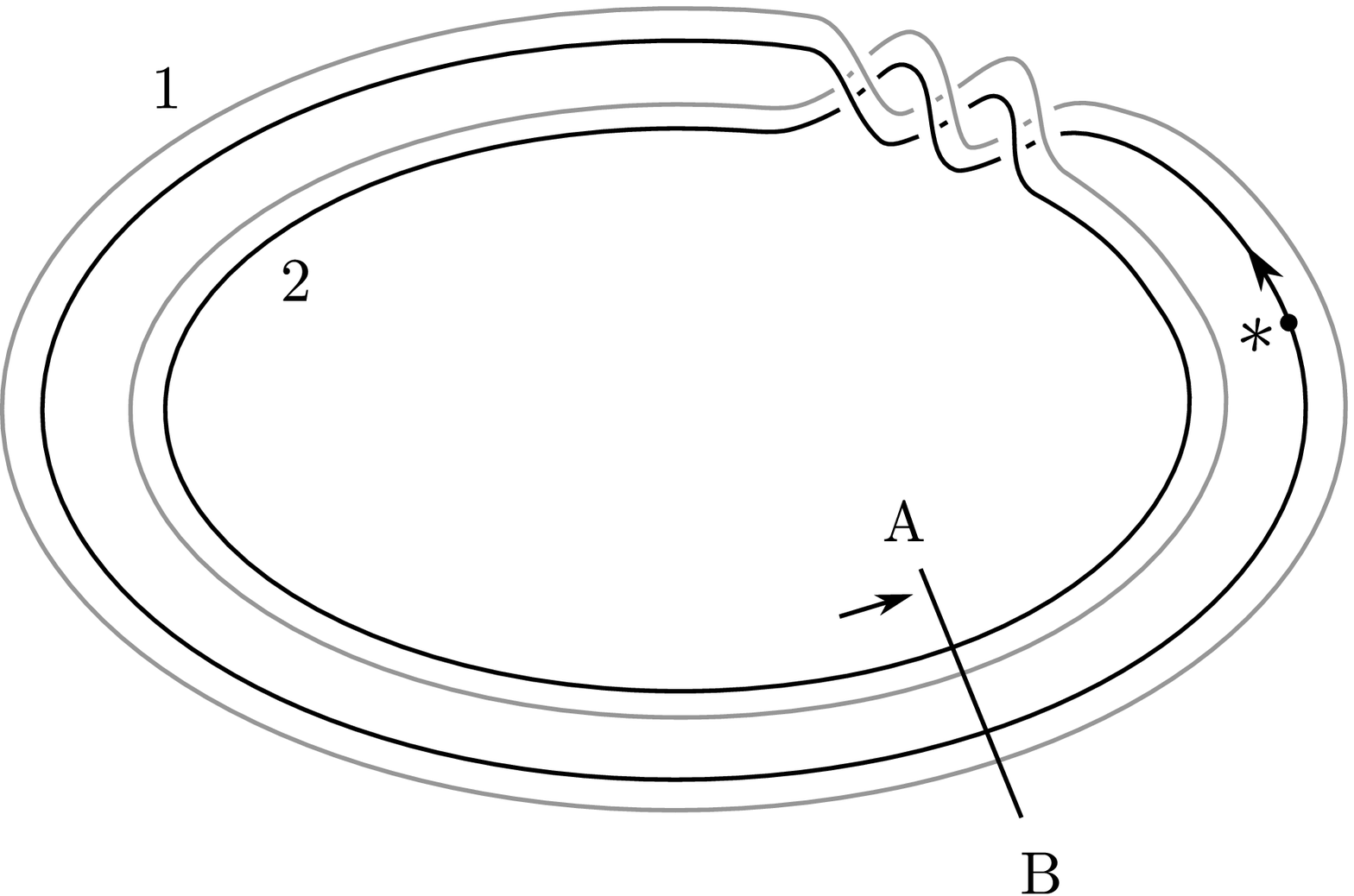}
\end{minipage}\hspace*{1cm}
\begin{minipage}{0.075\textwidth}
\includegraphics[width=1.0\textwidth]{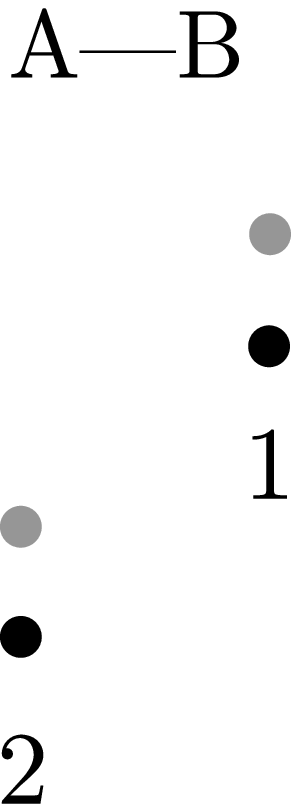}
\end{minipage}\\
\subfigure{\begin{minipage}[b]{0.55\textwidth}~\end{minipage}}
\subfigure{\begin{minipage}[b]{0.075\textwidth}~\end{minipage}}
\caption{Blackboard framing}\label{blackboardframing}
\end{figure}
In Figure \ref{blackboardframing} the right-handed trefoil knot with the blackboard framing is shown on the left side, on the right side a section through the knot from A to B with view direction in the direction of the arrow. This kind of representation ensures that the cords corresponding to critical points in $K\times K$ do not intersect the framing.\\
Using this framing we can now determine $D(C_{1}(K))$ as described above. However, the following adjustments will have to be made in order to get the cord algebra with respect to the Seifert framing:\\
To change the framing, $n$ additional windings of the framing around the knot $K$ are added. If we look in the direction of the orientation of the knot, these additional windings run in a mathematically positive resp.~negative direction (i.e. counterclockwise resp.~clockwise) around the knot and are taken into account with a sign +1 resp.~-1. So we have $n\in\mathbb{Z}$. If we have determined $D(C_{1}(K))$ for a framing $K_{1}^{\prime}$, we get $D(C_{1}(K))$ for a framing $K_{2}^{\prime}$ by applying various transformations. The first transformation is $\lambda\mapsto\lambda\mu^{n}$, where $n$ is the number of additional windings, with sign, of $K_{2}^{\prime}$ compared to $K_{1}^{\prime}$ (see \cite{Cie3}, after Remark 2.3). This results from the following consideration: The additional windings will first be added near the base point. Thus, when crossing the base point each cord gets the factor $\mu^{\pm n}$ according to relation (ii) in addition to the factor $\lambda^{\pm1}$ according to relation~(iii).\\
These additional windings change the linking number of the knot and its framing and we get:
\[\lk(K,K_{2}^{\prime})=\lk(K,K_{1}^{\prime})-n,\]
since the addition of a winding in positive direction around the knot creates two additional left-handed crossings, and thus decreases the linking number by 1. Let $K_{1}^{\prime}=K_{b}^{\prime}$ be the blackboard framing of the knot $K$ and $K_{2}^{\prime}=K_{S}^{\prime}$ be the Seifert framing. Then we have $n=\lk(K,K_{b}^{\prime})$ since $\lk(K,K_{S}^{\prime})=0$. Thus, the first transformation is given by
\[\lambda\mapsto\lambda\mu^{\lk(K,K_{b}^{\prime})}.\]
Then the added windings are shifted along the knot in such a way that a Seifert framing obtained by the Seifert algorithm actually occurs. It can happen that one of these windings intersects a generator of $C_{0}(K)$ at its start or end point. These intersections must now be taken into account in further transformations according to relation (ii), as explained in the following example.
\begin{ex}\label{changeofframing}
For the right-handed trefoil knot $K$ we have $\lk(K,K_{b}^{\prime})=3$. Thus, the first transformation is
\[\lambda\mapsto\lambda\mu^{3}.\]
\begin{figure}[ht]\centering 
\subfigure[Blackboard framing]{\includegraphics[scale=0.39]{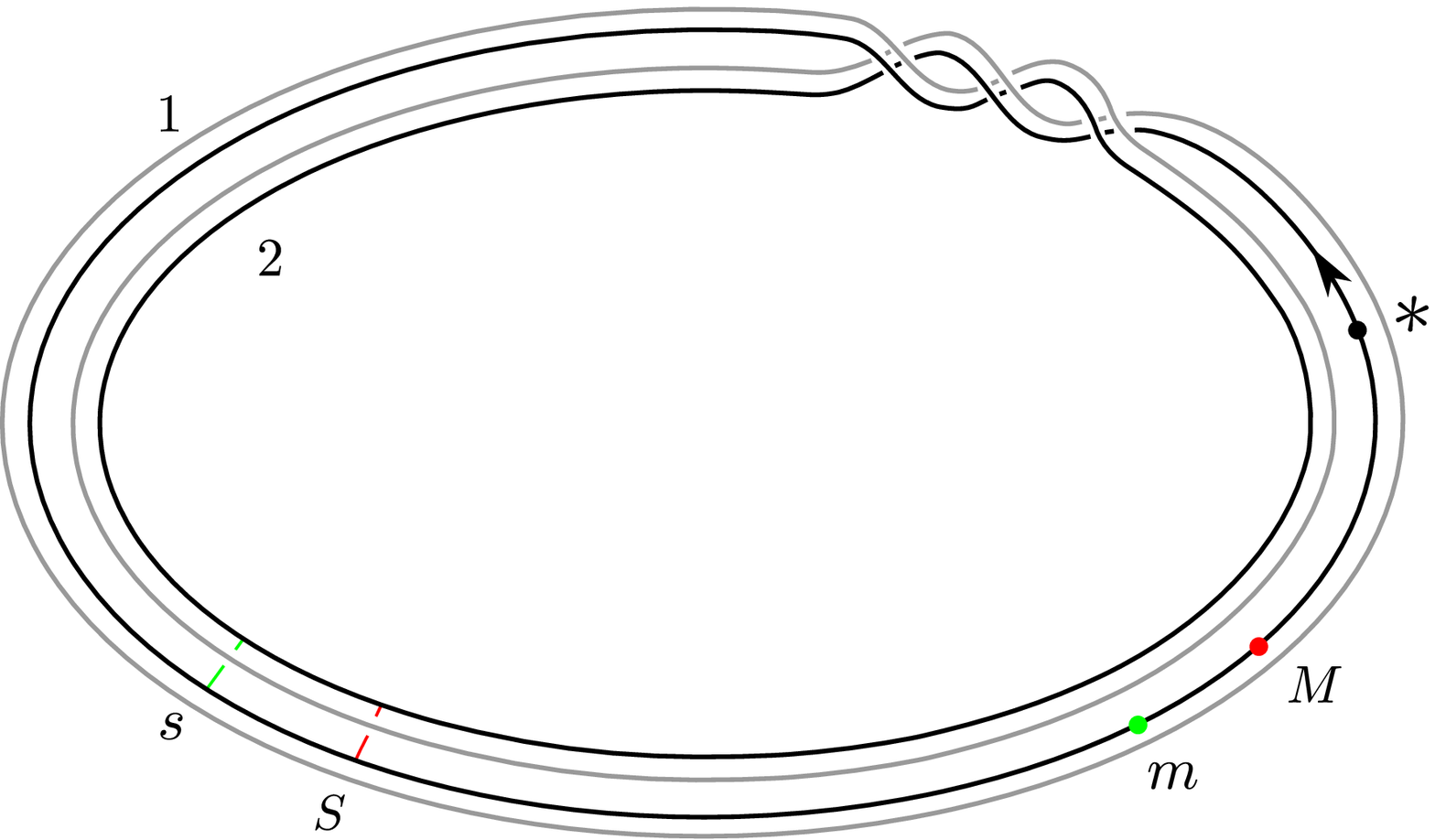}}
\subfigure[Adding 3 windings]{\includegraphics[scale=0.39]{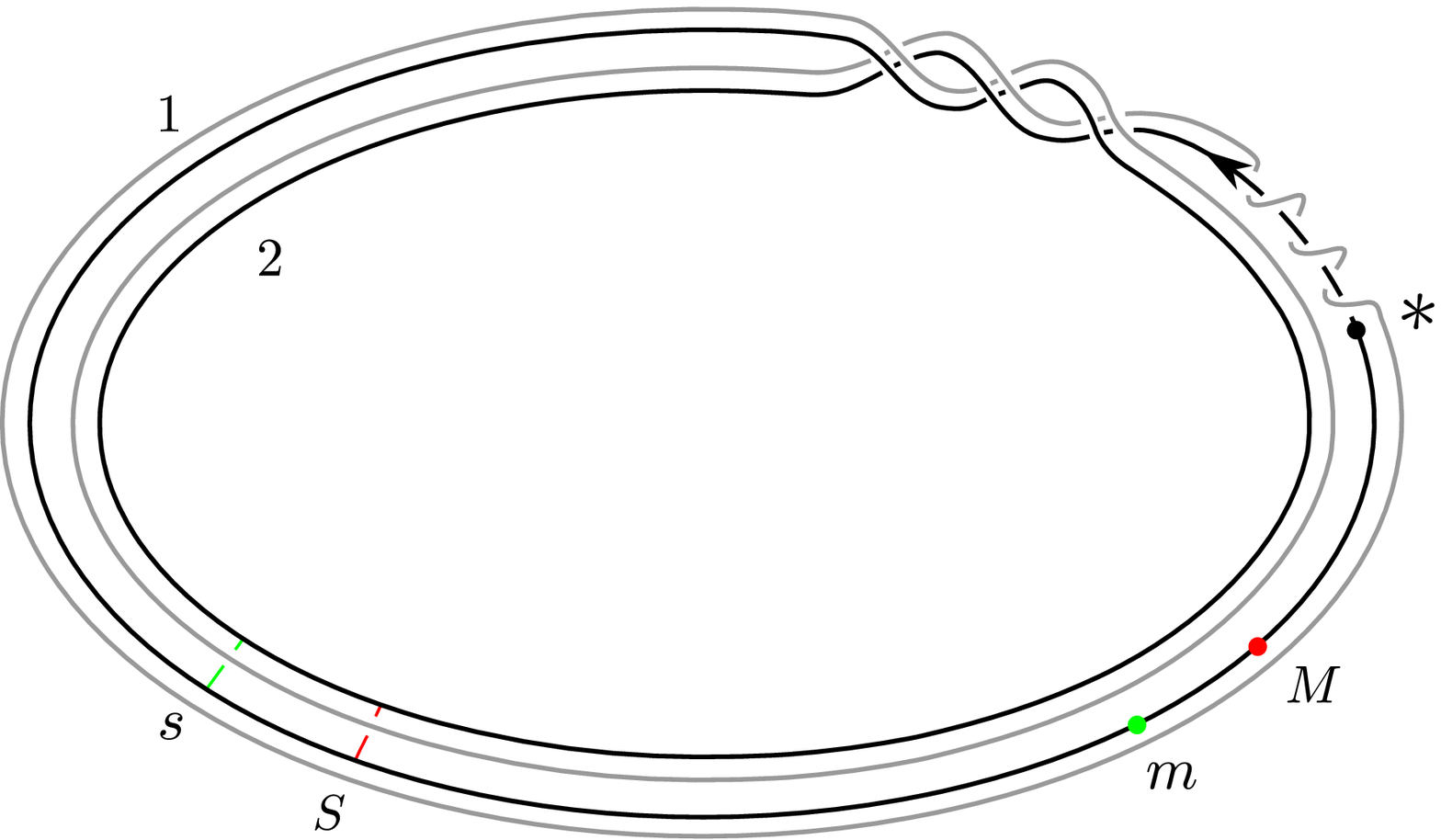}}
\subfigure[Shifting the additional windings]{\includegraphics[scale=0.4]{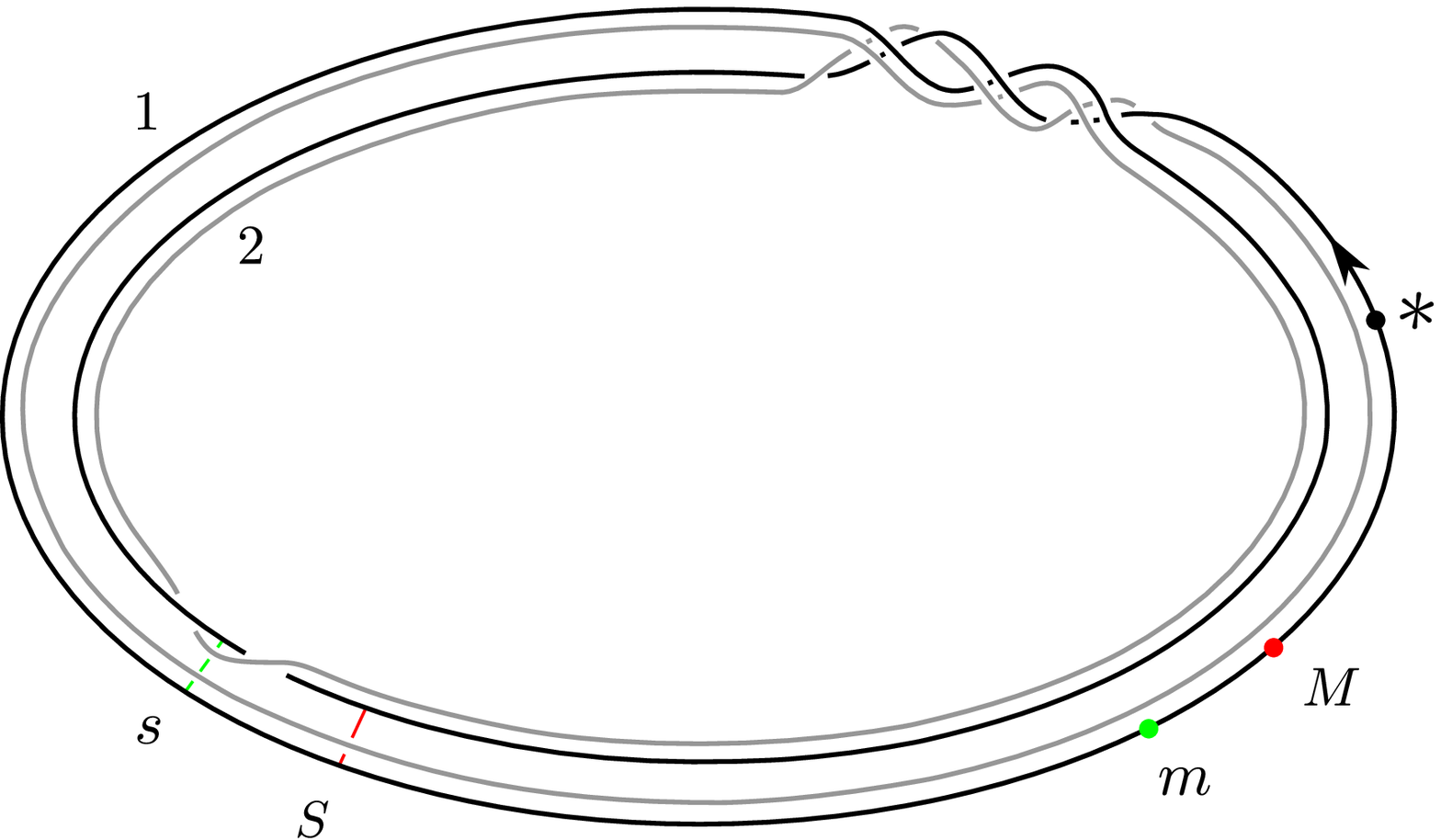}}
\subfigure[Seifert framing]{\includegraphics[scale=0.4]{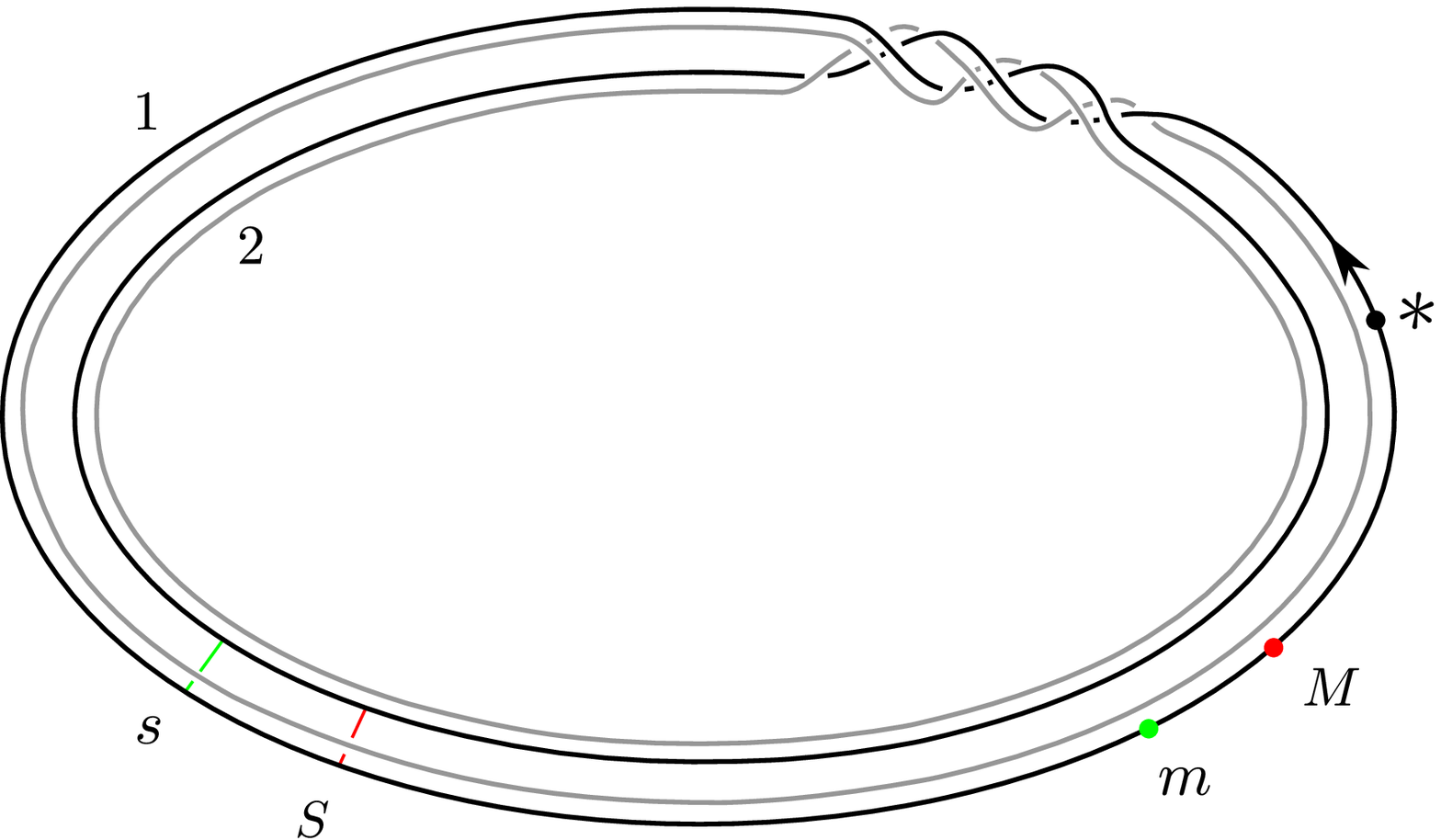}}
\caption{Changing from the blackboard framing to the Seifert framing in the example of the right-handed trefoil knot}\label{Changeofframing}
\end{figure}%
If the three windings of the framing added in the neighborhood of the base point are moved to the correct places, i.e. to the crossings of the knot, one of these windings intersects the cord $s$ during the movement, i.e. the cords $s^{s}$ and $s^{t}$ if the orientation is taken into account, see Figure~\ref{Changeofframing}. According to  relation (ii), the further transformations are 
\begin{align*}
s^{s}&\mapsto\mu^{-1}s^{s}\\
s^{t}&\mapsto s^{t}\mu.
\end{align*}
\end{ex}
\begin{rem}
The blackboard framing may have to be changed by a small perturbation in such a way that the assumptions of Remark \ref{framingassumption} for a generic framing are satisfied.
\end{rem}

\subsection{Examples: Unknot and right handed trefoil}\label{Examples}
\begin{ex}
First, we want to determine the cord algebra of the unknot $U$. We draw the unknot with the blackboard framing and choose a base point as shown in Figure \ref{exunknot}.
\begin{figure}[htbp]\centering
\includegraphics[scale=0.35]{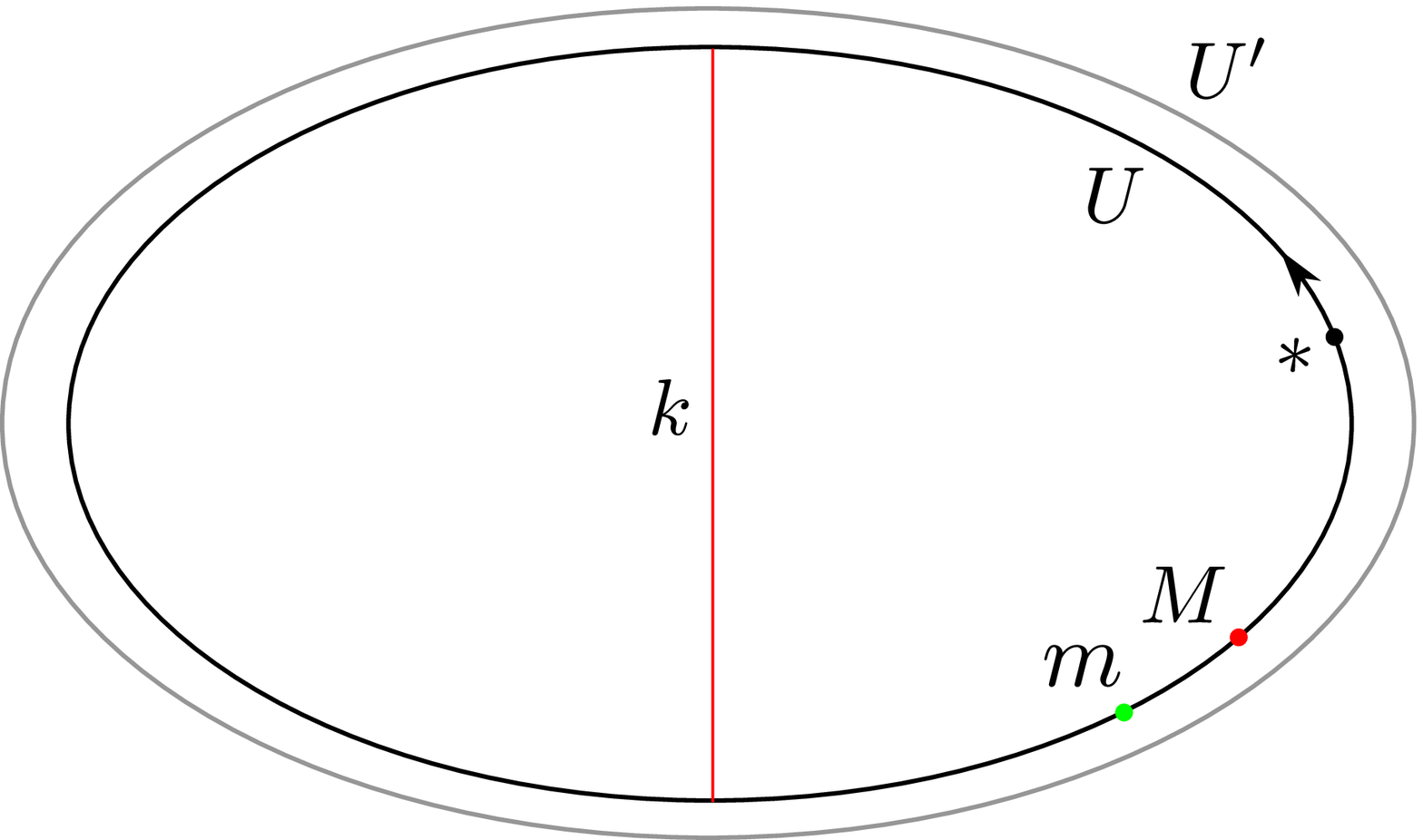}
\caption{The unknot equipped with the blackboard framing }\label{exunknot}
\end{figure}%
Let $R=\mathbb{Z}[\lambda^{\pm1},\mu^{\pm1}]$ be the ring as described above. The critical points of index 0 and 1, and thus $C_{0}$ and $C_{1}$, can be easily determined from the figure:
\begin{align*}
C_{0}&=\langle m\rangle_{R}\overset{\text{Relation (i)}}{=}\langle1-\mu\rangle_{R}=R& C_{1}&=\langle M,k^{s},k^{t}\rangle_{\mathbb{Z}}.
\end{align*}
Since the two non-trivial cords of index 1, $k^{s}$ and $k^{t}$, when moving along the gradient flow, intersect neither the framing nor the knot in their interior, relations (ii) and (iv) need not be considered. The image of $C_{1}$ under $D$ is therefore very easy to determine:
\begingroup
\allowdisplaybreaks
\begin{align*}
D(M)&=0\text{ (see Remark \ref{RemDMequalzero})}\\
D(k^{s})&=1-\mu-(1-\mu)\lambda^{-1}\\
D(k^{t})&=-(1-\mu)+\lambda(1-\mu)
\end{align*}
\endgroup
No transformation is necessary at the transition to the Seifert framing. Thus, in the quotient $C_{0}/I$, where $I=\langle D(C_{1})\rangle$ is the two-sided ideal generated by $C_{1}$ under the map $D$, we get the only relation
\[(\lambda-1)(\mu-1)=0\]
and hence
\[\Cord(U)=\mathbb{Z}[\lambda^{\pm1},\mu^{\pm1}]/((\lambda-1)(\mu-1)).\]
\end{ex}
\begin{ex}
To determine the cord algebra of the right-handed trefoil (also called torus(3,2)), the knot is equipped with the blackboard framing and a base point is chosen, see Figure~\ref{extorus32}.
\begin{figure}[H]\centering
\includegraphics[scale=0.55]{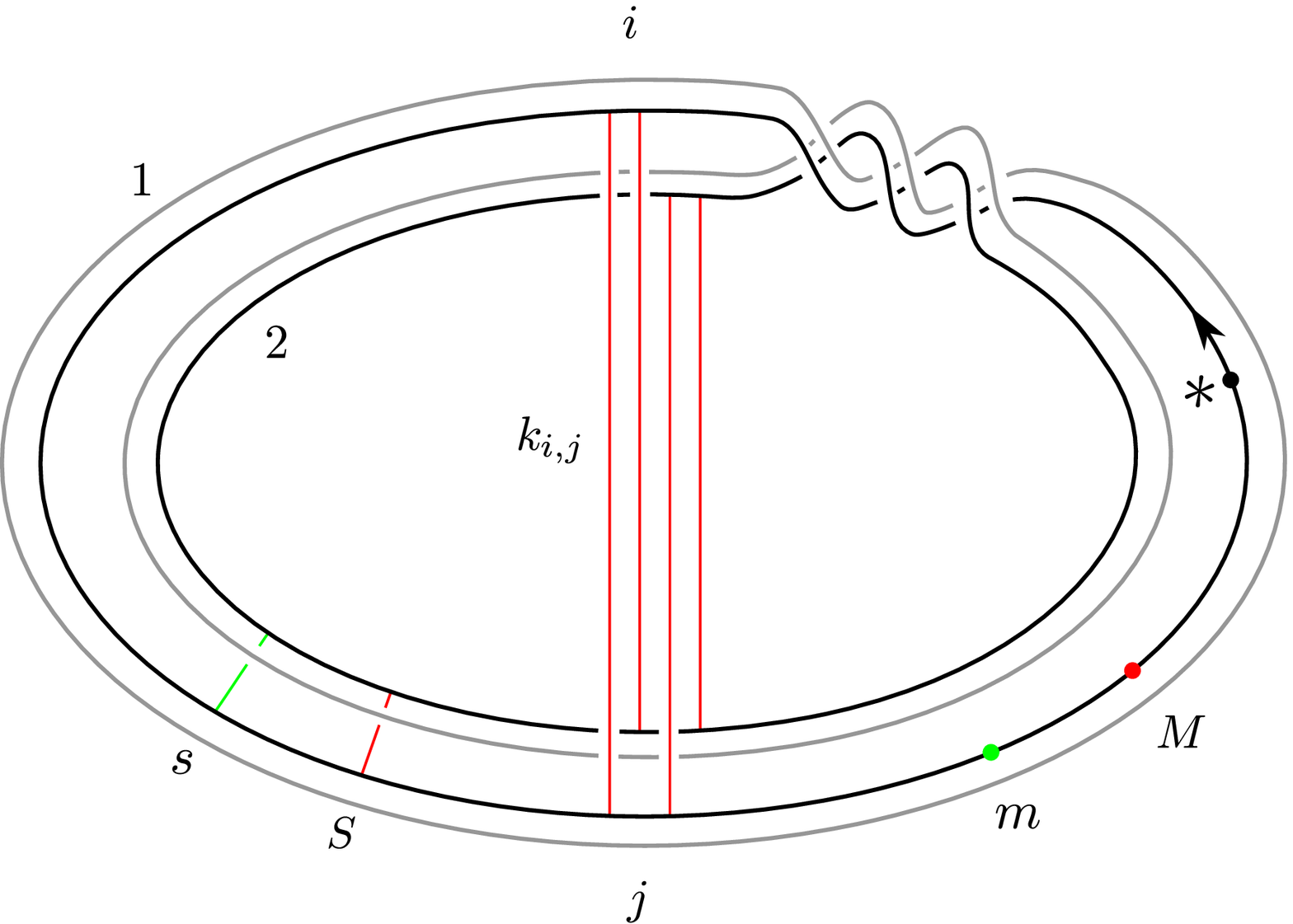}
\caption{The right-handed trefoil equipped with the blackboard framing}\label{extorus32}
\end{figure}
The critical points of index 0 and 1 can be determined from the figure. As before, $R=\mathbb{Z}[\lambda^{\pm1},\mu^{\pm1}]$. We get
\begin{align*}
C_{0}&=\langle m,s^{s},s^{t}\rangle_{R}\\
C_{1}&=\langle M,S^{s},S^{t},k_{1,1}^{s},k_{1,1}^{t},k_{1,2}^{s},k_{1,2}^{t},k_{2,1}^{s},k_{2,1}^{t} ,k_{2,2}^{s},k_{2,2}^{t}\rangle_{\mathbb{Z}}.
\end{align*}
Now we can get the image of $C_{1}$ under the map $D$ by applying $D$ to the generators of $C_{1}$: 
\begingroup
\allowdisplaybreaks
\begin{align*}
D(M)&=0\text{ (see Remark \ref{RemDMequalzero})}\\
D(S^{s})&=-s^{s}+\mu s^{t}\mu^{-1}\mu^{-1}\lambda^{-1}\\
&=-s^{s}+\mu s^{t}\lambda^{-1}\mu^{-2}\numberthis\label{extorus32_1}\\
D(S^{t})&=-s^{t}+\lambda\mu\mu s^{s}\mu^{-1}\\
&=-s^{t}+\lambda\mu^{2}s^{s}\mu^{-1}\numberthis\label{extorus32_2}\\
D(k_{1,1}^{s})&=(1-\mu)-\Big(\mu\mu s^{t}\mu^{-1}\mu^{-1}\lambda^{-1}-\mu s^{t}\mu^{-1}\mu^{-1}\big((1-\mu)\lambda^{-1}-\mu\mu s^{s}\mu^{-1}\mu^{-1}\mu\mu s^{t}\mu^{-1}\mu^{-1}\lambda^{-1}\big)\Big)\\
&=1-\mu-\mu^{2}s^{t}\lambda^{-1}\mu^{-2}+\mu s^{t}\lambda^{-1}\mu^{-2}-\mu s^{t}\lambda^{-1}\mu^{-1}-\mu s^{t}s^{s}s^{t}\lambda^{-1}\mu^{-2}\\
D(k_{1,1}^{t})&=-(1-\mu)+\lambda\mu\mu s^{s}\mu^{-1}\mu^{-1}+\big(\lambda(1-\mu)+\lambda\mu\mu s^{s}\mu^{-1}\mu^{-1}\mu s^{t}\mu^{-1}\mu^{-1}\big)\mu s^{s}\mu^{-1}\\
&=-1+\mu+\lambda\mu^{2}s^{s}\mu^{-2}+\lambda\mu s^{s}\mu^{-1}-\lambda\mu^{2}s^{s}\mu^{-1}+\lambda\mu^{2}s^{s}\mu^{-1}s^{t}\mu^{-1}s^{s}\mu^{-1}\\
D(k_{1,2}^{s})&=-\mu s^{s}+(1-\mu)+\big(\mu\mu s^{t}\mu^{-1}\mu^{-1}+(1-\mu)\mu s^{t}\mu^{-1}\mu^{-1}\big)\mu s^{s}\mu^{-1}\\
&=-\mu s^{s}+1-\mu+\mu s^{t}\mu^{-1}s^{s}\mu^{-1}\numberthis\label{extorus32_3}\\
D(k_{1,2}^{t})&=s^{t}\mu^{-1}-\Big((1-\mu)-\mu s^{t}\mu^{-1}\mu^{-1}\big(\mu\mu s^{s}\mu^{-1}\mu^{-1}-\mu\mu s^{s}\mu^{-1}\mu^{-1}(1-\mu)\big)\Big)\\
&=s^{t}\mu^{-1}-1+\mu+\mu s^{t}s^{s}\mu^{-1}\numberthis\label{extorus32_4}\\
D(k_{2,1}^{s})&=\mu s^{s}-\bigg((1-\mu)\lambda^{-1}-\mu\mu s^{s}\mu^{-1}\mu^{-1}\mu\mu s^{t}\mu^{-1}\mu^{-1}\lambda^{-1}-\mu s^{s}\mu^{-1}\Big(\mu\mu s^{t}\mu^{-1}\mu^{-1}\lambda^{-1}-\\
&\hspace*{.5cm}-\mu s^{t}\mu^{-1}\mu^{-1}\big((1-\mu)\lambda^{-1}-\mu\mu s^{s}\mu^{-1}\mu^{-1}\mu\mu s^{t}\mu^{-1}\mu^{-1}\lambda^{-1}\big)\Big)\bigg)\\
&=\mu s^{s}-\lambda^{-1}+\lambda^{-1}\mu+\mu^{2}s^{s}s^{t}\lambda^{-1}\mu^{-2}+\mu s^{s}\mu s^{t}\lambda^{-1}\mu^{-2}-\mu s^{s}s^{t}\lambda^{-1}\mu^{-2}+\mu s^{s}s^{t}\lambda^{-1}\mu^{-1}+\\
&\hspace*{.5cm}+\mu s^{s}s^{t}s^{s}s^{t}\lambda^{-1}\mu^{-2}\\
D(k_{2,1}^{t})&=-s^{t}\mu^{-1}+\lambda(1-\mu)+\lambda\mu\mu s^{s}\mu^{-1}\mu^{-1}\mu s^{t}\mu^{-1}\mu^{-1}+\\
&\hspace*{.5cm}+\Big(\lambda\mu\mu s^{s}\mu^{-1}\mu^{-1}+\big(\lambda(1-\mu)+\lambda\mu\mu s^{s}\mu^{-1}\mu^{-1}\mu s^{t}\mu^{-1}\mu^{-1}\big)\mu s^{s}\mu^{-1}\Big)s^{t}\mu^{-1}\\
&=-s^{t}\mu^{-1}+\lambda-\lambda\mu+\lambda\mu^{2}s^{s}\mu^{-1}s^{t}\mu^{-2}+\lambda\mu^{2}s^{s}\mu^{-2}s^{t}\mu^{-1}+\lambda\mu s^{s}\mu^{-1}s^{t}\mu^{-1}-\\
&\hspace*{.5cm}-\lambda\mu^{2}s^{s}\mu^{-1}s^{t}\mu^{-1}+\lambda\mu^{2}s^{s}\mu^{-1}s^{t}\mu^{-1}s^{s}\mu^{-1}s^{t}\mu^{-1}\\
D(k_{2,2}^{s})&=(1-\mu)-\bigg(\mu\mu s^{s}\mu^{-1}\mu^{-1}-\mu\mu s^{s}\mu^{-1}\mu^{-1}(1-\mu)-\\
&\hspace*{.5cm}-\mu s^{s}\mu^{-1}\Big((1-\mu)-\mu s^{t}\mu^{-1}\mu^{-1}\big(\mu\mu s^{s}\mu^{-1}\mu^{-1}-\mu\mu s^{s}\mu^{-1}\mu^{-1}(1-\mu)\big)\Big)\bigg)\\
&=1-\mu-\mu^{2}s^{s}\mu^{-1}+\mu s^{s}\mu^{-1}-\mu s^{s}-\mu s^{s}s^{t}s^{s}\mu^{-1}\\
D(k_{2,2}^{t})&=-(1-\mu)+\mu\mu s^{t}\mu^{-1}\mu^{-1}+(1-\mu)\mu s^{t}\mu^{-1}\mu^{-1}+\\
&\hspace*{.5cm}+\Big((1-\mu)+\big(\mu\mu s^{t}\mu^{-1}\mu^{-1}+(1-\mu)\mu s^{t}\mu^{-1}\mu^{-1}\big)\mu s^{s}\mu^{-1}\Big)s^{t}\mu^{-1}\\
&=-1+\mu+\mu s^{t}\mu^{-2}+s^{t}\mu^{-1}-\mu s^{t}\mu^{-1}+\mu s^{t}\mu^{-1}s^{s}\mu^{-1}s^{t}\mu^{-1}
\end{align*}
\endgroup
According to Example \ref{changeofframing}, the transformations to change from the blackboard framing to the Seifert framing are: 
\begin{align*}
\lambda&\mapsto\lambda\mu^{3}\\
s^{s}&\mapsto\mu^{-1}s^{s}\\
s^{t}&\mapsto s^{t}\mu.
\end{align*}
If these transformations are applied to the above results, the lines \eqref{extorus32_1}, \eqref{extorus32_2}, \eqref{extorus32_3} and \eqref{extorus32_4} in the quotient $C_{0}/\langle D(C_{1})\rangle$ result in the following equations: 
\begin{align}
-\mu^{-1}s^{s}+\mu s^{t}\lambda^{-1}\mu^{-4}=0\label{extorus32_5}\\
-s^{t}\mu+\lambda\mu^{4}s^{s}\mu^{-1}=0\label{extorus32_6}\\
-s^{s}+1-\mu+\mu s^{t}\mu^{-1}s^{s}\mu^{-1}=0\label{extorus32_7}\\
s^{t}-1+\mu+\mu s^{t}s^{s}\mu^{-1}=0\label{extorus32_8}
\end{align}
Therefore, we get
\begin{align}
\eqref{extorus32_5}&\Leftrightarrow s^{t}=\mu^{-2}s^{s}\lambda\mu^{4}\label{extorus32_9}\\
\eqref{extorus32_6}&\Leftrightarrow s^{t}=\lambda\mu^{4}s^{s}\mu^{-2}\label{extorus32_10}
\end{align}
In the quotient only one generator exists because $s^{t}$ can be expressed by a term with $s^{s}$. For simplicity, in the following we denote $s^{s}$ by $s$. If we now eliminate $s^{t}$ from \eqref{extorus32_9} and \eqref{extorus32_10}, we get 
\[s\lambda\mu^{6}=\lambda\mu^{6}s.\]
We put \eqref{extorus32_10} into \eqref{extorus32_7} and \eqref{extorus32_8} and get the following relations:
\begin{align*}
1-\mu-s+\lambda\mu^{5}s\mu^{-3}s\mu^{-1}&=0\\
-1+\mu+\lambda\mu^{4}s\mu^{-2}+\lambda\mu^{5}s\mu^{-2}s\mu^{-1}&=0.
\end{align*}
If the transformations are applied to the remaining results from $D(C_{1})$, the resulting equations do not yield new relations. All in all we get:
\[\Cord(\text{Torus(3,2)})=\langle s\rangle_{R}/(s\lambda\mu^{6}-\lambda\mu^{6}s,1-\mu-s+\lambda\mu^{5}s\mu^{-3}s\mu^{-1},-1+\mu+\lambda\mu^{4}s\mu^{-2}+\lambda\mu^{5}s\mu^{-2}s\mu^{-1}).\]
\end{ex}
\section{The Cord Algebra as a Knot Invariant}\label{Knotinv}
The cord algebra in its original definition, as explained at the beginning of Section \ref{Corddef}, is a knot invariant \cite{Cie3}. Now we will prove that the cord algebra in our definition using Morse Theory is also a knot invariant. So we have to show: For generic knots $K_{0}$ and $K_{1}$ for which there exists a smooth isotopy $(K_{r})_{r\in[0,1]}$ of knots, the following holds:
\[\Cord(K_{0})\cong\Cord(K_{1}).\]
So let $K_{0}$ and $K_{1}$ be two knots that are connected by a smooth isotopy $(K_{r})_{r\in[0,1]}$ of knots. For all $r\in[0,1]$ let $\gamma_{r}:[0,L_{r}]\to\mathbb{R}^{3}$ be an arclength parametrization of $K_{r}$ such that the map $r\mapsto\gamma_{r}(0)$ is continuous. Let $\gamma_{r}(0)$ be the base point of $K_{r}$. Thus, the set $B=(S^{1}\times\lbrace0\rbrace)\cup(\lbrace0\rbrace\times S^{1})\subset T^{2}\cong K_{r}\times K_{r}$ is the same for all $K_{r}$. Denote by $S_{r}\subset K_{r}\times K_{r}$ resp.~$F_{r}\subset K_{r}\times K_{r}$ the sets $S$ resp.~$F$ with respect to the knot $K_{r}$. Likewise, let $E_{r}:K_{r}\times K_{r}\to \mathbb{R},(x,y)\mapsto\frac{1}{2}\vert x-y\vert^{2},$ be the energy function with respect to the knot $K_{r}$, $X_{r}:K_{r}\times K_{r}\to\mathbb{R}^{2},(s,t)\mapsto-\nabla E_{r}(s,t)$, the associated gradient vector field on $K_{r}\times K_{r}$, $\varphi_{r}^{s}$ the flow of $X_{r}$, and $W_{r}^{s}(k_{r})$ resp.~$W_{r}^{u}(k_{r})$ the stable resp.~unstable manifold of a critical point $k_{r}$ of the function $E_{r}$. Furthermore, let $\widehat{D}_{r}:(K_{r}\times K_{r})\setminus A_{r}\to C_{0}(K_{r})$ and $D_{r}:C_{1}(K_{r})\to C_{0}(K_{r})$ be the maps belonging to $K_{r}$, where $A_{r}\subset K_{r}\times K_{r}$ is the exceptional set belonging to $K_{r}$.\\[.5em]
First, we need that only finitely many knots in this isotopy are non-generic. These finitely many knots each violate exactly one of the properties shown in the Lemmata \ref{cordlemmaadd1}, \ref{cordlemmaadd2} and \ref{cordlemmaadd3} for generic knots or at these knots pairs of critical points appear or disappear. This statement is formulated in the following lemma. This lemma can be proven by extending the arguments used in the lemmata mentioned above to 1-parameter families of knots and energy functions.
\begin{lem}\label{genericisotopy}
For a generic smooth isotopy $(K_{r})_{r\in[0,1]}$ of knots, where $K_{0}$ and $K_{1}$ are generic, the following holds:
\begin{itemize}
\item[(i)] $K_{r}$ is generic for all $r\in[0,1]\setminus\lbrace r_{1},\dots,r_{n}\rbrace$ for an $n\in\mathbb{N}$ and for $\lbrace r_{1},\dots,r_{n}\rbrace\subset[0,1]$.
\item[(ii)] For $i=1,\dots,n$ the following holds: $K_{r_{i}}$ is generic except exactly one of the following cases:
\begin{itemize}
\item[(1)] $W^{u}_{r_{i}}(k)\ntransv B$ for a critical point $k$ of index 1, and $W^{u}_{r_{i}}(k)$ is tangent to $B$ at most of order 1.
\item[(2)] $W^{u}_{r_{i}}(k)\ntransv S_{r_{i}}$ for a critical point $k$ of index 1, and $W^{u}_{r_{i}}(k)$ is tangent to $S_{r_{i}}$ at most of order 1.
\item[(3)] $W^{u}_{r_{i}}(k)\ntransv F_{r_{i}}$ for a critical point $k$ of index 1, and $W^{u}_{r_{i}}(k)$ is tangent to $F_{r_{i}}$ at most of order 1.
\item[(4)] $W^{u}_{r_{i}}(k)\cap\partial S_{r_{i}}\neq\emptyset$, and thus $W^{u}_{r_{i}}(k)\cap\partial F_{r_{i}}\neq\emptyset$ according to Lemma \ref{framinglemma}, for a critical point $k$ of index 1.
\item[(5)] $W^{u}_{r_{i}}(k)\cap S_{2,r_{i}}\neq\emptyset$ for a critical point $k$ of index 1, where $S_{2,r_{i}}\subset S_{r_{i}}$ is the set of cords that intersect the knot $K_{r_{i}}$ twice in their interior.
\item[(6)] $W^{u}_{r_{i}}(k)\cap B\cap S_{r_{i}}\neq\emptyset$ for a critical point $k$ of index 1.
\item[(7)] $W^{u}_{r_{i}}(k)\cap F_{r_{i}}\cap S_{r_{i}}\neq\emptyset$ for a critical point $k$ of index 1.
\item[(8)] $k\in B$ for a critical point $k$ of index 0 or 1.
\item[(9)] $k\in S_{r_{i}}$ for a critical point $k$ of index 0 or 1.
\item[(10)] $k\in F_{r_{i}}$ for a critical point $k$ of index 0 or 1.
\item[(11)] There exists a trajectory between two critical points of index 1 along the vector field $-\nabla E_{r_{i}}$.
\item[(12)] One of the cases (1) to (10) holds for a cord which is generated by the application of relation~(iv) (where in the statements (1) to (7) the unstable manifolds are to be replaced by trajectories along the vector field $-\nabla E_{r_{i}}$, starting at this cord).
\item[(13)] A cord which is generated by the application of relation~(iv) runs along the vector field $-\nabla E_{r_{i}}$ to a critical point of index 1.
\item[(14)] Creation / cancellation of a pair of critical points of index 0 and 1 or of index 1 and 2.
\end{itemize}
\end{itemize}
\end{lem}
\begin{rem}
According to Lemma \ref{OneparamfamVF}(ii), only degenerations of birth-death type occur in this 1-parameter family of gradient vector fields. Therefore, case (ii, 14) is the only one dealing with degeneracies of the family of vector fields. 
\end{rem}
To show that the cord algebra is a knot invariant, we first need the following lemma:
\begin{lem}\label{ExnbhdDxequalsDq}
For a generic knot $K$ the following holds:\\
Let $q\in(K\times K)\setminus A$, where $A$ is the exceptional set defined at the beginning of Section \ref{Cord}. Then there exists an open neighborhood $V\subset K\times K$ of $q$ with $\widehat{D}(x)=\widehat{D}(q)$ for all $x\in V$.
\end{lem}
\begin{proof}
Since $K$ is generic, all properties from the lemmata \ref{cordlemmaadd1}, \ref{cordlemmaadd2} and \ref{cordlemmaadd3} are satisfied.\\
There exist only finitely many critical points of index 0. Denote these by $g_{1},\dots,g_{n}$ for an $n\in\mathbb{N}$.\\
Then there exist $\varepsilon_{1},\dots,\varepsilon_{n}>0$ such that the following holds for all $i=1,\dots,n$:
\begin{align}
&\forall\ x\in\B_{2\varepsilon_{i}}(g_{i}):\lim_{s\to\infty}\varphi^{s}(x)=g_{i}\label{epsilonprop1}\\
&\B_{2\varepsilon_{i}}(g_{i})\cap(S\cup B\cup F)=\emptyset\label{epsilonprop2}
\end{align}
(\ref{epsilonprop1}) can be achieved, since $\dim W^{s}(g_{i})=2$, and (\ref{epsilonprop2}) can be achieved, since $g_{i}\notin(S\cup B\cup F)$ (according to Lemma \ref{cordlemmaadd1}) and this is an open condition.\\[1em]
Since $\widehat{D}(q)\in C_{0}(K)$ according to the assumption, we have $q\in W^{s}(g_{i_{q}})$ for an $i_{q}\in\lbrace1,\dots,n\rbrace$. Therefore, a $T_{q}\geq0$ exists with
\[\varphi^{T_{q}}(q)\in\B_{\varepsilon_{i_{q}}}(g_{i_{q}}).\]
The solution of a differential equation depends continuously on the initial conditions. So the following holds: 
\begin{align}
\forall\ T<\infty\ \forall\ \varepsilon>0\ \exists\ \delta(T,\varepsilon)>0\ \forall\ x\in K\times K\ \forall\ s\in[0,T]:\vert x-q\vert<\delta\Rightarrow\vert\varphi^{s}(x)-\varphi^{s}(q)\vert<\varepsilon\label{TheoryODE}
\end{align}
Choose $T=T_{q}$ and $\varepsilon=\varepsilon_{i_{q}}$. Then, according to the statement (\ref{TheoryODE}), there exists a $\delta_{q}>0$ such that the following holds for all $x\in\B_{\delta_{q}}(q)$:
\begin{align}
\varphi^{T_{q}}(x)\in\B_{2\varepsilon_{i_{q}}}(g_{i_{q}}).\label{deltaq}
\end{align}
So far, we've shown that each $x\in\B_{\delta_{q}}(q)\cap E^{-1}(E(q))$ flows along the vector field $X=-\nabla E$ to the same critical point of index 0 as $q$, namely to $g_{i_{q}}$.\\
To prove the lemma, we have to consider the four relations that we need to determine $D(q)$:\\[.5em]
Relation (i): Due to the properties (\ref{epsilonprop1}) and (\ref{deltaq}) relation (i) does not lead to different results when determining $\widehat{D}(x)$ for all $x\in\B_{\delta_{q}}(q)\cap E^{-1}(E(q))$.\\[.5em]
Relation (ii): Let $F_{q}:=\lbrace x\in F:\exists\ T>0:\varphi^{T}(q)=x\rbrace$. Since $K$ is generic, we have $\varphi^{s}(q)\pitchfork F$ and $\vert F_{q}\vert<\infty$. So let $F_{q}=\lbrace f_{1},\dots,f_{m_{F}}\rbrace$ for an $m_{F}\in\mathbb{N}_{0}$. For $j=1,\dots,m_{F}$ let $T_{f_{j}}$ be the time for which $\varphi^{T_{f_{j}}}(q)=f_{j}$. Because of $\varphi^{s}(q)\pitchfork F$, there exists an $\varepsilon_{F}>0$ with the following properties:
\begin{itemize}
\item[(i)]For all $j=1,\dots,m_{F}$ and all $x\in E^{-1}(E(q))$ the following holds: There exists at most one $T_{j,x}>0$ such that
\begin{align}
\varphi^{T_{j,x}}(x)\in\B_{\varepsilon_{F}}(f_{j})\cap F,\label{epsilonprop1relii}
\end{align}
i.e. the trajectory starting at $x$ intersects $F$ in $\B_{\varepsilon_{F}}(f_{j})$ at most once.
\item[(ii)]For all $x\notin\left(\bigcup\limits_{j=1}^{m_{F}}\B_{\varepsilon_{F}}(f_{j})\right)\cap F$ the following holds:
\begin{align}
\min_{s\in\mathbb{R}}\vert\varphi^{s}(q)-x\vert<\varepsilon_{F}\Rightarrow x\notin F,\label{epsilonprop2relii}
\end{align}
i.e. outside of $\B_{\varepsilon_{F}}(f_{j})$, $F$ is no closer than $\varepsilon_{F}$ to the trajectory starting from $q$. 
\end{itemize}
For all $j=1,\dots,m_{F}$ there exists, according to statement (\ref{TheoryODE}) with $T=T_{f_{j}}$ and $\varepsilon=\varepsilon_{F}$, a $\delta_{f_{j}}>0$ such that for all $x\in\B_{\delta_{f_{j}}}(q)$ the following holds:
\[\varphi^{T_{f_{j}}}(x)\in\B_{\varepsilon_{F}}(f_{j}).\]
So after reducing $\delta_{f_{j}}$ if necessary, the following holds for all $x\in\B_{\delta_{f_{j}}}(q)\cap E^{-1}(E(q))$:
\[\varphi^{s}(x)\in\B_{\varepsilon_{F}}(f_{j})\cap F\text{ for an } s>0.\]
Let $\delta_{F}:=\min\limits_{j}\delta_{f_{j}}$.\\[.5em]
Relation (iii) (analogous to the above consideration): Let $B_{q}:=\lbrace x\in B:\exists\ T>0:\varphi^{T}(q)=x\rbrace$. Since $K$ is generic, we have $\varphi^{s}(q)\pitchfork B$ and $\vert B_{q}\vert<\infty$. So let $B_{q}=\lbrace b_{1},\dots,b_{m_{B}}\rbrace$ for an $m_{B}\in\mathbb{N}_{0}$. For $j=1,\dots,m_{B}$ let $T_{b_{j}}$ be the time for which $\varphi^{T_{b_{j}}}(q)=b_{j}$. Because of $\varphi^{s}(q)\pitchfork B$, there exists an $\varepsilon_{B}>0$ with the following properties:
\begin{itemize}
\item[(i)]For all $j=1,\dots,m_{B}$ and all $x\in E^{-1}(E(q))$ the following holds: There exists at most one $T_{j,x}>0$ such that
\begin{align}
\varphi^{T_{j,x}}(x)\in\B_{\varepsilon_{B}}(b_{j})\cap B,\label{epsilonprop1reliii}
\end{align}
i.e. the trajectory starting at $x$ intersects $B$ in $\B_{\varepsilon_{B}}(b_{j})$ at most once.
\item[(ii)]For all $x\notin\left(\bigcup\limits_{j=1}^{m_{B}}\B_{\varepsilon_{B}}(b_{j})\right)\cap B$ the following holds:
\begin{align}
\min_{s\in\mathbb{R}}\vert\varphi^{s}(q)-x\vert<\varepsilon_{B}\Rightarrow x\notin B,\label{epsilonprop2reliii}
\end{align}
i.e. outside of $\B_{\varepsilon_{B}}(b_{j})$, $B$ is no closer than $\varepsilon_{B}$ to the trajectory starting at $q$.
\end{itemize}
For all $j=1,\dots,m_{B}$ there exists, according to statement (\ref{TheoryODE}) with $T=T_{b_{j}}$ and $\varepsilon=\varepsilon_{B}$, a $\delta_{b_{j}}>0$ such that for all $x\in\B_{\delta_{b_{j}}}(q)$ the following holds:
\[\varphi^{T_{b_{j}}}(x)\in\B_{\varepsilon_{B}}(b_{j}).\]
So after reducing $\delta_{b_{j}}$ if necessary, the following holds for all $x\in\B_{\delta_{b_{j}}}(q)\cap E^{-1}(E(q))$:
\[\varphi^{s}(x)\in\B_{\varepsilon_{B}}(b_{j})\cap B\text{ for an } s>0.\]
Let $\delta_{B}:=\min\limits_{j}\delta_{b_{j}}$.\\[.5em]
Relation (iv): Let $S_{q}:=\lbrace x\in S:\varphi^{s}(q)=x\text{ for an }s>0\rbrace\subset K\times K$. According to Lemma \ref{finitelymanyintersections}, there exist only finitely many intersections of $\varphi^{s}(q)$ and $S$, i.e. $\vert S_{q}\vert<\infty$. So let $S_{q}=\lbrace q_{1},\dots,q_{m_{S}}\rbrace$ for an $m_{S}\in\mathbb{N}_{0}$. For all $q_{j}\in S_{q}$ let $T_{q_{j}}>0$ be the time for which $\varphi^{T_{q_{j}}}(q)=q_{j}$. Let $S_{2}\subset S$ be the (according to Lemma \ref{cordlemma}(ii) finite) set of self-intersections of $S$, i.e. the set of cords that intersect $K$ twice in their interior.\\
Then there exists an $\varepsilon_{S}>0$ with the property that the following holds for all $j=1,\dots,m_{S}$:
\begingroup
\allowdisplaybreaks
\begin{align}
&\hspace*{-2.7cm}\text{(i)}\ \B_{\varepsilon_{S}}(q_{j})\cap\partial S=\emptyset\label{epsilonprop1reliv}\\
&\hspace*{-2.8cm}\text{(ii)}\ \B_{\varepsilon_{S}}(q_{j})\cap S_{2}=\emptyset\label{epsilonprop2reliv}\\
&\hspace*{-2.9cm}\text{(iii)\hspace{.18cm}For all }x\in E^{-1}(E(q))\text{ there exists at most one }T_{j,x}>0\text{ such that}:\notag\\
&\hspace*{1.5cm}\varphi^{T_{j,x}}(x)\in\B_{\varepsilon_{S}}(q_{j})\cap S\label{epsilonprop3reliv}\\
&\hspace*{-2.9cm}\text{(iv)\hspace{.18cm}For all }x\notin\left(\bigcup\limits_{j=1}^{m_{S}}\B_{\varepsilon_{S}}(q_{j})\right)\cap S\text{ the following holds:}\notag\\
&\hspace*{1cm}\min_{s\in\mathbb{R}}\vert\varphi^{s}(q)-x\vert<\varepsilon_{S}\Rightarrow x\notin S.\label{epsilonprop4reliv}
\end{align}
\endgroup
Such an $\varepsilon_{S}$ exists since:\\
(i) $S_{q}\cap\partial S=\emptyset$\\
(ii) $S_{q}\cap S_{2}=\emptyset$\\
(iii) $\varphi^{s}(q)\pitchfork S$\\
(iv) $\varphi^{s}(q)\cap S=S_{q}$ and $\varphi^{s}(q)\pitchfork S$\\
and (i) - (iii) are open conditions.\\
Therefore, the following holds for $j=1,\dots,m_{S}$: According to statement (\ref{TheoryODE}) with $T=T_{q_{j}}$ and $\varepsilon=\varepsilon_{S}$, there exists a $\delta_{q_{j}}>0$ such that the following holds for all $x\in\B_{\delta_{q_{j}}}(q)$:
\[\varphi^{T_{q_{j}}}(x)\in\B_{\varepsilon_{S}}(q_{j}).\]
Let $\delta_{S}:=\min\limits_{j}\delta_{q_{j}}$.\\[.5em]
Every $x\in S\setminus S_{2}$ splits according to relation (iv) in two cords $x^{1}$ and $x^{2}$, see Figure \ref{splittingofacord}, and we can see:
\begin{itemize}\itemsep0pt
\item$x^{1}$ has the same startpoint as $x$.
\item$x^{2}$ has the same endpoint as $x$.
\item The endpoint of $x^{1}$ is the startpoint of $x^{2}$.
\end{itemize}
\begin{figure}[htbp]\centering 
\includegraphics[scale=0.43]{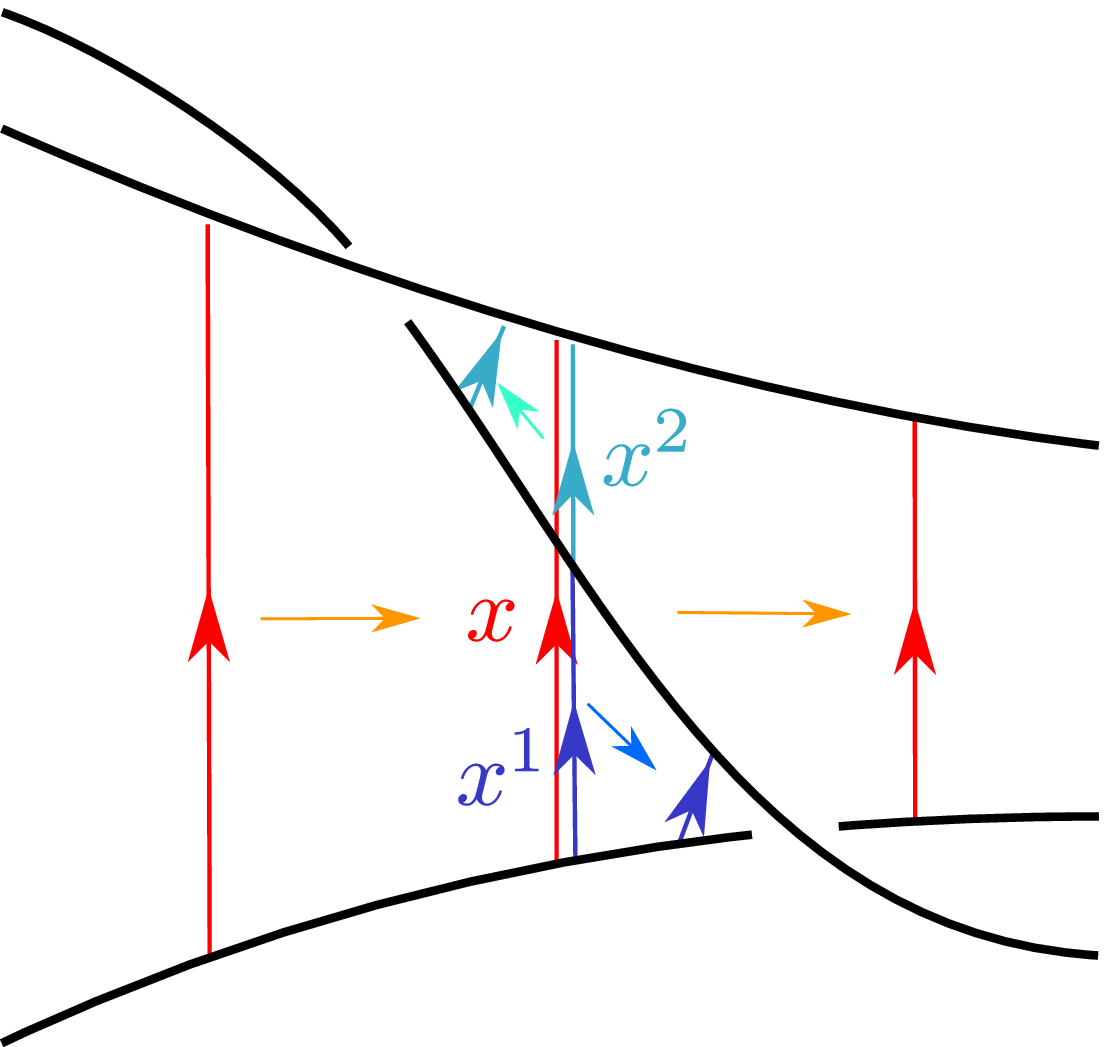}
\caption{Splitting of a cord according to relation (iv)}\label{splittingofacord}
\end{figure}%
Now consider the maps
\begin{align*}
\Psi_{1}:S\setminus S_{2}&\to K\times K,&\Psi_{2}:S\setminus S_{2}&\to K\times K\\
x&\mapsto x^{1}&x&\mapsto x^{2}
\end{align*}
For $j=1,\dots,m_{S}$ let
\begin{align*}
\Psi_{q^{1}_{j}}&:=\Psi_{1}\left(\B_{\varepsilon_{S}}(q_{j})\cap S\right)\\
\Psi_{q^{2}_{j}}&:=\Psi_{2}\left(\B_{\varepsilon_{S}}(q_{j})\cap S\right).
\end{align*}
$\Psi_{q^{1}_{j}}$ and $\Psi_{q^{2}_{j}}$ are welldefined since $\B_{\varepsilon_{S}}(q_{j})\cap S_{2}=\emptyset$. In addition, they are connected sets, since the sets $\B_{\varepsilon_{S}}(q_{j})\cap S$ are connected, the above observation concerning the startpoints and endpoints of $x^{1}$ and $x^{2}$ holds, and $K$ is smooth. See also Figure \ref{splittingofcords}.\\
\begin{figure}[htbp]\centering 
\includegraphics[scale=0.4]{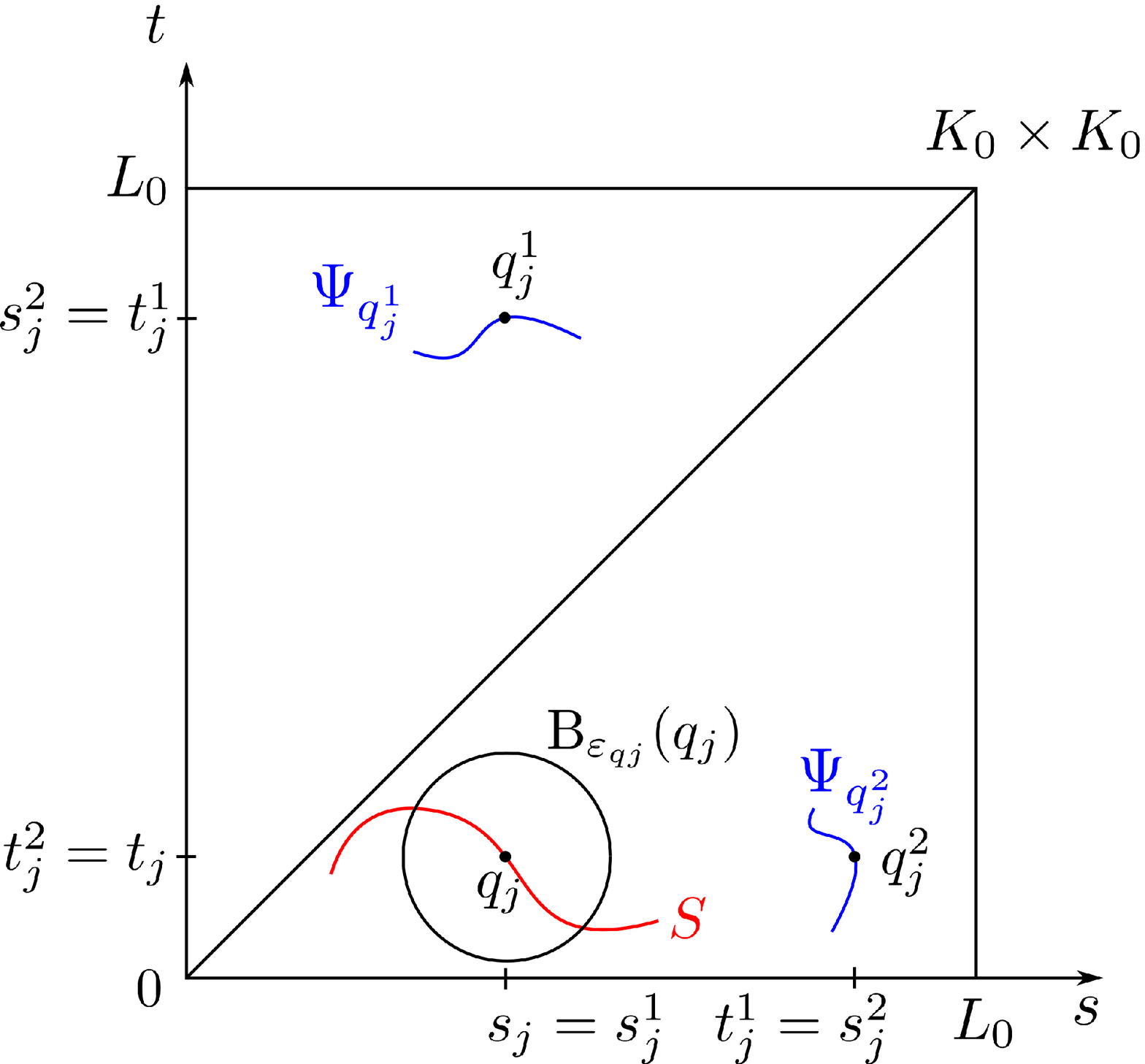}
\caption{Splitting of cords according to relation (iv)}\label{splittingofcords}
\end{figure}%
Now consider the sets $\Psi_{q^{1}_{j}}$ resp.~$\Psi_{q^{2}_{j}}$ under the gradient flow analogous to the consideration of $E^{-1}(E(q))$ under the gradient flow where $q_{j}^{1}=\Psi_{1}(q_{j})$ resp.~$q_{j}^{2}=\Psi_{2}(q_{j})$ take the role of $q$. Again, the relations (i) to (iv) must be taken into account. Thus, for $i=1,2$ and $j=1,\dots,m_{S}$ we get analogous to the above considerations: 
\begin{itemize}
\item $\delta_{q_{j}^{i},F}$ such that for all $x\in\B_{\delta_{q_{j}^{i},F}}(q_{j}^{i})\cap\Psi_{q_{j}^{i}}$ the following holds: $\varphi^{s}(x)$ and $\varphi^{s}(q_{j}^{i})$ intersect the set $F$ for $s\geq0$ so that relation (ii) produces the same result with respect to $\mu^{\pm1}$.
\item $\delta_{q_{j}^{i},B}$ such that for all $x\in\B_{\delta_{q_{j}^{i},B}}(q_{j}^{i})\cap\Psi_{q_{j}^{i}}$ the following holds: $\varphi^{s}(x)$ and $\varphi^{s}(q_{j}^{i})$ intersect the set $B$ for $s\geq0$ so that relation (iii) produces the same result with respect to $\lambda^{\pm1}$.
\item $\delta_{q^{i}_{j,k^{i}}}$ for $k^{i}=1,\dots,m_{S}^{i}$, with $m_{S}^{i}\in\mathbb{N}$, where $q^{i}_{j,k^{i}}$ denote the intersections of $\varphi^{s}(q_{j}^{i})$, for $s\geq0$, with $S$. The following holds for all $x\in\B_{\delta_{q^{i}_{j,k^{i}}}}(q_{j}^{i})\cap\Psi_{q_{j}^{i}}$: $\varphi^{s}(x)$ intersects $S$ in an $\varepsilon_{S}^{i}$-neighborhood of $q^{i}_{j,k^{i}}$ exactly once, but not within an $\varepsilon_{S}^{i}$-tube around $\varphi^{s}(q_{j}^{i})$ outside these neighborhoods.
\end{itemize}
According to Lemma \ref{finitelymanyintersections}, this process ends after finitely many steps.\\
For $i=1,2$ let $\hat{\delta}_{q_{j}^{i}}:=\min\lbrace\delta_{q_{j}^{i},F},\delta_{q_{j}^{i},B},\delta_{q_{j}^{i},1},\dots,\delta_{q_{j}^{i},m_{S}^{i}}\rbrace$. Let $\varepsilon_{q_{j}^{i}}>0$ so that
\[\B_{\varepsilon_{q_{j}^{i}}}(q_{j})\cap S\subset\Psi_{i}^{-1}\left(\B_{\hat{\delta}_{q_{j}^{i}}}(\Psi_{i}(q_{j}))\cap\Psi_{q_{j}^{i}}\right).\]
$\Psi_{1}$ and $\Psi_{2}$ are injective: Suppose $\Psi_{1}$ is not injective. Then there exist $x,y\in S\setminus S_{2}$ with $x\neq y$ and $\Psi_{1}(x)=\Psi_{1}(y)$, i.e. $x^{1}=y^{1}$. Thus, $x$ and $y$ have the same startpoint and a common intersection point with the knot in their interior. But since $x\neq y$, there are only the two possible positions shown in Figure \ref{positionsofxandy}.
\begin{figure}[hbt]\centering 
\subfigure{\includegraphics[scale=0.35]{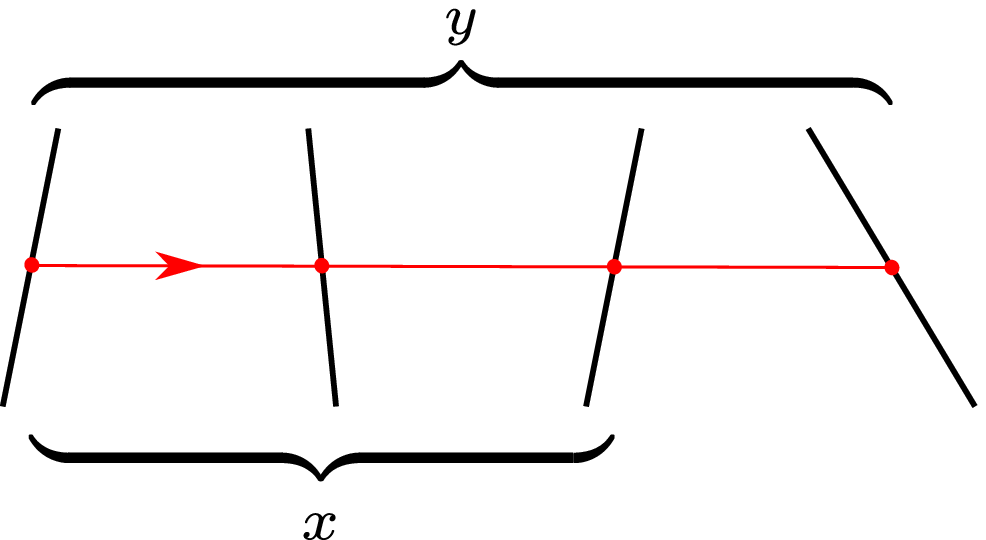}}
\subfigure{\hspace*{1cm}\includegraphics[scale=0.35]{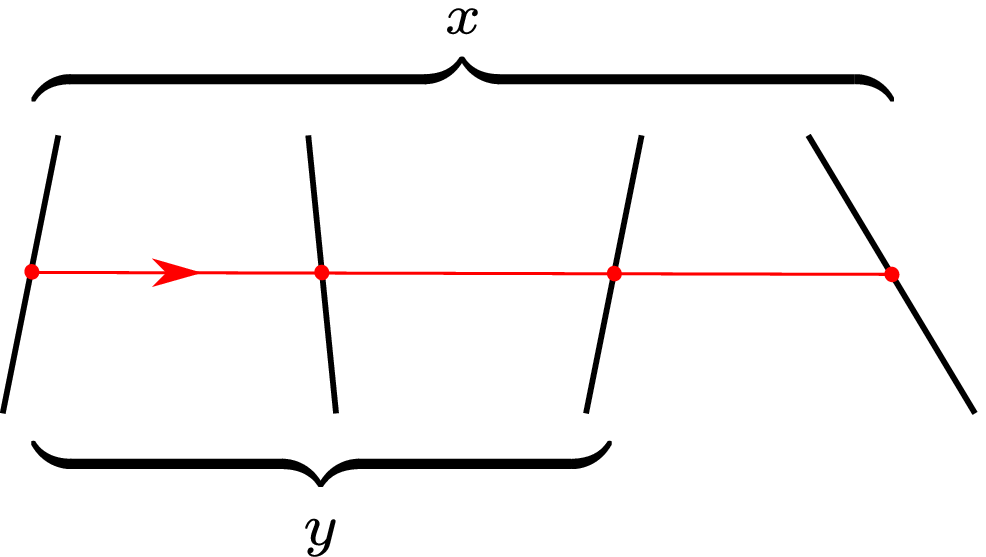}}
\caption{Possible positions of the cords $x$ and $y$}\label{positionsofxandy}
\end{figure}%
It follows that $x\in S_{2}$ or $y\in S_{2}$, but this contradicts the definition of $\Psi_{1}$. Similarly, the injectivity of $\Psi_{2}$ can be shown.\\
Because of this injectivity and the definition of $\Psi_{q_{j}^{i}}$, we get $\Psi_{i}^{-1}\left(\B_{\hat{\delta}_{q_{j}^{i}}}(\Psi_{i}(q_{j}))\cap\Psi_{q_{j}^{i}}\right)\subset\B_{\varepsilon_{S}}(q_{j})$ and thus $\B_{\varepsilon_{q_{j}^{i}}}(q_{j})\cap S\subset\B_{\varepsilon_{S}}(q_{j})$.\\
According to statement (\ref{TheoryODE}) with $T=T_{q_{j}}$ and $\varepsilon=\varepsilon_{q_{j}^{i}}$, there exists a $\delta_{q_{j}^{i}}>0$ such that for all $x\in\B_{\delta_{q_{j}^{i}}}(q)$ the following holds:
\[\varphi^{T_{q_{j}}}(x)\in\B_{\varepsilon_{q_{j}^{i}}}(q_{j}).\]
The same procedure is used for all intersections $q^{i}_{j,k^{i}}$, for $i=1,2,\ j=1,\dots,m_{S}$ and $k^{i}=1,\dots,m_{S}^{i}$, the $\delta$s obtained thereby having to be pulled back by multiple application of $\Psi_{i}^{-1}$ and statement~(\ref{TheoryODE}) into the set $E^{-1}(E(q))$.\\
In total we get the finite set $\Delta=\lbrace\delta_{q},\delta_{q_{1}},...,\delta_{q_{m_{S}}},\delta_{q_{1}^{1}},...,\delta_{q_{m_{S}}^{1}},\delta_{q_{1}^{2}},...,\delta_{q_{m_{S}}^{2}},...\rbrace$. Let $\delta:=\min\Delta$ and $\widetilde{V}:=\B_{\delta}(q)\cap E^{-1}(E(q))$. Then we have $\widehat{D}(x)=\widehat{D}(q)$ for all $x\in \widetilde{V}$ because of the properties~(\ref{epsilonprop1}) and (\ref{epsilonprop2}) and the construction of $\delta$. Since the solution of a differential equation depends continuously on the initial conditions, $\widetilde{V}$ can be extended to an open neighborhood $V$ of $q$ with $\widetilde{V}\subset V$ such that the following holds for all $x\in V$:
\[\widehat{D}(x)=\widehat{D}(q).\]
\end{proof}
\begin{thm}
The cord algebra as in Definition \ref{cordalgebradef} is a knot invariant.
\end{thm}
\begin{proof}
We consider a generic smooth isotopy $(K_{r})_{r\in[0,1]}$ of knots where $K_{0}$ and $K_{1}$ are generic, and want to show that $\Cord(K_{0})\cong\Cord(K_{1})$. Therefore, we will look at all the cases listed in Lemma \ref{genericisotopy}. Since there are, according to Lemma \ref{genericisotopy}, only finitely many non-generic knots during this isotopy, we can simplify notation by considering a generic isotopy $(K_{r})_{r\in[-\varepsilon,\varepsilon]}, \varepsilon>0$, such that
\begin{itemize}
\item in case (i) of Lemma \ref{genericisotopy} $K_{r}$ is generic for all $r\in[-\varepsilon,\varepsilon]$.
\item in case (ii) of Lemma \ref{genericisotopy} $K_{r}$ is generic for all $r\in[-\varepsilon,\varepsilon]\setminus\lbrace0\rbrace$ and $K_{0}$ satisfies exactly one of the cases (ii,1) to (ii,14).
\end{itemize}
First, we construct isomorphisms which we will need in the cases (i) and (ii,1) to (ii,13): Let
\[n_{0}:=\vert Crit_{0}(K_{0})\vert\text{ and }n_{1}:=\vert Crit_{1}(K_{0})\vert.\]
Since no critical points appear or disappear (in the cases (i) and (ii,1) to (ii,13)), the following holds for all $r\in[-\varepsilon,\varepsilon]$:
\[\vert Crit_{0}(K_{r})\vert=n_{0}\text{ and }\vert Crit_{1}(K_{r})\vert=n_{1}.\]
So for $r\in[-\varepsilon,\varepsilon]$ let
\[Crit_{0}(K_{r})=\lbrace g_{r}^{1},\dots,g_{r}^{n_{0}}\rbrace\text{ and }Crit_{1}(K_{r})=\lbrace k_{r}^{1},\dots,k_{r}^{n_{1}}\rbrace.\]
Now we define the sets
\begin{align*}
Crit_{0}^{[-\varepsilon,\varepsilon]}&:=\bigcup_{r\in[-\varepsilon,\varepsilon]}\lbrace r\rbrace\times\lbrace g_{r}^{1},\dots,g_{r}^{n_{0}}\rbrace\subset[-\varepsilon,\varepsilon]\times T^{2}\\
Crit_{1}^{[-\varepsilon,\varepsilon]}&:=\bigcup_{r\in[-\varepsilon,\varepsilon]}\lbrace r\rbrace\times\lbrace k_{r}^{1},\dots,k_{r}^{n_{1}}\rbrace\subset[-\varepsilon,\varepsilon]\times T^{2}.
\end{align*}
Since the isotopy is smooth, we can number the critical points such that the maps

\begin{align*}
\Psi_{0}:[-\varepsilon,\varepsilon]\times Crit_{0}(K_{-\varepsilon})&\to Crit_{0}^{[-\varepsilon,\varepsilon]}\\
(r,g_{-\varepsilon}^{i})&\mapsto g_{r}^{i},\ i=1,\dots,n_{0}\\
\Psi_{1}:[-\varepsilon,\varepsilon]\times Crit_{1}(K_{-\varepsilon})&\to Crit_{1}^{[-\varepsilon,\varepsilon]}\\
(r,k_{-\varepsilon}^{i})&\mapsto k_{r}^{i},\ i=1,\dots,n_{1}
\end{align*}
are continuous with respect to the first component. For all $r\in[-\varepsilon,\varepsilon]$ we define the linear maps $\Phi_{0,r}$ and $\Phi_{1,r}$ on generators by
\begin{align*}
\Phi_{0,r}:C_{0}(K_{-\varepsilon})&\to C_{0}(K_{r})\\
g_{-\varepsilon}^{i}&\mapsto\Psi_{0}(r,g_{-\varepsilon}^{i})=g_{r}^{i},\ i=1,\dots,n_{0}\\
\lambda^{\pm1}&\mapsto\lambda^{\pm1}\\
\mu^{\pm1}&\mapsto\mu^{\pm1}\\
\Phi_{1,r}:C_{1}(K_{-\varepsilon})&\to C_{1}(K_{r})\\
k_{-\varepsilon}^{i}&\mapsto\Psi_{1}(r,k_{-\varepsilon}^{i})=k_{r}^{i},\ i=1,\dots,n_{1}.
\end{align*}
For all $r\in[-\varepsilon,\varepsilon]$ we extend $\Phi_{0,r}$ to an algebra homomorphism. Obviously, $\Phi_{0,r}$ and $\Phi_{1,r}$ are isomorphisms for all $r\in[-\varepsilon,\varepsilon]$.\\
For each $k_{r}^{i},\ r\in[-\varepsilon,\varepsilon],i=1,\dots,n_{1}$, choose $k_{r,+}^{i},k_{r,-}^{i}\in W_{r}^{u}(k_{r}^{i})$ to determine $D_{r}(k_{r}^{i})$. Choose $T<\infty$ such that the following holds for all $i=1,\dots,n_{1}$:
\begin{align*}
\varphi_{0}^{T}(k_{0,+}^{i})&\in\B_{\varepsilon_{0}^{j_{i}}}(g_{0}^{j_{i}})\\
\varphi_{0}^{T}(k_{0,-}^{i})&\in\B_{\varepsilon_{0}^{l_{i}}}(g_{0}^{l_{i}}),
\end{align*}
where the $\varepsilon_{0}^{j_{i}},\varepsilon_{0}^{l_{i}}$ are chosen such that the properties \eqref{epsilonprop1} and \eqref{epsilonprop2} are satisfied. This must also hold for all cords created by splitting a cord according to relation~(iv) during its movement along the unstable manifold of $k_{r}^{i}$. Since, according to Lemma \ref{finitelymanyintersections}, these are only finitely many cords and all cords which are moved along the gradient flow reach a $\B_{\varepsilon_{0}^{i}}(g_{0}^{i})$ after finite time, such a $T$ exists. After possibly increasing $T$ it can be achieved that for all $r\in[-\varepsilon,\varepsilon],\ i=1,\dots,n_{1}$, and matching $\varepsilon_{r}^{j_{i}}$ and $\varepsilon_{r}^{l_{i}}$ the following holds (since $[-\varepsilon,\varepsilon]\times T^{2}$ is compact):
\begin{align*}
\varphi_{r}^{T}(k_{r,+}^{i})&\in\B_{\varepsilon_{r}^{j_{i}}}(g_{r}^{j_{i}})\\
\varphi_{r}^{T}(k_{r,-}^{i})&\in\B_{\varepsilon_{r}^{l_{i}}}(g_{r}^{l_{i}}).
\end{align*}
Lemma \ref{ExnbhdDxequalsDq} implies that there exists a $\bar{\delta}>0$ such that the following holds for all $i=1,\dots,n_{1}$ and all $x\in\B_{\bar{\delta}}(k_{0,\pm}^{i})$:
\[\widehat{D}_{0}(x)=\widehat{D}_{0}(k_{0,\pm}^{i}).\]
Let $\hat{\varepsilon}>0$ be small enough such that for all $r\in[-\varepsilon,\varepsilon]$ the properties \eqref{epsilonprop1relii} to \eqref{epsilonprop4reliv} required in the proof of Lemma \ref{ExnbhdDxequalsDq} are satisfied (which are formulated there for $\varepsilon_{F},\varepsilon_{B}$ and $\varepsilon_{S}$). Such an $\hat{\varepsilon}$ exists since the required properties are open and $[-\varepsilon,\varepsilon]$ is compact. The starting times are always $s_{0}=0$. Since the solution of a differential equation depends continuously on the initial conditions and the function, there exists a $\delta>0$ with $\delta<\bar{\delta}$ such that the following holds for all $i=1,\dots,n_{1}$ and all $x\in T^{2}$ with $\vert x-k_{0,\pm}^{i}\vert<\delta$ and $\vert X_{r}(x)-X_{0}(x)\vert<\delta$ for all $s\in[0,T]$:
\[\vert\varphi_{r}^{s}(x)-\varphi_{0}^{s}(k_{0,\pm}^{i})\vert<\hat{\varepsilon}.\]
If $\varepsilon$ is chosen small enough, it follows that the following holds for all $i=1,\dots,n_{1}$ and all $r\in[-\varepsilon,\varepsilon]$:
\[\widehat{D}_{r}(k_{r,\pm}^{i})=\Phi_{0,r}\circ\widehat{D}_{-\varepsilon}(k_{-\varepsilon,\pm}^{i}).\]
So we get for all $i=1,\dots,n_{1}$ and all $r\in[-\varepsilon,\varepsilon]$:
\begin{align*}
D_{r}(k_{r}^{i})&=\Phi_{0,r}\circ D_{-\varepsilon}(k_{-\varepsilon}^{i})\\
&=\Phi_{0,r}\circ D_{-\varepsilon}\circ\Phi_{1,r}^{-1}(k_{r}^{i}).
\end{align*}
Therefore, we have in case (i):
\[D_{r}=\Phi_{0,r}\circ D_{-\varepsilon}\circ\Phi_{1,r}^{-1}.\]
Thus, where $\langle M\rangle$ means the ideal generated by the set $M$: 
\begin{align*}
\Cord(K_{r})&=C_{0}(K_{r})/I_{r}\\
&=C_{0}(K_{r})/\langle D_{r}(C_{1}(K_{r}))\rangle\\
&=\Phi_{0,r}(C_{0}(K_{-\varepsilon}))/\langle\Phi_{0,r}\circ D_{-\varepsilon}\circ\Phi_{1,r}^{-1}(C_{1}(K_{r}))\rangle\\
&=\Phi_{0,r}(C_{0}(K_{-\varepsilon}))/\langle\Phi_{0,r}\circ D_{-\varepsilon}(C_{1}(K_{-\varepsilon}))\rangle\\
&\hspace*{-.04cm}\overset{(1)}{=}\Phi_{0,r}\Big(C_{0}(K_{-\varepsilon})/\langle D_{-\varepsilon}(C_{1}(K_{-\varepsilon}))\rangle\Big)\\
&=\Phi_{0,r}\big(C_{0}(K_{-\varepsilon})/I_{-\varepsilon}\big)\\
&=\Phi_{0,r}(\Cord(K_{-\varepsilon}))\\
&\cong\Cord(K_{-\varepsilon}).
\end{align*}
Equality (1) holds since $\Phi_{0,r}$ is an algebra isomorphism for all $r\in[-\varepsilon,\varepsilon]$.\\
So in case (i) the cord algebra remains the same up to a canonical isomorphism.\\[.5em]
Now we assume that $K_{r}$ is generic for all $r\in[-\varepsilon,\varepsilon]\setminus\lbrace0\rbrace$ and $K_{0}$ satisfies exactly one of the cases (ii,1) to (ii,14) of Lemma \ref{genericisotopy}. We want to show that $\Cord(K_{-\varepsilon})\cong\Cord(K_{\varepsilon})$ in each of these cases. Without loss of generality, we can number the critical points such that in the cases (ii,1) to (ii,10) $k_{0}^{1}$ is the critical point of index 1 that is mentioned in these statements. The above consideration concerning $D_{r}(k_{r}^{i})$ only holds for $i=2,\dots,n_{1}$ and we have to look at $D_{r}(k_{r}^{1})$. Furthermore, in the cases (ii,8) to (ii,10) we number the critical points such that $g_{0}^{1}$ is the critical point of index 0 that is mentioned in these statements.\\[.5em]
Case (ii,1): $W_{0}^{u}(k_{0}^{1})\ntransv B$. We consider the situation as in Figure \ref{WuntransvB}.
\begin{figure}[ht]\centering
\subfigure{\includegraphics[scale=0.3]{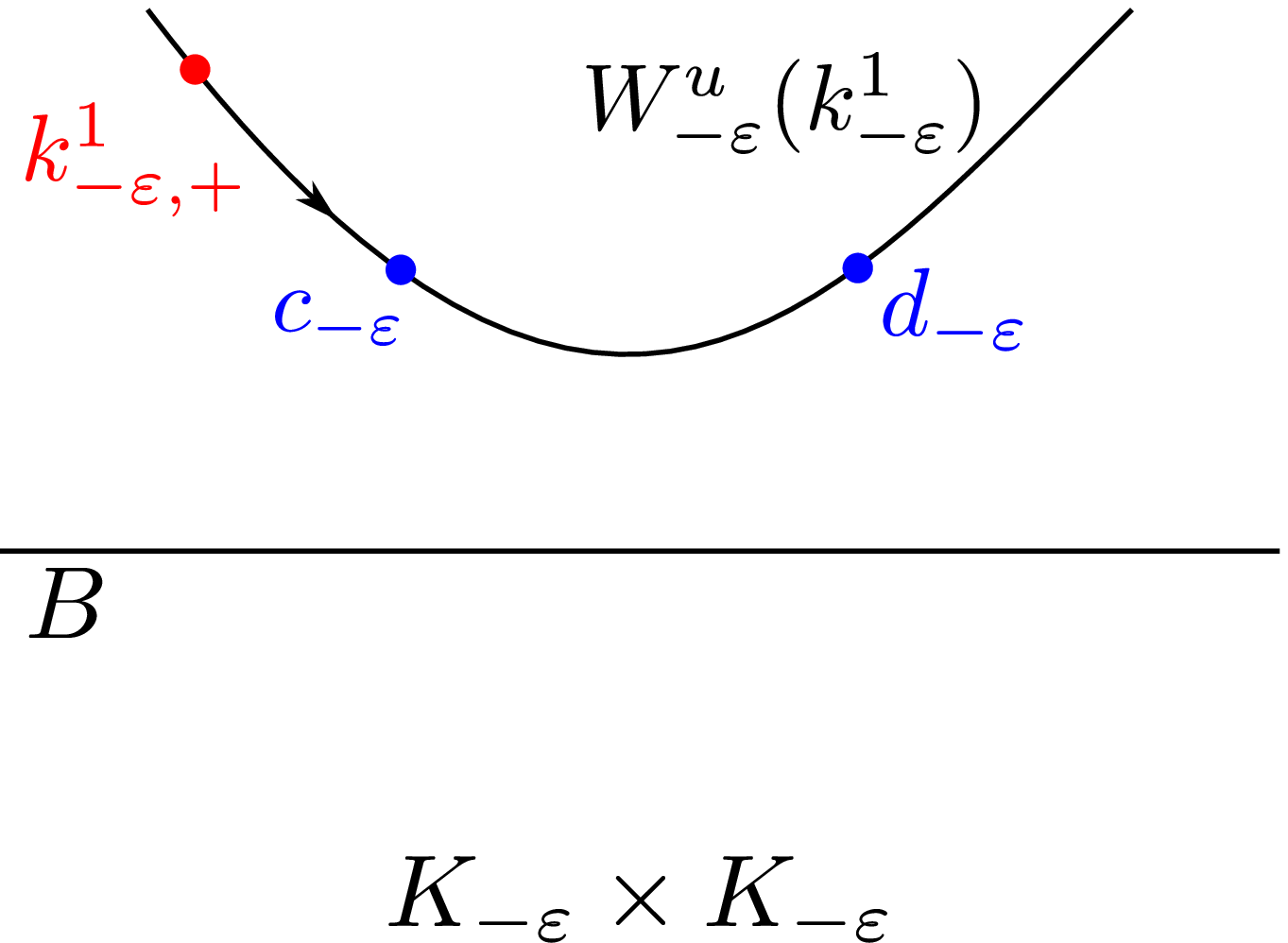}}
\subfigure{\hspace*{1cm}\includegraphics[scale=0.3]{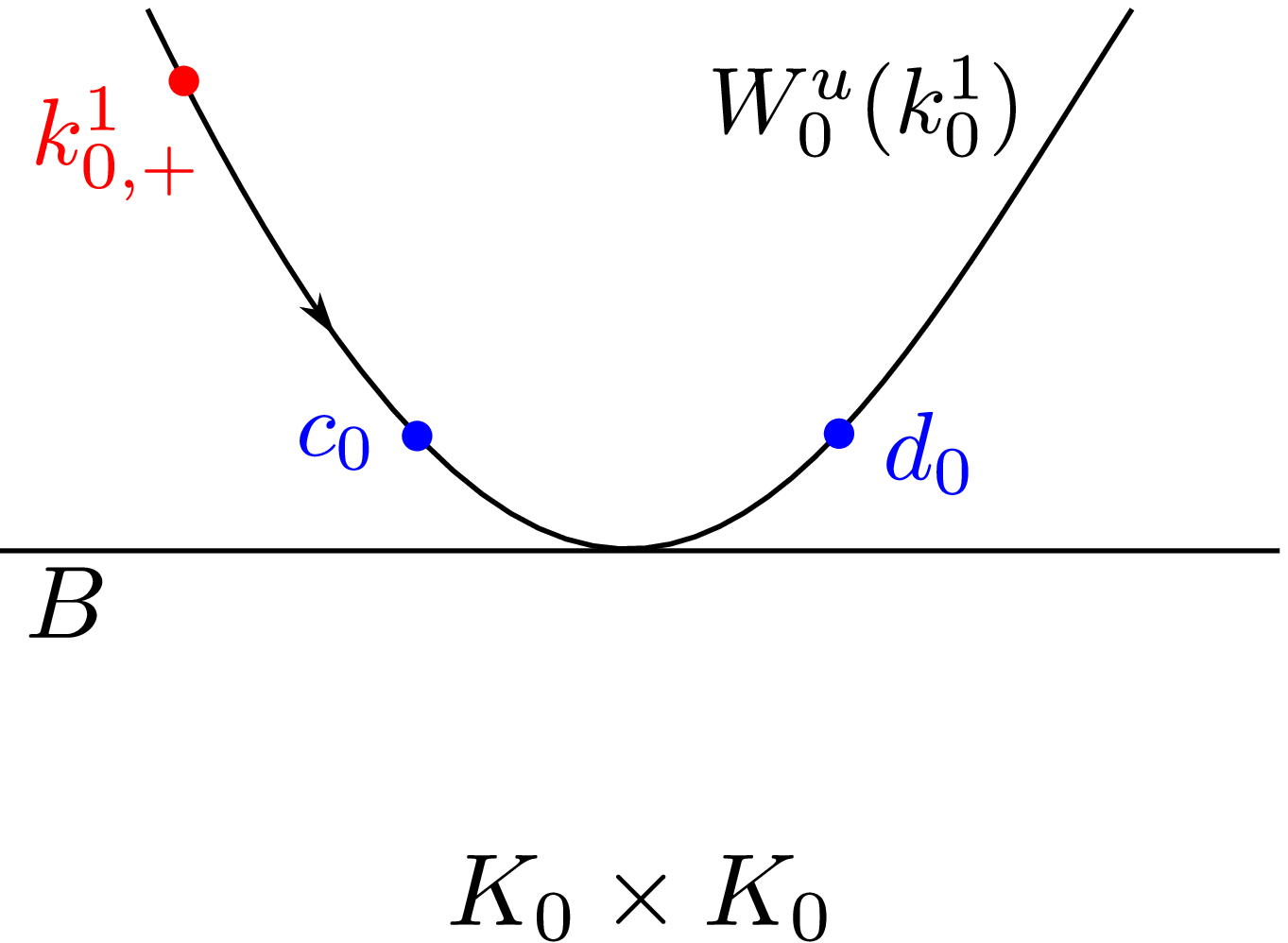}}
\subfigure{\hspace*{1.1cm}\includegraphics[scale=0.3]{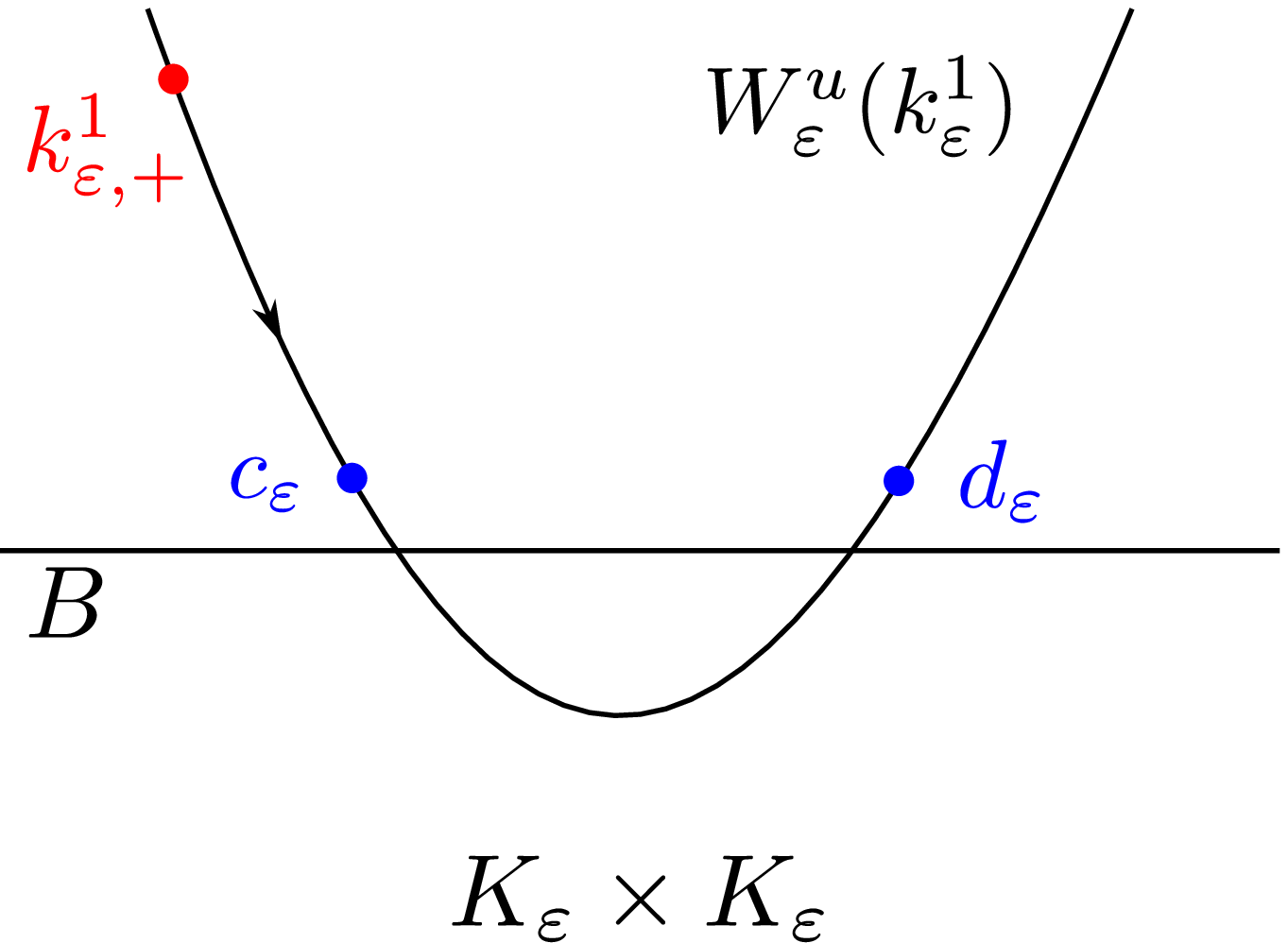}}
\caption{The unstable manifold of the cord $k_{0}^{1}$ is tangent to $B$\label{WuntransvB}}
\end{figure}
The following consideration holds analogously for the other possible situations: direction of the flow from the right to the left, $k_{r,-}^{1}$, and $W_{0}^{u}(k_{0}^{1})$ tangent to $(\lbrace0\rbrace\times K_{0})\subset B$.\\
For $r\in[-\varepsilon,\varepsilon]$ choose $c_{r},d_{r}$ as in Figure \ref{WuntransvB}. Then for $\varepsilon$ small enough we get on the one hand $\widehat{D}_{\varepsilon}(d_{\varepsilon})=\Phi_{0,\varepsilon}\circ\widehat{D}_{-\varepsilon}(d_{-\varepsilon})$ and on the other hand $\widehat{D}_{r}(c_{r})=\widehat{D}_{r}(d_{r})$ for all $r\in[-\varepsilon,0)$ and
\begin{align*}
\widehat{D}_{r}(c_{r})&\overset{\text{rel. (iii)}}{=}\widehat{D}_{r}(d_{r})\lambda\lambda^{-1}\\
&\hspace*{.36cm}=\widehat{D}_{r}(d_{r})
\end{align*}
for all $r\in(0,\varepsilon]$, since the endpoint of the cord intersects the base point once in the reversed direction of the orientation of the knot and once in the direction of the orientation of the knot. Thus, we get
\[D_{\varepsilon}(k_{\varepsilon}^{1})=\Phi_{0,\varepsilon}\circ D_{-\varepsilon}\circ\Phi_{1,\varepsilon}^{-1}(k_{\varepsilon}^{1}).\]
It follows that
\[D_{\varepsilon}=\Phi_{0,\varepsilon}\circ D_{-\varepsilon}\circ\Phi_{1,\varepsilon}^{-1}.\]
By the analogous computation as in case (i) with $r=\varepsilon$ we get $\Cord(K_{\varepsilon})\cong\Cord(K_{-\varepsilon})$.\\[.5em]
Case (ii,2): $W_{0}^{u}(k_{0}^{1})\ntransv S_{0}$. We consider the situation as in Figure \ref{WuntransvS}.
\begin{figure}[ht]\centering 
\subfigure{\includegraphics[scale=0.3]{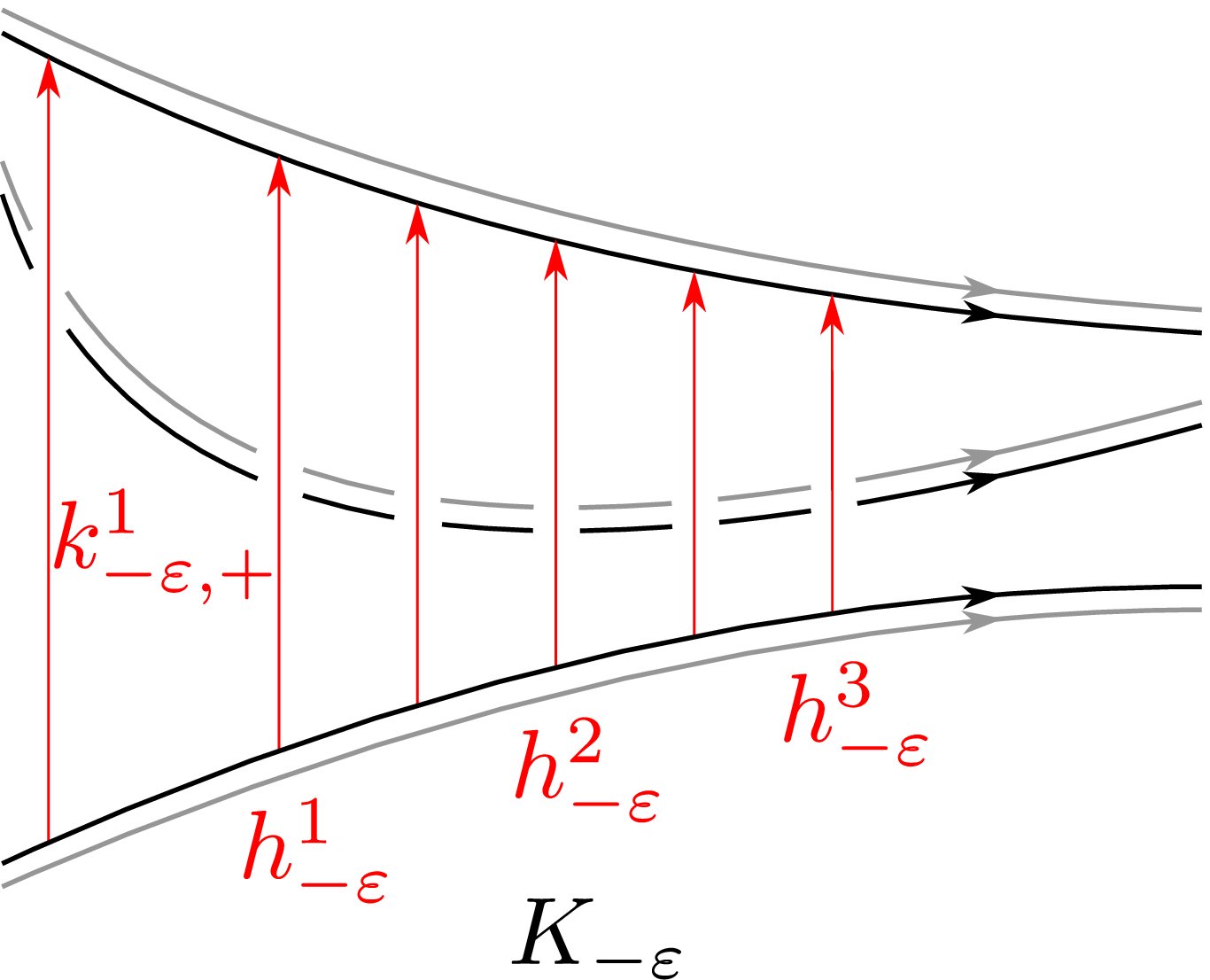}}
\subfigure{\hspace*{1cm}\includegraphics[scale=0.3]{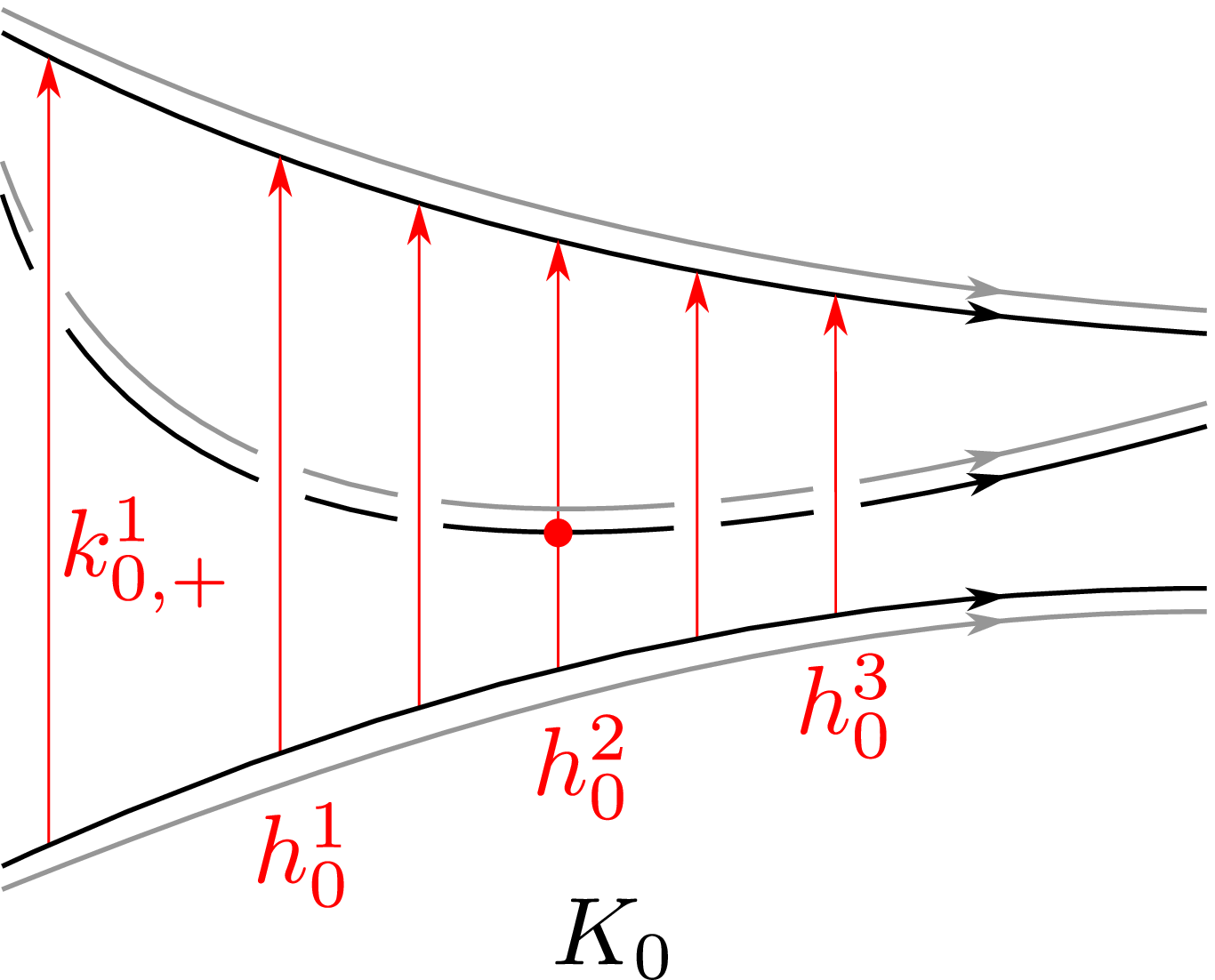}}
\subfigure{\hspace*{1.1cm}\includegraphics[scale=0.3]{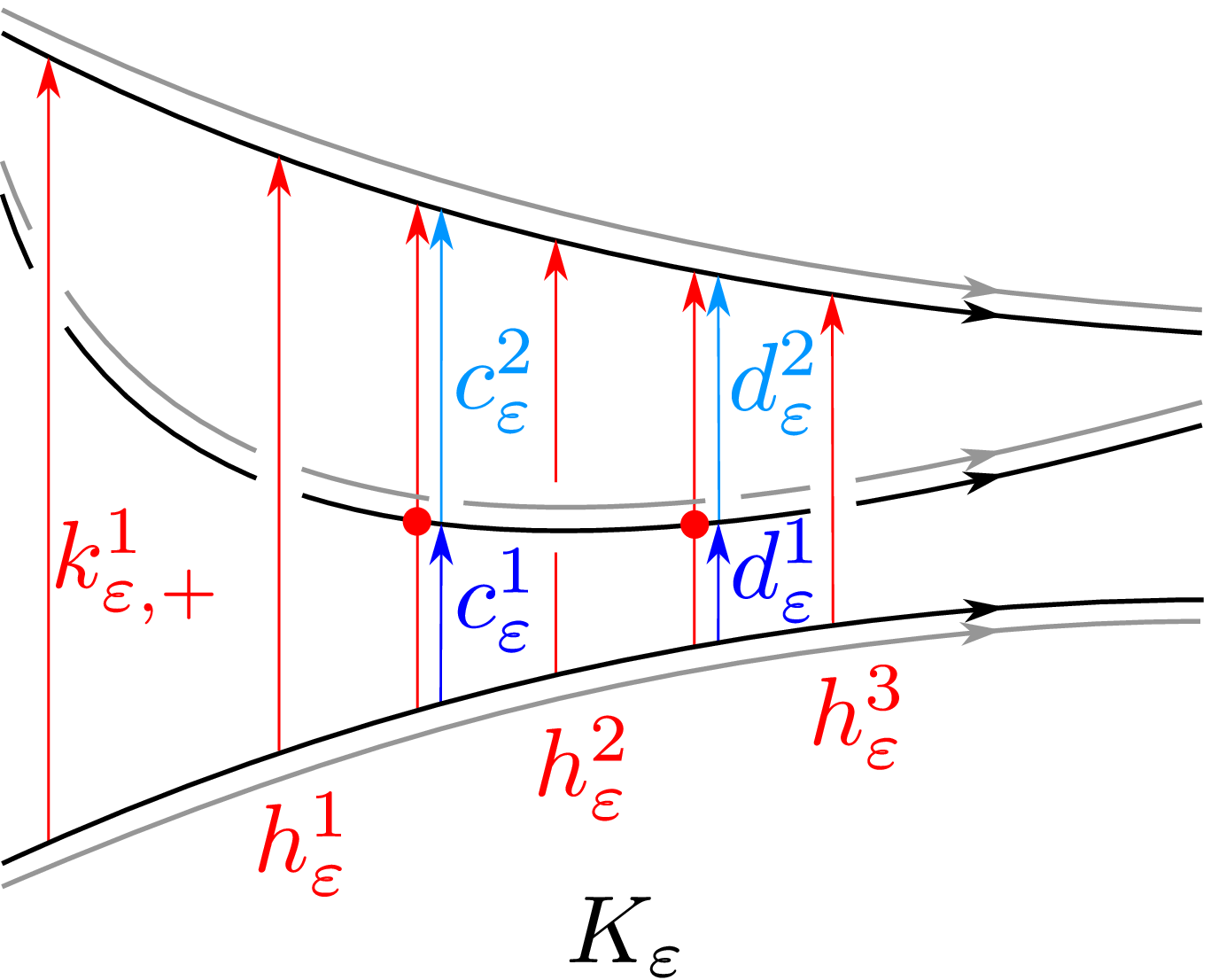}}
\caption{The unstable manifold of the cord $k_{0}^{1}$ is tangent to $S_{0}$\label{WuntransvS}}
\end{figure}
The cords labelled $h_{r}^{i}$, for $i=1,2,3,$ lie on the unstable manifold of $k_{r}^{1}$ such that $\varphi_{r}^{t_{i}}(k_{r,+}^{1})=h_{r}^{i}, t_{i}>0,$ and $t_{i}>t_{j}$ if $i>j$. (We will use this notation in the following cases, too.) For $r\in(0,\varepsilon]$ the cord intersects the knot in its interior at two different times when flowing along the negative gradient, and we get $c_{r}^{1}, c_{r}^{2}, d_{r}^{1}$, and $d_{r}^{2}$ by the application of relation (iv). According to Lemma \ref{genericisotopy}, we can assume that there is no citical point between $c_{r}^{1}$ and $d_{r}^{1}$ and between $c_{r}^{2}$ and $d_{r}^{2}$, otherwise we would have also the case (ii,13), but this would be a contradiction to this Lemma since only one of these cases occurs at any one time. So if $\varepsilon$ is chosen small enough, we have $\widehat{D}_{r}(c_{r}^{1})=\widehat{D}_{r}(d_{r}^{1})$ and $\widehat{D}_{r}(c_{r}^{2})=\widehat{D}_{r}(d_{r}^{2})$ for $r>0$. Thus, we get
\begin{align*}
\widehat{D}_{r}(h_{r}^{1})&\overset{\text{rel. (iv)}}{=}\widehat{D}_{r}(h_{r}^{2})+\widehat{D}_{r}(c_{r}^{1})\widehat{D}_{r}(c_{r}^{2})\\
&\overset{\text{rel. (iv)}}{=}\widehat{D}_{r}(h_{r}^{3})-\widehat{D}_{r}(d_{r}^{1})\widehat{D}_{r}(d_{r}^{2})+\widehat{D}_{r}(c_{r}^{1})\widehat{D}_{r}(c_{r}^{2})\\
&\hspace*{.36cm}=\widehat{D}_{r}(h_{r}^{3}).
\end{align*}
As in case (ii,1) we get $D_{\varepsilon}(k_{\varepsilon}^{1})=\Phi_{0,\varepsilon}\circ D_{-\varepsilon}\circ\Phi_{1,\varepsilon}^{-1}(k_{\varepsilon}^{1}), D_{\varepsilon}=\Phi_{0,\varepsilon}\circ D_{-\varepsilon}\circ\Phi_{1,\varepsilon}^{-1}$ and thus by the analogous computation as above $\Cord(K_{\varepsilon})\cong\Cord(K_{-\varepsilon})$.\\
The other possible situations are to be treated analogously.\\[.5em]
Case (ii,3): $W_{0}^{u}(k_{0}^{1})\ntransv F_{0}$: Analogous to case (ii,1) with relation (ii) instead of relation (iii), i.e. $\mu^{\pm1}$ instead of $\lambda^{\pm1}$.\\[.5em]
Case (ii,4): $W^{u}_{0}(k_{0}^{1})\cap\partial S_{0}\neq\emptyset$ and $W^{u}_{0}(k_{0}^{1})\cap\partial F_{0}\neq\emptyset$. We consider the situation as in Figure \ref{WucapbSbFnempty}.
\begin{figure}[htpb]\centering
\subfigure{\includegraphics[scale=0.3]{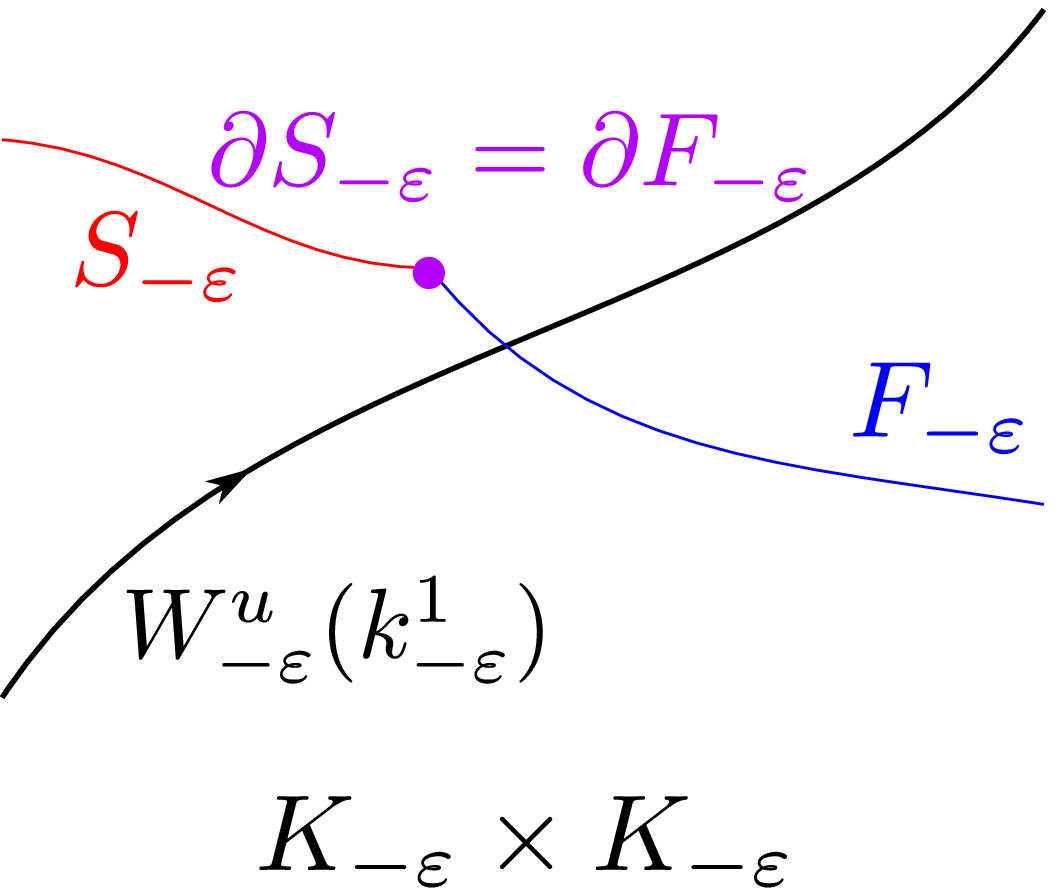}}
\subfigure{\hspace*{1.5cm}\includegraphics[scale=0.3]{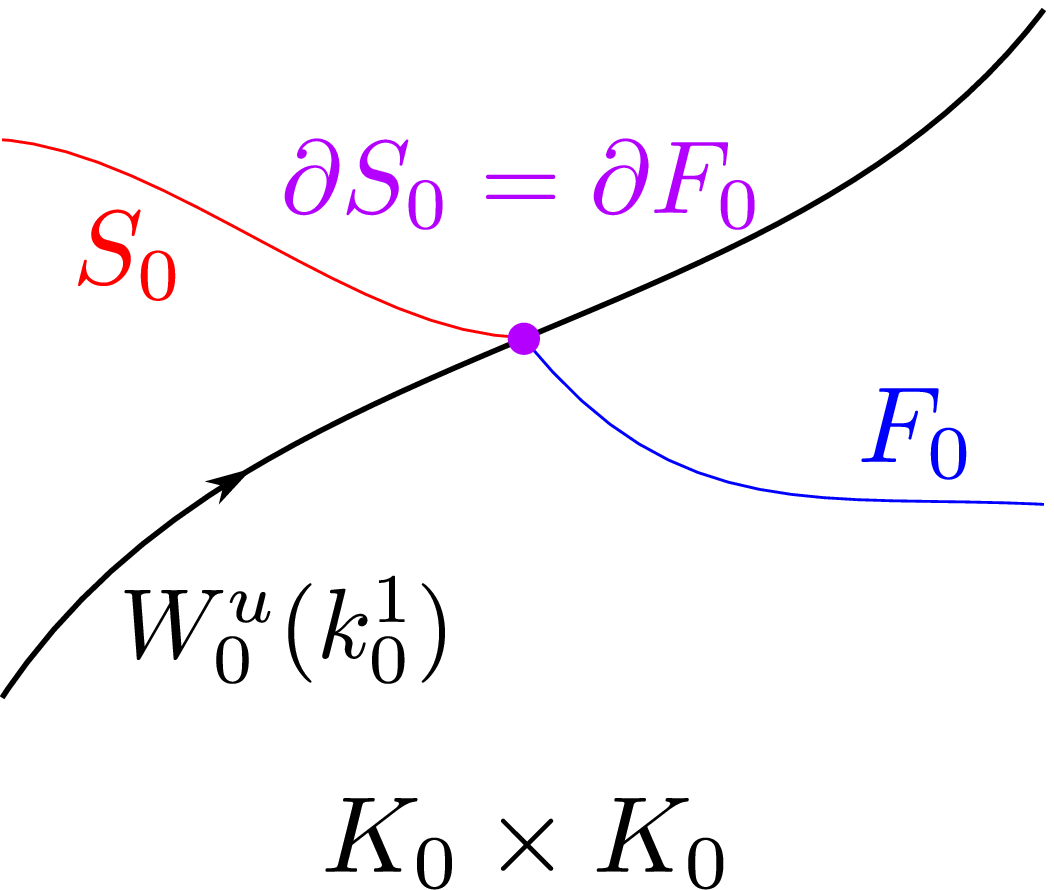}}
\subfigure{\hspace*{1.5cm}\includegraphics[scale=0.3]{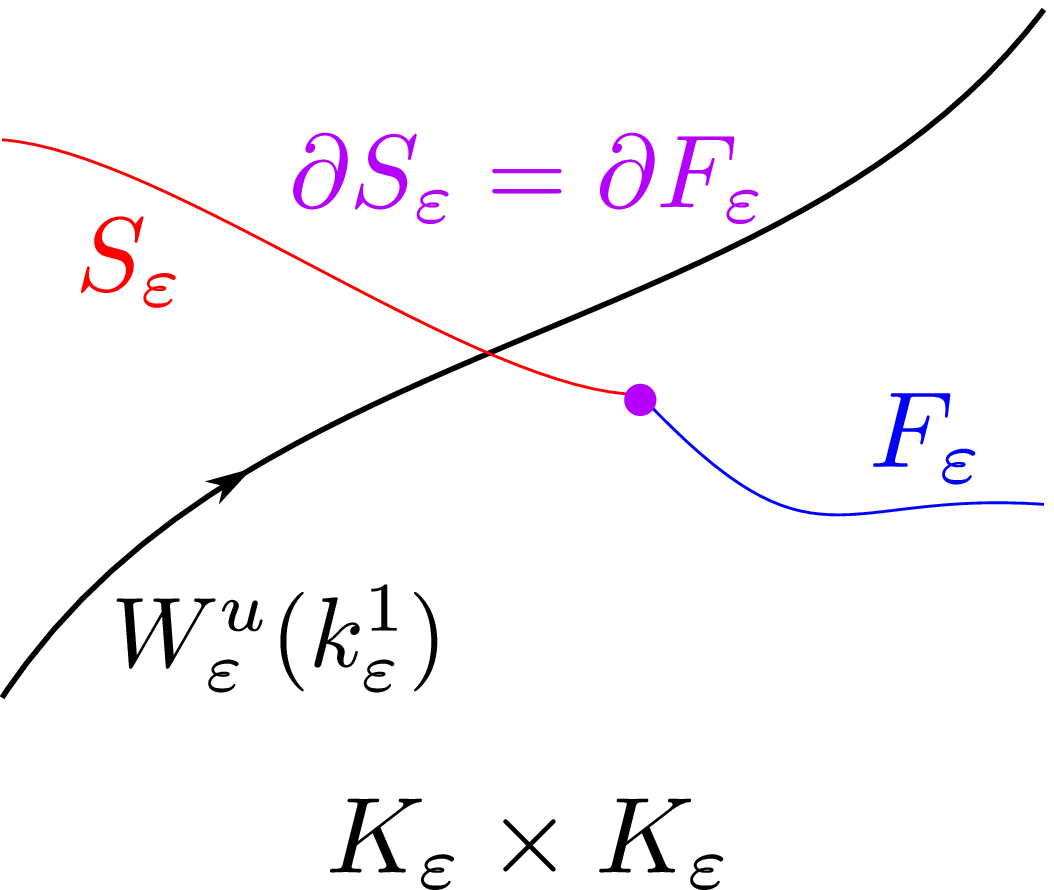}}
\subfigure{\includegraphics[scale=0.3]{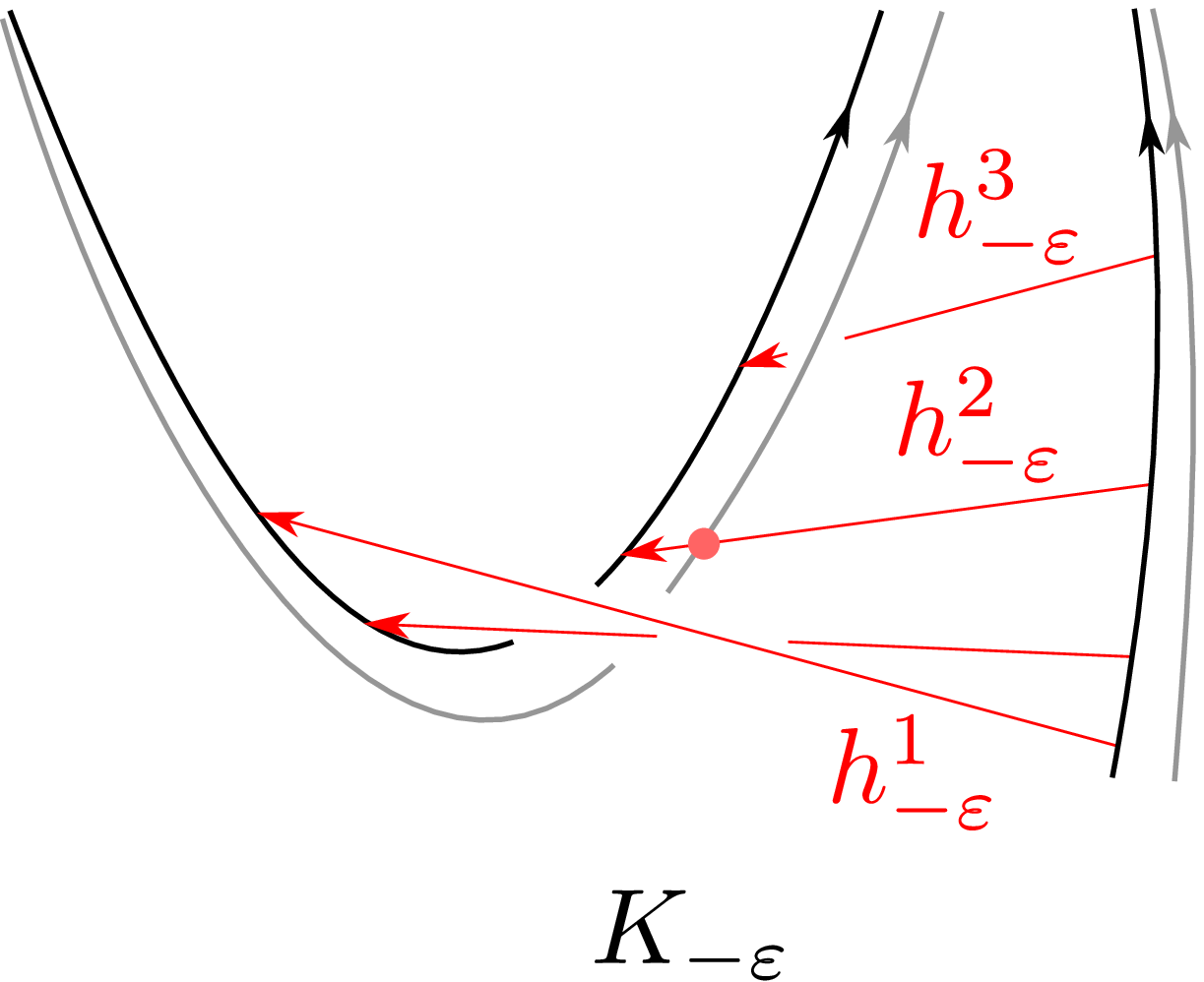}}
\subfigure{\hspace*{1cm}\includegraphics[scale=0.3]{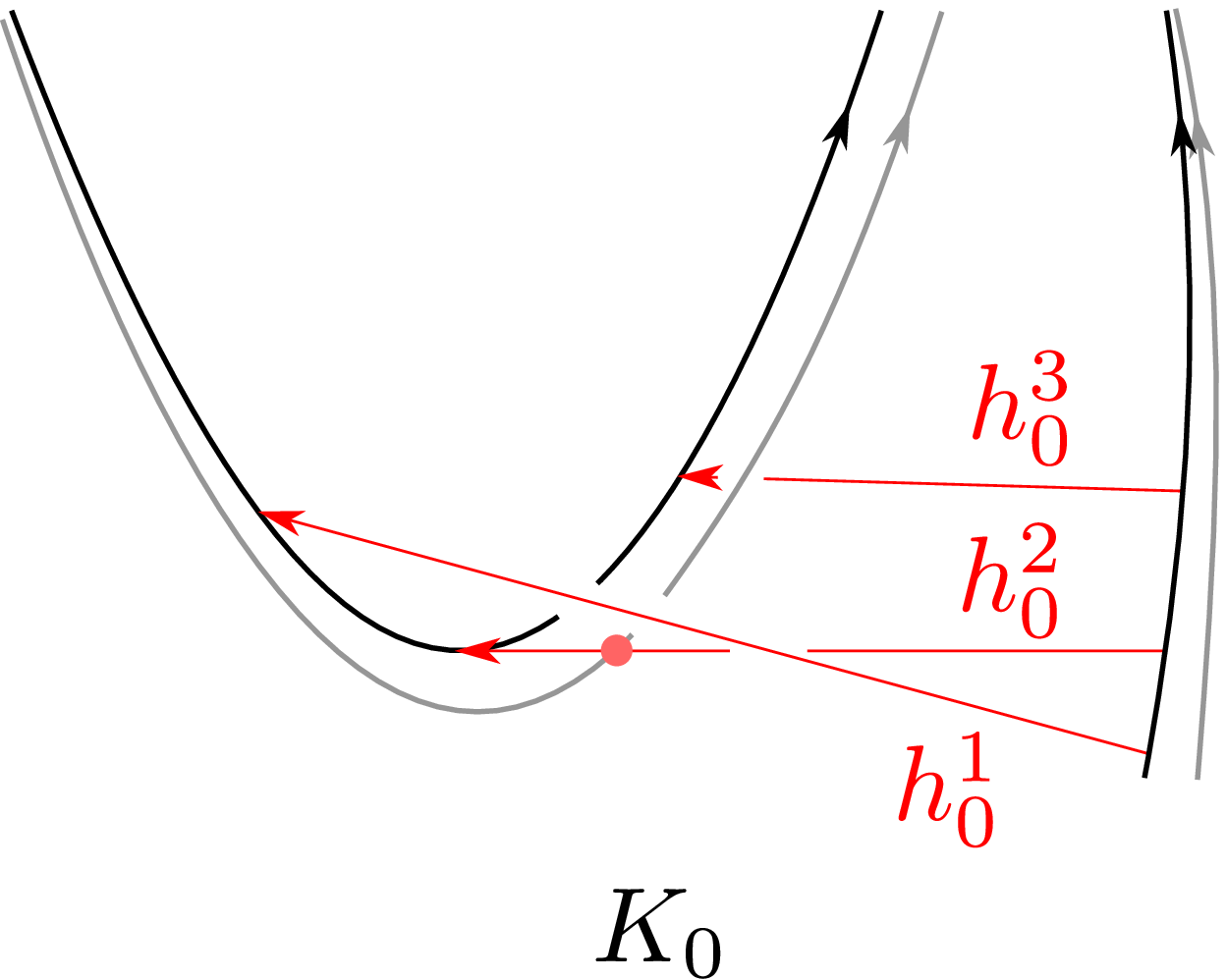}}
\subfigure{\hspace*{1cm}\includegraphics[scale=0.3]{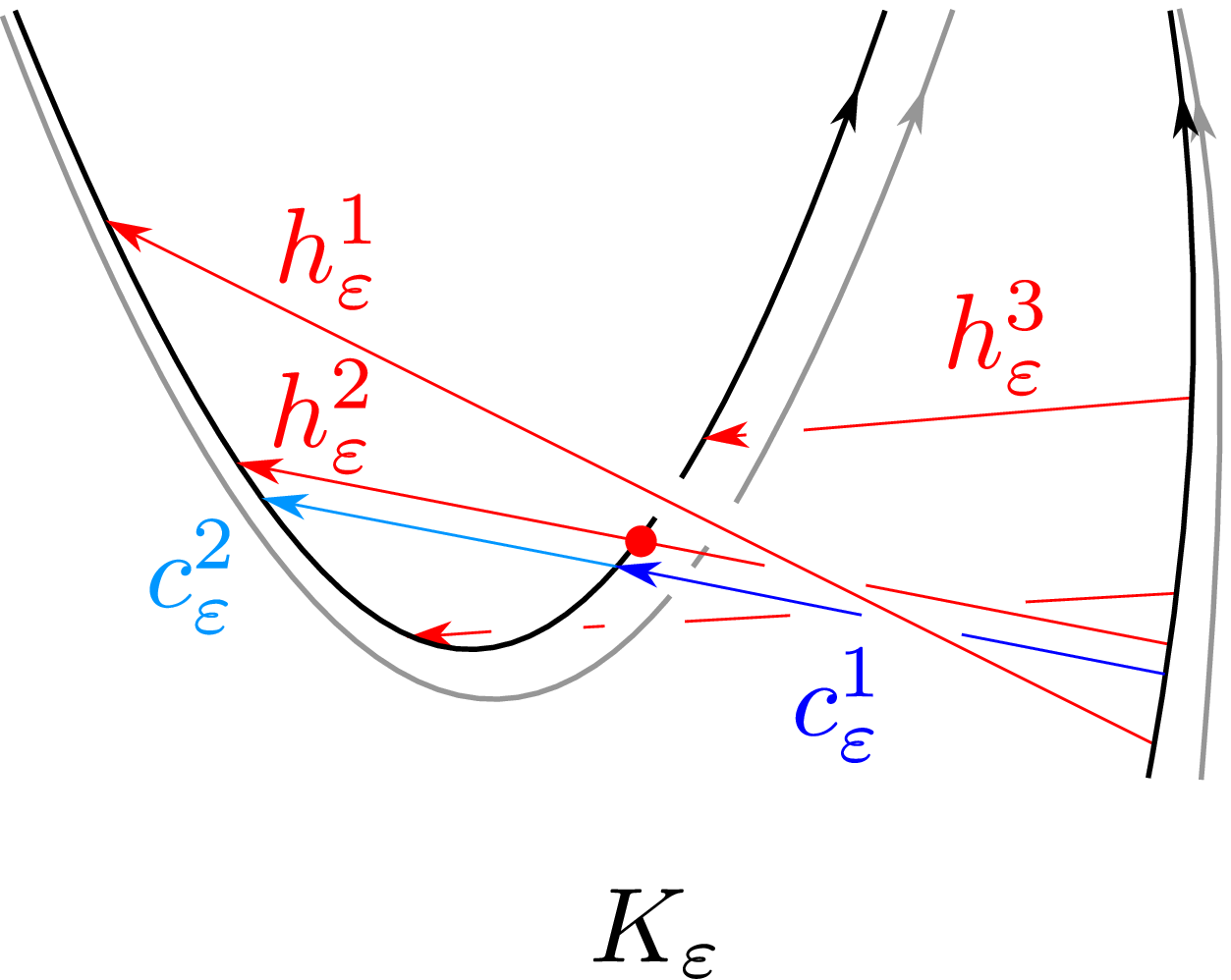}}
\caption{The unstable manifold of the cord $k_{0}^{1}$ intersects $\partial S_{0}$ and $\partial F_{0}$\label{WucapbSbFnempty}}
\end{figure}

As we can see from the figure, we get $\widehat{D}_{r}(h_{r}^{1})\overset{\text{rel. (ii)}}{=}\widehat{D}_{r}(h_{r}^{3})\mu^{-1}$ for $r<0$ and $\widehat{D}_{r}(h_{r}^{3})\overset{\text{rel. (ii)}}{=}\widehat{D}_{r}(c_{r}^{1})\mu$ for $r>0$. If $\varepsilon$ is chosen small enough, we have $\widehat{D}_{r}(h_{r}^{3})=\Phi_{0,r}\circ\widehat{D}_{-\varepsilon}(h_{-\varepsilon}^{3})$. Now we can compute for $r>0$:

\begin{align*}
\widehat{D}_{r}(h_{r}^{1})&\overset{\text{rel. (iv)}}{=}\widehat{D}_{r}(h_{r}^{3})+\widehat{D}_{r}(c_{r}^{1})\widehat{D}_{r}(c_{r}^{2})\\
&\hspace*{.07cm}\overset{\text{rel. (i)}}{=}\widehat{D}_{r}(h_{r}^{3})+\widehat{D}_{r}(h_{r}^{3})\mu^{-1}(1-\mu)\\
&\hspace{.34cm}=\widehat{D}_{r}(h_{r}^{3})\mu^{-1}\\
&\hspace*{.34cm}=\Phi_{0,r}\circ\widehat{D}_{-\varepsilon}(h_{-\varepsilon}^{3})\mu^{-1}\\
&\hspace*{.34cm}=\Phi_{0,r}\circ\widehat{D}_{-\varepsilon}(h_{-\varepsilon}^{1}).
\end{align*}
As in case (ii,1) we get $D_{\varepsilon}(k_{\varepsilon}^{1})=\Phi_{0,\varepsilon}\circ D_{-\varepsilon}\circ\Phi_{1,\varepsilon}^{-1}(k_{\varepsilon}^{1}), D_{\varepsilon}=\Phi_{0,\varepsilon}\circ D_{-\varepsilon}\circ\Phi_{1,\varepsilon}^{-1}$ and thus by the analogous computation as above $\Cord(K_{\varepsilon})\cong\Cord(K_{-\varepsilon})$.\\
The other possible situations are to be treated analogously.\\[.5em]
Case (ii,5): $W^{u}_{0}(k_{0}^{1})\cap S_{2,0}\neq\emptyset$. We consider the situation as in Figure \ref{WucapS2nempty}.
\begin{figure}[htpb]\centering
\subfigure{\includegraphics[scale=0.3]{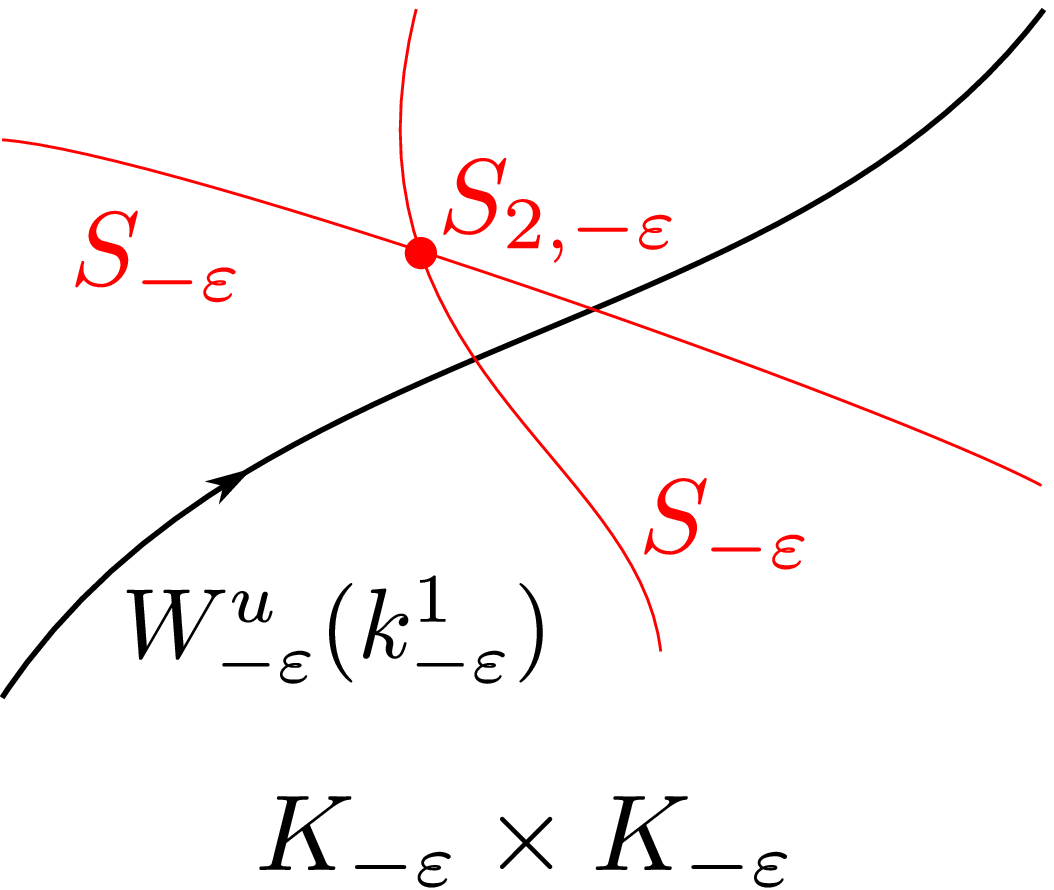}}
\subfigure{\hspace*{1.7cm}\includegraphics[scale=0.3]{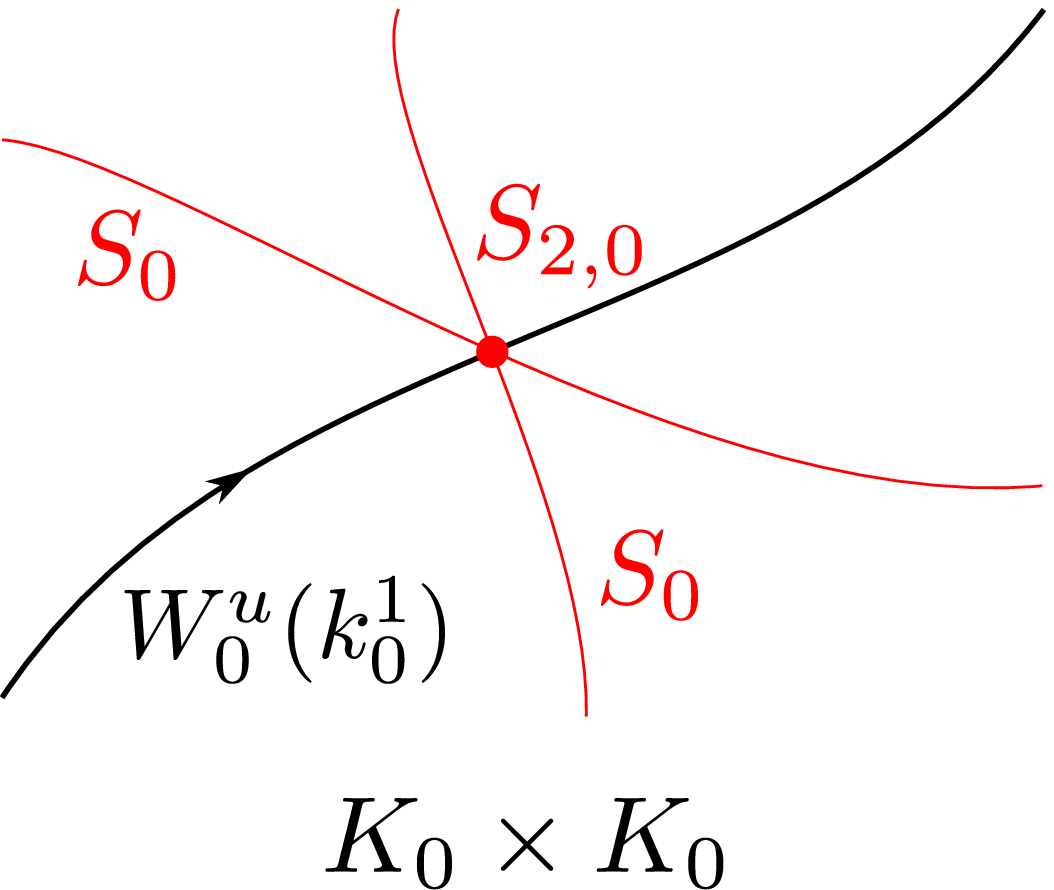}}
\subfigure{\hspace*{1.7cm}\includegraphics[scale=0.3]{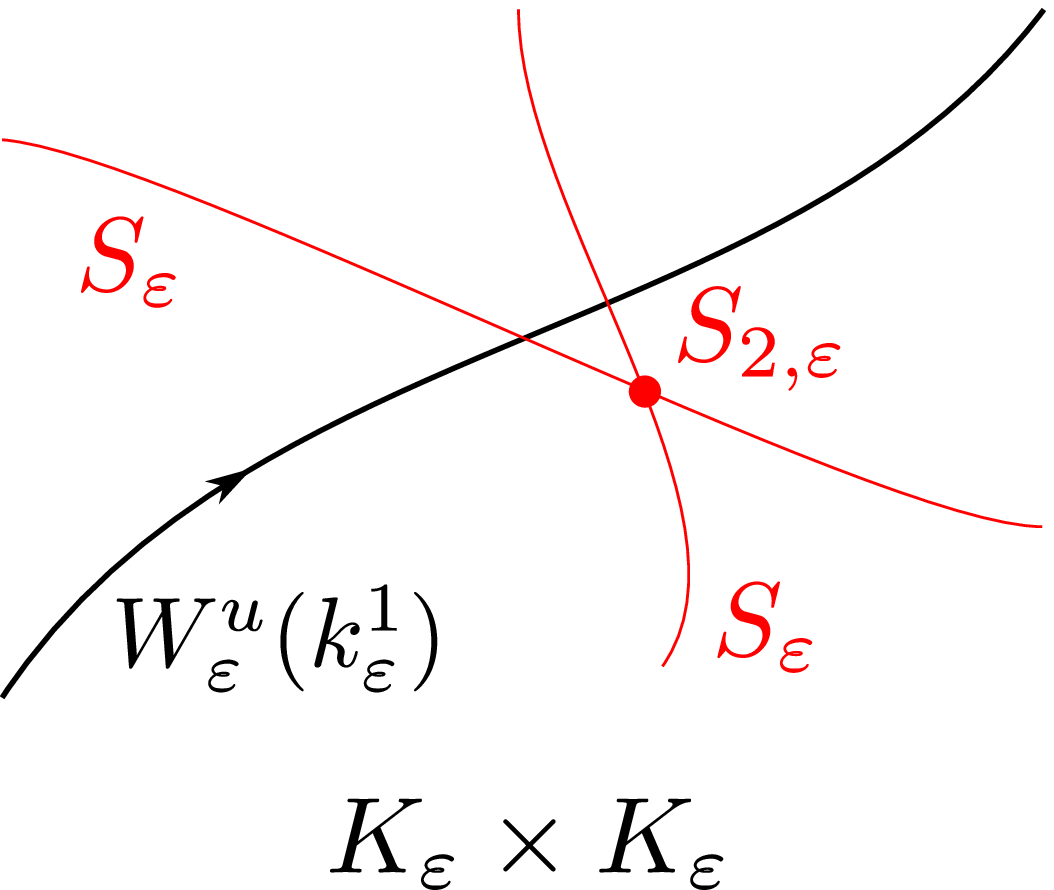}}
\subfigure{\includegraphics[scale=0.3]{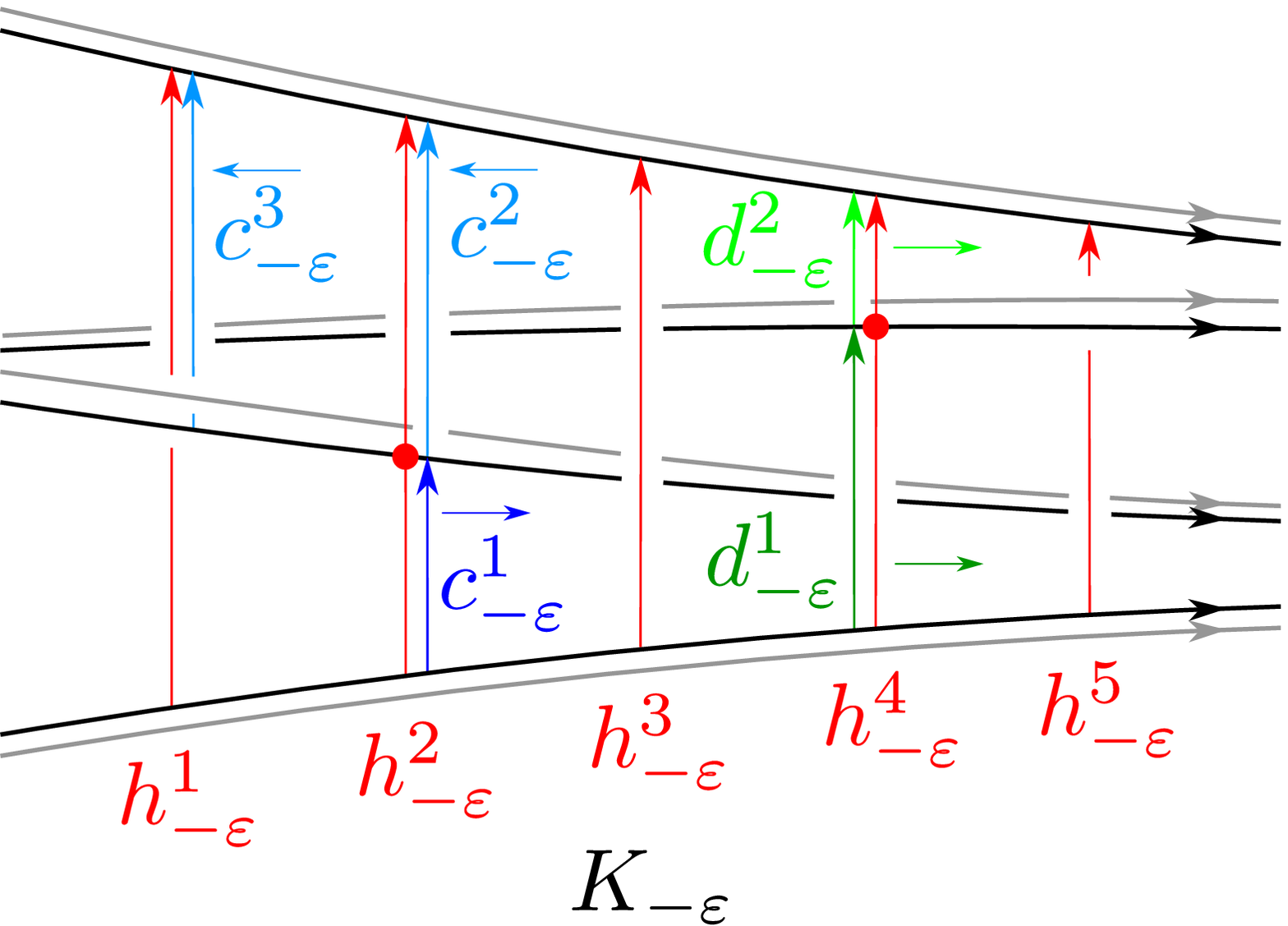}}
\subfigure{\hspace*{1cm}\includegraphics[scale=0.3]{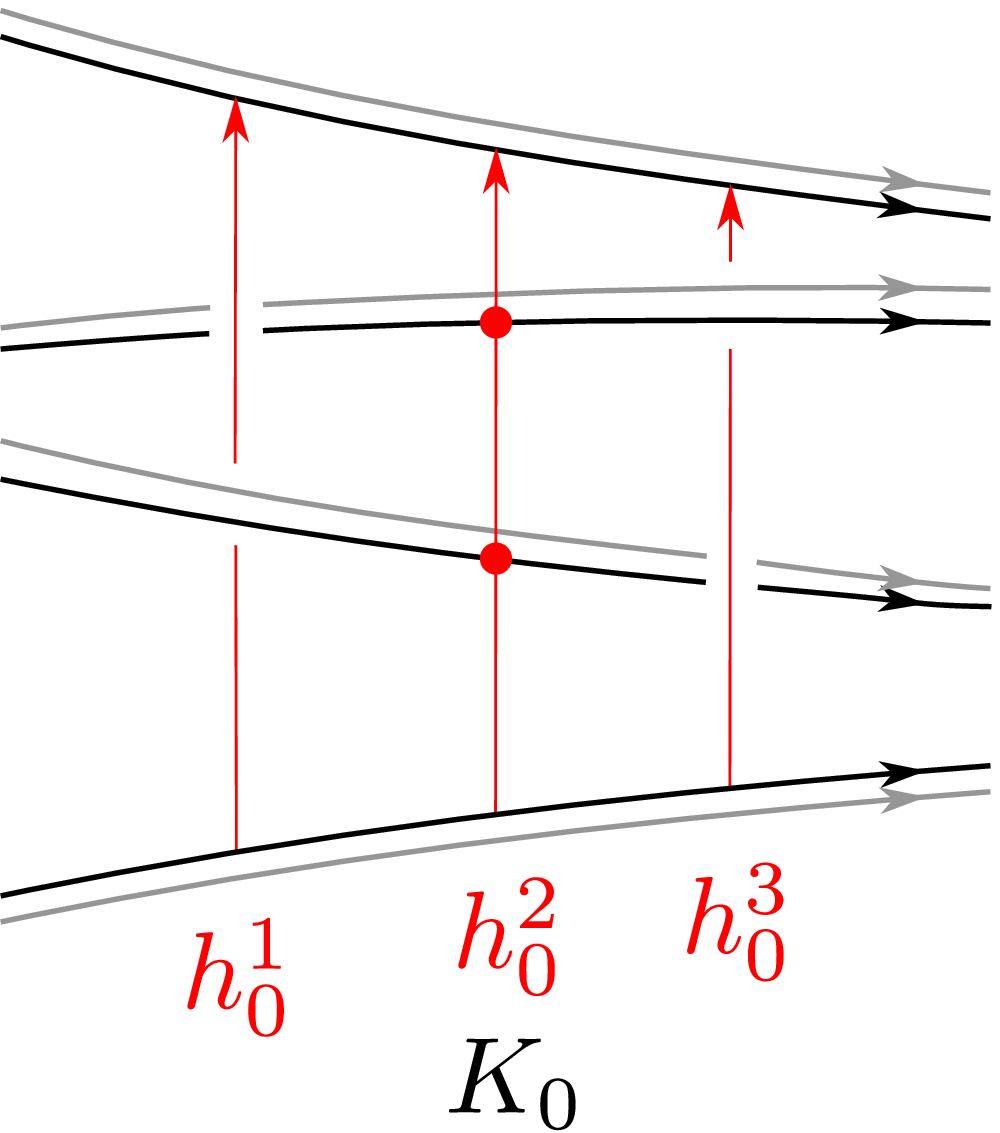}}
\subfigure{\hspace*{1cm}\includegraphics[scale=0.3]{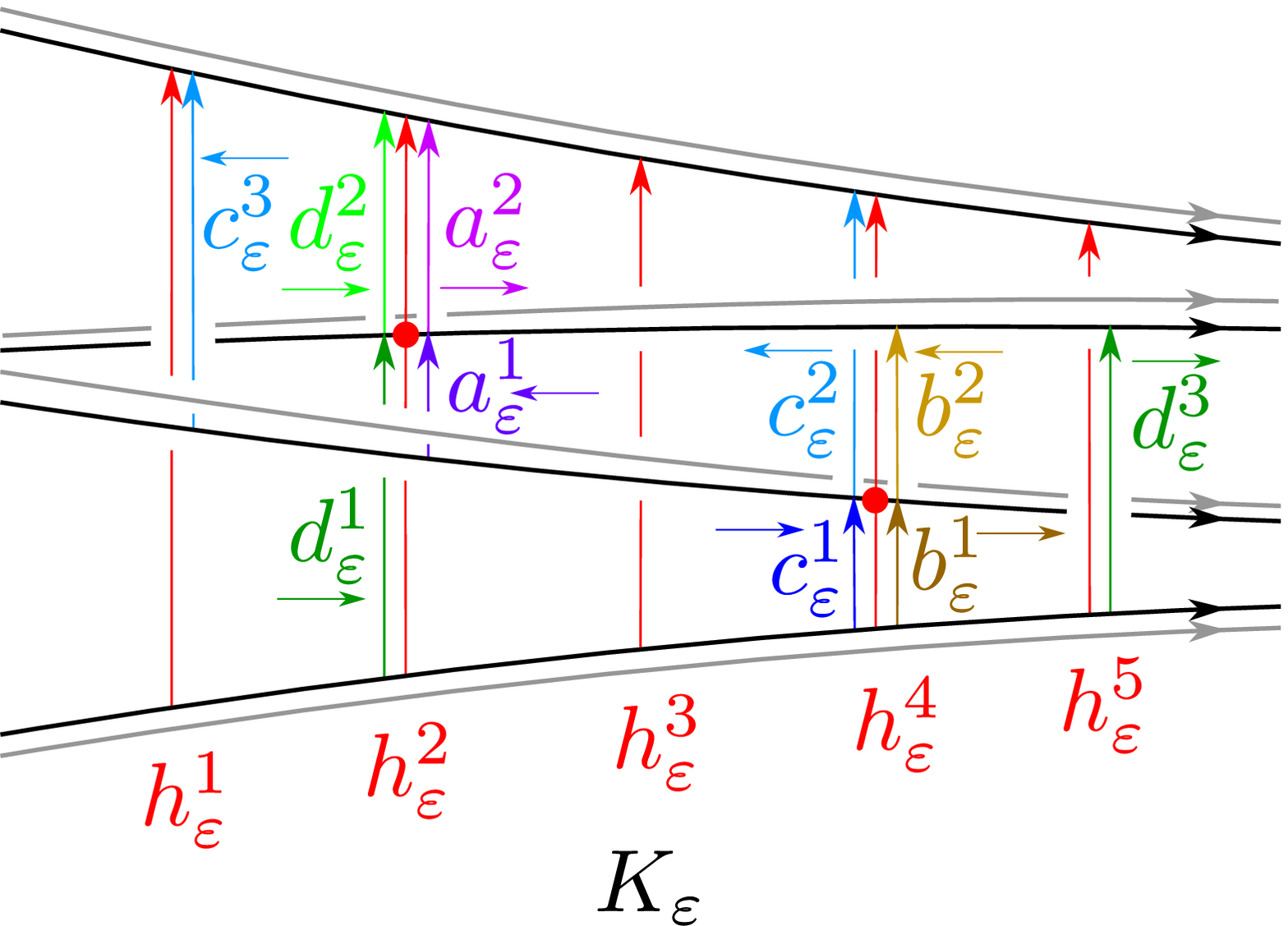}}
\caption{The unstable manifold of the cord $k_{0}^{1}$ intersects $S_{2,0}$\label{WucapS2nempty}}
\end{figure}
As in case (ii,2) we can assume that there are no critical points in the shown region. The small horizontal arrows indicate the direction in which the cords move along the gradient vector field. From the figure we get% the following equations:
\begin{align*}
\widehat{D}_{-\varepsilon}(c_{-\varepsilon}^{2})&\overset{\text{rel. (ii)}}{=}\mu^{-1}\widehat{D}_{-\varepsilon}(c_{-\varepsilon}^{3})&\widehat{D}_{\varepsilon}(a_{\varepsilon}^{2})&=\widehat{D}_{\varepsilon}(d_{\varepsilon}^{2})\\
\widehat{D}_{\varepsilon}(b_{\varepsilon}^{2})&\overset{\text{rel. (ii)}}{=}\mu^{-1}\widehat{D}_{\varepsilon}(a_{\varepsilon}^{1})&\widehat{D}_{\varepsilon}(b_{\varepsilon}^{1})&=\widehat{D}_{\varepsilon}(c_{\varepsilon}^{1}).\\
%\widehat{D}_{\varepsilon}(a_{\varepsilon}^{2})&=\widehat{D}_{\varepsilon}(d_{\varepsilon}^{2})\\
%\widehat{D}_{\varepsilon}(b_{\varepsilon}^{1})&=\widehat{D}_{\varepsilon}(c_{\varepsilon}^{1}).
\end{align*}
If $\varepsilon$ is chosen small enough, the following holds:
\begin{align*}
\widehat{D}_{\varepsilon}(c_{\varepsilon}^{1})&=\Phi_{0,\varepsilon}\circ\widehat{D}_{-\varepsilon}(c_{-\varepsilon}^{1})\\
\widehat{D}_{\varepsilon}(d_{\varepsilon}^{2})&=\Phi_{0,\varepsilon}\circ\widehat{D}_{-\varepsilon}(d_{-\varepsilon}^{2})\\
\widehat{D}_{\varepsilon}(d_{\varepsilon}^{3})&=\Phi_{0,\varepsilon}\circ\widehat{D}_{-\varepsilon}(d_{-\varepsilon}^{1})\\
\widehat{D}_{\varepsilon}(c_{\varepsilon}^{3})&=\Phi_{0,\varepsilon}\circ\widehat{D}_{-\varepsilon}(c_{-\varepsilon}^{3})\\
\widehat{D}_{\varepsilon}(h_{\varepsilon}^{5})&=\Phi_{0,\varepsilon}\circ\widehat{D}_{-\varepsilon}(h_{-\varepsilon}^{5}).
\end{align*}
Now we can compute
\begingroup
\allowdisplaybreaks
\begin{align*}
\widehat{D}_{-\varepsilon}(h_{-\varepsilon}^{1})&\overset{\text{rel. (iv)}}{=}\widehat{D}_{-\varepsilon}(h_{-\varepsilon}^{3})-\widehat{D}_{-\varepsilon}(c_{-\varepsilon}^{1})\widehat{D}_{-\varepsilon}(c_{-\varepsilon}^{2})\\
&\overset{\text{rel. (iv)}}{=}\widehat{D}_{-\varepsilon}(h_{-\varepsilon}^{5})+\widehat{D}_{-\varepsilon}(d_{-\varepsilon}^{1})\widehat{D}_{-\varepsilon}(d_{-\varepsilon}^{2})-\widehat{D}_{-\varepsilon}(c_{-\varepsilon}^{1})\mu^{-1}\widehat{D}_{-\varepsilon}(c_{-\varepsilon}^{3})\\
\widehat{D}_{\varepsilon}(h_{\varepsilon}^{1})&\overset{\text{rel. (iv)}}{=}\widehat{D}_{\varepsilon}(h_{\varepsilon}^{3})+\widehat{D}_{\varepsilon}(d_{\varepsilon}^{1})\widehat{D}_{\varepsilon}(d_{\varepsilon}^{2})\\
&\overset{\text{rel. (iv)}}{=}\widehat{D}_{\varepsilon}(h_{\varepsilon}^{5})-\widehat{D}_{\varepsilon}(c_{\varepsilon}^{1})\widehat{D}_{\varepsilon}(c_{\varepsilon}^{2})+\widehat{D}_{\varepsilon}(d_{\varepsilon}^{1})\widehat{D}_{\varepsilon}(d_{\varepsilon}^{2})\\
&\overset{\text{rel. (iv)}}{=}\widehat{D}_{\varepsilon}(h_{\varepsilon}^{5})-\widehat{D}_{\varepsilon}(c_{\varepsilon}^{1})\left(\mu^{-1}\widehat{D}_{\varepsilon}(c_{\varepsilon}^{3})-\mu^{-1}\widehat{D}_{\varepsilon}(a_{\varepsilon}^{1})\widehat{D}_{\varepsilon}(a_{\varepsilon}^{2})\right)\\
&\hspace*{1.1cm}+\left(\widehat{D}_{\varepsilon}(d_{\varepsilon}^{3})-\widehat{D}_{\varepsilon}(b_{\varepsilon}^{1})\widehat{D}_{\varepsilon}(b_{\varepsilon}^{2})\right)\widehat{D}_{\varepsilon}(d_{\varepsilon}^{2})\\
&\hspace*{.37cm}=\widehat{D}_{\varepsilon}(h_{\varepsilon}^{5})-\widehat{D}_{\varepsilon}(c_{\varepsilon}^{1})\mu^{-1}\widehat{D}_{\varepsilon}(c_{\varepsilon}^{3})+\widehat{D}_{\varepsilon}(b_{\varepsilon}^{1})\widehat{D}_{\varepsilon}(b_{\varepsilon}^{2})\widehat{D}_{\varepsilon}(d_{\varepsilon}^{2})+\widehat{D}_{\varepsilon}(d_{\varepsilon}^{3})\widehat{D}_{\varepsilon}(d_{\varepsilon}^{2})\\
&\hspace*{1.1cm}-\widehat{D}_{\varepsilon}(b_{\varepsilon}^{1})\widehat{D}_{\varepsilon}(b_{\varepsilon}^{2})\widehat{D}_{\varepsilon}(d_{\varepsilon}^{2})\\
&\hspace*{.37cm}=\Phi_{0,\varepsilon}\circ\left(\widehat{D}_{-\varepsilon}(h_{-\varepsilon}^{5})-\widehat{D}_{-\varepsilon}(c_{-\varepsilon}^{1})\mu^{-1}\widehat{D}_{-\varepsilon}(c_{-\varepsilon}^{3})+\widehat{D}_{-\varepsilon}(d_{-\varepsilon}^{1})\widehat{D}_{-\varepsilon}(d_{-\varepsilon}^{2})\right)\\
&\hspace*{.37cm}=\Phi_{0,\varepsilon}\circ\widehat{D}_{-\varepsilon}(h_{-\varepsilon}^{1}).
\end{align*}
\endgroup
As in case (ii,1) we get $D_{\varepsilon}(k_{\varepsilon}^{1})=\Phi_{0,\varepsilon}\circ D_{-\varepsilon}\circ\Phi_{1,\varepsilon}^{-1}(k_{\varepsilon}^{1}), D_{\varepsilon}=\Phi_{0,\varepsilon}\circ D_{-\varepsilon}\circ\Phi_{1,\varepsilon}^{-1}$ and thus by the analogous computation as above $\Cord(K_{\varepsilon})\cong\Cord(K_{-\varepsilon})$.\\
The other possible situations are to be treated analogously.\\[.5em]
Case (ii,6): $W_{0}^{u}(k_{0}^{1})\cap B\cap S_{0}\neq\emptyset$. We consider the situation as in Figure \ref{WucapBcapSnempty}.
\begin{figure}[htpb]\centering
\subfigure{\includegraphics[scale=0.3]{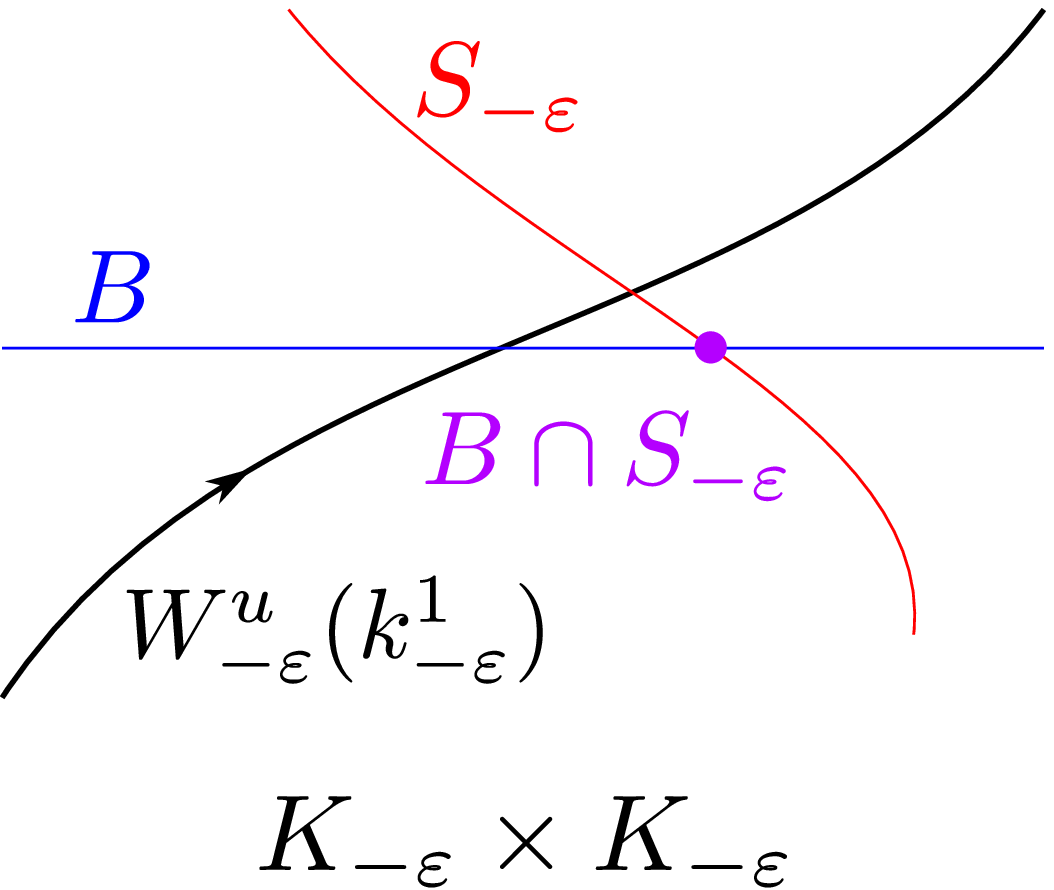}}
\subfigure{\hspace*{1.95cm}\includegraphics[scale=0.3]{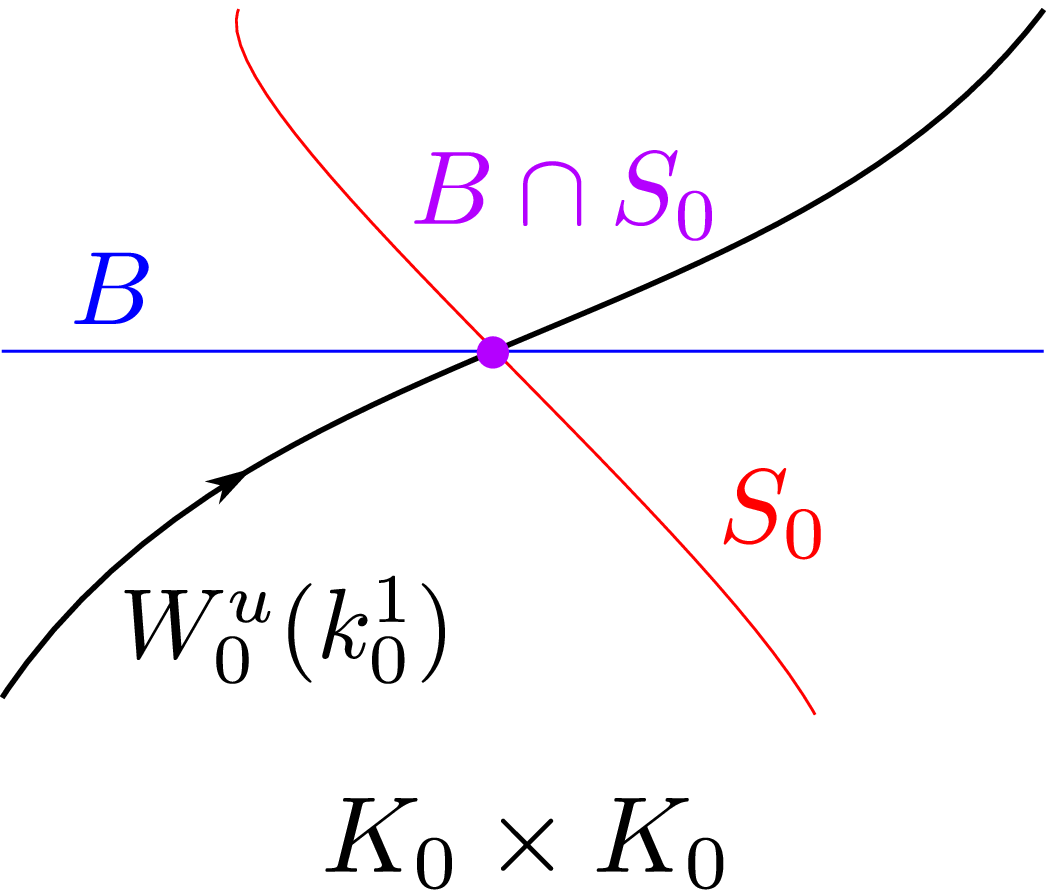}}
\subfigure{\hspace*{1.95cm}\includegraphics[scale=0.3]{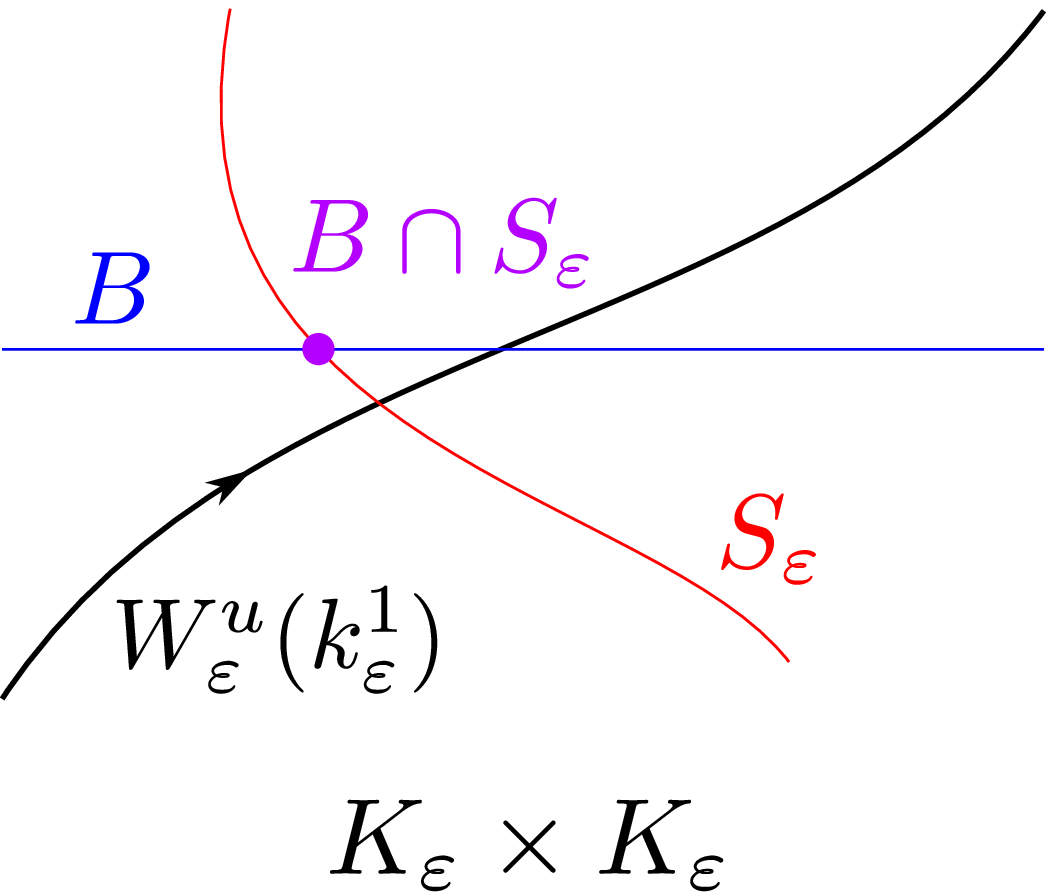}}
\subfigure{\includegraphics[scale=0.3]{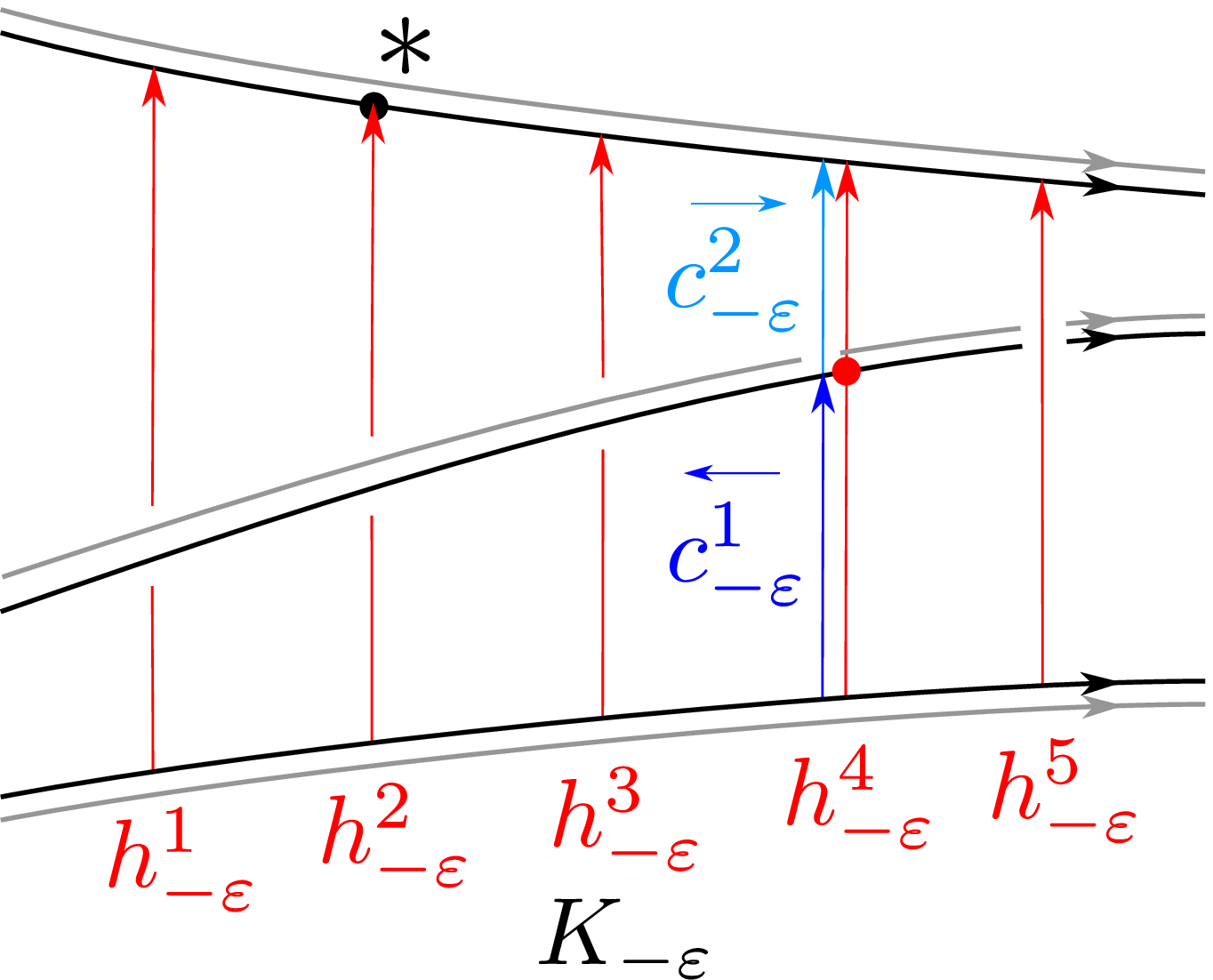}}
\subfigure{\hspace*{1cm}\includegraphics[scale=0.3]{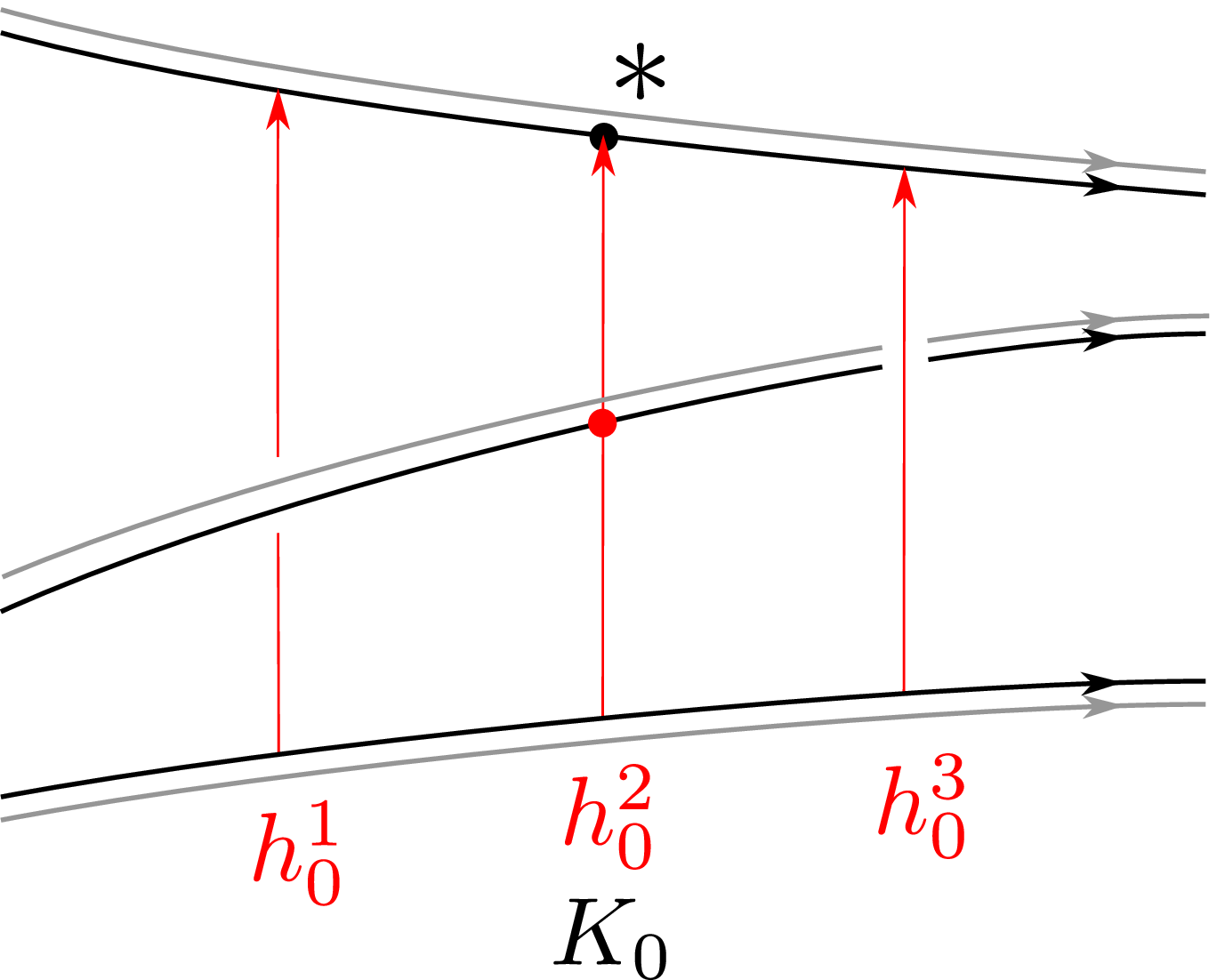}}
\subfigure{\hspace*{1cm}\includegraphics[scale=0.3]{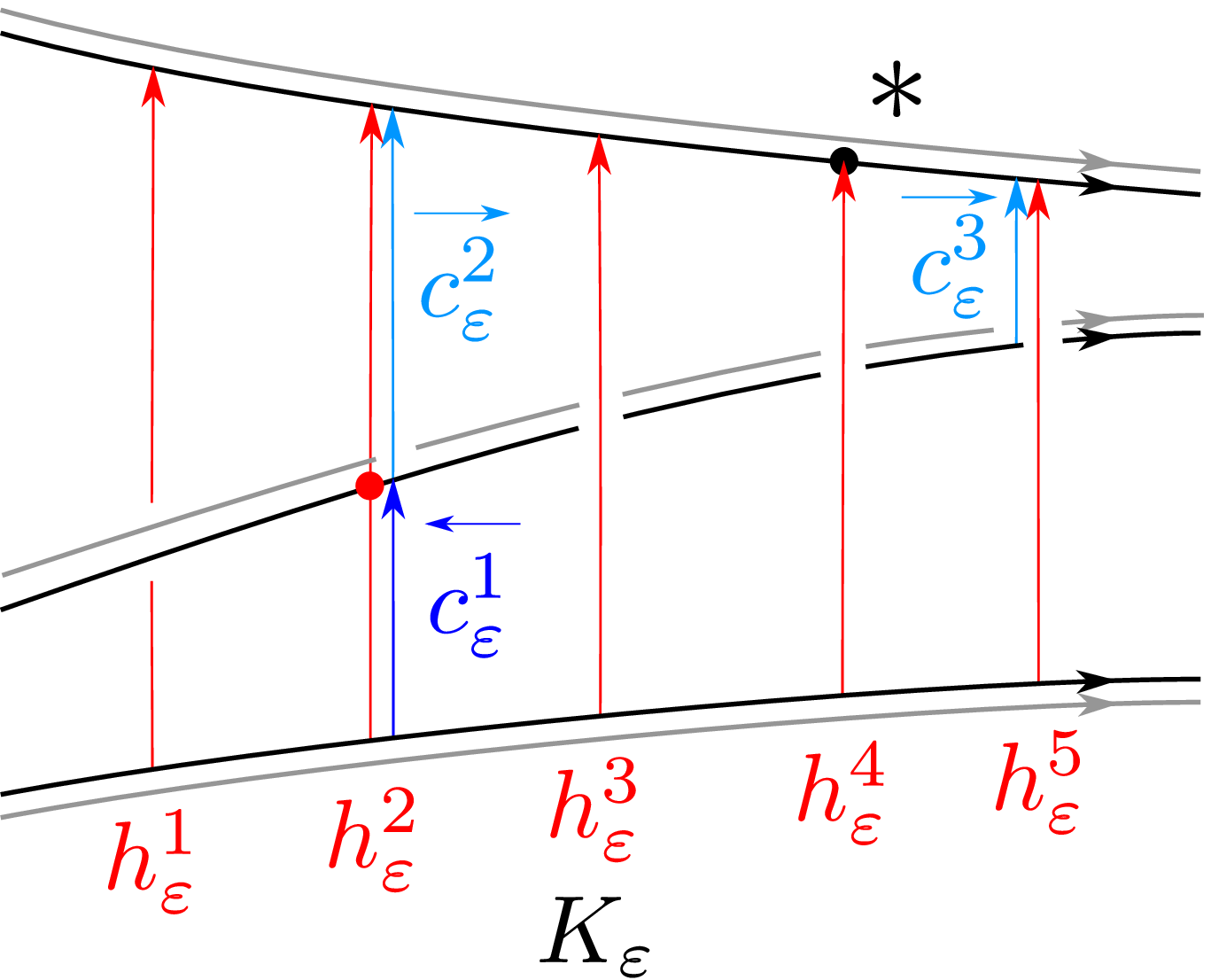}}
\caption{The unstable manifold of the cord $k_{0}^{1}$ intersects $B\cap S_{0}$\label{WucapBcapSnempty}}
\end{figure}
As in case (ii,2) we can assume that there are no critical points in the shown region. If $\varepsilon$ is chosen small enough, the following holds:
\begin{align*}
\widehat{D}_{\varepsilon}(c_{\varepsilon}^{1})&=\Phi_{0,\varepsilon}\circ\widehat{D}_{-\varepsilon}(c_{-\varepsilon}^{1})\\
\widehat{D}_{\varepsilon}(c_{\varepsilon}^{3})&=\Phi_{0,\varepsilon}\circ\widehat{D}_{-\varepsilon}(c_{-\varepsilon}^{2})\\
\widehat{D}_{\varepsilon}(h_{\varepsilon}^{5})&=\Phi_{0,\varepsilon}\circ\widehat{D}_{-\varepsilon}(h_{-\varepsilon}^{5}).
\end{align*}
Now we can compute
\begingroup
\allowdisplaybreaks
\begin{align*}
\widehat{D}_{-\varepsilon}(h_{-\varepsilon}^{1})&\overset{\text{rel. (iii)}}{=}\widehat{D}_{-\varepsilon}(h_{-\varepsilon}^{3})\lambda^{-1}\\
&\overset{\text{rel. (iv)}}{=}\widehat{D}_{-\varepsilon}(h_{-\varepsilon}^{5})\lambda^{-1}-\widehat{D}_{-\varepsilon}(c_{-\varepsilon}^{1})\widehat{D}_{-\varepsilon}(c_{-\varepsilon}^{2})\lambda^{-1}\\
\widehat{D}_{\varepsilon}(h_{\varepsilon}^{1})&\overset{\text{rel. (iv)}}{=}\widehat{D}_{\varepsilon}(h_{\varepsilon}^{3})-\widehat{D}_{\varepsilon}(c_{\varepsilon}^{1})\widehat{D}_{\varepsilon}(c_{\varepsilon}^{2})\\
&\overset{\text{rel. (iii)}}{=}\widehat{D}_{\varepsilon}(h_{\varepsilon}^{5})\lambda^{-1}-\widehat{D}_{\varepsilon}(c_{\varepsilon}^{1})\widehat{D}_{\varepsilon}(c_{\varepsilon}^{3})\lambda^{-1}\\
&\hspace*{.36cm}=\Phi_{0,\varepsilon}\circ\left(\widehat{D}_{-\varepsilon}(h_{-\varepsilon}^{5})\lambda^{-1}-\widehat{D}_{-\varepsilon}(c_{-\varepsilon}^{1})\widehat{D}_{-\varepsilon}(c_{-\varepsilon}^{2})\lambda^{-1}\right)\\
&\hspace*{.36cm}=\Phi_{0,\varepsilon}\circ\widehat{D}_{-\varepsilon}(h_{-\varepsilon}^{1}).
\end{align*}
\endgroup
As in case (ii,1) we get $D_{\varepsilon}(k_{\varepsilon}^{1})=\Phi_{0,\varepsilon}\circ D_{-\varepsilon}\circ\Phi_{1,\varepsilon}^{-1}(k_{\varepsilon}^{1}), D_{\varepsilon}=\Phi_{0,\varepsilon}\circ D_{-\varepsilon}\circ\Phi_{1,\varepsilon}^{-1}$ and thus by the analogous computation as above $\Cord(K_{\varepsilon})\cong\Cord(K_{-\varepsilon})$.\\
The other possible situations are to be treated analogously.\\[.5em]
Case (ii,7): $W_{0}^{u}(k_{0}^{1})\cap F_{0}\cap S_{0}\neq\emptyset$: Analogous to case (ii,6) with relation (ii) instead of relation (iii), i.e. $\mu^{\pm1}$ instead of $\lambda^{\pm1}$.\\[.5em]
Case (ii,8): a) $g_{0}^{1}\in B$. We consider the situation as in Figure \ref{kinBcapCrit0}.
\begin{figure}[ht]\centering
\subfigure{\includegraphics[scale=0.3]{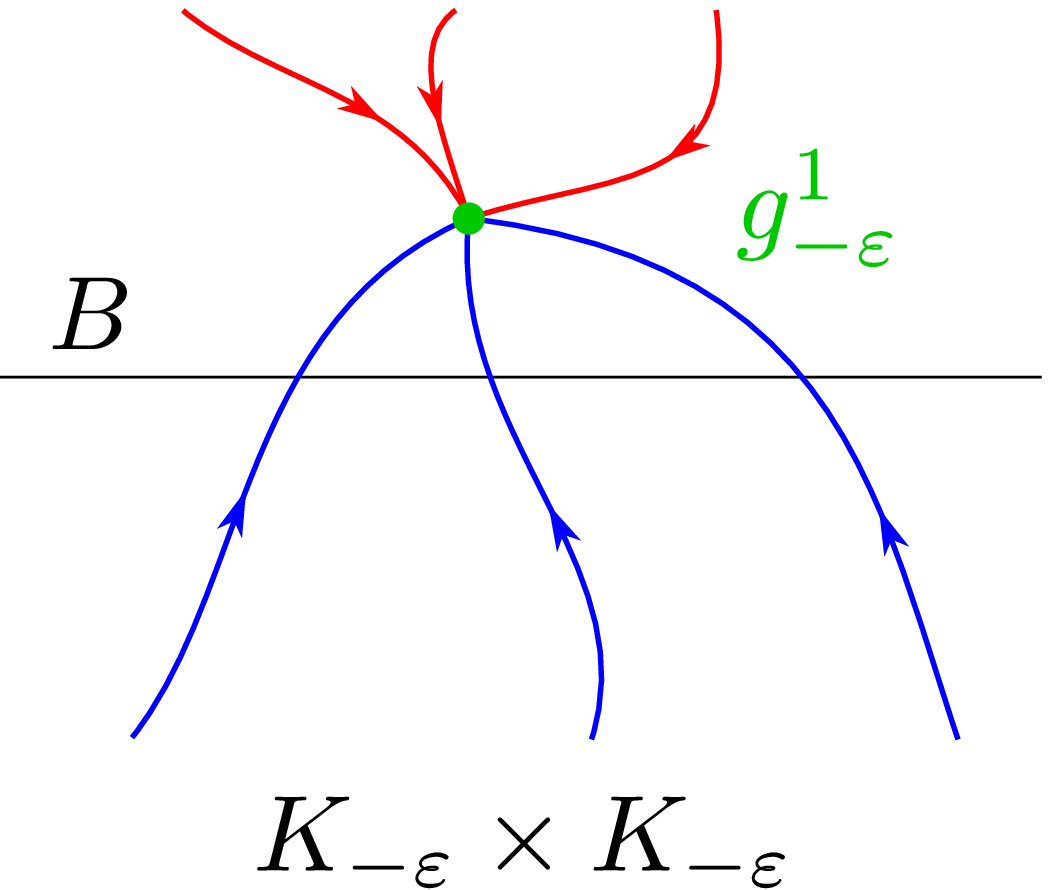}}
\subfigure{\hspace*{1cm}\includegraphics[scale=0.3]{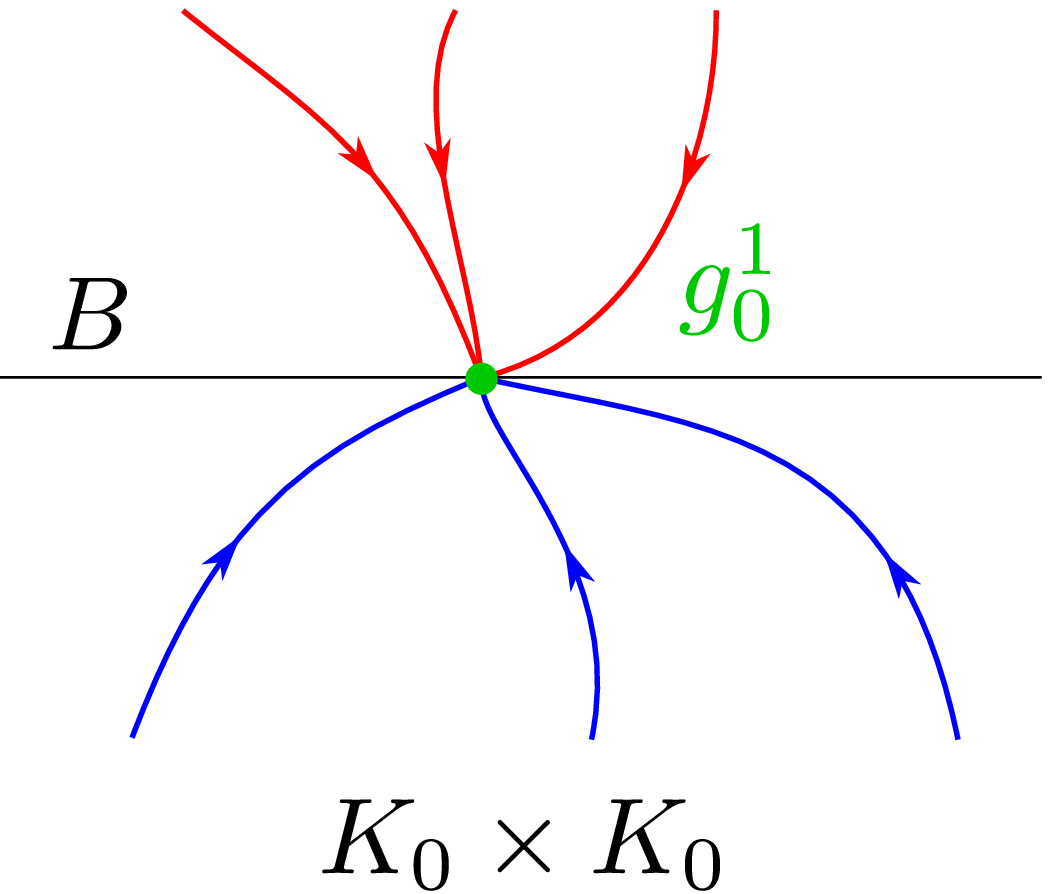}}
\subfigure{\hspace*{1.1cm}\includegraphics[scale=0.3]{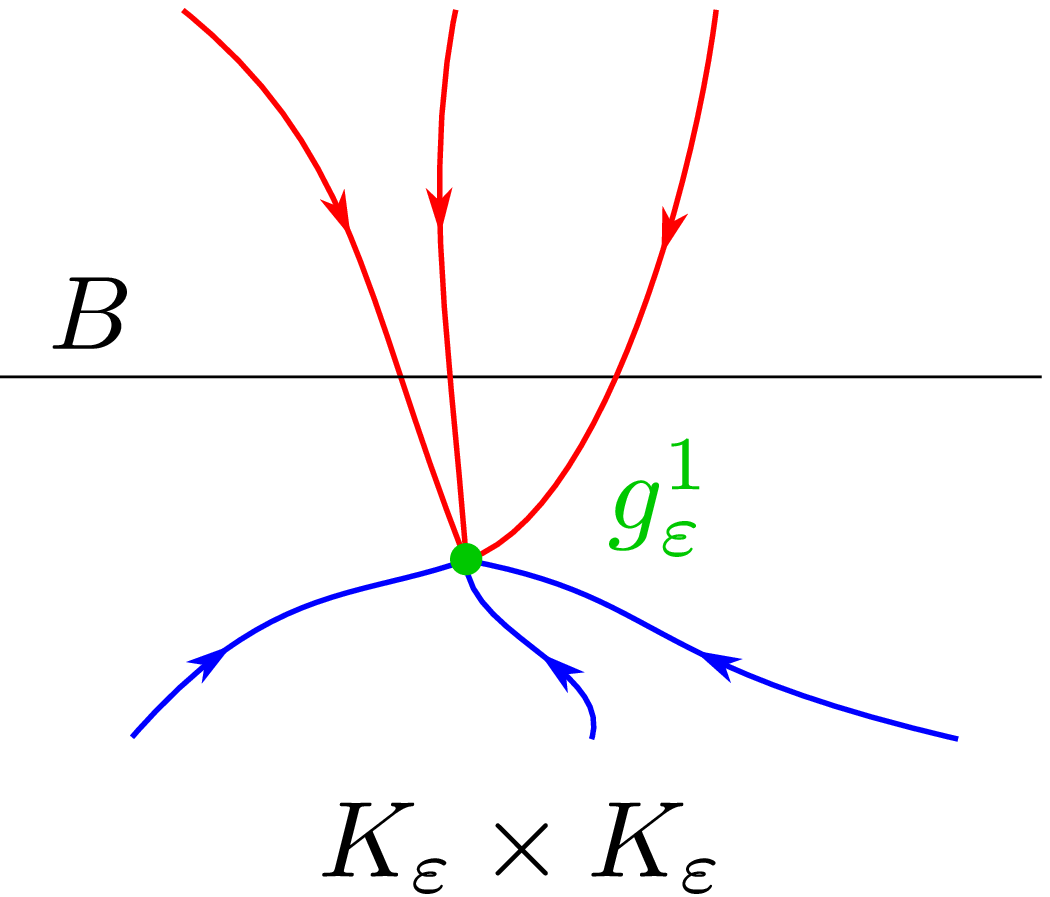}}
\caption{The critical point $g_{0}^{1}$ of index 0 intersects the basepoint\label{kinBcapCrit0}}
\end{figure}
The red and blue trajectories are unstable manifolds of critical points of index 1 or trajectories starting at cords that are generated by splitting according to relation (iv). According to Lemma \ref{finitelymanyintersections}, there are only finitely many such trajectories. In $K_{0}$ none of the red and blue trajectories is tangent to $B$, otherwise we would also have case (ii,1) or (ii,12) of Lemma \ref{genericisotopy}, but this would be a contradiction to this Lemma since only one of these cases occurs at any one time. If $\varepsilon$ is chosen small enough, we can guarantee the following: In $K_{-\varepsilon}$ all of the blue trajectories intersect $B$ in a small neighborhood of $g_{-\varepsilon}^{1}$, but none of the red ones, and in $K_{\varepsilon}$ all of the red trajectories intersect $B$ in a small neighborhood of $g_{\varepsilon}^{1}$, but none of the blue ones. It follows that in $K_{-\varepsilon}$ we have a contribution of $g_{-\varepsilon}^{1}$ along the red trajectories and of $g_{-\varepsilon}^{1}\lambda^{-1}$ along the blue trajectories. In $K_{\varepsilon}$ we have a contribution of $g_{\varepsilon}^{1}\lambda$ along the red trajectories and of $g_{\varepsilon}^{1}$ along the blue trajectories. So we can construct the following canonical isomorphism:
\begin{align*}
\Cord(K_{-\varepsilon})&\overset{\sim}{\to}\Cord(K_{\varepsilon})\\
g_{-\varepsilon}^{1}&\mapsto g_{\varepsilon}^{1}\lambda\\
g_{-\varepsilon}^{i}&\mapsto g_{\varepsilon}^{i},\ i=2,\dots,n_{0}\\
\lambda^{\pm1}&\mapsto\lambda^{\pm1}\\
\mu^{\pm1}&\mapsto\mu^{\pm1}.
\end{align*}
The other possible situations are to be treated analogously.\\[.5em]
b) $k_{0}^{1}\in B$. We consider the situation as in Figure \ref{kinBcapCrit1}.
\begin{figure}[ht]\centering
\subfigure{\includegraphics[scale=0.3]{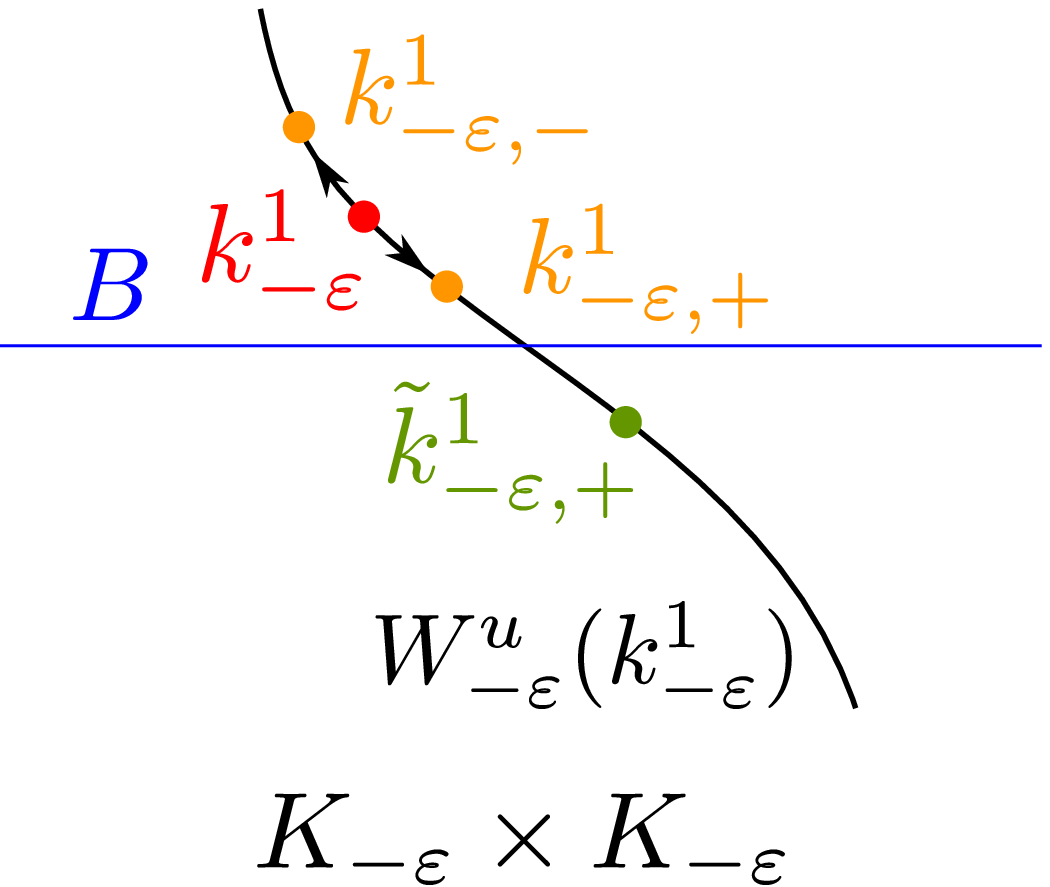}}
\subfigure{\hspace*{1cm}\includegraphics[scale=0.3]{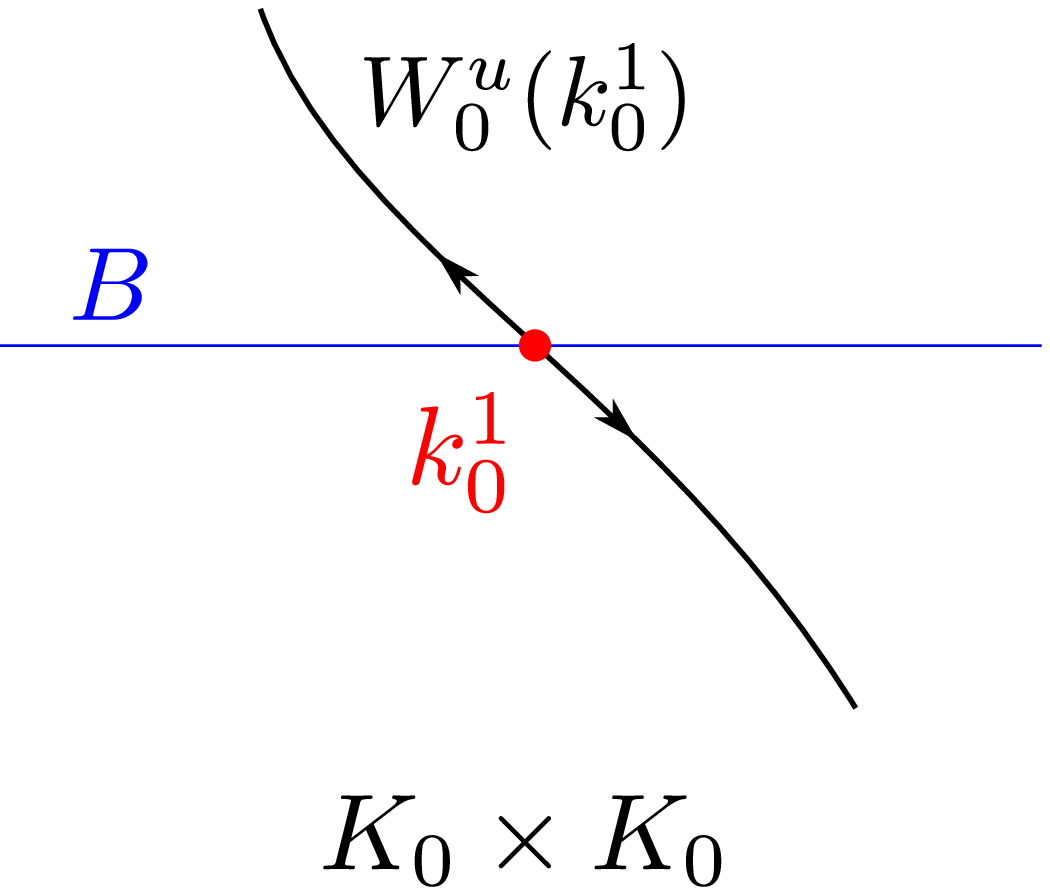}}
\subfigure{\hspace*{1.1cm}\includegraphics[scale=0.3]{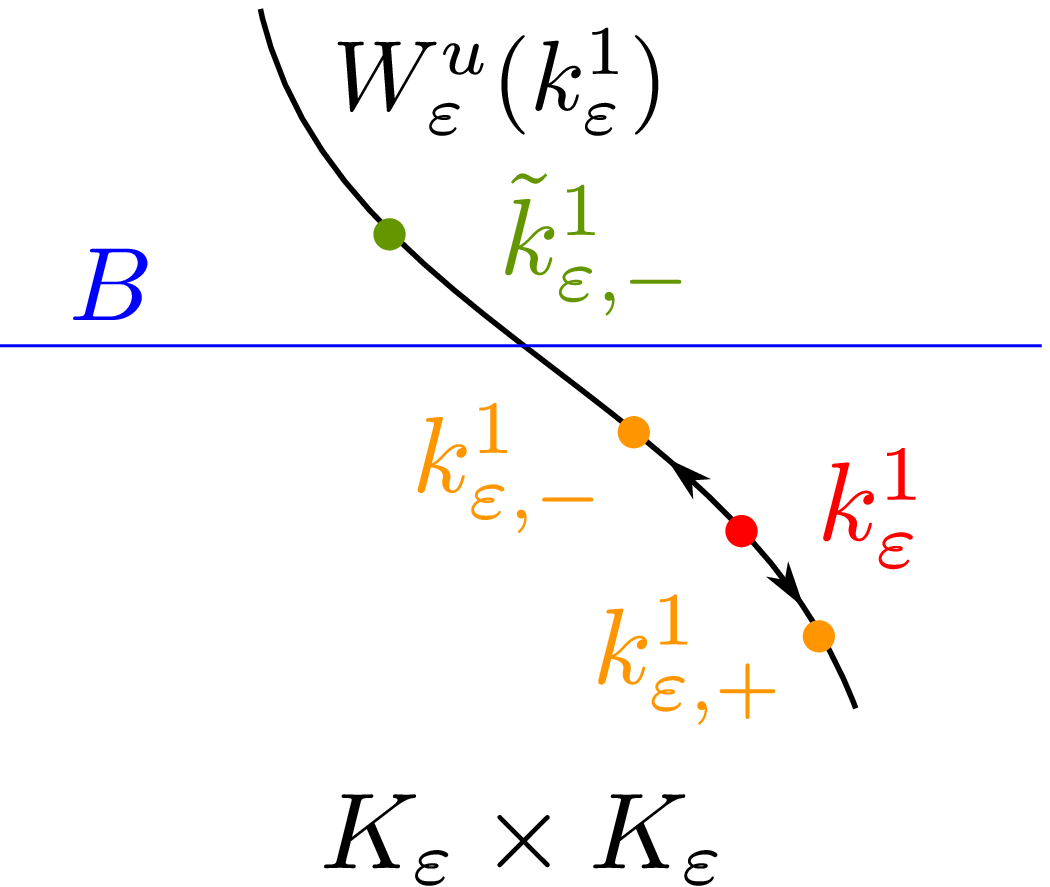}}
\caption{The critical point $k_{0}^{1}$ of index 1 intersects the basepoint\label{kinBcapCrit1}}
\end{figure}
If $\varepsilon$ is chosen small enough, the following holds:
\begin{align*}
\widehat{D}_{\varepsilon}(k_{\varepsilon,+}^{1})&=\Phi_{0,\varepsilon}\circ\widehat{D}_{-\varepsilon}(\tilde{k}_{-\varepsilon,+}^{1})\\
\widehat{D}_{\varepsilon}(\tilde{k}_{\varepsilon,-}^{1})&=\Phi_{0,\varepsilon}\circ\widehat{D}_{-\varepsilon}(k_{-\varepsilon,-}^{1}).
\end{align*}
Now we can compute
\begin{align*}
D_{\varepsilon}(k_{\varepsilon}^{1})&\hspace*{.36cm}=\widehat{D}_{\varepsilon}(k_{\varepsilon,+}^{1})-\widehat{D}_{\varepsilon}(k_{\varepsilon,-}^{1})\\
&\overset{\text{rel. (iii)}}{=}\widehat{D}_{\varepsilon}(k_{\varepsilon,+}^{1})-\widehat{D}_{\varepsilon}(\tilde{k}_{\varepsilon,-}^{1})\lambda^{-1}\\
&\hspace*{.36cm}=\Phi_{0,\varepsilon}\circ\left(\widehat{D}_{-\varepsilon}(\tilde{k}_{-\varepsilon,+}^{1})-\widehat{D}_{-\varepsilon}(k_{-\varepsilon,-}^{1})\lambda^{-1}\right)\\
&\hspace*{.36cm}=\Phi_{0,\varepsilon}\circ\left(\left(\widehat{D}_{-\varepsilon}(\tilde{k}_{-\varepsilon,+}^{1})\lambda-\widehat{D}_{-\varepsilon}(k_{-\varepsilon,-}^{1})\right)\lambda^{-1}\right)\\
&\overset{\text{rel. (iii)}}{=}\Phi_{0,\varepsilon}\circ\left(\left(\widehat{D}_{-\varepsilon}(k_{-\varepsilon,+}^{1})-\widehat{D}_{-\varepsilon}(k_{-\varepsilon,-}^{1})\right)\lambda^{-1}\right)\\
&\hspace*{.36cm}=\Phi_{0,\varepsilon}\circ\left(D_{-\varepsilon}(k_{-\varepsilon}^{1})\lambda^{-1}\right)\\
&\hspace*{.36cm}=\left(\Phi_{0,\varepsilon}\circ D_{-\varepsilon}(k_{-\varepsilon}^{1})\right)\lambda^{-1}.
\end{align*}
With this result we get, where $R=\mathbb{Z}[\lambda^{\pm1},\mu^{\pm1}]$ is the commutative ring as described above and $\langle M\rangle_{R}$ is the ideal generated by the set $M$:
\begingroup
\allowdisplaybreaks
\begin{align*}
\Cord(K_{\varepsilon})&=C_{0}(K_{\varepsilon})/I_{\varepsilon}\\
&=C_{0}(K_{\varepsilon})/\langle D_{\varepsilon}(\lbrace k_{\varepsilon}^{2},\dots,k_{\varepsilon}^{n_{1}}\rbrace),D_{\varepsilon}(k_{\varepsilon}^{1})\rangle_{R}\\
&=C_{0}(K_{\varepsilon})/\langle\Phi_{0,\varepsilon}\circ D_{-\varepsilon}(\lbrace k_{-\varepsilon}^{2},\dots,k_{-\varepsilon}^{n_{1}}\rbrace),\Phi_{0,\varepsilon}\circ D_{-\varepsilon}(k_{-\varepsilon}^{1})\lambda^{-1}\rangle_{R}\\
&=C_{0}(K_{\varepsilon})/\langle\Phi_{0,\varepsilon}\circ D_{-\varepsilon}(\lbrace k_{-\varepsilon}^{2},\dots,k_{-\varepsilon}^{n_{1}}\rbrace),\Phi_{0,\varepsilon}\circ D_{-\varepsilon}(k_{-\varepsilon}^{1})\rangle_{R}\\
&=C_{0}(K_{\varepsilon})/\langle\Phi_{0,\varepsilon}\circ D_{-\varepsilon}(\lbrace k_{-\varepsilon}^{1},\dots,k_{-\varepsilon}^{n_{1}}\rbrace)\rangle_{R}\\
&=\Phi_{0,\varepsilon}(C_{0}(K_{-\varepsilon}))/\Phi_{0,\varepsilon}(I_{-\varepsilon})\\
&\overset{\text{(1)}}{=}\Phi_{0,\varepsilon}\left(C_{0}(K_{-\varepsilon})/I_{-\varepsilon}\right)\\
&=\Phi_{0,\varepsilon}(\Cord(K_{-\varepsilon}))\\
&\cong\Cord(K_{-\varepsilon}).
\end{align*}
\endgroup
Equality (1) holds since $\Phi_{0,\varepsilon}$ is an algebra isomorphism.\\[.5em]
The other possible situations are to be treated analogously.\\[.5em]
Case (ii,9): a) $g_{0}^{1}\in S$. We consider the situation as in Figure \ref{kinScapCrit0}.
\begin{figure}[ht]\centering
\subfigure{\includegraphics[scale=0.3]{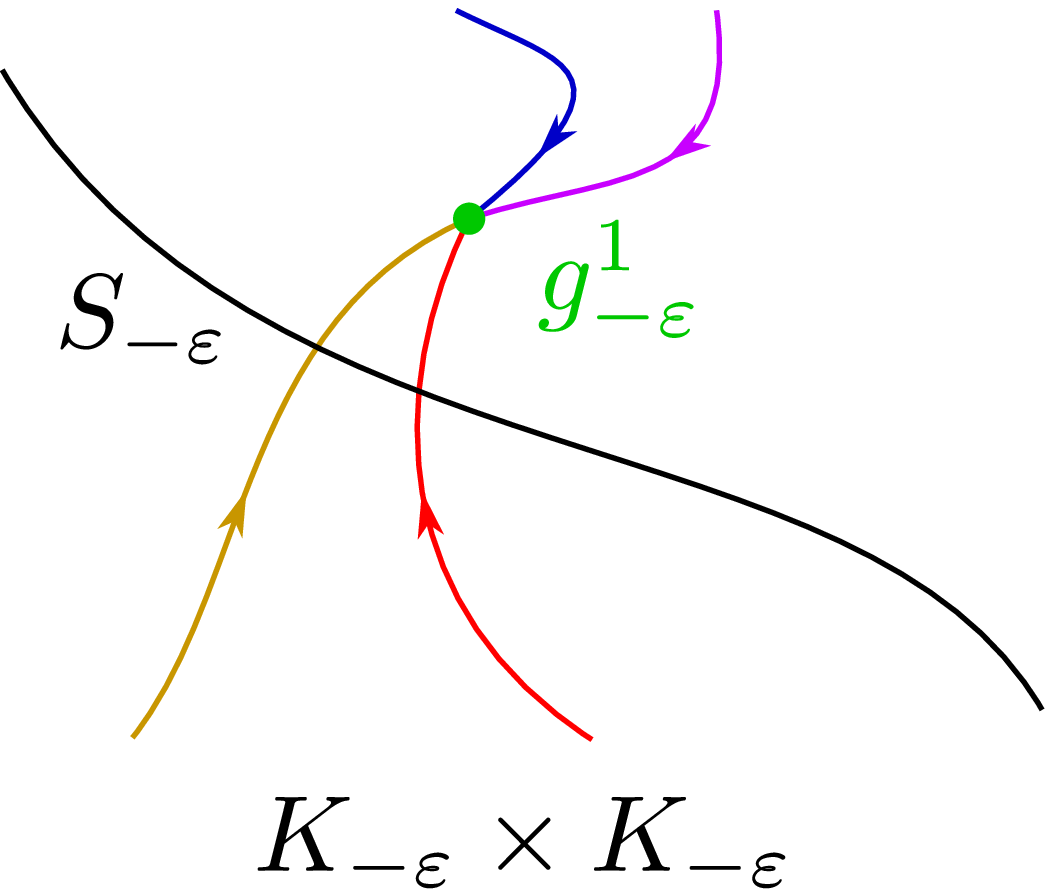}}
\subfigure{\hspace*{1.6cm}\includegraphics[scale=0.3]{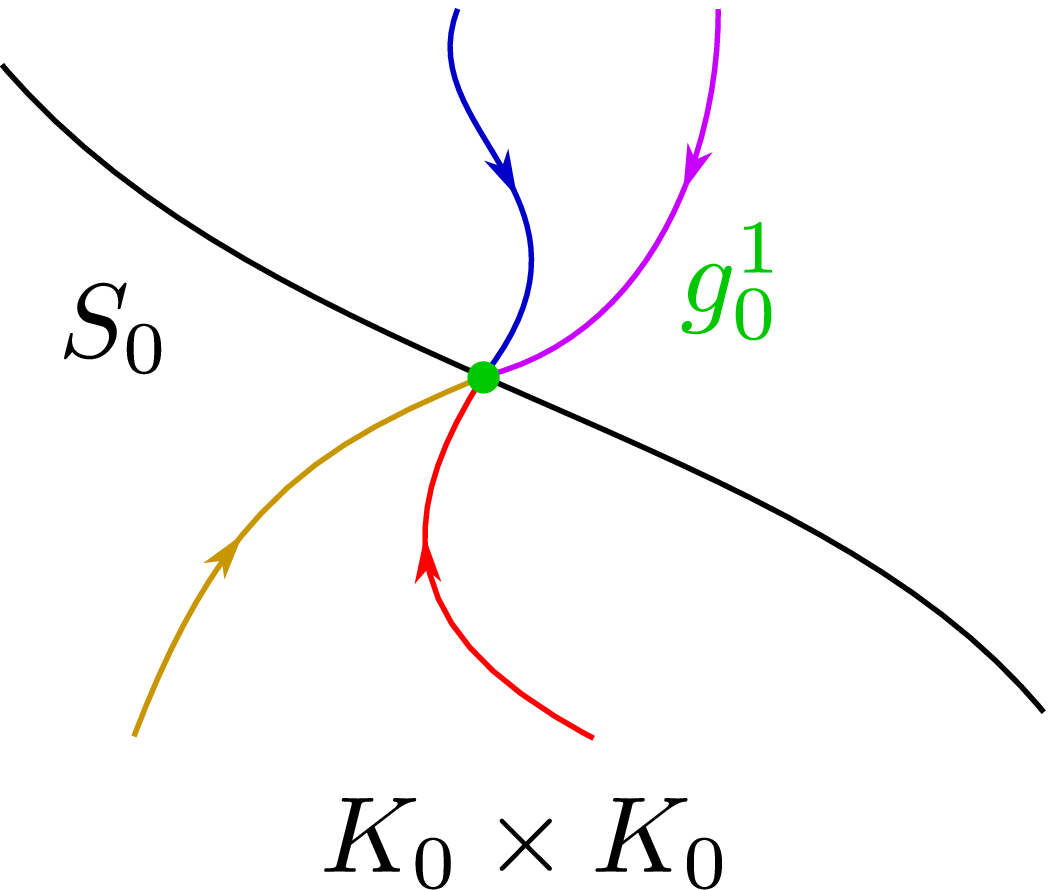}}
\subfigure{\hspace*{1.6cm}\includegraphics[scale=0.3]{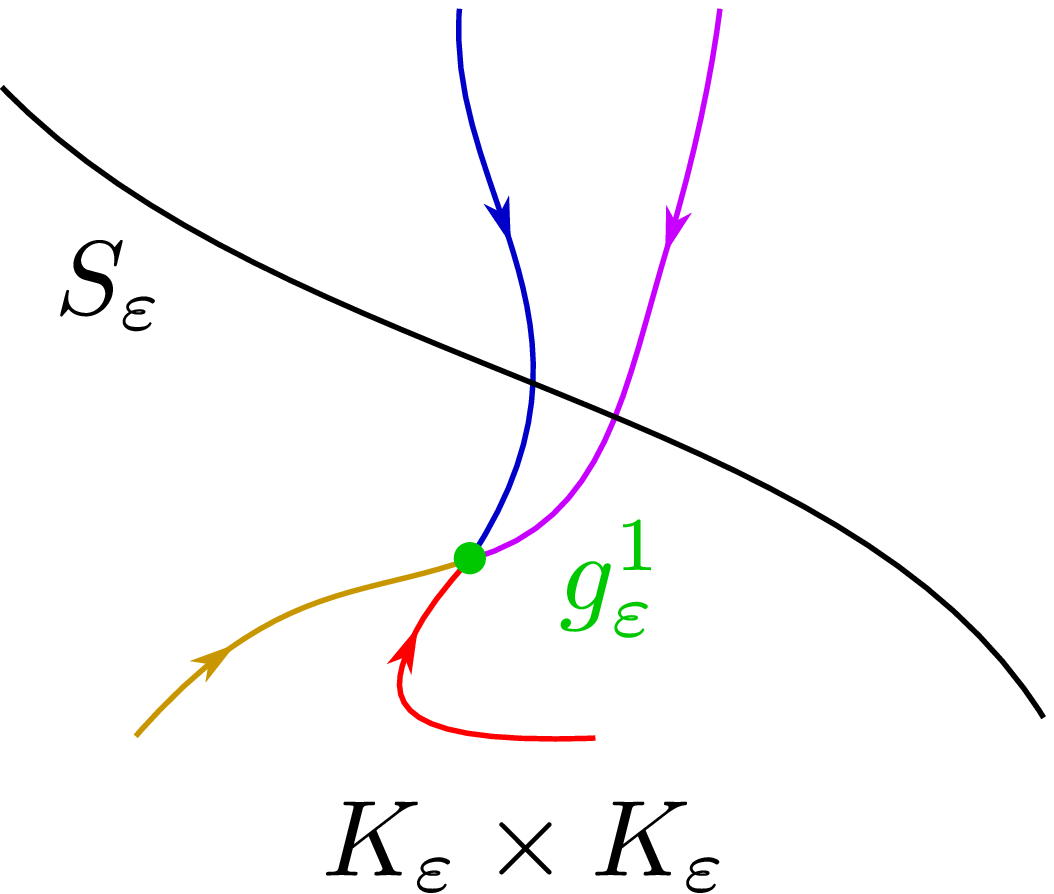}}
\subfigure{\includegraphics[scale=0.3]{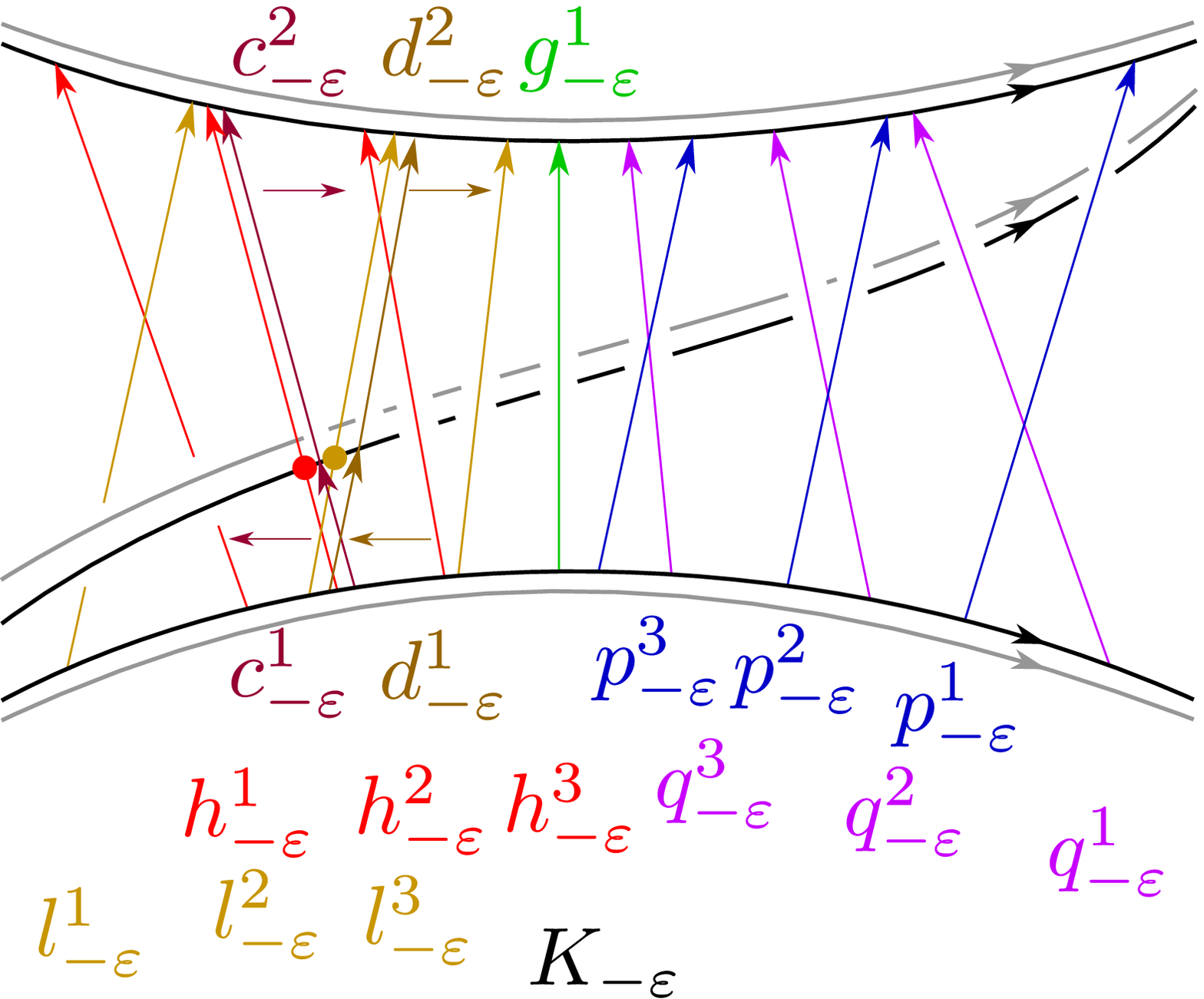}}
\subfigure{\includegraphics[scale=0.3]{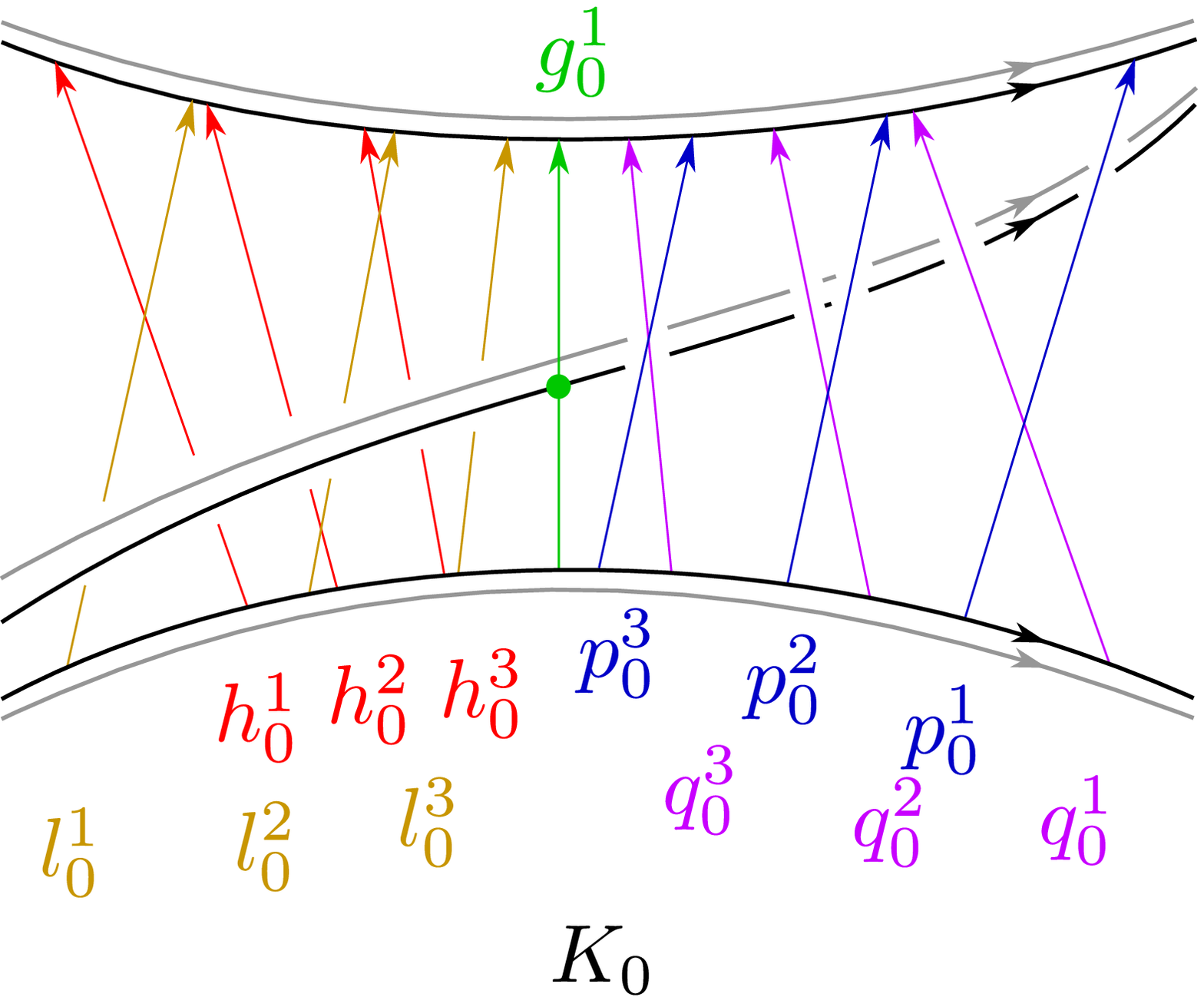}}
\subfigure{\includegraphics[scale=0.3]{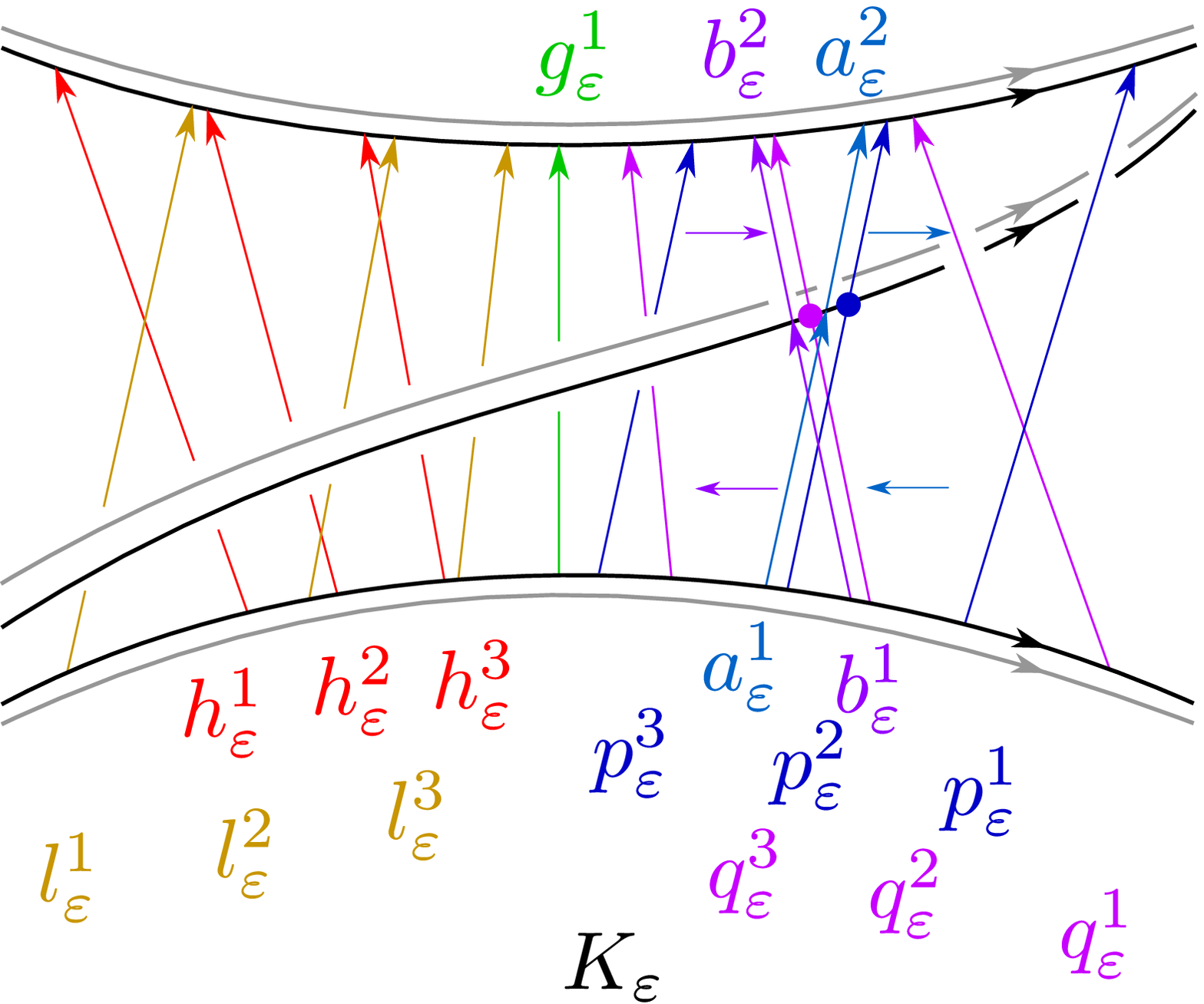}}
\caption{The critical point $g_{0}^{1}$ of index 0 intersects the knot in its interior\label{kinScapCrit0}}
\end{figure}
The trajectories drawn in the upper figures are unstable manifolds of critical points of index 1 or trajectories starting at cords that are generated by splitting according to relation (iv). According to Lemma \ref{finitelymanyintersections}, there are only finitely many such trajectories. In $K_{0}$ none of these trajectories is tangent to $S_{0}$, otherwise we would also have case (ii,2) or (ii,12) of Lemma \ref{genericisotopy}, but this would be a contradiction to this Lemma since only one of these cases occurs at any one time. If $\varepsilon$ is chosen small enough, we can guarantee the following (in the example of Figure \ref{kinScapCrit0}):
\begin{itemize}
\item In $K_{-\varepsilon}$ all of the ``lower'' (red and brown) trajectories intersect $S_{-\varepsilon}$ in a small neighborhood of $g_{-\varepsilon}^{1}$, but none of the ``upper'' (blue and purple) ones.
\item In $K_{-\varepsilon}$ we get for the cords that are split off from the ``lower'' trajectories according to relation (iv): $\widehat{D}_{-\varepsilon}(c_{-\varepsilon}^{1})=\widehat{D}_{-\varepsilon}(d_{-\varepsilon}^{1})$ and $\widehat{D}_{-\varepsilon}(c_{-\varepsilon}^{2})=\widehat{D}_{-\varepsilon}(d_{-\varepsilon}^{2})$.
\item In $K_{\varepsilon}$ all of the ``upper'' trajectories intersect $S_{\varepsilon}$ in a small neighborhood of $g_{\varepsilon}^{1}$, but none of the ``lower'' ones.
\item In $K_{\varepsilon}$ we get for the cords that are split off from the ``upper'' trajectories according to relation~(iv): $\widehat{D}_{\varepsilon}(a_{\varepsilon}^{1})=\widehat{D}_{\varepsilon}(b_{\varepsilon}^{1})$ and $\widehat{D}_{\varepsilon}(a_{\varepsilon}^{2})=\widehat{D}_{\varepsilon}(b_{\varepsilon}^{2})$.
\item Furthermore, we get $\widehat{D}_{\varepsilon}(a_{\varepsilon}^{1})=\Phi_{0,\varepsilon}\circ\widehat{D}_{-\varepsilon}(c_{-\varepsilon}^{1})$ and $\widehat{D}_{\varepsilon}(a_{\varepsilon}^{2})=\Phi_{0,\varepsilon}\circ\widehat{D}_{-\varepsilon}(c_{-\varepsilon}^{2})$.
\end{itemize}
It follows that in $K_{-\varepsilon}$ we have a contribution of $g_{-\varepsilon}^{1}$ along the ``upper'' trajectories and of
\[g_{-\varepsilon}^{1}-\widehat{D}_{-\varepsilon}(c_{-\varepsilon}^{1})\widehat{D}_{-\varepsilon}(c_{-\varepsilon}^{2})\]
along the ``lower'' trajectories according to relation (iv). In $K_{\varepsilon}$ we have a contribution of
\[g_{\varepsilon}^{1}+\widehat{D}_{\varepsilon}(a_{\varepsilon}^{1})\widehat{D}_{\varepsilon}(a_{\varepsilon}^{2})=g_{\varepsilon}^{1}+\Phi_{0,\varepsilon}\circ\left(\widehat{D}_{-\varepsilon}(c_{-\varepsilon}^{1})\widehat{D}_{-\varepsilon}(c_{-\varepsilon}^{2})\right)\]
along the ``upper'' trajectories and of $g_{\varepsilon}$ along the ``lower'' trajectories. So we can construct the following canonical isomorphism:
\begin{align*}
\Cord(K_{-\varepsilon})&\overset{\sim}{\to}\Cord(K_{\varepsilon})\\
g_{-\varepsilon}^{1}&\mapsto g_{\varepsilon}^{1}+\widehat{D}_{\varepsilon}(a_{\varepsilon}^{1})\widehat{D}_{\varepsilon}(a_{\varepsilon}^{2})\\
g_{-\varepsilon}^{i}&\mapsto g_{\varepsilon}^{i},\ i=2,\dots,n_{0}\\
\lambda^{\pm1}&\mapsto\lambda^{\pm1}\\
\mu^{\pm1}&\mapsto\mu^{\pm1}.
\end{align*}
The other possible situations are to be treated analogously.\\[.5em]
b) $k_{0}^{1}\in B$. We consider the situation as in Figure \ref{kinScapCrit1}.\begin{figure}[ht]\centering
\subfigure{\includegraphics[scale=0.3]{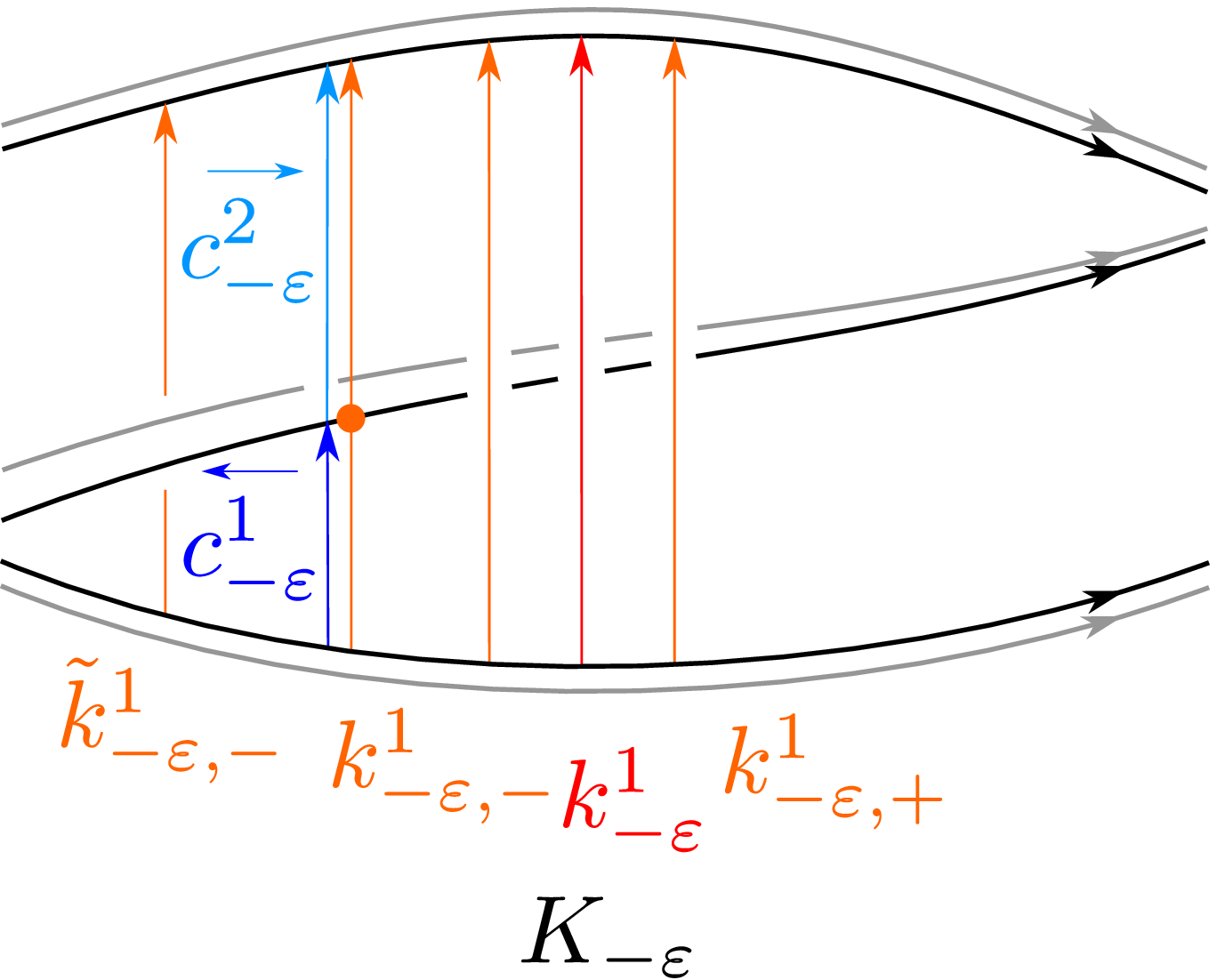}}
\subfigure{\hspace*{1cm}\includegraphics[scale=0.3]{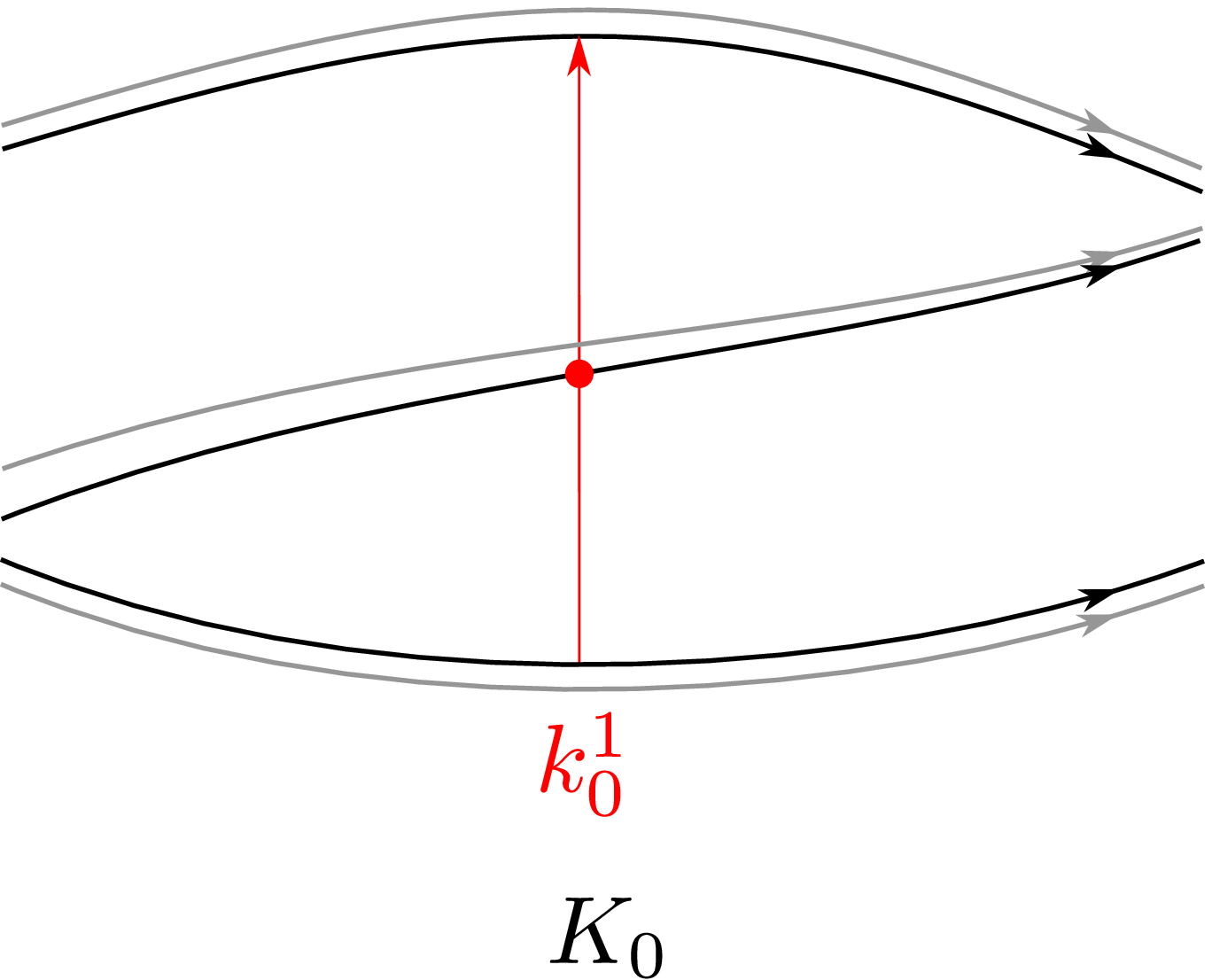}}
\subfigure{\hspace*{1cm}\includegraphics[scale=0.3]{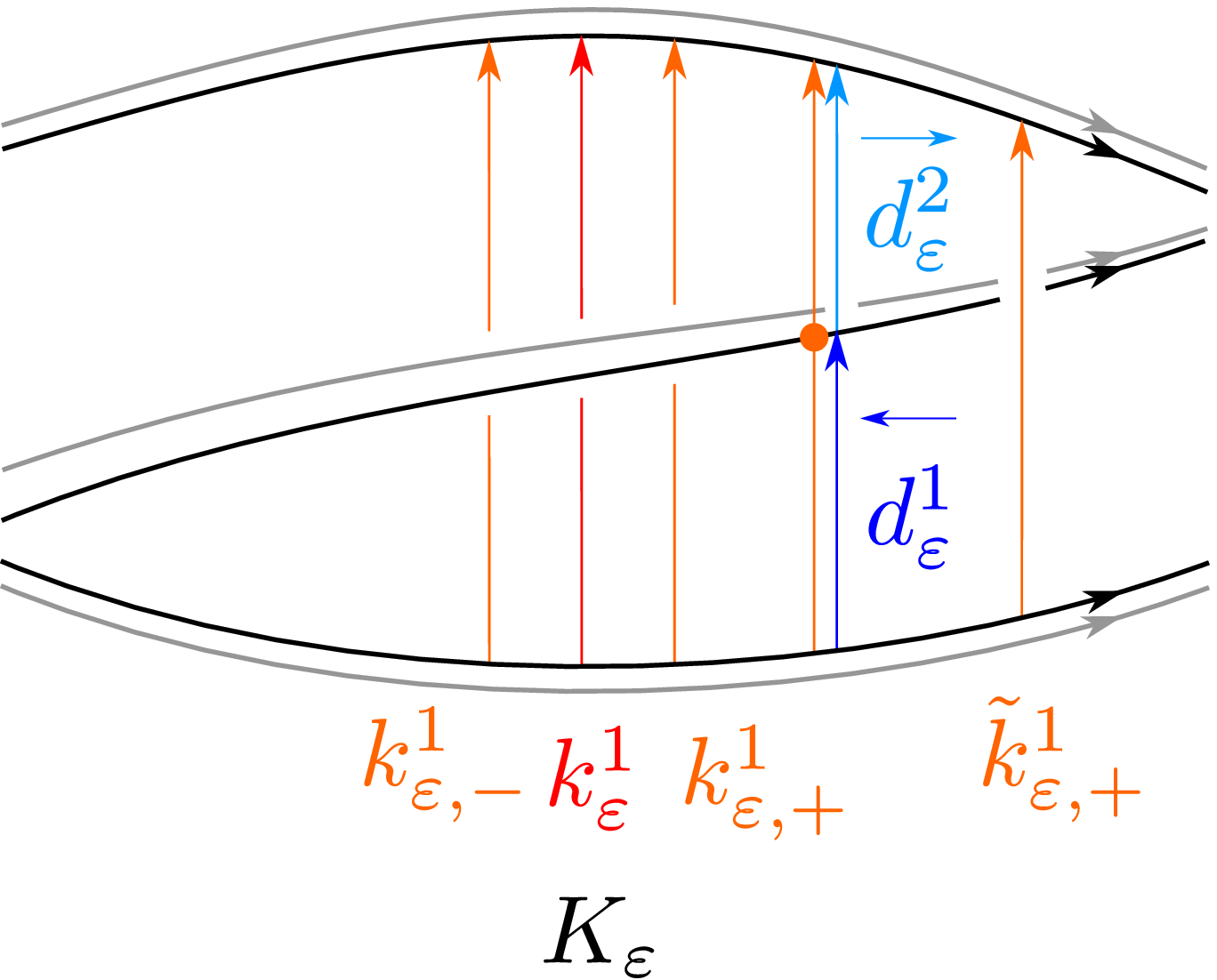}}
\caption{The critical point $k_{0}^{1}$ of index 1 intersects the knot in its interior\label{kinScapCrit1}}
\end{figure}
As before, we can assume, according to Lemma \ref{genericisotopy}, that the only critical point in the shown region is $k_{r}^{1}$. If $\varepsilon$ is chosen small enough, the following holds:
\begin{align*}
\widehat{D}_{\varepsilon}(d_{\varepsilon}^{1})&=\Phi_{0,\varepsilon}\circ\widehat{D}_{-\varepsilon}(c_{-\varepsilon}^{1})\\
\widehat{D}_{\varepsilon}(d_{\varepsilon}^{2})&=\Phi_{0,\varepsilon}\circ\widehat{D}_{-\varepsilon}(c_{-\varepsilon}^{2})\\
\widehat{D}_{\varepsilon}(\tilde{k}_{\varepsilon,+}^{1})&=\Phi_{0,\varepsilon}\circ\widehat{D}_{-\varepsilon}(k_{-\varepsilon,+}^{1})\\
\widehat{D}_{\varepsilon}(k_{\varepsilon,-}^{1})&=\Phi_{0,\varepsilon}\circ\widehat{D}_{-\varepsilon}(\tilde{k}_{-\varepsilon,-}^{1}).
\end{align*}
Now we can compute
\begingroup
\allowdisplaybreaks
\begin{align*}
D_{\varepsilon}(k_{\varepsilon}^{1})&\hspace*{.34cm}=\widehat{D}_{\varepsilon}(k_{\varepsilon,+}^{1})-\widehat{D}_{\varepsilon}(k_{\varepsilon,-}^{1})\\
&\overset{\text{rel. (iv)}}{=}\widehat{D}_{\varepsilon}(\tilde{k}_{\varepsilon,+}^{1})-\widehat{D}_{\varepsilon}(d_{\varepsilon}^{1})\widehat{D}_{\varepsilon}(d_{\varepsilon}^{2})-\widehat{D}_{\varepsilon}(k_{\varepsilon,-}^{1})\\
&\hspace{.34cm}=\Phi_{0,\varepsilon}\circ\left(\widehat{D}_{-\varepsilon}(k_{-\varepsilon,+}^{1})-\widehat{D}_{-\varepsilon}(c_{-\varepsilon}^{1})\widehat{D}_{-\varepsilon}(c_{-\varepsilon}^{2})-\widehat{D}_{-\varepsilon}(\tilde{k}_{-\varepsilon,-}^{1})\right)\\
&\overset{\text{rel. (iv)}}{=}\Phi_{0,\varepsilon}\circ\left(\widehat{D}_{-\varepsilon}(k_{-\varepsilon,+}^{1})-\widehat{D}_{-\varepsilon}(c_{-\varepsilon}^{1})\widehat{D}_{-\varepsilon}(c_{-\varepsilon}^{2})-\left(\widehat{D}_{-\varepsilon}(k_{-\varepsilon,-}^{1})-\widehat{D}_{-\varepsilon}(c_{-\varepsilon}^{1})\widehat{D}_{-\varepsilon}(c_{-\varepsilon}^{2})\right)\right)\\
&\hspace{.34cm}=\Phi_{0,\varepsilon}\circ\left(\widehat{D}_{-\varepsilon}(k_{-\varepsilon,+}^{1})-\widehat{D}_{-\varepsilon}(k_{-\varepsilon,-}^{1})\right)\\
&\hspace{.34cm}=\Phi_{0,\varepsilon}\circ D_{-\varepsilon}(k_{-\varepsilon}^{1})
\end{align*}
\endgroup
As before we get $D_{\varepsilon}(k_{\varepsilon}^{1})=\Phi_{0,\varepsilon}\circ D_{-\varepsilon}\circ\Phi_{1,\varepsilon}^{-1}(k_{\varepsilon}^{1}), D_{\varepsilon}=\Phi_{0,\varepsilon}\circ D_{-\varepsilon}\circ\Phi_{1,\varepsilon}^{-1}$ and thus by the analogous computation as above $\Cord(K_{\varepsilon})\cong\Cord(K_{-\varepsilon})$.\\
The other possible situations are to be treated analogously.\\[.5em]
Case (ii,10): $k\in F_{0}$ for a critical point $k$ of index 0 or 1: Analogous to case (ii,8) with relation (ii) instead of relation (iii).\\[.5em]
Case (ii,11): There exists a trajectory between two critical points of index 1 along the vector field $-\nabla E_{0}$. We consider the situation as in Figure \ref{WuendsinCrit1}.
\begin{figure}[htpb]\centering
\subfigure{\includegraphics[scale=0.3]{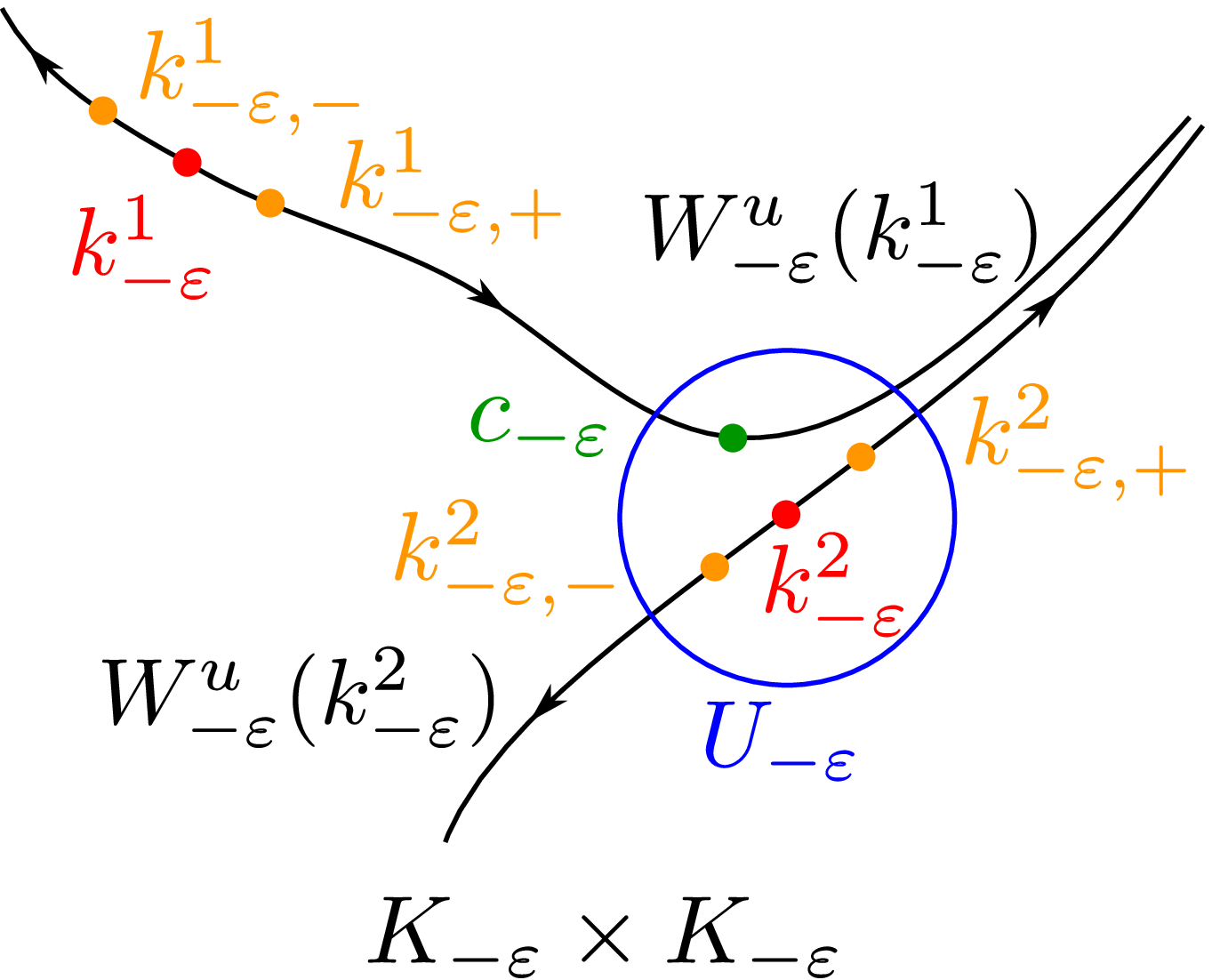}}
\subfigure{\hspace*{1cm}\includegraphics[scale=0.3]{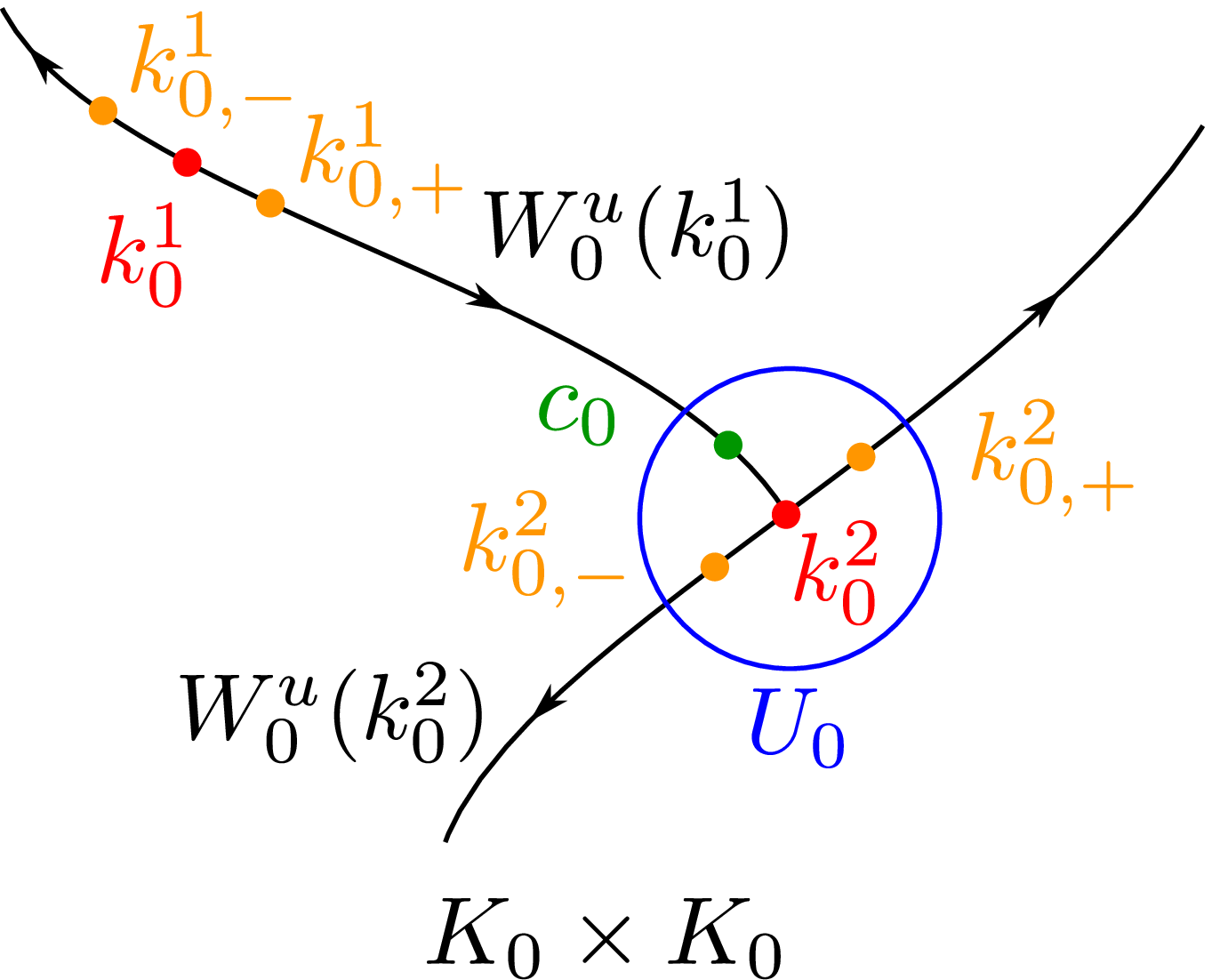}}
\subfigure{\hspace*{1cm}\includegraphics[scale=0.3]{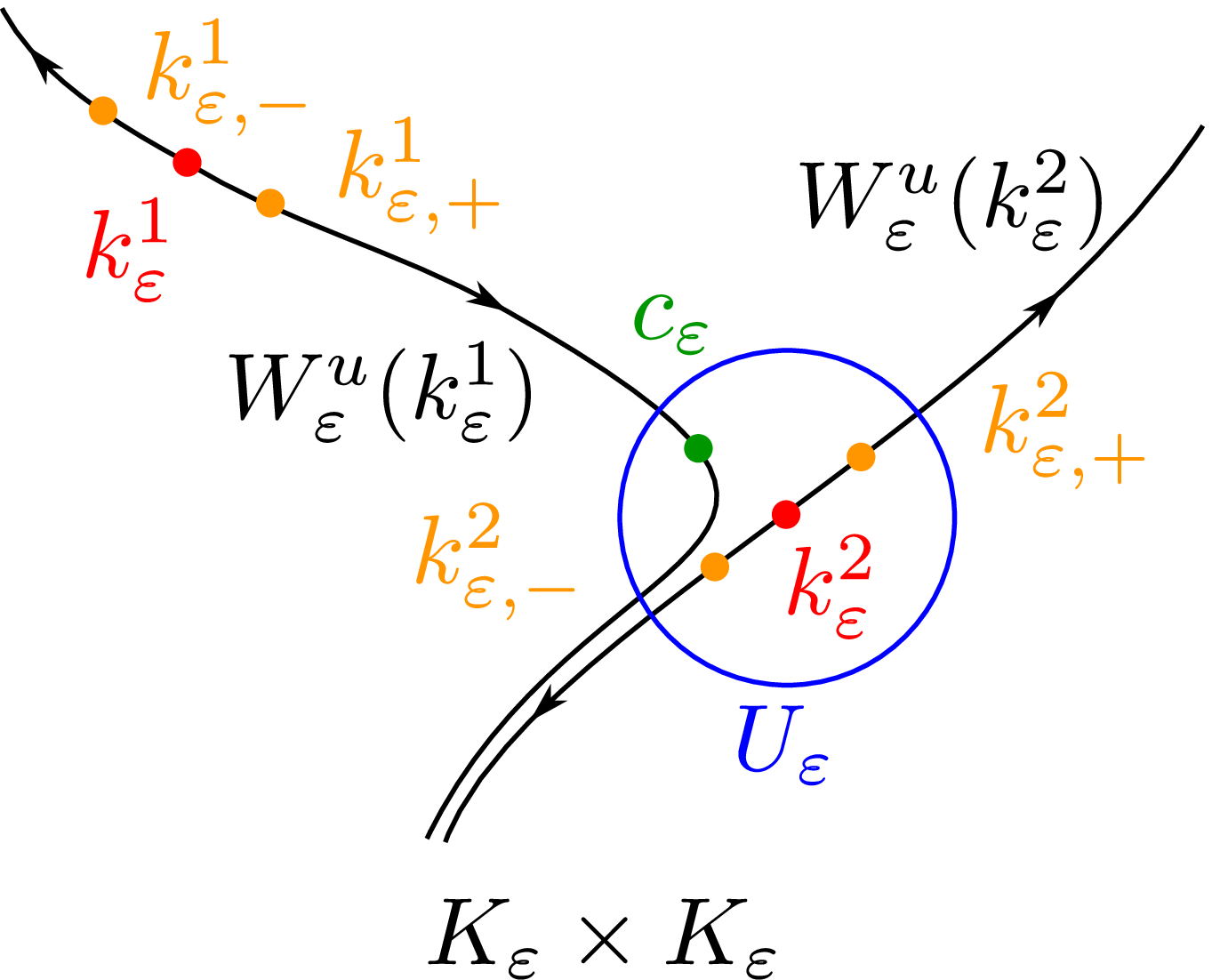}}
\caption{The unstable manifold of the cord $k_{0}^{1}$ ends at the critical point $k_{0}^{2}$ of index 1\label{WuendsinCrit1}}
\end{figure}

According to Lemma \ref{genericisotopy}, we have $k_{r}^{2}\notin(B\cup S_{r}\cup F_{r})$ for all $r\in[-\varepsilon,\varepsilon]$. Since this is an open property, we can choose neighborhoods $U_{r}$ of every $k_{r}^{2}$ such that $U_{r}\cap(B\cup S_{r}\cup F_{r})=\emptyset$. For all $r\in[-\varepsilon,\varepsilon]$ choose $c_{r}\in W^{u}_{r}(k_{r}^{1})\cap U_{r}$. If $\varepsilon$ is small enough, we get $\widehat{D}_{r}(c_{r})=\widehat{D}_{r}(k_{r,+}^{2})$ for $r<0$ and $\widehat{D}_{r}(c_{r})=\widehat{D}_{r}(k_{r,-}^{2})$ for $r>0$. If $W_{0}^{u}(k_{0}^{1})$ intersects $B,F_{0}$, or $S_{0}$ between $k_{0,+}^{1}$ and $c_{0}$, we can choose $\varepsilon$ small enough such that $W_{r}^{u}(k_{r}^{1})$ intersects $B,F_{r}$, and $S_{r}$ for all $r$ in the same way between $k_{r,+}^{1}$ and $c_{r}$. That means, while computing $\widehat{D}_{r}(k_{r,+}^{1})$ we get the same result up to the point $c_{r}$ for all $r\in[-\varepsilon,\varepsilon]$. So without loss of generality we can assume $\widehat{D}_{r}(k_{r,+}^{1})=\widehat{D}_{r}(c_{r})$. Therefore, we get
\begin{align*}
\widehat{D}_{-\varepsilon}(k_{-\varepsilon,+}^{1})&=\widehat{D}_{-\varepsilon}(k_{-\varepsilon,+}^{2})\\
\widehat{D}_{\varepsilon}(k_{\varepsilon,+}^{1})&=\widehat{D}_{\varepsilon}(k_{\varepsilon,-}^{2}).
\end{align*}
If $\varepsilon$ is small enough, the following holds:
\begin{align*}
\widehat{D}_{\varepsilon}(k_{\varepsilon,+}^{2})&=\Phi_{0,\varepsilon}\circ\widehat{D}_{-\varepsilon}(k_{-\varepsilon,+}^{2})\\
\widehat{D}_{\varepsilon}(k_{\varepsilon,-}^{2})&=\Phi_{0,\varepsilon}\circ\widehat{D}_{-\varepsilon}(k_{-\varepsilon,-}^{2})\\
\widehat{D}_{\varepsilon}(k_{\varepsilon,-}^{1})&=\Phi_{0,\varepsilon}\circ\widehat{D}_{-\varepsilon}(k_{-\varepsilon,-}^{1}).
\end{align*}
With this we can compute the ideal
\begingroup
\allowdisplaybreaks
\begin{align*}
\langle D_{\varepsilon}(C_{1}(K_{\varepsilon}))\rangle_{R}&=\langle D_{\varepsilon}(k_{\varepsilon}^{1}),D_{\varepsilon}(\lbrace k_{\varepsilon}^{2},\dots,k_{\varepsilon}^{n_{1}}\rbrace)\rangle_{R}\\
&=\langle D_{\varepsilon}(k_{\varepsilon}^{1})+D_{\varepsilon}(k_{\varepsilon}^{2}),D_{\varepsilon}(\lbrace k_{\varepsilon}^{2},\dots,k_{\varepsilon}^{n_{1}}\rbrace)\rangle_{R}\\
&=\langle\widehat{D}_{\varepsilon}(k_{\varepsilon,+}^{1})-\widehat{D}_{\varepsilon}(k_{\varepsilon,-}^{1})+\widehat{D}_{\varepsilon}(k_{\varepsilon,+}^{2})-\widehat{D}_{\varepsilon}(k_{\varepsilon,-}^{2}),D_{\varepsilon}(\lbrace k_{\varepsilon}^{2},\dots,k_{\varepsilon}^{n_{1}}\rbrace)\rangle_{R}\\
&=\langle\widehat{D}_{\varepsilon}(k_{\varepsilon,+}^{2})-\widehat{D}_{\varepsilon}(k_{\varepsilon,-}^{1}),D_{\varepsilon}(\lbrace k_{\varepsilon}^{2},\dots,k_{\varepsilon}^{n_{1}}\rbrace)\rangle_{R}\\
&=\langle\Phi_{0,\varepsilon}\circ(\widehat{D}_{-\varepsilon}(k_{-\varepsilon,+}^{2})-\widehat{D}_{-\varepsilon}(k_{-\varepsilon,-}^{1})),\Phi_{0,\varepsilon}\circ D_{-\varepsilon}(\lbrace k_{-\varepsilon}^{2},\dots,k_{-\varepsilon}^{n_{1}}\rbrace)\rangle_{R}\\
&=\langle\Phi_{0,\varepsilon}\circ(\widehat{D}_{-\varepsilon}(k_{-\varepsilon,+}^{1})-\widehat{D}_{-\varepsilon}(k_{-\varepsilon,-}^{1})),\Phi_{0,\varepsilon}\circ D_{-\varepsilon}(\lbrace k_{-\varepsilon}^{2},\dots,k_{-\varepsilon}^{n_{1}}\rbrace)\rangle_{R}\\
&=\langle\Phi_{0,\varepsilon}\circ D_{-\varepsilon}(k_{-\varepsilon}^{1}),\Phi_{0,\varepsilon}\circ D_{-\varepsilon}(\lbrace k_{-\varepsilon}^{2},\dots,k_{-\varepsilon}^{n_{1}}\rbrace)\rangle_{R}\\
&=\langle\Phi_{0,\varepsilon}\circ D_{-\varepsilon}(C_{1}(K_{-\varepsilon}))\rangle_{R}.
\end{align*}
\endgroup
It follows
\begin{align*}
\Cord(K_{\varepsilon})&=\Phi_{0,\varepsilon}(\Cord(K_{-\varepsilon}))\\
&\cong\Cord(K_{-\varepsilon}).
\end{align*}
The other possible situations are to be treated analogously.\\[.5em]
Case (ii,12): One of the cases (ii,1) to (ii,10) holds for a cord which is generated by the application of relation~(iv) (where in the cases (ii,1) to (ii,7) the unstable manifolds are to be replaced by trajectories along the vector field $-\nabla E_{0}$, starting at this cord): Analogous to the cases (ii,1) to (ii,10).\\[.5em]
Case (ii,13): A cord which is generated by the application of relation~(iv) runs along the vector field $-\nabla E_{0}$ to a critical point of index 1: Analogous to case (ii,11).\\[.5em]
Case (ii,14): Creation / cancellation of a pair of critical points of index 0 and 1 or of index 1 and 2. It suffices to consider the creation of a pair of critical points since the other case corresponds to the inverse isotopy.\\[.5em]
a) Creation of a pair of critical points of index 0 and 1. So let $K_{0}\times K_{0}$ contain a critical point $p$ of birth type of index 0 (see Definition \ref{emryonicbirthdeathdef}).\\[.5em]
Since $K_{0}$ is generic with one exception, the proofs of the lemmata \ref{cordlemmaadd1}, \ref{cordlemmaadd2} and \ref{cordlemmaadd3} also work for $p$ and $W_{0}^{u}(p)$. With a similar transversality argument as in these lemmata we can show that $p$ does not lie on the unstable manifold of a critical point of index 1.\\
According to Remark \ref{UcaphatWrem}, we can choose neighborhoods $U_{r}$ and coordinates such that we can represent the situation as in Figure \ref{Index01}. Here, $g_{\varepsilon}$ and $k_{\varepsilon}$ in the right figure are nondegenerate critical points of index 0 and 1, respectively.
\begin{figure}[ht]\centering 
\subfigure{\includegraphics[scale=0.3]{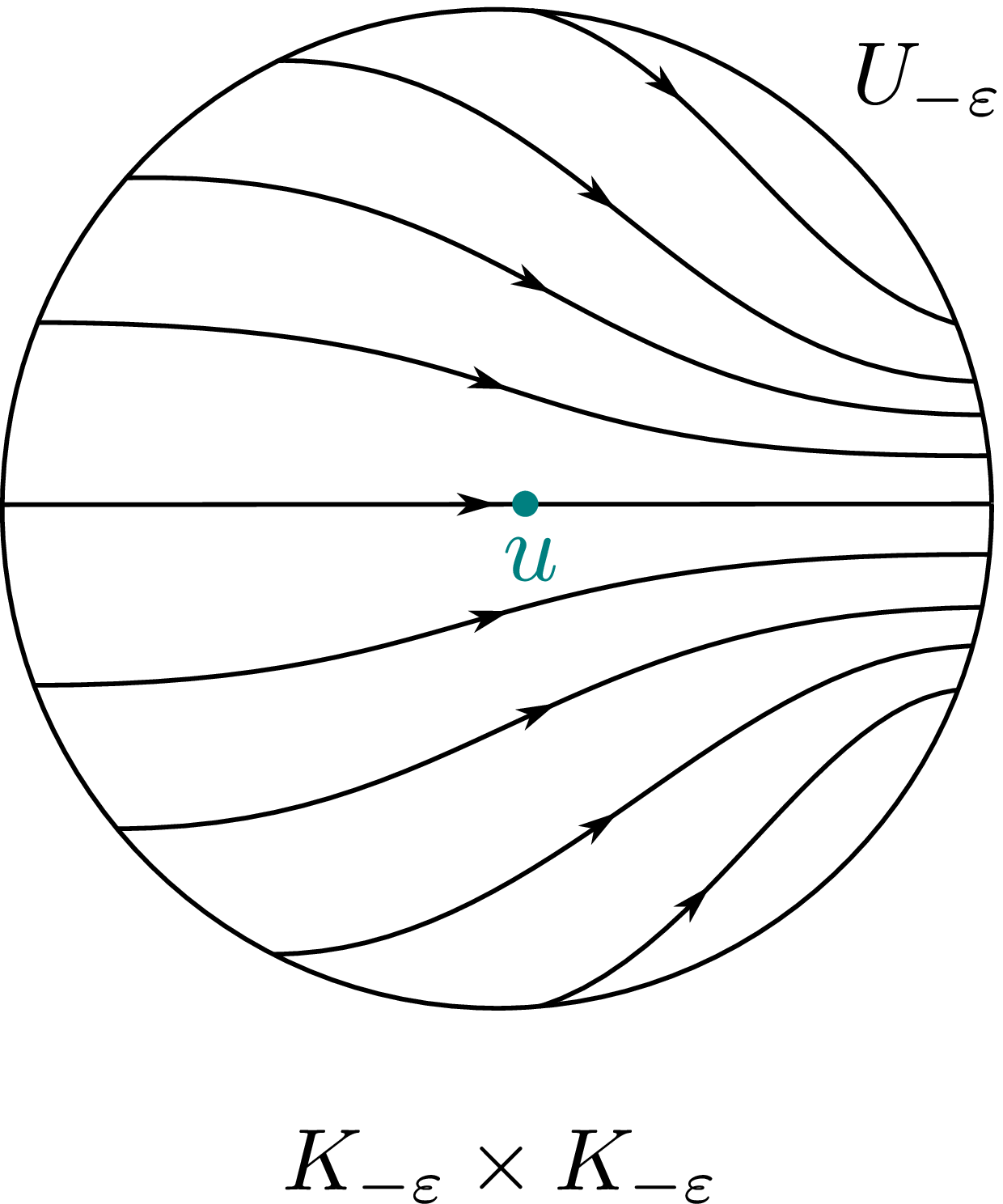}}
\subfigure{\hspace*{1cm}\includegraphics[scale=0.3]{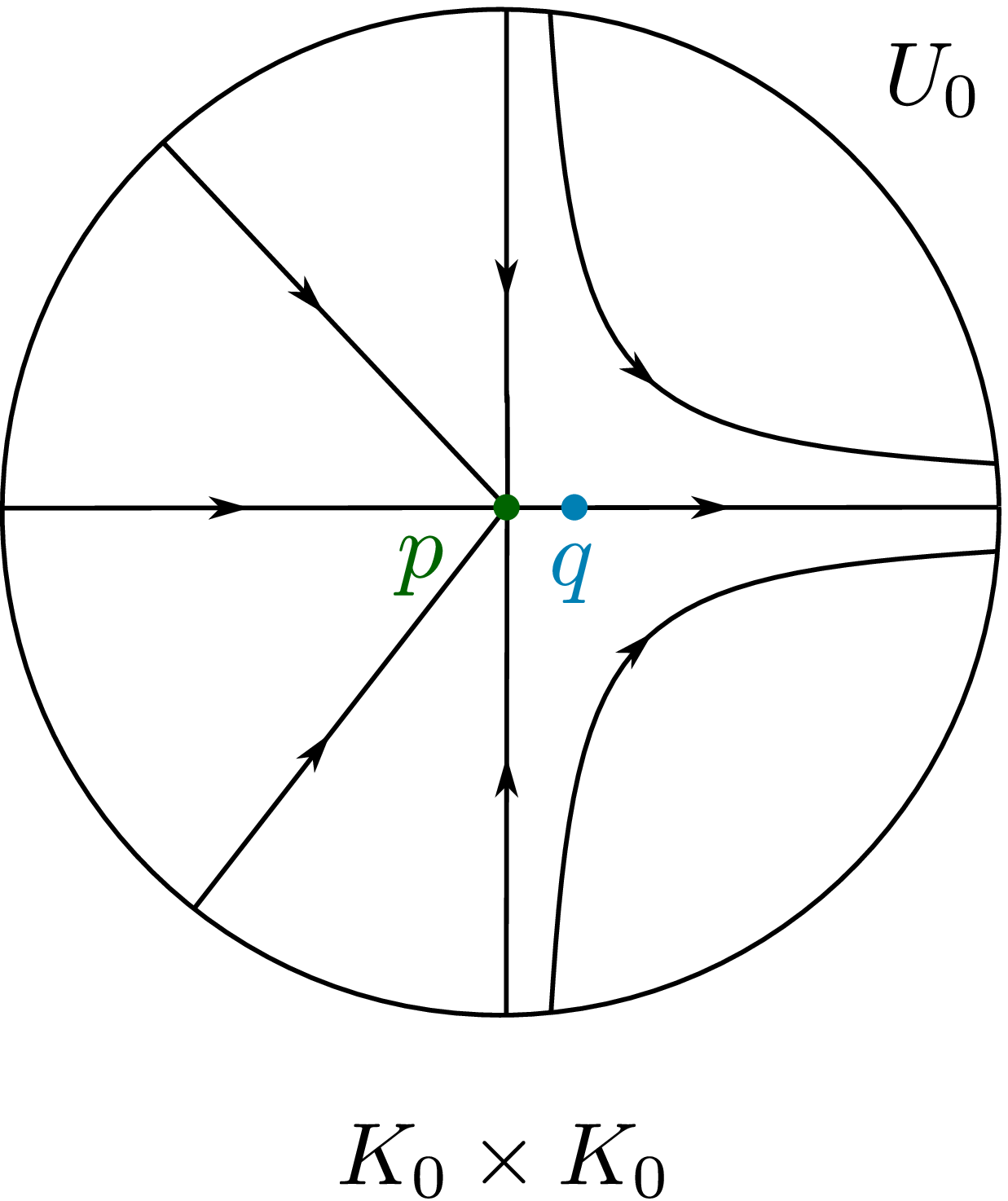}}
\subfigure{\hspace*{1.1cm}\includegraphics[scale=0.3]{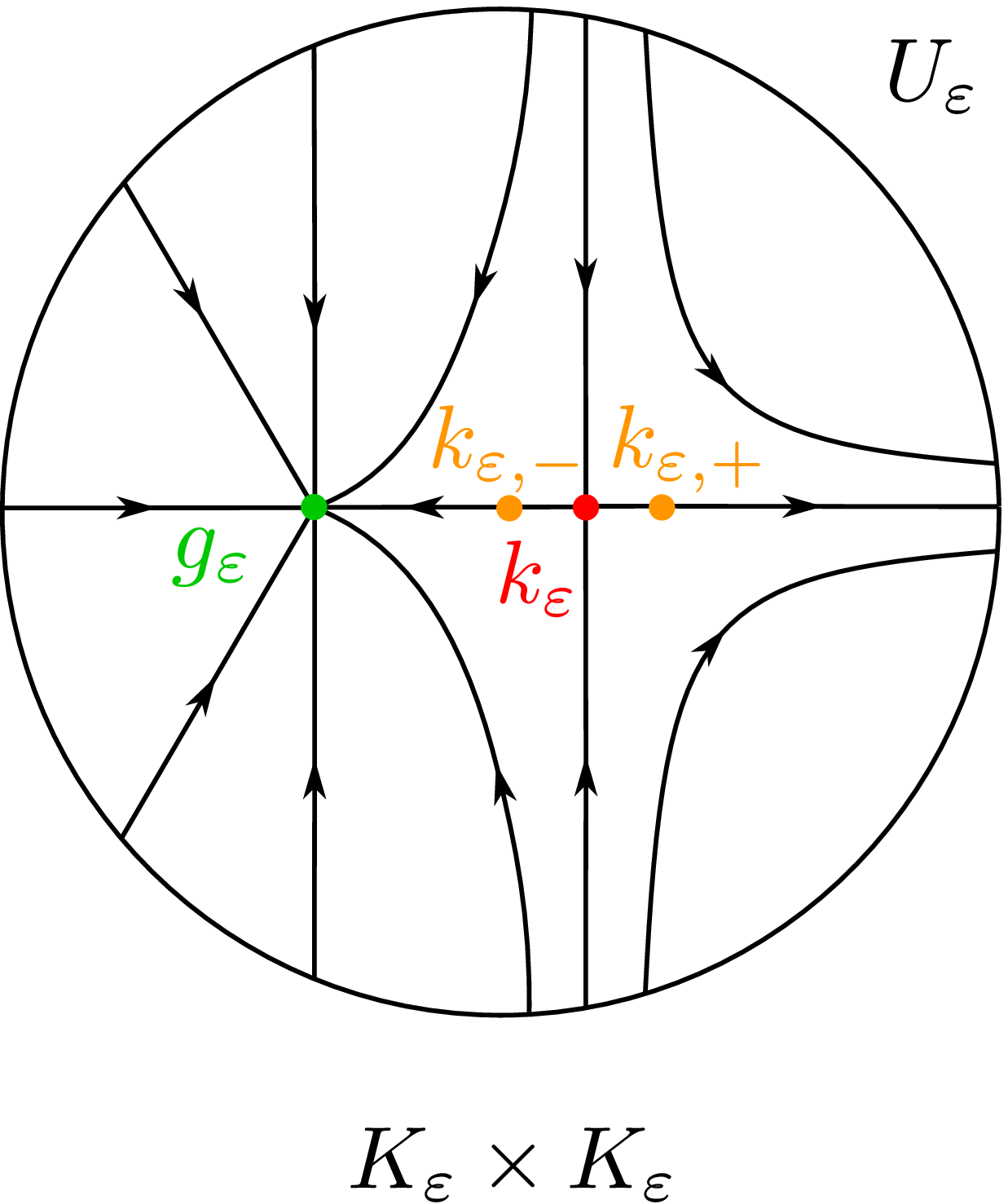}}
\caption{Creation of a pair of critical points of index 0 and 1\label{Index01}}
\end{figure}

Since we get two additional critical points during the isotopy, we have to modify the linear maps $\Phi_{0,r}$ and $\Phi_{1,r}$ a little bit:\\
Let
\[n_{0}:=\vert Crit_{0}(K_{-\varepsilon})\vert\text{ and }n_{1}:=\vert Crit_{1}(K_{-\varepsilon})\vert.\]
Then for all $r\in[-\varepsilon,0)$ we have
\[\vert Crit_{0}(K_{r})\vert=n_{0}\text{ and }\vert Crit_{1}(K_{r})\vert=n_{1}\]
and for all $r\in(0,\varepsilon]$ we have
\[\vert Crit_{0}(K_{r})\vert=n_{0}+1\text{ and }\vert Crit_{1}(K_{r})\vert=n_{1}+1.\]
So let
\begin{align*}
Crit_{0}(K_{r})&=
\begin{cases}
\lbrace g_{r}^{1},\dots,g_{r}^{n_{0}}\rbrace&r\in[-\varepsilon,0)\\
\lbrace g_{0}^{1},\dots,g_{0}^{n_{0}},p\rbrace&r=0\\
\lbrace g_{r}^{1},\dots,g_{r}^{n_{0}},g_{r}\rbrace&r\in(0,\varepsilon]
\end{cases}\\
Crit_{1}(K_{r})&=
\begin{cases}
\lbrace k_{r}^{1},\dots,k_{r}^{n_{1}}\rbrace&r\in[-\varepsilon,0]\\
\lbrace k_{r}^{1},\dots,k_{r}^{n_{1}},k_{r}\rbrace&r\in(0,\varepsilon].
\end{cases}
\end{align*}
Now we define for $i=0,1$
\[Crit_{i}^{[-\varepsilon,\varepsilon]}:=\bigcup_{r\in[-\varepsilon,\varepsilon]}\lbrace r\rbrace\times Crit_{i}(K_{r})\subset[-\varepsilon,\varepsilon]\times T^{2}.\]
%Since the solution of a differential equation depends continuously on the initial condition, we can number the critical points such that the maps
Since the isotopy is smooth, we can number the critical points such that the maps

\begin{align*}
\Psi_{0}:[-\varepsilon,\varepsilon]\times Crit_{0}(K_{-\varepsilon})&\to Crit_{0}^{[-\varepsilon,\varepsilon]}\\
(r,g_{-\varepsilon}^{i})&\mapsto g_{r}^{i},\ i=1,\dots,n_{0}\\
\Psi_{1}:[-\varepsilon,\varepsilon]\times Crit_{1}(K_{-\varepsilon})&\to Crit_{1}^{[-\varepsilon,\varepsilon]}\\
(r,k_{-\varepsilon}^{i})&\mapsto k_{r}^{i},\ i=1,\dots,n_{1}
\end{align*}
are continuous with respect to the first component. For all $r\in[-\varepsilon,\varepsilon]$ we define the linear maps $\Phi_{0,r}$ and $\Phi_{1,r}$ on generators by
\begin{align*}
\Phi_{0,r}:C_{0}(K_{-\varepsilon})&\to\langle g_{r}^{1},\dots,g_{r}^{n_{0}}\rangle_{R}\\
g_{-\varepsilon}^{i}&\mapsto\Psi_{0}(r,g_{-\varepsilon}^{i})=g_{r}^{i},\ i=1,\dots,n_{0}\\
\lambda^{\pm1}&\mapsto\lambda^{\pm1}\\
\mu^{\pm1}&\mapsto\mu^{\pm1}\\
\Phi_{1,r}:C_{1}(K_{-\varepsilon})&\to\langle k_{r}^{1},\dots,k_{r}^{n_{1}}\rangle_{\mathbb{Z}}\\
k_{-\varepsilon}^{i}&\mapsto\Psi_{1}(r,k_{-\varepsilon}^{i})=k_{r}^{i},\ i=1,\dots,n_{1}.
\end{align*}
For all $r\in[-\varepsilon,\varepsilon]$ we extend $\Phi_{0,r}$ to an algebra homomorphism. Obviously, $\Phi_{0,r}$ and $\Phi_{1,r}$ are isomorphisms for all $r\in[-\varepsilon,\varepsilon]$.\\
By the same consideration as above, it follows that for $\varepsilon$ small enough the following holds for all $i=1,\dots,n_{1}$:
\[\widehat{D}_{\varepsilon}(k_{\varepsilon,\pm}^{i})=\Phi_{0,\varepsilon}\circ\widehat{D}_{-\varepsilon}(k_{-\varepsilon,\pm}^{i}).\]
Thus, we have for all $i=1,\dots,n_{1}$:
\begin{align*}
D_{\varepsilon}(k_{\varepsilon}^{i})&=\Phi_{0,\varepsilon}\circ D_{-\varepsilon}(k_{-\varepsilon}^{i})\\
&=\Phi_{0,\varepsilon}\circ D_{-\varepsilon}\circ\Phi_{1,\varepsilon}^{-1}(k_{\varepsilon}^{i}).
\end{align*}
Now we choose $q\in W^{u}_{0}(p)$ such that $\lbrace\varphi_{0}^{s}(q):s\leq0\rbrace\cap(S\cup F\cup B)=\emptyset$. Analogous to the above consideration it can be shown that there exists an open neighborhood $U$ of $(0,q)$ in $[-\varepsilon,\varepsilon]\times T^{2}$ such that for all $(r,x)\in U$ the following holds:
\[\Phi_{0,0}\circ\Phi_{0,r}^{-1}\circ\widehat{D}_{r}(x)=\widehat{D}_{0}(q).\]
For $\varepsilon$ small enough we can assume that we have $U\cap(\lbrace r\rbrace\times T^{2})\neq\emptyset$ for all $r\in[-\varepsilon,\varepsilon]$. Choose $u\in U\cap(\lbrace-\varepsilon\rbrace\times T^{2})$. Also choose $k_{\varepsilon,+}$ and $k_{\varepsilon,-}$ to determine $D_{\varepsilon}(k_{\varepsilon})$. Without loss of generality, $k_{\varepsilon,+}$ and $k_{\varepsilon,-}$ can be chosen as shown in Figure \ref{Index01}. If the orientation of the knot is reversed, $k_{\varepsilon,+}$ and $k_{\varepsilon,-}$ are swapped and the sign of $D_{\varepsilon}(k_{\varepsilon})$ changes. However, nothing changes in the quotient. Then we get:
\[\Phi_{0,0}(\widehat{D}_{-\varepsilon}(u))=\widehat{D}_{0}(q)=\Phi_{0,0}\circ\Phi_{0,\varepsilon}^{-1}(\widehat{D}_{\varepsilon}(k_{\varepsilon,+})).\]
In addition we get
\[\widehat{D}_{\varepsilon}(k_{\varepsilon,-})=g_{\varepsilon},\]
since $\dim W^{s}_{\varepsilon}(g_{\varepsilon})=2$ and $g_{\varepsilon}$ lies arbitrarily close to $k_{\varepsilon}$ for $\varepsilon$ sufficiently small. Therefore, it follows that
\begin{align*}
D_{\varepsilon}(k_{\varepsilon})&=\widehat{D}_{\varepsilon}(k_{\varepsilon,+})-\widehat{D}_{\varepsilon}(k_{\varepsilon,-})\\
&=\Phi_{0,\varepsilon}(\widehat{D}_{-\varepsilon}(u))-g_{\varepsilon}.
\end{align*}
Further we get
\begin{align*}
C_{0}(K_{\varepsilon})&=\langle\Phi_{0,\varepsilon}(C_{0}(K_{-\varepsilon})),g_{\varepsilon}\rangle_{R}\\
C_{1}(K_{\varepsilon})&=\langle\Phi_{1,\varepsilon}(C_{1}(K_{-\varepsilon})),k_{\varepsilon}\rangle_{\mathbb{Z}}.
\end{align*}
Now we can compute
\begingroup
\allowdisplaybreaks
\begin{align*}
\Cord(K_{\varepsilon})&=C_{0}(K_{\varepsilon})/I_\varepsilon\\
&=C_{0}(K_{\varepsilon})/\langle D_{\varepsilon}(C_{1}(K_{\varepsilon}))\rangle\\
&=C_{0}(K_{\varepsilon})/\langle D_{\varepsilon}(\lbrace k_{\varepsilon}^{1},\dots,k_{\varepsilon}^{n_{1}}\rbrace),D_{\varepsilon}(k_{\varepsilon})\rangle\\
&=\langle\Phi_{0,\varepsilon}(C_{0}(K_{-\varepsilon})),g_{\varepsilon}\rangle_{R}/\langle\Phi_{0,\varepsilon}\circ D_{-\varepsilon}\circ\Phi_{1,\varepsilon}^{-1}(\lbrace k_{\varepsilon}^{1},\dots,k_{\varepsilon}^{n_{1}}\rbrace),D_{\varepsilon}(k_{\varepsilon})\rangle\\
&=\langle\Phi_{0,\varepsilon}(C_{0}(K_{-\varepsilon})),g_{\varepsilon}\rangle_{R}/\langle\Phi_{0,\varepsilon}\circ D_{-\varepsilon}(\lbrace k_{-\varepsilon}^{1},\dots,k_{-\varepsilon}^{n_{1}}\rbrace),\Phi_{0,\varepsilon}(\widehat{D}_{-\varepsilon}(u))-g_{\varepsilon}\rangle\\
&=\langle\Phi_{0,\varepsilon}(C_{0}(K_{-\varepsilon})),g_{\varepsilon}\rangle_{R}/\langle\Phi_{0,\varepsilon}(I_{-\varepsilon}),\Phi_{0,\varepsilon}(\widehat{D}_{-\varepsilon}(u))-g_{\varepsilon}\rangle\\
&\hspace*{-.04cm}\overset{(1)}{=}\Phi_{0,\varepsilon}(C_{0}(K_{-\varepsilon}))/\Phi_{0,\varepsilon}(I_{-\varepsilon})\\
&\hspace*{-.04cm}\overset{(2)}{=}\Phi_{0,\varepsilon}\big(C_{0}(K_{-\varepsilon})/I_{-\varepsilon}\big)\\
&=\Phi_{0,\varepsilon}(\Cord(K_{-\varepsilon}))\\
&\cong\Cord(K_{-\varepsilon}).
\end{align*}
\endgroup
Equality (1) holds because on the one hand we have $\widehat{D}_{-\varepsilon}(u)\in C_{0}(K_{-\varepsilon})$ and on the other hand we have in the quotient: $g_{\varepsilon}=\Phi_{0,\varepsilon}(\widehat{D}_{-\varepsilon}(u))$. It follows that $\langle g_{\varepsilon}\rangle\subset\Phi_{0,\varepsilon}(C_{0}(K_{-\varepsilon}))$. Equality (2) holds because $\Phi_{0,\varepsilon}$ is an algebra isomorphism.\\[.5em]
b) Creation of a pair of critical points of index 1 and 2. So let $K_{0}\times K_{0}$ contain a critical point $p$ of birth type of index 1.\\
With a similar transversality argument as before we can show that $p$ does not lie on the stable or unstable manifold of a critical point of index 1.\\
According to Remark \ref{UcaphatWrem}, we can choose neighborhoods $U_{r}$ and coordinates such that we can represent the situation as in Figure \ref{Index12}. Here, $k_{\varepsilon}$ and $l_{\varepsilon}$ in the right figure are nondegenerate critical points of index 1 and 2, respectively.
\begin{figure}[ht]\centering 
\subfigure{\includegraphics[scale=0.3]{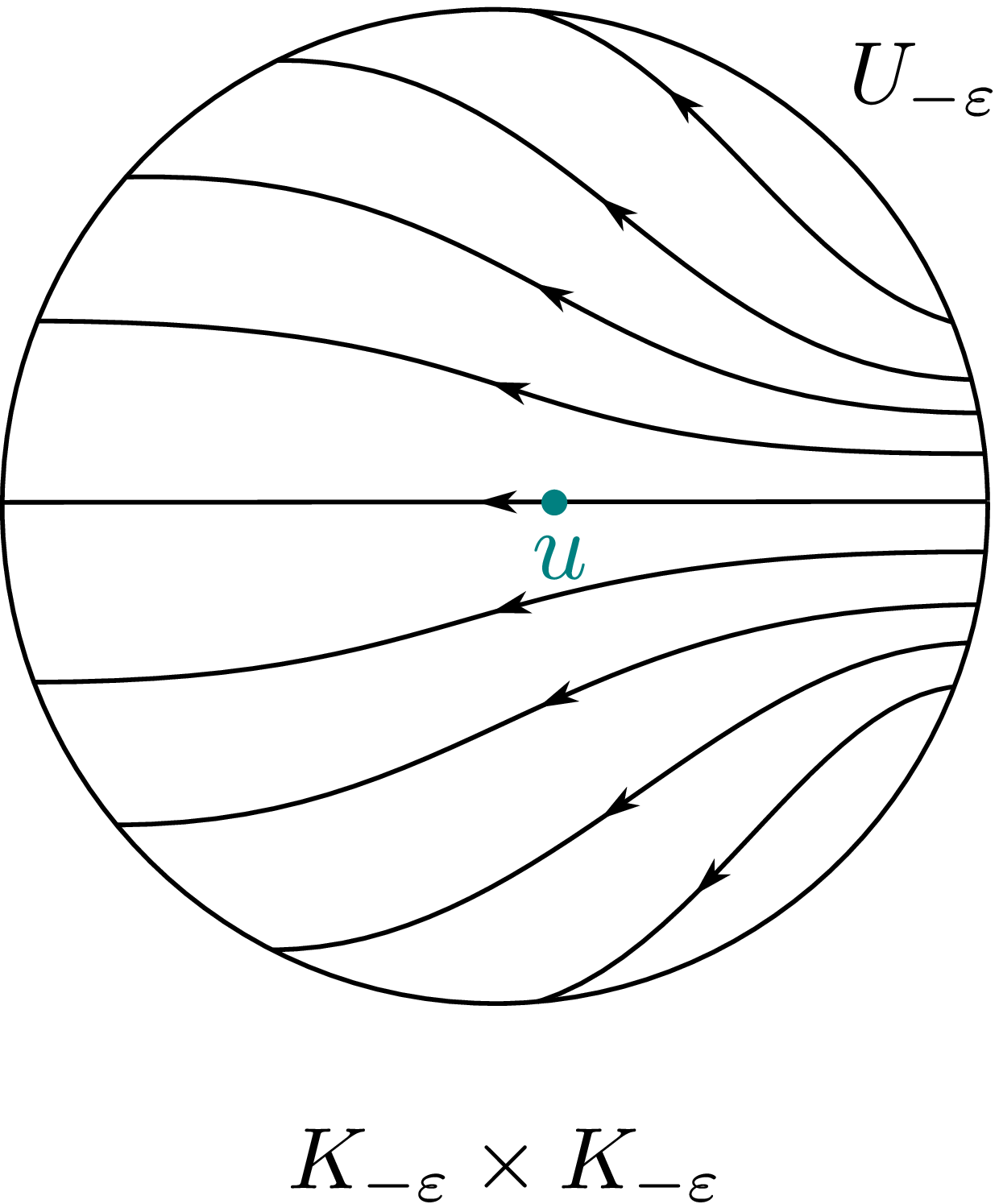}}
\subfigure{\hspace*{1cm}\includegraphics[scale=0.3]{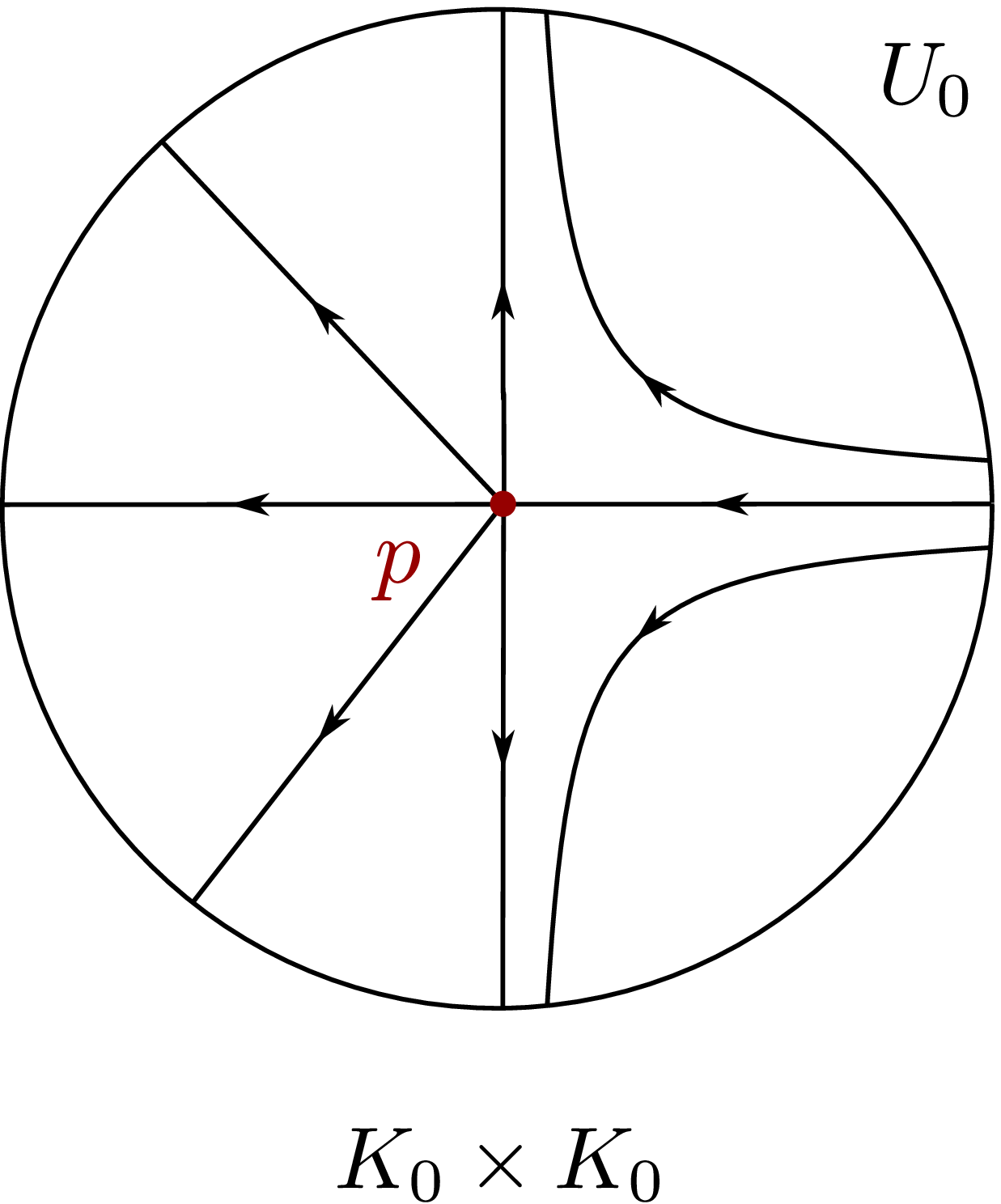}}
\subfigure{\hspace*{1.1cm}\includegraphics[scale=0.3]{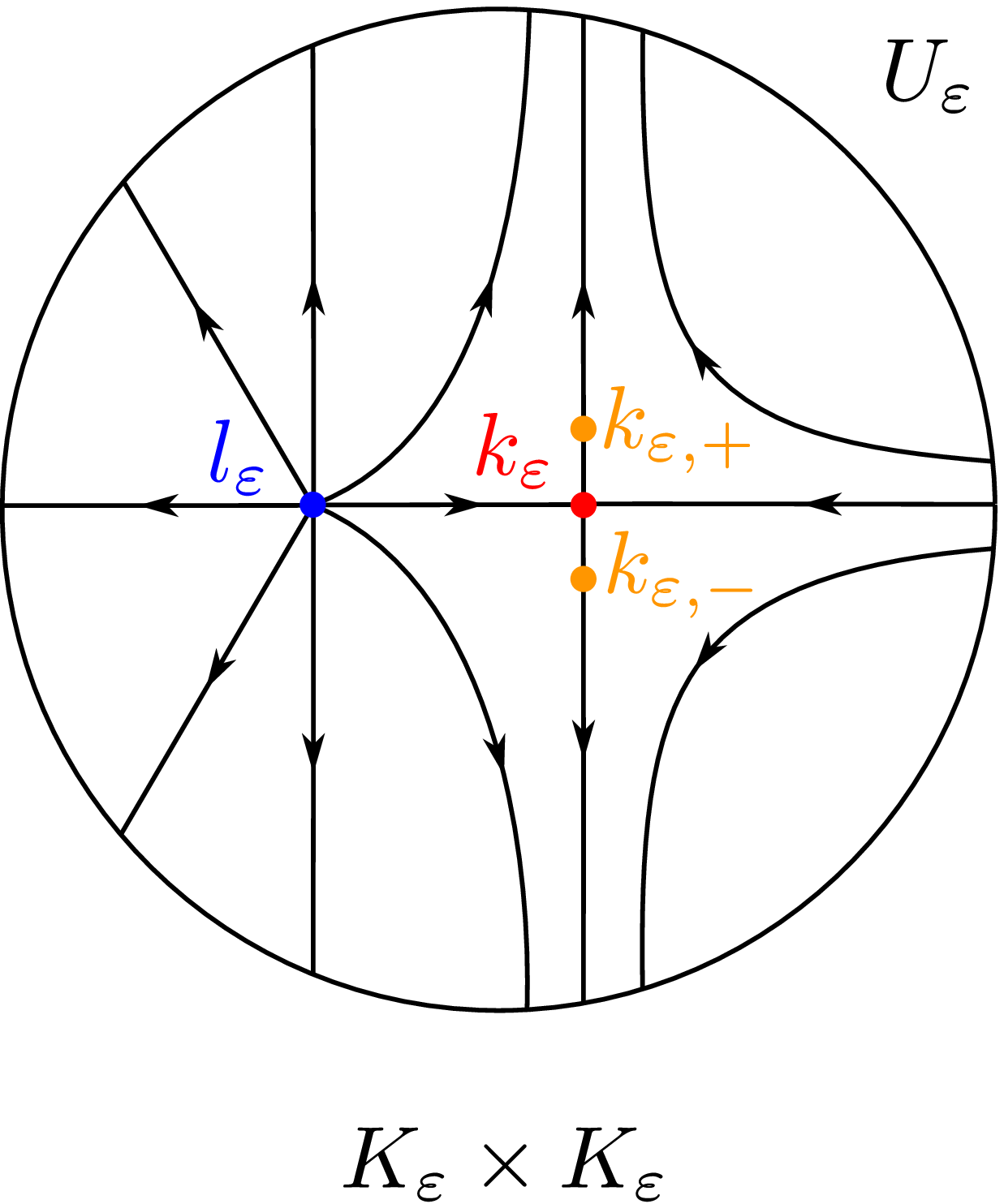}}
\caption{Creation of a pair of critical points of index 1 and 2\label{Index12}}
\end{figure}
We can divide $K_{r}\times K_{r}$ into open subsets bounded by ``broken'' trajectories \cite{Aud} such that $K_{r}\times K_{r}$ is the union of the closures of these open sets. Figure~\ref{brokentraj} shows this schematically on the left, where $l_{r},k_{r}^{1},k_{r}^{2}$ and $g_{r}$ are critical points with $\Ind(l_{r})=2,\Ind(k_{r}^{1})=\Ind(k_{r}^{2})=1$, and $\Ind(g_{r})=0$. These open subsets are of the form
\[\lbrace x\in K_{r}\times K_{r}:\lim_{s\to\infty}\varphi_{r}^{s}(x)=g_{r}\text{ and}\lim_{s\to-\infty}\varphi_{r}^{s}(x)=l_{r}\rbrace.\]
\begin{figure}[ht]\centering 
\subfigure{\includegraphics[scale=0.4]{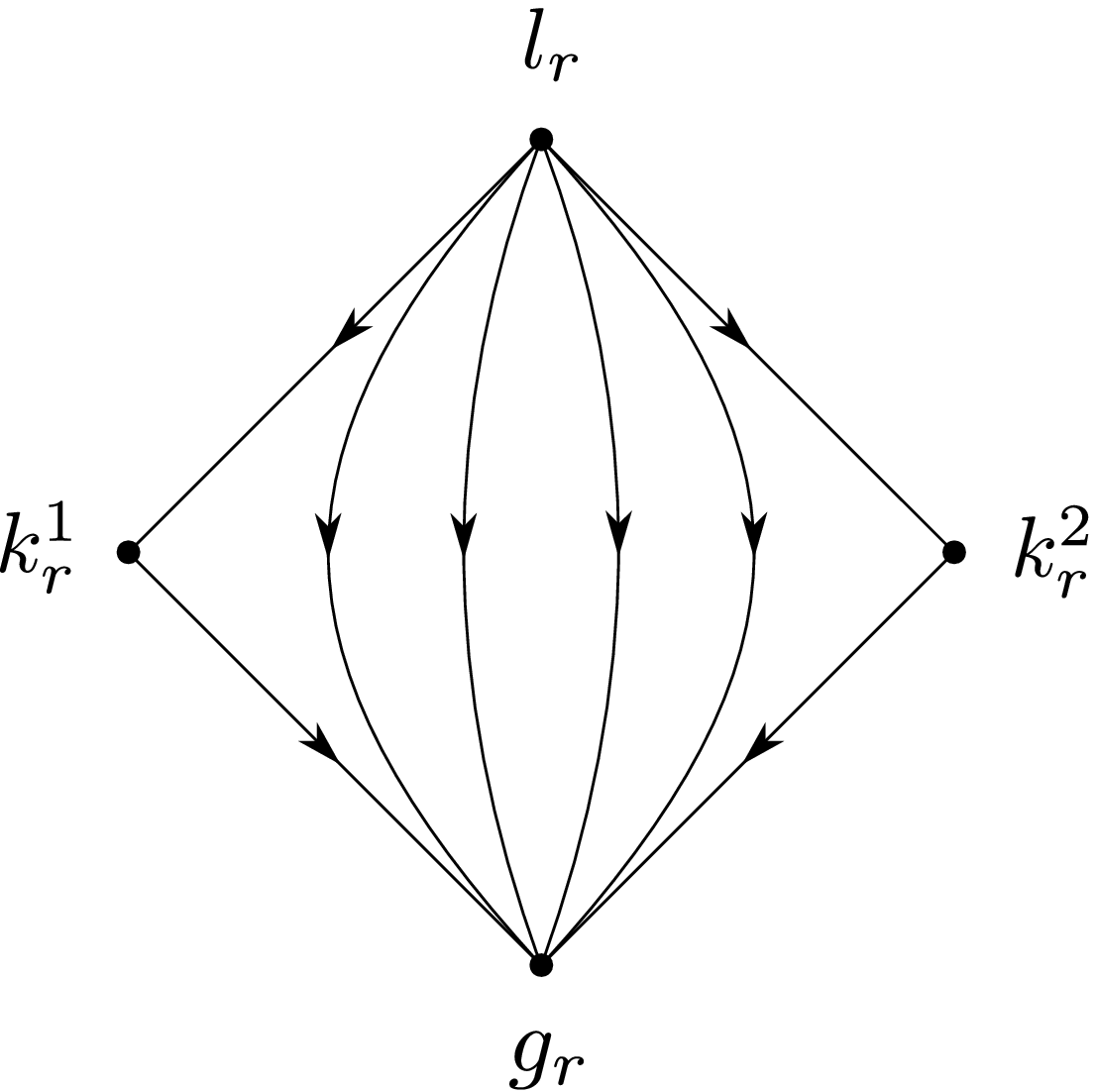}}
\subfigure{\hspace*{1cm}\includegraphics[scale=0.4]{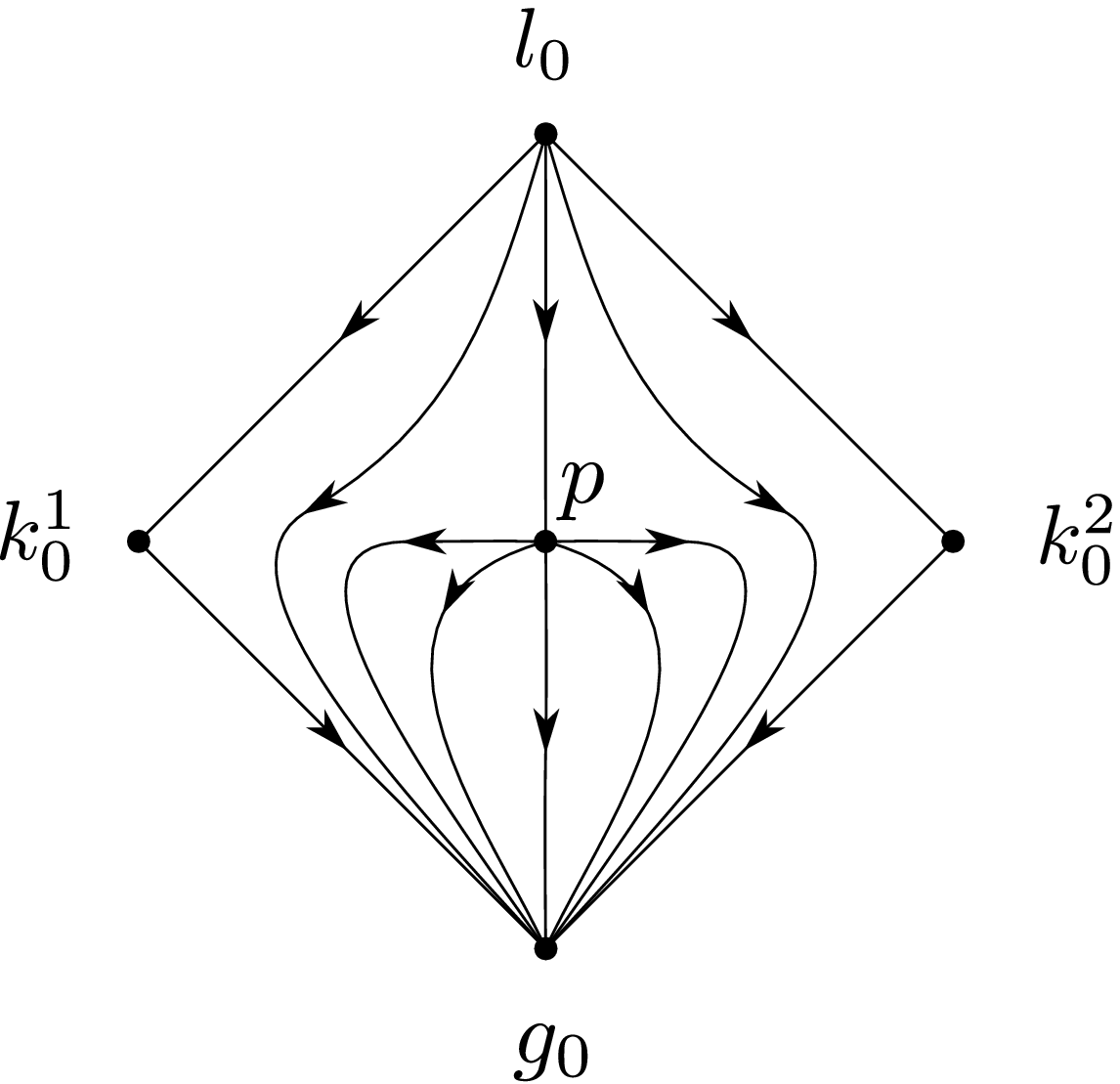}}
\caption{Broken trajectories and position of the critical point $p$ of birth type of index 1\label{brokentraj}}
\end{figure}%
A ``broken'' trajectory starts at a critical point of index 2, runs along a stable manifold of a critical point of index 1 to that point, then on the unstable manifold of that point to a critical point of index 0.\\
So $p$ lies in such an open subset of $K_{0}\times K_{0}$, as shown in Figure \ref{brokentraj} on the right. Thus, it is obvious that the trajectories starting from $k_{\varepsilon,+}$ and $k_{\varepsilon,-}$ end at the same critical point $g_{\varepsilon}$ of index 0. If these trajectories intersect the sets $S_{\varepsilon},F_{\varepsilon}$, and $B$ in the same way (as well as all flow lines resulting from the application of relation (iv)) such that $\widehat{D}_{\varepsilon}(k_{\varepsilon,+})=\widehat{D}_{\varepsilon}(k_{\varepsilon,-})$, we get $D_{\varepsilon}(k_{\varepsilon})=0$, and therefore $\Cord(K_{-\varepsilon})\cong\Cord(K_{\varepsilon})$.\\
If this is not the case, we perturb the 1-parameter family $(E_{r})_{r\in[-\varepsilon,\varepsilon]}$ of energy functions with the help of Lemma \ref{perturbationofWug} such that we get the 1-parameter family $(E^{\prime}_{r})_{r\in[-\varepsilon,\varepsilon]}$ with the following properties (see also Figure~\ref{Perturbenfct}):
\begin{itemize}\itemsep0pt
\item[(i)]$E_{-\varepsilon}^{\prime}=E_{-\varepsilon}$ and $E_{\varepsilon}^{\prime}=E_{\varepsilon}$.
\item[(ii)]The flow lines starting at $p$ are arbitrarily close to each other, so they run inside a $\delta$-tube for a $\delta>0$.
\item[(iii)]The flow lines starting at $k_{r,+}^{\prime}$ and $k_{r,-}^{\prime}$ run also within a $\delta$-tube for all $r\in(0,\hat{\varepsilon})$ where $0<\hat{\varepsilon}<\varepsilon$. $\hat{\varepsilon}$ is chosen such that the energy functions $E_{r}^{\prime}$ are generic for all $r\in[-\hat{\varepsilon},\hat{\varepsilon}]\setminus\lbrace0\rbrace$.
\end{itemize}
\begin{figure}[htbp]\centering 
\includegraphics[scale=0.5]{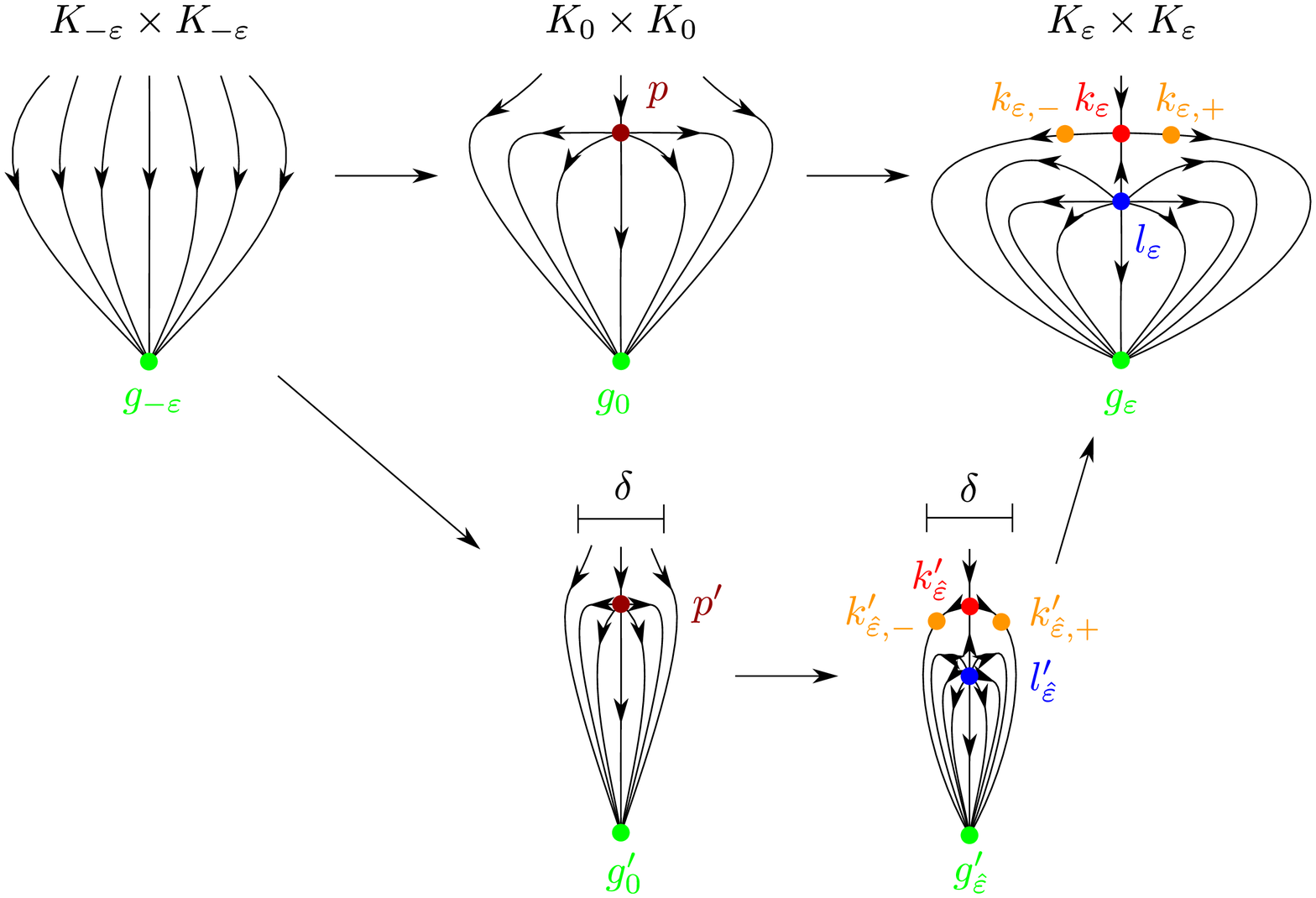}
\caption{Perturbation of the 1-parameter family of energy functions}\label{Perturbenfct}
\end{figure}%
Now we choose $\delta$ small enough such that the flow lines starting at $k_{\hat{\varepsilon},+}^{\prime}$ and $k_{\hat{\varepsilon},-}^{\prime}$ intersect the sets $S_{\hat{\varepsilon}},F_{\hat{\varepsilon}}$ and $B$ in the same way (and also all flow lines that result from the application of relation~(iv)). We now determine the cord algebras for the knots $K_{r}, r\in[-\varepsilon,\varepsilon]$ with respect to the energy functions $E_{r}^{\prime}$, except in the finitely many cases where $K_{r}$ is non-generic, and get:
\begin{itemize}\itemsep0pt
\item[(i)]For $r\in[-\varepsilon,0)$ and $r\in(\hat{\varepsilon},\varepsilon]$ the cases (ii, 1) to (ii, 13) from Lemma \ref{genericisotopy} may occur. We have shown yet that the cord algebra does not change in these cases.
\item[(ii)]For $r\in(0,\hat{\varepsilon})$ the cord algebra does not change since only generic energy functions occur. 
\item[(iii)]Since $D^{\prime}_{\hat{\varepsilon}}(k_{\hat{\varepsilon}}^{\prime})=\widehat{D}^{\prime}_{\hat{\varepsilon}}(k_{\hat{\varepsilon},+}^{\prime})-\widehat{D}^{\prime}_{\hat{\varepsilon}}(k_{\hat{\varepsilon},-}^{\prime})=0$, it follows that the cord algebra stays the same for all $r\in(0,\varepsilon]$.
\end{itemize}
It remains to be shown that the cord algebra does not change at the transition over $r=0$. For this we use the linear maps as described above:
\begingroup
\allowdisplaybreaks
\begin{align*}
\Phi_{0,r}:C_{0}(K_{-\hat{\varepsilon}})&\to C_{0}(K_{r})\\
g_{-\hat{\varepsilon}}^{i}&\mapsto\Psi_{0}(r,g_{-\hat{\varepsilon}}^{i})=g_{r}^{i},\ i=1,\dots,n_{0}\\
\lambda^{\pm1}&\mapsto\lambda^{\pm1}\\
\mu^{\pm1}&\mapsto\mu^{\pm1}\displaybreak\\
\Phi_{1,r}:C_{1}(K_{-\hat{\varepsilon}})&\to\langle k_{r}^{1},\dots,k_{r}^{n_{1}}\rangle_{\mathbb{Z}}\\
k_{-\hat{\varepsilon}}^{i}&\mapsto\Psi_{1}(r,k_{-\hat{\varepsilon}}^{i})=k_{r}^{i},\ i=1,\dots,n_{1},
\end{align*}
\endgroup
where $r\in[-\hat{\varepsilon},\hat{\varepsilon}]$. By the same consideration as above, it follows that for $\varepsilon$ small enough the following holds for all $i=1,\dots,n_{1}$:
\begin{align*}
D_{\hat{\varepsilon}}^{\prime}(k_{\hat{\varepsilon}}^{i})&=\Phi_{0,\hat{\varepsilon}}\circ D_{-\hat{\varepsilon}}^{\prime}(k_{-\hat{\varepsilon}}^{i})\\
&=\Phi_{0,\hat{\varepsilon}}\circ D_{-\hat{\varepsilon}}^{\prime}\circ\Phi_{1,\hat{\varepsilon}}^{-1}(k_{\hat{\varepsilon}}^{i}).
\end{align*}
Now we can compute
\begingroup
\allowdisplaybreaks
\begin{align*}
\Cord(K_{\hat{\varepsilon}})&=C_{0}(K_{\hat{\varepsilon}})/I_{\hat{\varepsilon}}\\
&=C_{0}(K_{\hat{\varepsilon}})/\langle D_{\hat{\varepsilon}}^{\prime}(C_{1}(K_{\hat{\varepsilon}}))\rangle\\
&=C_{0}(K_{\hat{\varepsilon}})/\langle D_{\hat{\varepsilon}}^{\prime}(\lbrace k_{\hat{\varepsilon}}^{1},\dots,k_{\hat{\varepsilon}}^{n_{1}}\rbrace),\underbrace{D_{\hat{\varepsilon}}^{\prime}(k_{\hat{\varepsilon}}^{\prime})}_{=0}\rangle\\
&=C_{0}(K_{\hat{\varepsilon}})/\langle D_{\hat{\varepsilon}}^{\prime}(\lbrace k_{\hat{\varepsilon}}^{1},\dots,k_{\hat{\varepsilon}}^{n_{1}}\rbrace)\rangle\\
&=\Phi_{0,\hat{\varepsilon}}(C_{0}(K_{-\hat{\varepsilon}}))/\langle\Phi_{0,\hat{\varepsilon}}\circ D_{-\hat{\varepsilon}}\circ\Phi_{1,\hat{\varepsilon}}^{-1}(\lbrace k_{\hat{\varepsilon}}^{1},\dots,k_{\hat{\varepsilon}}^{n_{1}}\rbrace)\rangle\\
&=\Phi_{0,\hat{\varepsilon}}(C_{0}(K_{-\hat{\varepsilon}}))/\langle\Phi_{0,\hat{\varepsilon}}\circ D_{-\hat{\varepsilon}}(C_{1}(K_{-\hat{\varepsilon}}))\rangle\\
&=\Phi_{0,\hat{\varepsilon}}(C_{0}(K_{-\hat{\varepsilon}}))/\Phi_{0,\hat{\varepsilon}}(I_{-\hat{\varepsilon}})\\
&=\Phi_{0,\hat{\varepsilon}}\big(C_{0}(K_{-\hat{\varepsilon}})/I_{-\hat{\varepsilon}}\big)\\
&=\Phi_{0,\hat{\varepsilon}}(\Cord(K_{-\hat{\varepsilon}}))\\
&\cong\Cord(K_{-\hat{\varepsilon}}).
\end{align*}
\endgroup
\end{proof}
\section{Final Remark}
Finally, we will briefly compare the topological definition of Lenhard Ng's cord algebra with our definition using Morse theory.\\
The original version of the cord algebra is easy to define. However, when calculating the cord algebra for a given knot the following must be considered: If one has found several generators and relations, it is still to be shown that no further generators or relations exist. This proof must be given for each knot individually.\\
The definition of the cord algebra with the help of Morse theory is very complex, since first some properties of generic knots and a generic framing are to be shown, and then the boundary map is to be defined. As can be seen in the examples, the determination of the individual relations is also laborious and must be carried out very carefully. However, since there are only finitely many critical points of index 1, one has surely found all relations and generators, as soon as one has determined the boundary map for each of these critical points.

\section*{Appendix}
\addcontentsline{toc}{section}{Appendix}
\appendix
\section{Background material}\label{Backgroundmaterial}
In this section we will recollect some statements from differential topology and Morse theory that we need in the Sections \ref{Corddef} and \ref{Knotinv}. All statements are presented without proofs. These can be found in the respective literature. 
\subsection{Differential topology}
This section is taken from \cite{Hir} (Chapters 2.4 and 3.2) except for the relative versions of the transversality theorem and the jet transversality theorem.\\[.5em]
We denote by $C(X,Y)$ the set of continuous maps from a space $X$ to a
space $Y$. The \textit{compact open topology} on $C(X,Y)$ is generated by the subbase
comprising all sets of the form
\[\lbrace f\in C(X,Y):f(K)\subset V\rbrace\]
where $K\subset X$ is compact and $V\subset Y$ is open. We also call this the \textit{weak topology} to contrast it with another topology defined below. The resulting
topological space is denoted by $C_{W}(X,Y)$.\\
The space $C_{S}(X,Y)$ is the set $C(X,Y)$ with the following \textit{strong topology}. Let $\Gamma_{f}\subset X\times Y$ denote the graph of the map $f\in C(X,Y)$. If $W\subset X\times Y$ is an open set containing $\Gamma_{f}$, let
\[\mathcal{N}(f,W):=\lbrace g\in C(X,Y):\Gamma_{g}\subset W\rbrace.\]
These sets, for all $f$ and $W$, form a base for the strong topology. The induced
topology on a subset of $C(X,Y)$ is also called \textit{strong}.\\
If $X$ is compact, the weak and strong topologies are the same.\\[.5em]
Let $M, N$ be $C^{r}$ manifolds, $0\leq r<\infty$. An \textit{$r$-jet from $M$ to $N$} is an equivalence class $[x,f,U]_{r}$ of triples $(x,f,U)$, where $U\subset M$ is an open set, $x\in U$, and $f:U\to N$ is a $C^{r}$ map. The equivalence relation is: $[x,f,U]_{r}=[x^{\prime},f^{\prime},U^{\prime}]_{r}$ if $x=x^{\prime}$, $f(x)=f^{\prime}(x)$ and in some (and hence any) pair of charts adapted to $f$ at $x$, $f$ and $f^{\prime}$ have the same derivatives up to order $r$. We use the notation
\[[x,f,U]_{r}=j^{r}_{x}=j^{r}f(x),\]
to denote the \textit{$r$-jet of $f$ at $x$}. The set of all $r$-jets from $M$ to $N$ is denoted by $J^{r}(M,N)$. In fact, if $M,N$ are $C^{r+s}$ manifolds, $J^{r}(M,N)$ has differentiability class $C^{s}$.\\
For each $C^{r}$ map $f:M\to N$ we define a map 
\begin{align*}
j^{r}f:M&\to J^{r}(M,N)\\
x&\mapsto j^{r}f(x).
\end{align*}
This \textit{$r$-prolongation} of $f$ is continuous and in fact $C^{s}$ if $M$
and $N$ are $C^{r+s}$. We consider $j^{r}f$ as a kind of intrinsic $r$'th derivative of $f$. It is clear that $j^{r}$ is injective.\\[.5em]
Let $f:M\to N$ be a $C^{1}$ map and $A\subset N$ a submanifold. If $K\subset M$, we write $f\pitchfork_{K}A$ to mean that $f$ is \textit{transverse to $A$ along $K$}, that is, whenever $x\in K$ and $f(x)=y\in A$, the tangent space $T_{y}N$ is spanned by $T_{y}A$ and the image of $D_{x}f$. When $K=M$ we simply write $f\pitchfork A$.\\
If $f\pitchfork A$, then $f^{-1}(A)$ is a submanifold (under certain restrictions on boundary behavior) and $\codim(f^{-1}(A)\subset M)=\codim(A\subset N)$. This is one of the main reasons for the importance of transversality.\\
Define
\[\pitchfork^{r}_{K}\mspace{-4mu}(M,N;A):=\lbrace f\in C^{r}(M,N):f\pitchfork_{K}A\rbrace\]
and
\[\pitchfork^{r}\mspace{-4mu}(M,N;A):=\ \pitchfork^{r}_{M}\mspace{-4mu}(M,N;A).\]
Recall that a \textit{residual} subset of a space $X$ is one which contains the intersection of countably many dense open sets. All manifolds and submanifolds are tacitly assumed to be $C^{\infty}$.
\begin{thm}[Transversality Theorem]\label{Transvthm}
Let $M,N$ be manifolds and $A\subset N$ a submanifold. Let $1\leq r<\infty$. Then:
\begin{itemize}\itemsep0pt
\item[(i)]$\pitchfork^{r}\mspace{-6mu}(M,N;A)$ is residual (and therefore dense) in $C^{r}(M,N)$ for both the strong and weak topologies.
\item[(ii)]Suppose $A$ is closed in $N$. If $L\subset M$ is closed (resp.~compact), then
$\pitchfork^{r}_{L}\mspace{-4mu}(M,N;A)$ is dense and open in $C^{r}_{S}(M,N)$ (resp.~in $C^{r}_{W}(M,N)$).
\end{itemize}
\end{thm}
\begin{thm}[Relative Transversality Theorem, \cite{Gol}]\label{Transvthmrel}
Let $U_{1}$ and $U_{2}$ be open subsets of $M$ with $\overline{U}_{1}\subset U_{2}$. Let $f$ be in $C^{\infty}(M,N)$ and $V$ be an open neighborhood of $f$ in $C^{\infty}(M,N)$. Then there is a smooth mapping $g:M\to N$ in $V$ such that $g=f$ on $U_{1}$ and $g\pitchfork A$ off $U_{2}$.
\end{thm}
We define the \textit{jet map}
\[j^{r}:C^{s}(M,N)\to C^{s-r}(M,J^{r}(M,N)),\]
where $1\leq r<s\leq\infty$. Let $A\subset J^{r}(M,N)$ be a submanifold. Let $g:M\to N$ be a $C^{s}$ map. We try to approximate $g$ by another $C^{s}$ map $h$ whose prolongation $j^{r}h:M\to J^{r}(M,N)$ is transverse to $A$. Denote the set of such maps $h$ by $\pitchfork^{s}\mspace{-4mu}(M,N;j^{r},A)$.
\begin{thm}[Jet Transversality Theorem]\label{JetTransvthm}
Let $M,N$ be $C^{\infty}$ manifolds without boundary, and let $A\subset J^{r}(M,N)$ be a $C^{\infty}$ submanifold. Suppose $1\leq r<s\leq\infty$.\\
Then $\pitchfork^{s}\mspace{-4mu}(M,N;j^{r},A)$ is residual and thus dense in $C^{s}_{S}(M,N)$, and open if $A$ is closed.
\end{thm}
\begin{thm}[Relative Jet Transversality Theorem]\label{JetTransvthmrel}
Let $M,N$ be $C^{\infty}$ manifolds without boundary, and let $A\subset J^{r}(M,N)$ be a $C^{\infty}$ submanifold. Let $K\subset M$ be a compact subset. Let $f_{0}\in C^{s}(M,N)$ with $f_{0}\pitchfork A$ on $\overline{M\setminus K}$. Let $1\leq r<s\leq\infty$.\\
Then $\pitchfork_{K}^{s}\mspace{-5mu}(M,N;A,f_{0}):=\lbrace f\in C^{s}(M,N):f\pitchfork A,f=f_{0}\text{ on }M\setminus K\rbrace$ is open and dense in $\lbrace f\in C^{s}(M,N):f=f_{0}\text{ on }M\setminus K\rbrace$.
\end{thm}
\begin{proof}
Openness follows as in the case of the jet transversality theorem.\\
To proove density, let $L\subset\mathring{K}\subset M$ be a compact subset and $\delta>0$. $L$ and $\delta$ will be determined more precisely later. Let $f\in C^{s}(M,N)$ with $f=f_{0}$ on $M\setminus K$ be given. It follows from the jet transversality theorem that there exists a function $g\in C^{s}(M,N)$ with $g\pitchfork A$ and
\[d_{C^{s}(M,N)}(g\vert_{K},f\vert_{K})<\delta,\]
where $d_{C^{s}(M,N)}$ is the distance in the $C^{s}(M,N)$ norm. Let $\varphi\in C^{\infty}(M,[0,1])$ be a smooth cutter function with
\vspace*{-.4cm}
\begin{align*}
\varphi\vert_{L}&\equiv1\text{ and}\\
\varphi\vert_{M\setminus K}&\equiv0.\vspace*{-.2cm}
\end{align*}%
\vspace*{-.2cm}
With this function $\varphi$ we define the map
\[h:=\varphi g+(1-\varphi)f.\]
The map $h$ satisfies:
\begin{itemize}\itemsep0pt\vspace*{-.2cm}
\item[(i)]$h=f_{0}$ on $M\setminus K$.
\item[(ii)]Let $\varepsilon>0$ be given. Since $\delta$ can be chosen arbitrarily small, $d_{C^{s}(M,N)}(h,f)<\varepsilon$ can be achieved.
\item[(iii)]$h\pitchfork A$ since:\\
Obviously, $h\vert_{L}\pitchfork A$ and $h\vert_{M\setminus K}\pitchfork A$. So it still remains to show $h\vert_{K\setminus L}\pitchfork A$:\\
Since $f_{0}$ is transverse to $A$ on $\overline{M\setminus K}$, it follows: Because of the openness of transversality, $L$ can be chosen sufficiently large such that $f_{0}\pitchfork A$ on $\overline{M\setminus L}$. Furthermore, it holds for $\tilde{\varepsilon}>0$: If $L$ is chosen big enough and $\delta$ is chosen small enough, it can be guaranteed that 
\[d_{C^{s}\left(\overline{K\setminus L},N\right)}(f_{0},h)\leq d_{C^{s}\left(\overline{K\setminus L},N\right)}(f_{0},f)+d_{C^{s}\left(\overline{K\setminus L},N\right)}(f,h)<\tilde{\varepsilon}.\]
It follows: If $\tilde{\varepsilon}$ is chosen small enough, the following holds because of the openness of transversality:
\[h\vert_{\overline{K\setminus L}}\pitchfork A.\]
\end{itemize}
\vspace*{-.4cm}
\end{proof}

The jet transversality theorem can be extended to families of submanifolds as follows: 
\begin{thm}\label{JetTransvthmadd}
Let $A_{0},\dots,A_{q}$ be $C^{\infty}$ submanifolds of $J^{r}(M,N)$. If $1\leq r<s\leq\infty$, the set
\[\lbrace f\in C^{s}(M,N):j^{r}f\pitchfork A_{k},k=0,\dots q\rbrace\]
is residual in $C_{S}^{s}(M,N)$.
\end{thm}
\subsection{Morse Theory}\label{MorseTheory}
\subsubsection{Morse functions and gradient fields}\label{gradientfields}
In this section we recollect some definitions and statements of Morse theory as presented in \cite{Aud}, with one exception.\\[.5em]
If $f:\mathbb{R}^{n}\rightarrow\mathbb{R}$ is a differentiable function, we are familiar with its gradient, the vector field $\grad f$, whose coordinates in the canonical basis of $\mathbb{R}^{n}$ are
\[\grad_{p}f=\Big(\frac{\partial f}{\partial p_{1}},\dots,\frac{\partial f}{\partial p_{n}}\Big).\]
More succinctly, it is (also) the vector field defined by
\[\langle\grad_{p}f,x\rangle=D_{p}f\cdot x\]
for every vector $x\in\mathbb{R}^{n}$ (where, of course, the angle brackets $\langle\cdot,\cdot\rangle$ denote the usual Euclidean inner product in $\mathbb{R}^{n}$). The most important properties of this vector field are due to the fact that this inner product is a positive definite symmetric bilinear form:
\begin{itemize}\itemsep0pt
\item[(i)]It vanishes exactly at the critical points of the function $f$.
\item[(ii)]The function $f$ is decreasing along the flow lines of the field $-\grad f$.
\end{itemize}
The gradient is also denoted by $\nabla$.
\begin{defi}
We say that a function $f:M\to\mathbb{R}$ on a manifold $M$ is a \textit{Morse function} if all its critical points are nondegenerate. The \textit{(Morse) index} $\Ind(p)$ of a nondegenerate critical point $p$ of $f$ is the dimension of the largest subspace of the tangent space to $M$ at $p$ on which the Hessian of $f$ is negative definite.
\end{defi}
\begin{defi}
Let $f:M\rightarrow\mathbb{R}$ be a Morse function on a manifold $M$ and $p\in M$ be a critical point of $f$. Denote by $\varphi^{s}$ the flow of the gradient field $-\nabla f$, i.e., $\varphi^{s}$ is the solution of the differential equation $\frac{d}{ds}\varphi^{s}=-\nabla f(\varphi^{s})$.
\begin{itemize}\itemsep0pt
\item[(i)]The \textit{stable manifold of p} is
\[\vspace*{-.3cm} W^{s}(p):=\{x\in M: \lim_{s\rightarrow+\infty}\varphi^{s}(x)=p\}.\]
\item[(ii)]The \textit{unstable manifold of p} is
\[W^{u}(p):=\{x\in M: \lim_{s\rightarrow-\infty}\varphi^{s}(x)=p\}.\]
\end{itemize}
\end{defi}
\begin{defi}[\cite{Ste}]
Let $M$ be a smooth manifold, $K,N\subset M$ smooth submanifolds, $p\in K\cap N$, and $T_{p}M,T_{p}N,T_{p}K$ the respective tangent spaces.\\
We say, $K$ is \textit{transverse} to $N$, written as $K\pitchfork N$, if the inclusion map $i:K\hookrightarrow M$ is transverse to~$N$, i.e. if the following holds:
\[T_{p}M=T_{p}N+T_{p}K\ \forall p\in K\cap N.\]
\end{defi}
\begin{defi}\label{smaledef}
We say that a gradient field of a Morse function $f$ satisfies the \textit{Smale condition} if all stable and unstable manifolds of its critical points meet transversally, that is, if for all critical points $p,q$ of $f$,
\[W^{u}(p)\pitchfork W^{s}(q).\]
\end{defi}
The following conclusion results from the Smale condition:
\begin{cor}\label{smale}
If $p$ and $q$ are two distinct critical points of a Morse function $f:M\to\mathbb{R}$, where $M$ is a smooth manifold, the gradient of $f$ satisfies the Smale condition, and $W^{u}(p)\cap W^{s}(q)\neq\emptyset$, then
\[\Ind(p)>\Ind(q).\]
In other words, the index decreases along the gradient lines.
\end{cor}
\subsubsection{1-parameter families of functions and vector fields}\label{Oneparamfam}
In Section \ref{Knotinv} we will proove that the cord algebra defined using Morse theory is a knot invariant. For this we consider a smooth isotopy of knots. Therefore, we get 1-parameter families of functions and vector fields, which we will take a closer look at. In the following, some properties of such 1-parameter families are shown. This section is taken from \cite{Cie1}.\\[.5em]
Throughout this section, $V$ denotes a smooth manifold of dimension m. First, we describe the critical points that occur in a generic 1-parameter family of functions $\phi_{t}:V\to\mathbb{R},t\in\mathbb{R}$.
\begin{defi}\ \\\label{emryonicbirthdeathdef}
\vspace*{-.5cm}
\begin{itemize}\itemsep0pt
\item[(i)]A critical point $p$ of a function $\phi:V\to\mathbb{R}$ is
called \textit{embryonic} if $\ker\Hess_{p}\phi$ is one-dimensional and the third derivative of $f$ in the direction of $\ker\Hess_{p}\phi$ is nonzero.
\item[(ii)]We say that a 1-parameter family of functions $\phi_{t}:V\to\mathbb{R},t\in\mathbb{R}$, has a \textit{birth-death type} critical point $p\in V$ at $t=0$ if $p$ is an embryonic critical point of $\phi_{0}$ and $(0,p)$ is a nondegenerate critical point of the function $(t,x)\mapsto\phi_{t}(x)$.
\end{itemize}
\end{defi}
With a family of functions $\phi_{t}:V\to\mathbb{R},t\in\mathbb{R}$, one can associate its \textit{profile} (or Cerf diagram). This is the subset $C(\lbrace\phi_{t}\rbrace)\subset\mathbb{R}\times\mathbb{R}$ such that $C(\lbrace\phi_{t}\rbrace)\cap(t\times\mathbb{R})$ is the set of critical values of the function~$\phi_{t}$. If $\phi_{t}$ is a family of Morse functions, then $C(\lbrace\phi_{t}\rbrace)$ is a collection of graphs of smooth functions. Part (ii) of the following theorem shows that birth-death points correspond to cusps of the profile.
\begin{thm}
\begin{itemize}
\item[(i)]Near an embryonic critical point $p$ of $\phi$ of index $k-1$ there exist coordinates $(x,y,z)\in\mathbb{R}^{m-k}\oplus\mathbb{R}^{k-1}\oplus\mathbb{R}$ in which $\phi$ has the form
\[\phi(x,y,z)=\phi(p)+\vert x\vert^{2}-\vert y\vert^{2}+z^{3}.\]
\item[(ii)]Suppose that $p$ is a birth-death type critical point of index $k-1$ for the family of functions $\phi_{t}:V\to\mathbb{R},t\in\mathbb{R}$, at $t=0$. Then there exist families of local diffeomorphisms $f_{t}:\mathcal{O}p\ p\to\mathcal{O}p\ 0\subset\mathbb{R}^{m}$ and $g_{t}:\mathcal{O}p\ \phi_{0}(p)\to\mathcal{O}p\ 0\subset\mathbb{R},t\in\mathcal{O}p\ 0$ such that the family of functions $\psi_{t}=g_{t}\circ\phi_{t}\circ f_{t}^{-1}$ has the form
\begin{align}
\psi_{t}(x,y,z)=\vert x\vert^{2}-\vert y\vert^{2}+z^{3}\pm tz\label{birthdeathcoord}
\end{align}
for $(x,y,z)\in\mathbb{R}^{m-k}\oplus\mathbb{R}^{k-1}\oplus\mathbb{R}$.
\item[(iii)]Let $\phi_{t},\tilde{\phi}_{t}:V\to\mathbb{R}$ be two families of functions with birth-death type critical points $p,\tilde{p}$ at $t=0$ of the same index and with the same profile. Then there exist a family of local diffeomorphisms $h_{t}:\mathcal{O}p\ p\to\mathcal{O}p\ \tilde{p},t\in\mathcal{O}p\ 0$ such that $\tilde{\phi}_{t}\circ h_{t}=\phi_{t}$.
\item[(iv)]A generic 1-parameter family of functions $\phi_{t}:V\to\mathbb{R}$ has only nondegenerate and birth-death type critical points.
\end{itemize}
\end{thm}
In particular, part (i) shows that embryonic critical points are isolated. We say that a birth-death type critical point $p$ is of \textit{birth type} if the sign in front of $t$ in formula \eqref{birthdeathcoord} is minus, and of \textit{death type} otherwise. Note that near a birth type critical point a pair of nondegenerate critical points of indices $k$ and $k-1$ appears at $t=0$, and near a death type critical point such a pair disappears.\\[.5em]
Let $X$ be a smooth vector field on $V$ und $p\in V$ a zero of $X$. The differential $D_{p}X:T_{p}V\to T_{p}V$ induces a splitting into invariant subspaces
\[T_{p}V=E_{p}^{+}\oplus E_{p}^{-}\oplus E_{p}^{0},\]
where $E_{p}^{+}$ (resp.~$E_{p}^{-},E_{p}^{0}$) is spanned by the generalized eigenvectors corresponding to eigenvalues with positive (resp.~negative, vanishing) real part. The dimension of $E_{p}^{-}$ is called the \textit{(Morse) index} of $X$ at $p$. Denote by $\varphi^{s}:V\to V,s\in\mathbb{R}$, the flow of $X$.
\begin{thm}[Center manifold theorem]
Let $p\in V$ be a zero of a $C^{r+1}$ vector field $X,r\in\mathbb{N}$. Then there exist the following local $\varphi^{s}$-invariant manifolds through $p$:
\begin{itemize}
\item $W_{p}^{0\pm}$ tangent to $E_{p}^{0}\oplus E_{p}^{\pm}$ of class $C^{r+1}$,
\item $W_{p}^{\pm}\subset W_{p}^{0\pm}$ tangent to $E_{p}^{\pm}$ of class $C^{r}$,
\item $W_{p}^{0}=W_{p}^{0+}\cap W_{p}^{0-}$ tangent to $E_{p}^{0}$ of class $C^{r+1}$.
\end{itemize}
The $W_{p}^{\pm}$ are unique, and they are smooth (resp.~real analytic) if $X$ is.
\end{thm}
$W_{p}^{-}$ (resp.~$W_{p}^{+},\ W_{p}^{0},\ W_{p}^{0-},\ W_{p}^{0+}$)is called the \textit{local stable (resp.~unstable, center, center-stable, center-unstable) manifold} at $p$. The center, center-stable, and center-unstable manifolds are in general not unique, and they need not be smooth even if $X$ is. By the center manifold theorem we can choose $C^{r}$-coordinates $Z=(x,y,z)\in E_{p}^{+}\oplus E_{p}^{-}\oplus E_{p}^{0}$ in which $W_{p}^{\pm}$ and $W_{p}^{0\pm}$ correspond to $E_{p}^{\pm}$ (resp.~$E_{p}^{0}\oplus E_{p}^{0\pm}$). In these coordinates $X$ is of the form
\begin{align}
X(x,y,z)=(A^{+}x+O(\vert x\vert\vert Z\vert),A^{-}y+O(\vert y\vert\vert Z\vert),A^{0}z+O(\vert z\vert\vert Z\vert+\vert x\vert\vert y\vert))\label{vfcoord}
\end{align}
with linear maps $A^{+}$ (resp.~$A^{-},\ A^{0}$) all of whose eigenvalues have positive (resp.~negative, zero) real part. The specific form of the higher order terms follows from tangency of $X$ to $W_{p}^{\pm}$ and $W_{p}^{0\pm}$.\\
A zero $p$ of a vector field $X$ is called \textit{nondegenerate} if all its eigenvalues are nonzero. It is called \textit{hyperbolic} if $E_{p}^{0}=0$, i.e., all eigenvalues of $D_{p}X$ have nonzero real part. In this case we have \textit{global} stable and unstable manifolds characterized by
\[W_{p}^{\pm}=\lbrace x\in V:\lim_{s\to\mp\infty}\varphi^{s}(x)=p\rbrace.\]
\begin{rem}
In Section \ref{gradientfields} we denote the stable (resp.~unstable) manifold by $W_{p}^{s}$ (resp.~$W_{p}^{u}$). In Section \ref{Corddef} we use this notation. However, in this section we have retained the $W_{p}^{\pm}$ notation used in \cite{Cie1} to illustrate the relationship with $E_{p}^{\pm}$. 
\end{rem}
These manifolds are injectively immersed (but not necessarily embedded) in $V$. For a hyperbolic zero the local representation \eqref{vfcoord} simplifies to
\[X(x,y)=(A^{+}x+O(\vert x\vert\vert Z\vert),A^{-}y+O(\vert y\vert\vert Z\vert)).\]
Let us call a zero $p$ \textit{embryonic} if $E_{p}^{0}$ is one-dimensional and the restriction of $X$ to a center manifold $W_{p}^{0}$ has nonvanishing second derivative at $p$ (for some local coordinate on $W_{p}^{0}\cong\mathbb{R}$; the definition depends neither on this local coordinate nor on the choice of $W_{p}^{0}$). It follows that in suitable coordinates $Z=(x,y,z)\in\mathbb{R}^{m-k}\oplus\mathbb{R}^{k-1}\oplus\mathbb{R}$ near $p$ the vector field is of the form
\begin{align}
X(x,y,z)=\left( A^{+}x+O(\vert x\vert\vert Z\vert),A^{-}y+O(\vert y\vert\vert Z\vert),z^{2}+O(\vert z\vert(\vert x\vert+\vert y\vert+\vert z\vert^{2})+\vert x\vert\vert y\vert)\right)\label{vfcoordembr}
\end{align}
with linear maps $A^{+},\ A^{-}$ all of whose eigenvalues have positive (resp.~negative) real part.
\begin{lem}
Let $p$ be an embryonic zero of a smooth vector field $X$. Then
\[\widehat{W}_{p}^{\pm}:=\lbrace x\in V:\lim_{s\to\mp\infty}\varphi^{s}(x)=p\rbrace\]
is an injectively immersed smooth manifold with boundary $W_{p}^{\pm}$.
\end{lem}
\begin{rem}\label{UcaphatWrem}
If we choose coordinates $Z=(x,y,z)$ on a neighborhood $U$ of $p$ in which $X$ is of the form \eqref{vfcoordembr}, then
\[U\cap\widehat{W}_{p}^{-}=\lbrace(x,y,z)\in U:x=0,z\leq0\rbrace,\]
see Figure \ref{UcaphatW}. An analogous statement holds for $\widehat{W}^{+}_{p}$.
\begin{figure}[ht]\centering
\includegraphics[scale=0.35]{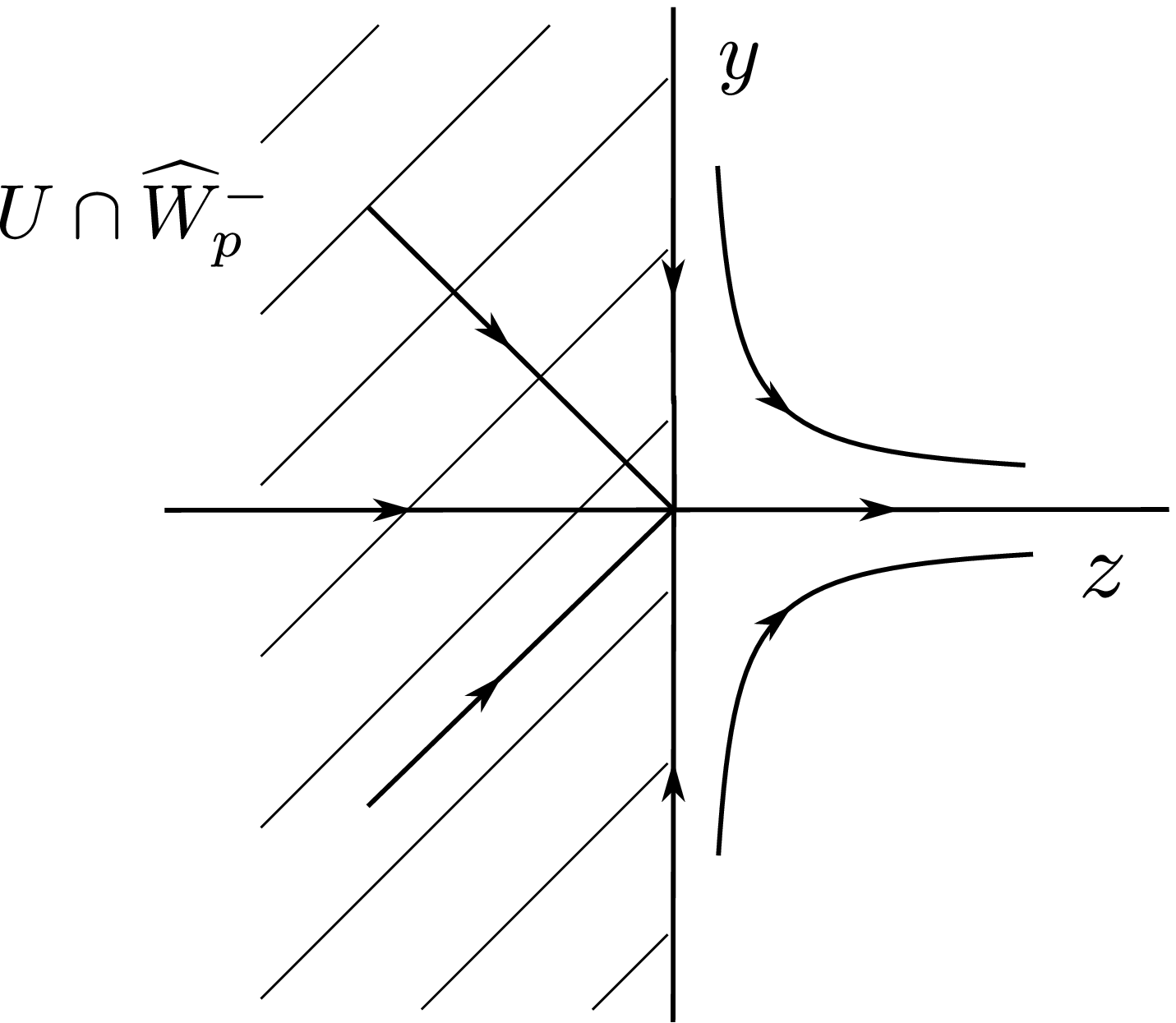}
\caption{The flow near an embryonic zero}\label{UcaphatW}
\end{figure}%
\end{rem}
We say that a 1-parameter family $X_{t},\ t\in(-\varepsilon,\varepsilon)$, of vector fields near $p\in V$ is of \textit{birth-death type} if $p$ is an embryonic zero of $X_{0}$ and the section $(t,Z)\mapsto X_{t}(Z)$ is transverse to the zero section of the bundle $TV\to\mathbb{R}\times V$ at $(0,p)$. It follows that in suitable coordinates $Z=(x,y,z)\in\mathbb{R}^{m-k}\oplus\mathbb{R}^{k-1}\oplus\mathbb{R}$ near $p$ the family is of the form
\begin{align}
X_{t}(x,y,z)=\left( A_{t}^{+}x+O(\vert x\vert\vert Z\vert),A_{t}^{-}y+O(\vert y\vert\vert Z\vert),z^{2}\pm t+O(\vert z\vert(\vert x\vert+\vert y\vert+\vert z\vert^{2})+\vert x\vert\vert y\vert)\right)\label{vffamcoord}
\end{align}
with smooth families of linear maps $A_{t}^{\pm}$ all of whose eigenvalues have positive (resp.~negative) real part. (The specific form of the higher order terms follows from tangency of the vector field $\widehat{X}(t,Z)=(0,X_{t}(Z))$ on $\mathbb{R}\times V$ to $\lbrace0\rbrace\times W_{p}^{\pm}$ and $\mathbb{R}\times W_{p}^{0}$, plus the fact that in suitable coordinates the zero set of $\widehat{X}$ is the curve $\lbrace x=y=z^{2}\pm t=0\rbrace$.)\\
We say that the family is of \textit{birth type} if the sign $z^{2}\pm t$ in \eqref{vffamcoord} is minus, and of \textit{death type} otherwise. Note that in a birth type family a pair of hyperbolic zeroes of indices $k$ and $k-1$ appears at $t=0$ and in a death type family such a pair disappears.
\begin{lem}\label{OneparamfamVF}
\begin{itemize}
\item[(i)]A generic vector field has only hyperbolic zeroes.
\item[(ii)]In a generic 1-parameter family of vector fields without nonconstant periodic
orbits, only birth-death type degeneracies appear.
\end{itemize}
\end{lem}
\begin{rem}
In Section \ref{Corddef} we will look at a gradient vector field of a Morse function. In this case, the zeros of the vector field correspond to the critical points of the function. In Section \ref{Knotinv} we will deal with an isotopy of knots. This results in a 1-parameter family of Morse functions and an associated 1-parameter family of gradient fields. In this context we will need the second statement of the previous lemma.
\end{rem}
\section{Proof of Lemma \ref{framinglemma}}\label{framinglemmaproof}
\begin{rem}\label{framingassumption}
Let $K$ be a knot of length $L$ and $\gamma:S^{1}\cong\mathbb{R}/L\mathbb{Z}\to\mathbb{R}^{3}$ be an arclength parametrization of $K$. Assume $\ddot{\gamma}(s)\neq0$ for all $s\in S^{1}$. In the proof of the lemma it is used that the framing~$\nu$ satisfies the following conditions: 
\begin{itemize}\itemsep0pt
\item[(i)]$\dot{\nu}(s)\neq0$ at all points $\gamma(s)$ for which a cord exists that is tangent to $K$ at $\gamma(s)$. According to Lemma \ref{cordlemma}, these are only finitely many points.
\item[(ii)]Consider the map $\bar{n}:S^{1}\to S^{2},s\mapsto\frac{\ddot{\gamma}(s)}{\vert\ddot{\gamma}(s)\vert}$. For all points $\gamma(s)$ with $\nu(s)=\bar{n}(s)$ the following holds: $\dot{\nu}(s)\neq\dot{\bar{n}}(s)$.
\item[(iii)]For two particular finite sets $R_{1},R_{2}\subset K\times K$ which are described in more detail in the proof the following holds: $F^{s}\cap(R_{1}\cup R_{2})=\emptyset$. 
\end{itemize}
\end{rem}
\begin{proof}[Proof of Lemma \ref{framinglemma}]
Let $K$ be a knot and $\gamma:S^{1}\to\mathbb{R}^{3}$ be an arclength parametrization of $K$. Assume $\ddot{\gamma}(s)\neq0$ for all $s\in S^{1}$. Let $\nu$ be a framing of $K$.\\[.5em]
1) Let $(p,q)\in\partial S$ be a cord that is tangent to $K$ at the point $p=\gamma(s)$. We can assume that $\dot{\gamma}(s)$ points in the direction of the cord $(p,q)$, otherwise we can reparametrize. We choose local coordinates $(x,y,z)$ in $\mathbb{R}^{3}$, where the $x$-axis points in the direction of the cord $(p,q)$, therefore in the direction of $\dot{\gamma}(s)$, the $y$-axis points in the direction of $\ddot{\gamma}(s)$, and the $z$-axis points in the direction of $\dot{\gamma}(s)\times\ddot{\gamma}(s)$ such that we have $p=(0,0,0)$ and $q=(1,0,0)$. Then $K$ near $p$ can be written as a graph over the $x$-axis (with $y=\kappa x^{2}+O(x^{3})$ and $z=O(x^{3})$) and near $q$ as a graph over the $z$-axis (with $x=1+O(z^{2})$ and $y=O(z^{2})$). There $2\kappa\neq0$ is the curvature of $K$ at $p$ (cf. \cite{Cie3}, in the proof of Lemma~7.10). So we can assume a local model of $K$ near $p$ and $q$ (see Figure~\ref{boundaryofF}):\\
$K$ can be written near $p$ as a graph over the $x$-axis:
\[y=\kappa x^{2},\ z=0\]
and near $q$ as a graph over the $z$-axis:
\[x=1,\ y=0.\]
\begin{figure}[ht]\centering
\includegraphics[scale=0.4]{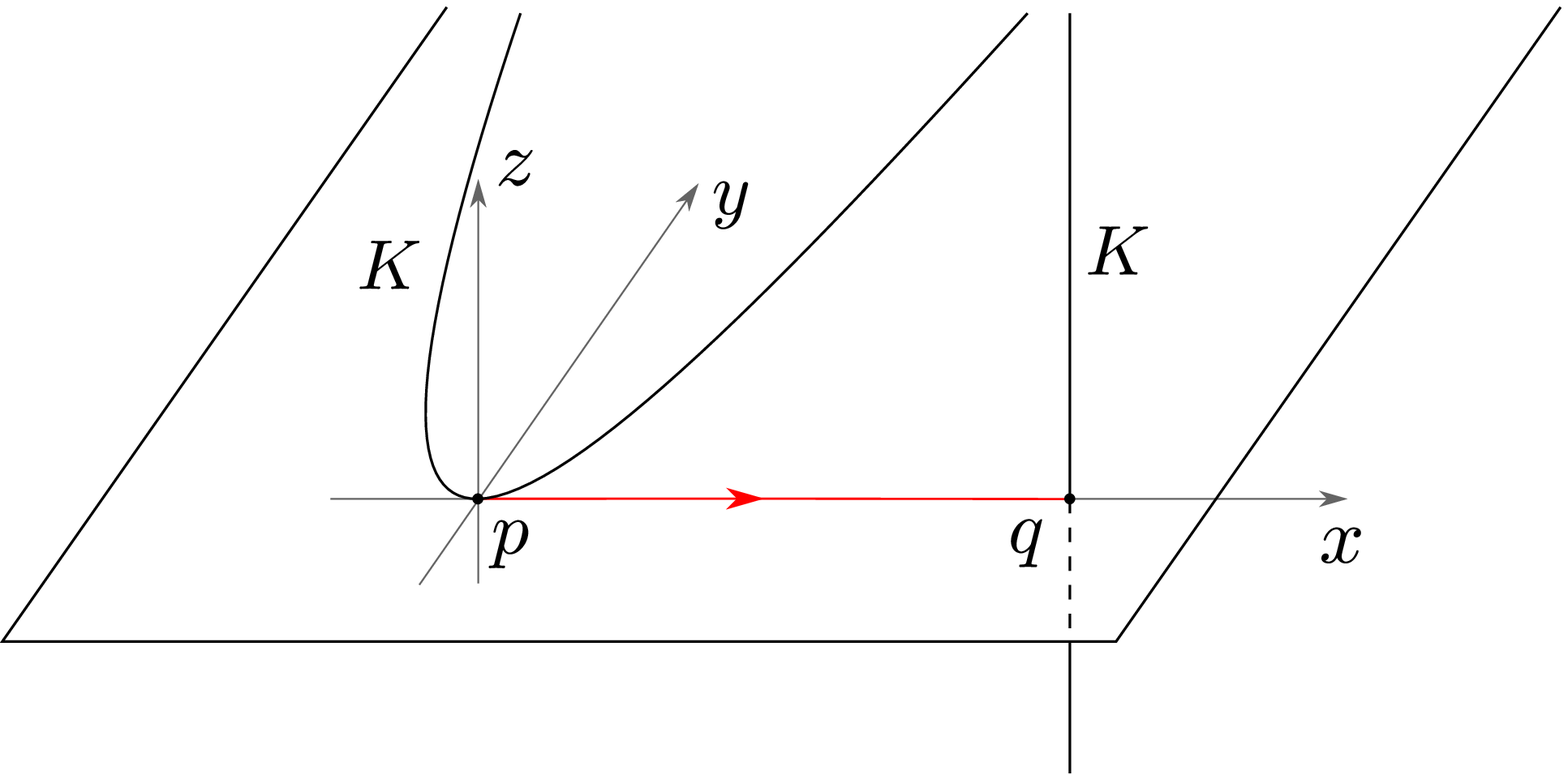}
\caption{Local model of $K$ in a neighborhood of a boundary point of $S$}\label{boundaryofF}
\end{figure}%

Thus, cords in a neighborhood of $(p,q)$ can be written as $(s(x),t(z))\in K\times K$ such that
\[\gamma(s(x))=
\begin{pmatrix}
x\\\kappa x^{2}\\0
\end{pmatrix}
\text{ and }\gamma(t(z))=
\begin{pmatrix}
1\\0\\z
\end{pmatrix}.\]
With this we have
\[p=\gamma(s(0))\text{ and }q=\gamma(t(0)).\]
The normal plane at the point $\gamma(s(x))$ near $p$ is
\begin{align*}
N(x)&=\lbrace v\in\mathbb{R}^{3}:\langle v,\frac{d}{dx}\gamma(s(x))\rangle=0\rbrace\\
&=\spann\left(
\begin{pmatrix}
-2\kappa x\\1\\0
\end{pmatrix},
\begin{pmatrix}
0\\0\\1
\end{pmatrix}\right).
\end{align*}
With $\alpha(x):=\begin{pmatrix}
-2\kappa x\\1\\0
\end{pmatrix}$ and $\beta:=\begin{pmatrix}
0\\0\\1
\end{pmatrix}$ we get
\[N(x)=\spann(\alpha(x),\beta).\]
Thus, the following holds:
\begin{align*}
\langle\gamma(t(z))-\gamma(s(x)),\alpha(x)\rangle&=\langle
\begin{pmatrix}
1-x\\-\kappa x^{2}\\z
\end{pmatrix},
\begin{pmatrix}
-2\kappa x\\1\\0
\end{pmatrix}\rangle\\
&=-2\kappa x+2\kappa x^{2}-\kappa x^{2}\\
&=\kappa x^{2}-2\kappa\\
\langle\gamma(t(z))-\gamma(s(x)),\beta\rangle&=\langle
\begin{pmatrix}
1-x\\-\kappa x^{2}\\z
\end{pmatrix},
\begin{pmatrix}
0\\0\\1
\end{pmatrix}\rangle\\
&=z.
\end{align*}
We define the map
\begin{align*}
n:D^{2}((p,q))\setminus\lbrace(p,q)\rbrace&\to S^{1}\\
(s(x),t(z))&\mapsto\frac{\left(\langle\gamma(t(z))-\gamma(s(x)),\alpha(x)\rangle,\langle\gamma(t(z))-\gamma(s(x)),\beta\rangle\right)}{\vert\left(\langle\gamma(t(z))-\gamma(s(x)),\alpha(x)\rangle,\langle\gamma(t(z))-\gamma(s(x)),\beta\rangle\right)\vert}\\
&\ \ \ \ =\frac{(\kappa x^{2}-2\kappa x,z)}{\vert(\kappa x^{2}-2\kappa x,z)\vert}.
\end{align*}
Now we set $x(u)=\cos u$ and $z(u)=\sin u$ for $u\in[0,2\pi]$ and consider the map
\begin{align*}
\hat{n}:S^{1}&\to S^{1}\\
u&\mapsto n(s(x(u)),t(z(u)))\\
&\ \ \ \ =\frac{(\kappa x^{2}(u)-2\kappa x(u),z(u))}{\vert(\kappa x^{2}(u)-2\kappa x(u),z(u))\vert}\\
&\ \ \ \ =\frac{(\kappa\cos^{2}u-2\kappa\cos u,\sin u)}{\vert(\kappa\cos^{2}u-2\kappa\cos u,\sin u)\vert}\\
&\ \ \ \ =\frac{((\cos u-2)\kappa\cos u,\sin u)}{\vert((\cos u-2)\kappa\cos u,\sin u)\vert}.
\end{align*}
The mapping degree of $\hat{n}$ can be determined immediately:
\[\deg(\hat{n})=1,\]
i.e. $(p,q)$ is a boundary point of $F^{s}$ and hence we have shown $\partial^{s}S\subset\partial F^{s}$.\\[.5em]
Let us now consider a framing $\nu(x)=a(x)\frac{\alpha(x)}{\vert\alpha(x)\vert}+b(x)\beta$ with $a^{2}(x)+b^{2}(x)=1$. The following holds:
\[(s(x),t(z))\in F^{s}\Leftrightarrow
\begin{pmatrix}
a(x)\\b(x)
\end{pmatrix}
=n(s(x),t(z))=\frac{1}{((\kappa x^{2}-2\kappa x)^{2}+z^{2})^{\frac{1}{2}}}
\begin{pmatrix}
\kappa x^{2}-2\kappa x\\z
\end{pmatrix}.\]
If the two equations are squared, both yield the same equation
\begin{align}
b^{2}(x)(\kappa x^{2}-2\kappa x)^{2}=a^{2}(x)z^{2}.\label{zofx}
\end{align}
\textbf{Case 1}: $a(0),b(0)\neq0\ \Leftrightarrow\ a(0),b(0)\neq\pm1$.\\
Then there exists a neighborhood of 0 with $a(x),b(x)\neq0$ for all $x$ in that neighborhood. In this neighborhood, equation \eqref{zofx} can be solved for $z^{2}$ and $z$ can be written as a function of $x$:
\[z^{2}(x)=\frac{b^{2}(x)(\kappa x^{2}-2\kappa x)^{2}}{a^{2}(x)}.\]
In addition, the following has to hold:
\begin{align*}
\sign(z(x))&\overset{!}{=}\sign(b(x))=const\text{ in a neighborhood of }x=0\\
\sign(\kappa x^{2}-2\kappa x)&\overset{!}{=}\sign(a(x))=const\text{ in a neighborhood of }x=0.
\end{align*}
It follows
\[x
\begin{cases}
>0&a(0)<0\\
<0&a(0)>0.
\end{cases}\]
Thus, in a neighborhood of $(p,q)$, $F^{s}$ is a smooth curve with boundary point $(p,q)$.\\[.5em]
\textbf{Case 2}: $b(0)=0\ \Leftrightarrow\ a(0)=\pm1$.\\
From equation \eqref{zofx} it follows immediately $z=0$ and the following must hold:
\begin{align*}
\sign(\kappa x^{2}-2\kappa x)&\overset{!}{=}\sign(a(x))=const\text{ in a neighborhood of }x=0.
\end{align*}
It follows
\[x
\begin{cases}
>0&a(0)=-1\\
<0&a(0)=1.
\end{cases}\]
Thus, in a neighborhood of $(p,q)$, $F^{s}$ is also a smooth curve with boundary point $(p,q)$.\\[.5em]
\textbf{Case 3}: $a(0)=0\ \Leftrightarrow\ b(0)=\pm1$.\\
From equation \eqref{zofx} it follows immediately $x=0$ (the second solution $x=2$ is not relevant since $x$ is only considered in a small neighborhood of 0) and the following has to hold:
\begin{align*}
\sign(z)&\overset{!}{=}\sign(b(x))=const\text{ in a neighborhood of }x=0.
\end{align*}
It follows
\[z
\begin{cases}
>0&b(0)=1\\
<0&b(0)=-1.
\end{cases}.\]
Thus, in a neighborhood of $(p,q)$, $F^{s}$ is also a smooth curve with boundary point $(p,q)$.\\[.5em]
We look again at equation \eqref{zofx} and calculate 
\begingroup
\allowdisplaybreaks
\begin{align*}
\lim_{x\to0}z^{2}(x)&=\lim_{x\to0}\frac{b^{2}(x)(\kappa x^{2}-2\kappa x)^{2}}{a^{2}(x)}\\
&=\lim_{x\to0}\frac{2b(x)b^{\prime}(x)(\kappa x^{2}-2\kappa x)^{2}+2b^{2}(x)(\kappa x^{2}-2\kappa x)(2\kappa x-2\kappa)}{2a(x)a^{\prime}(x)}\\
&=\lim_{x\to0}\Big(\frac{((b^{\prime}(x))^{2}+b(x)b^{\prime\prime}(x))(\kappa x^{2}-2\kappa x)^{2}+2b(x)b^{\prime}(x)(\kappa x^{2}-2\kappa x)(2\kappa x-2\kappa)}{(a^{\prime}(x))^{2}+a(x)a^{\prime\prime}(x)}\\
&\hspace*{1.3cm}+\frac{2b(x)b^{\prime}(x)(\kappa x^{2}-2\kappa x)(2\kappa x-2\kappa)+b^{2}(x)(2\kappa x-2\kappa)^{2}+2b^{2}(x)(\kappa x^{2}-2\kappa x)\kappa}{(a^{\prime}(x))^{2}+a(x)a^{\prime\prime}(x)}\Big)\\
&=\frac{4\kappa^{2}}{(a^{\prime}(0))^{2}}.
\end{align*}
\endgroup
The second and third equality result from the rule of l'Hospital. The last equality holds for $a^{\prime}(0)\neq0$. According to condition (i) from Remark \ref{framingassumption}, we have $\dot{\nu}(s(0))\neq0$ because the cord $(p,q)$ is tangent to $K$ at $p=\gamma(s(0))$. Since $b(0)=\pm1$, and thus $b$ reaches a maximum or minimum, we get $b^{\prime}(0)=0$. So $a^{\prime}(0)\neq0$ is satisfied. It follows 
\[\lim_{x\to0}z(x)=\sign(b(0))\frac{2\kappa}{\vert a^{\prime}(0)\vert}.\]
This means that $F^{s}$ can approach the cord $(p,q)$ in $K\times K$ very close, but in this case does not pass through this cord. Together with the above, this results in a picture like the one in Figure \ref{Fatpq}.\\[.5em]
\begin{figure}[ht]\centering
\includegraphics[scale=0.4]{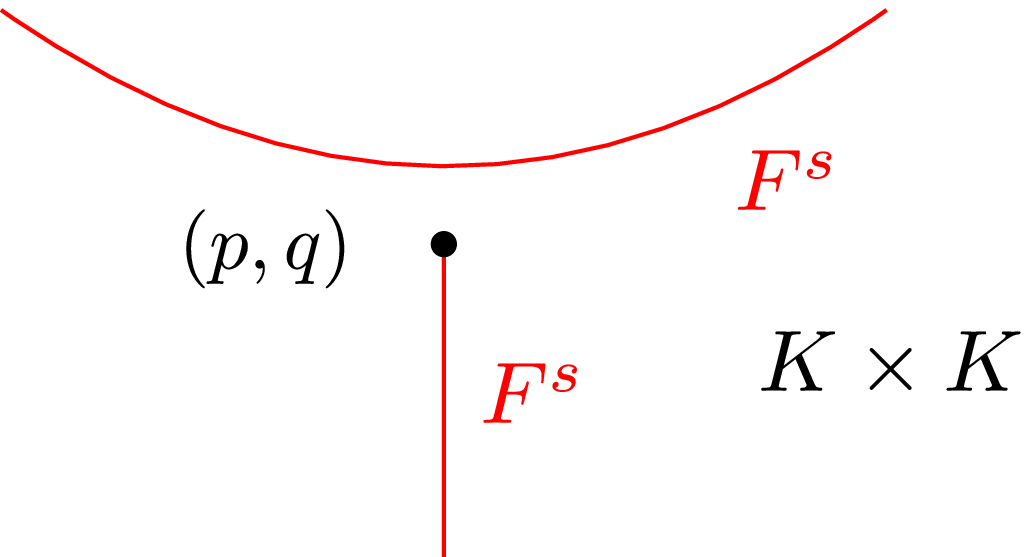}
\caption{$F^{s}\subset K\times K$ in a neighborhood of the cord $(p,q)$}\label{Fatpq}
\end{figure}%

So far we have only considered a local model of $K$ and neglected the terms of higher order. These must now be taken into account in a further step. However, we will not carry out this consideration here in detail, but only note the following:\\
Let $p=\gamma(s_{0})$ and $q=\gamma(t_{0})$. Now we only look at cords with startpoint $\gamma(s)$ for $s\in(s_{0}-\varepsilon,s_{0}+\varepsilon)$ and endpoint $\gamma(t)$ for $t\in(t_{0}-\delta,t_{0}+\delta)$, where $\varepsilon$ and $\delta$ have to be chosen sufficiently small. The above calculations will then be carried out again, taking into account the higher-order terms. In the course of the calculations, $\varepsilon$ and $\delta$ may have to be reduced several times in order to guarantee that the terms of higher order are small enough. Compare also \cite{Cie3} in the proof of Lemma 7.10.\\[.5em]
2) Now we consider the behavior of $F$ near the diagonal $\Delta\subset K\times K$. Let $p=\gamma(s)\in K$. We choose coordinates $(x,y,z)$ as in 1) such that $(\dot{\gamma}(s),\ddot{\gamma}(s),\dot{\gamma}(s)\times\ddot{\gamma}(s))$ is a basis of this coordinate system. Thus, $K$ can be considered as a graph over the $x$-axis with $y=\kappa x^{2}$ and $z=0$. So, the normal plane at the point $\gamma(s(x))$ is as in 1)
\[N(x)=\spann(\alpha(x),\beta)\text{ where }\alpha(x)=\begin{pmatrix}
-2\kappa x\\1\\0
\end{pmatrix},\beta=\begin{pmatrix}
0\\0\\1
\end{pmatrix}.\]
We define the following map, which is the projection onto the normal plane $N(x)$ and a normalization: 
\begin{align*}
n_{x}:\mathbb{R}^{3}&\to S^{1}\subset N(x)\\
v&\mapsto\frac{\left(\langle v,\frac{\alpha(x)}{\vert\alpha(x)\vert}\rangle,\langle v,\beta\rangle\right)}{\Bigl|\left(\langle v,\frac{\alpha(x)}{\vert\alpha(x)\vert}\rangle,\langle v,\beta\rangle\right)\Bigl|}
\end{align*}
Let $(x_{1},x_{2})$, where $x_{1}\neq x_{2}$, be a cord in a neighborhood $U$ of the cord $(p,p)\in\Delta$, see Figure~\ref{Fatpp}. Then
\begin{align*}
n_{x_{1}}(\gamma(s(x_{2}))-\gamma(s(x_{1})))&=n_{x_{1}}\left(\begin{pmatrix}
x_{2}-x_{1}\\
\kappa x_{2}^{2}-\kappa x_{1}^{2}\\
0-0
\end{pmatrix}\right)\\
&=\frac{(-2\kappa x_{2}x_{1}+2\kappa x_{1}^{2}+\kappa x_{2}^{2}-\kappa x_{1}^{2},0)}{\vert(-2\kappa x_{2}x_{1}+2\kappa x_{1}^{2}+\kappa x_{2}^{2}-\kappa x_{1}^{2},0)\vert}\\
&=\frac{(\kappa(x_{2}-x_{1})^{2},0)}{\vert(\kappa(x_{2}-x_{1})^{2},0)\vert}\\
&=(1,0).
\end{align*}
In the local coordinates of $N(0)$ the following holds: $\frac{\ddot{\gamma}(s(0))}{\vert\ddot{\gamma}(s(0))\vert}=(1,0)$. So if $\nu(0)=(1,0)\in N(0)$, we have $\nu(x)\neq(1,0)$ in a neighborhood of $x=0$  according to condition (ii) from Remark \ref{framingassumption}, and so we get
\[F^{s}\cap U=\lbrace(x_{1},x_{2})\in U:x_{1}=0, x_{2}\neq x_{1}\rbrace.\]
The set $F^{s}$ is not defined on the diagonal because there the projection of a cord onto the normal plane is the zero vector. However, we can extend $F^{s}$ in a  canonical way on the diagonal (see Figure~\ref{Fatpp})
\[F^{s}\cap U=\lbrace(x_{1},x_{2})\in U:x_{1}=0\rbrace.\]
\begin{figure}[ht]\centering
\includegraphics[scale=0.45]{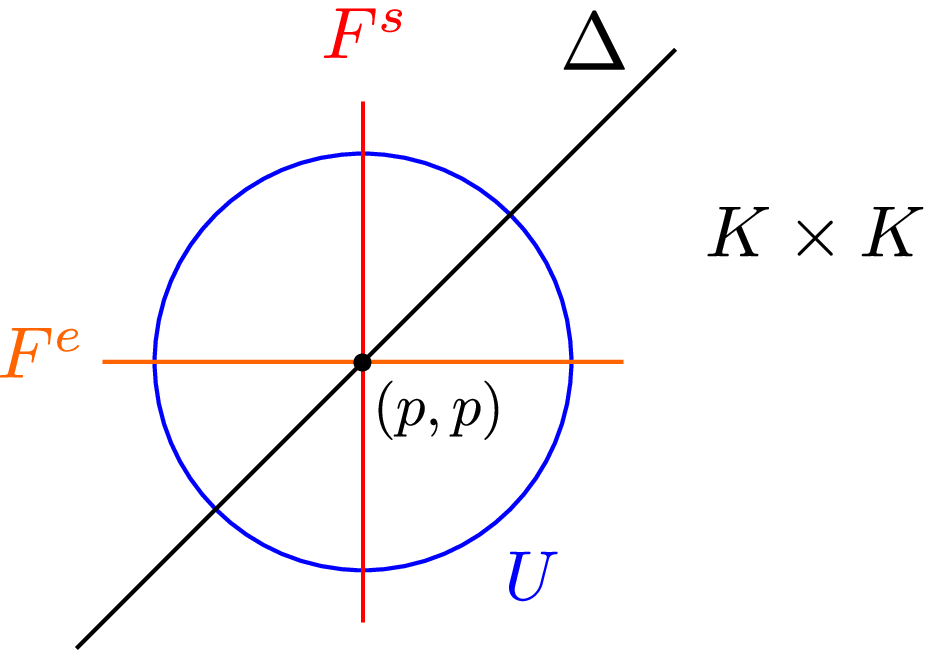}
\caption{$F\subset K\times K$ in a neighborhood of the cord $(p,p)$}\label{Fatpp}
\end{figure}%

Assumption (ii) from Remark \ref{framingassumption} is necessary for the following reason: If $\nu(x)\equiv(1,0)$ in a neighborhood of $x=0$, $\dim F^{s}=2$ would follow because every cord in that neighborhood would intersect the framing.\\[.5em]
Also here we will not carry out the consideration of the terms of higher order. However, this is to be realized analogous to the above consideration.\\[.5em]
3) It remains to show that $F^{s}\setminus(\partial S\cup\Delta)$ is a one-dimensional submanifold.\\
We define the maps $\nu_{1},\nu_{2}\in C^{\infty}(S^{1},S^{2})$ by $\nu_{1}(s)=\frac{\ddot{\gamma}(s)}{\vert\ddot{\gamma}(s)\vert}$ and $\nu_{2}(s)=\frac{\dot{\gamma}(s)\times\ddot{\gamma}(s)}{\vert\dot{\gamma}(s)\times\ddot{\gamma}(s)\vert}$. Thus, we have $(\dot{\gamma}(s),\nu_{1}(s),\nu_{2}(s))$ as an orthonormal basis of $\mathbb{R}^{3}$ for all $s\in S^{1}$. So the framing can be written as
\begin{align*}
\nu:S^{1}&\to S^{2}\\
s&\mapsto a_{1}(s)\nu_{1}(s)+a_{2}(s)\nu_{2}(s)
\end{align*}
where $(a_{1},a_{2})=:\hat{a}\in C^{\infty}(S^{1},S^{1})$.
We define the map
\begin{align*}
a:S^{1}&\to S^{1}\times S^{1}\\
s&\mapsto(s,\hat{a}(s)).
\end{align*}
We also define the following map, which is the orthogonal projection of a cord onto the normal plane at the startpoint of the cord and a normalization: 
\begin{align*}
\hat{n}:T^{2}\setminus(\partial S\cup\Delta)&\to S^{1}\\
(t,u)&\mapsto
\frac{
\begin{pmatrix}
\langle\gamma(u)-\gamma(t),\nu_{1}(t)\rangle\\
\langle\gamma(u)-\gamma(t),\nu_{2}(t)\rangle
\end{pmatrix}
}
{\biggl|
\begin{pmatrix}
\langle\gamma(u)-\gamma(t),\nu_{1}(t)\rangle\\
\langle\gamma(u)-\gamma(t),\nu_{2}(t)\rangle
\end{pmatrix}
\biggl|}.
\end{align*}
This map is well defined since the boundary of $S$ and the diagonal are excluded. According to Definition \ref{framingintersectiondef}, a cord $(t,u)$ intersects the framing at its startpoint if and only if $\hat{n}(t,u)=\hat{a}(t)$. We consider the map 
\begin{align*}
n:T^{2}\setminus(\partial S\cup\Delta)&\to S^{1}\times S^{1}\\
(t,u)&\mapsto(t,\hat{n}(t,u)).
\end{align*}
We want to show that $a$ is transverse to $n$. This holds if and only if the map 
\begin{align*}
a\times n:S^{1}\times(T^{2}\setminus(\partial S\cup\Delta))&\to(S^{1}\times S^{1})\times(S^{1}\times S^{1})\\
(s,t,u)&\mapsto(a(s),n(t,u))=(s,\hat{a}(s),t,\hat{n}(t,u))
\end{align*}
is transverse to $\Delta_{(S^{1}\times S^{1})\times(S^{1}\times S^{1})}$ where
\[\Delta_{(S^{1}\times S^{1})\times(S^{1}\times S^{1})}=\lbrace(s,t,u,v)\in (S^{1}\times S^{1})\times(S^{1}\times S^{1}):(s,t)=(u,v)\rbrace\]
is the diagonal in $(S^{1}\times S^{1})\times(S^{1}\times S^{1})$. The tangent space at points of this diagonal is
\[T_{(s,t,s,t)}\Delta_{(S^{1}\times S^{1})\times(S^{1}\times S^{1})}=\spann\left(
\begin{pmatrix}
1\\0\\1\\0
\end{pmatrix},
\begin{pmatrix}
0\\1\\0\\1
\end{pmatrix}
\right).\]
In addition we have
\[D_{(s,t,u)}(a\times n)=\left(\begin{array}{ccc}
1&0&0\\
\dot{\hat{a}}(s)&0&0\\
0&1&0\\
0&\frac{\partial\hat{n}}{\partial t}(t,u)&\frac{\partial\hat{n}}{\partial u}(t,u)
\end{array}\right).\]
The derivatives of $\hat{a}$ and $\hat{n}$ each have two components, but can be regarded as real-valued by the identification $S^{1}\cong\mathbb{R}/\mathbb{Z}$. The matrix is to be understood in this sense. Furthermore, we have
\[(a\times n)(s,t,u)\in\Delta_{(S^{1}\times S^{1})\times(S^{1}\times S^{1})}\Leftrightarrow s=t,\hat{a}(s)=\hat{n}(t,u).\]
Thus, $a\times n$ is transverse to the diagonal $\Delta_{(S^{1}\times S^{1})\times(S^{1}\times S^{1})}$ if and only if for all points $(s,t,s,t)\in(S^{1}\times S^{1})\times(S^{1}\times S^{1})$ with $(s,t,s,t)=(a\times n)(s,s,u)$ the following holds:
\[\spann\left(
\begin{pmatrix}
1\\0\\1\\0
\end{pmatrix},
\begin{pmatrix}
0\\1\\0\\1
\end{pmatrix},
\begin{pmatrix}
1\\\dot{\hat{a}}(s)\\0\\0
\end{pmatrix},
\begin{pmatrix}
0\\0\\1\\\frac{\partial\hat{n}}{\partial s}(s,u)
\end{pmatrix},
\begin{pmatrix}
0\\0\\0\\\frac{\partial\hat{n}}{\partial u}(s,u)
\end{pmatrix}
\right)=\mathbb{R}^{4}.\]
This is satisfied if
\begin{align}
\dot{\hat{a}}(s)\neq\frac{\partial\hat{n}}{\partial s}(s,u)\hspace*{1cm}\text{ or }\hspace*{1.2cm}\dot{\hat{a}}(s)=\frac{\partial\hat{n}}{\partial s}(s,u)\text{ and }\frac{\partial\hat{n}}{\partial u}(s,u)\neq0.\label{transvcond}
\end{align}
Since $\codim(\Delta_{(S^{1}\times S^{1})\times(S^{1}\times S^{1})}\subset(S^{1}\times S^{1})\times(S^{1}\times S^{1}))=2$, $(a\times n)^{-1}(\Delta_{(S^{1}\times S^{1})\times(S^{1}\times S^{1})})$ is a one-dimensional submanifold of $S^{1}\times(T^{2}\setminus(\partial S\cup\Delta))$. The following holds:
\begin{align*}
(a\times n)^{-1}(\Delta_{(S^{1}\times S^{1})\times(S^{1}\times S^{1})})&=\lbrace(s,t,u):a(s)=n(t,u)\rbrace\\
&=\lbrace(s,t,u):(s,\hat{a}(s))=(t,\hat{n}(t,u))\rbrace\\
&=\lbrace(s,s,u):(s,\hat{a}(s))=(s,\hat{n}(s,u))\rbrace.
\end{align*}
The latter set can be considered as
\[\lbrace(s,u)\in T^{2}\setminus(\partial S\cup\Delta):\hat{a}(s)=\hat{n}(s,u)\rbrace,\]
i.e. the set of all cords that intersect the framing at its startpoint. Hence $F^{s}\setminus(\partial S\cup\Delta)\subset K\times K$ is a one-dimensional submanifold.\\[.5em]
Together with the results from 1) and 2) the following holds: $F^{s}\subset K\times K$ is a one-dimensional submanifold with boundary and $\partial F^{s}=\partial^{s}S$.\\[.5em]
So it remains to show that condition \eqref{transvcond} holds generically. For this we want to show the following: \\
(i) $M_{1}:=\lbrace(s,u):\frac{\partial\hat{n}}{\partial u}(s,u)=0\rbrace\subset T^{2}\setminus(\partial S\cup\Delta)$ is a one-dimensional submanifold.\\
(ii) Let $M_{2}:=\lbrace(s,u):\hat{a}(s)=\hat{n}(s,u),\dot{\hat{a}}(s)=\frac{\partial\hat{n}}{\partial s}(s,u)\rbrace\subset T^{2}\setminus(\partial S\cup\Delta)$. Then $M_{1}\cap M_{2}=\emptyset$.\\[.5em]
Proof of these two statements:\\
(i) In Figure \ref{nhatsu} we can see the following:
\begin{align*}
\frac{\partial\hat{n}}{\partial u}(s,u)=0&\Leftrightarrow\frac{\partial(\gamma(u)-\gamma(s))}{\partial u}\in\spann(\dot{\gamma}(s),\hat{n}(s,u))\\
&\Leftrightarrow\dot{\gamma}(u)\in\spann(\dot{\gamma}(s),\hat{n}(s,u))\\
&\Leftrightarrow\langle\dot{\gamma}(u),\hat{n}(s,u)\times\dot{\gamma}(s)\rangle=0\\
&\Leftrightarrow\langle\dot{\gamma}(u),(\gamma(u)-\gamma(s))\times\dot{\gamma}(s)\rangle=0\\
&\Leftrightarrow\langle\dot{\gamma}(s)\times\dot{\gamma}(u),\gamma(u)-\gamma(s)\rangle=0.
\end{align*}
\begin{figure}[ht]\centering
\includegraphics[scale=0.5]{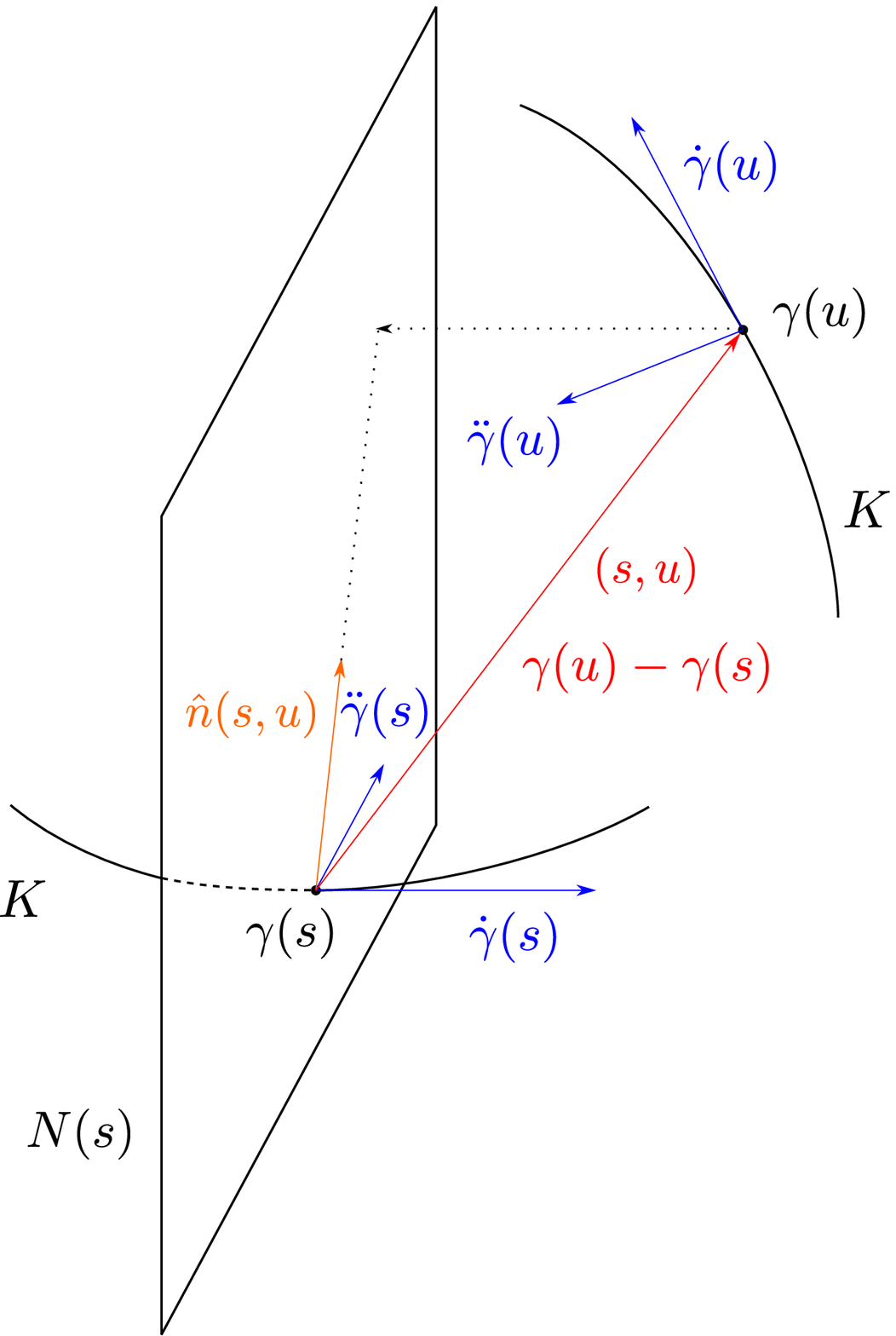}
\caption{The cord $(s,u)$ with the first and second derivatives of the parametrization $\gamma$ at start and end point of the cord and $\hat{n}(s,u)$}\label{nhatsu}
\end{figure}%

We define the maps (for $k$ big enough)
\begin{align*}
f:C^{k}(S^{1},\mathbb{R}^{3})\times(T^{2}\setminus(\partial S\cup\Delta))&\to\mathbb{R}\\
(\gamma,s,u)&\mapsto\langle\dot{\gamma}(s)\times\dot{\gamma}(u),\gamma(u)-\gamma(s)\rangle
\end{align*}
and
\begin{align*}
f_{\gamma}:T^{2}\setminus(\partial S\cup\Delta)&\to\mathbb{R}\\
(s,u)&\mapsto f(\gamma,s,u).
\end{align*}
We want to show that 0 is a regular value for $f_{\gamma}$. First, we show that the map
\[D_{(\gamma,s,u)}f:C^{k}(S^{1},\mathbb{R}^{3})\times\mathbb{R}^{2}\to\mathbb{R}\]
is surjective for all $(\gamma,s,u)$ with $f(\gamma,s,u)=0$:\\[.5em]
Let $(\gamma,s,u)$ be such that $f(\gamma,s,u)=\langle\dot{\gamma}(s)\times\dot{\gamma}(u),\gamma(u)-\gamma(s)\rangle=0$.\\
We assume that $D_{(\gamma,s,u)}f$ is the zero map and determine $D_{(\gamma,s,u)}f\cdot(\hat{\gamma},\hat{s},\hat{u})$:
\begin{align*}
\frac{\partial f}{\partial s}(\gamma,s,u)\cdot\hat{s}&=\langle\ddot{\gamma}(s)\times\dot{\gamma}(u),\gamma(u)-\gamma(s)\rangle\hat{s}+\underbrace{\langle\dot{\gamma}(s)\times\dot{\gamma}(u),-\dot{\gamma}(s)\rangle}_{=0}\hat{s}\\[-1em]
&=\langle\ddot{\gamma}(s)\times\dot{\gamma}(u),\gamma(u)-\gamma(s)\rangle\hat{s}\\
\frac{\partial f}{\partial u}(\gamma,s,u)\cdot\hat{u}&=\langle\dot{\gamma}(s)\times\ddot{\gamma}(u),\gamma(u)-\gamma(s)\rangle\hat{u}+\underbrace{\langle\dot{\gamma}(s)\times\dot{\gamma}(u),\dot{\gamma}(u)\rangle}_{=0}\hat{u}\\[-1em]
&=\langle\dot{\gamma}(s)\times\ddot{\gamma}(u),\gamma(u)-\gamma(s)\rangle\hat{u}\\
\frac{\partial f}{\partial\gamma}(\gamma,s,u)\cdot\hat{\gamma}&=\langle\dot{\hat{\gamma}}(s)\times\dot{\gamma}(u),\gamma(u)-\gamma(s)\rangle\\[-.7em]
&\hspace*{.4cm}+\langle\dot{\gamma}(s)\times\dot{\hat{\gamma}}(u),\gamma(u)-\gamma(s)\rangle\\
&\hspace*{.4cm}+\langle\dot{\gamma}(s)\times\dot{\gamma}(u),\hat{\gamma}(u)-\hat{\gamma}(s)\rangle.
\end{align*}
Since $\langle\dot{\gamma}(s)\times\dot{\gamma}(u),\gamma(u)-\gamma(s)\rangle=0$, we can also write:
\begin{align*}
\frac{\partial f}{\partial s}(\gamma,s,u)&=\langle(\ddot{\gamma}(s)+\lambda\dot{\gamma}(s))\times\dot{\gamma}(u),\gamma(u)-\gamma(s)\rangle\\
\frac{\partial f}{\partial u}(\gamma,s,u)&=\langle\dot{\gamma}(s)\times(\ddot{\gamma}(u)+\mu\dot{\gamma}(u)),\gamma(u)-\gamma(s)\rangle
\end{align*}
for any $\lambda,\mu\in\mathbb{R}$.\\[.5em]
Assuming $\gamma(u)-\gamma(s)=0$, it follows that $s=u$ and hence $(s,u)\in\Delta$. However, this is excluded by the definition of $f$.\\[.5em]
Assuming $(\gamma(u)-\gamma(s))\parallel\dot{\gamma}(u)$, it follows that the cord $(s,u)$ is tangent to $K$. Thus, we have $(s,u)\in\partial S$. However, this is also excluded by the definition of $f$. Analog for $(\gamma(u)-\gamma(s))\parallel\dot{\gamma}(s)$.\\[.5em]
Assume that $\ddot{\gamma}(s),\dot{\gamma}(s)$ and $\dot{\gamma}(u)$ are linearly independent. The set
\[A:=\lbrace(\ddot{\gamma}(s)+\lambda\dot{\gamma}(s))\times\dot{\gamma}(u):\lambda\in\mathbb{R}\rbrace\]
describes a straight line with $0\notin A$. So $\spann(A)$ is a plane through the origin with $\spann(A)\perp\dot{\gamma}(u)$. Thus, the following holds:
\[\langle(\ddot{\gamma}(s)+\lambda\dot{\gamma}(s))\times\dot{\gamma}(u),\gamma(u)-\gamma(s)\rangle=0\ \forall\ \lambda\in\mathbb{R}\ \ \Leftrightarrow\ \ (\gamma(u)-\gamma(s))\parallel\dot{\gamma}(u).\]
It follows that $(s,u)\in\partial S$. This is a contradiction to the definition of $f$. So $\ddot{\gamma}(s),\dot{\gamma}(s)$ and $\dot{\gamma}(u)$ are linearly dependent and we get
\[\dot{\gamma}(u)=\lambda_{1}\ddot{\gamma}(s)+\lambda_{2}\dot{\gamma}(s)\text{ with }\lambda_{1},\lambda_{2}\in\mathbb{R}\]
since $\ddot{\gamma}(s)$ and $\dot{\gamma}(s)$ are linearly independent. It follows that the set
\begin{align*}
\lbrace(\ddot{\gamma}(s)+\lambda\dot{\gamma}(s))\times\dot{\gamma}(u):\lambda\in\mathbb{R}\rbrace&=\lbrace(\ddot{\gamma}(s)+\lambda\dot{\gamma}(s))\times(\lambda_{1}\ddot{\gamma}(s)+\lambda_{2}\dot{\gamma}(s)):\lambda\in\mathbb{R}\rbrace\\
&=\lbrace\lambda\lambda_{1}\dot{\gamma}(s)\times\ddot{\gamma}(s)+\lambda_{2}\ddot{\gamma}(s)\times\dot{\gamma}(s):\lambda\in\mathbb{R}\rbrace\\
&=\lbrace(\lambda\lambda_{1}-\lambda_{2})\dot{\gamma}(s)\times\ddot{\gamma}(s):\lambda\in\mathbb{R}\rbrace
\end{align*}
is a straight line through the origin. Since $\frac{\partial f}{\partial s}(\gamma,s,u)=\langle(\ddot{\gamma}(s)+\lambda\dot{\gamma}(s))\times\dot{\gamma}(u),\gamma(u)-\gamma(s)\rangle=0$ for all $\lambda\in\mathbb{R}$ according to the asumption, it follows that
\[\gamma(u)-\gamma(s)\in\spann(\dot{\gamma}(s),\ddot{\gamma}(s)),\]
hence
\[\gamma(u)-\gamma(s)=\lambda_{3}\ddot{\gamma}(s)+\lambda_{4}\dot{\gamma}(s)\text{ with }\lambda_{3},\lambda_{4}\in\mathbb{R}.\]
Now we can rewrite the three summands of $\frac{\partial f}{\partial\gamma}(\gamma,s,u)\cdot\hat{\gamma}$:
\begin{align*}
\langle\dot{\hat{\gamma}}(s)\times\dot{\gamma}(u),\gamma(u)-\gamma(s)\rangle&=\langle\dot{\hat{\gamma}}(s)\times(\lambda_{1}\ddot{\gamma}(s)+\lambda_{2}\dot{\gamma}(s)),\lambda_{3}\ddot{\gamma}(s)+\lambda_{4}\dot{\gamma}(s)\rangle\\
&=\lambda_{1}\lambda_{4}\langle\dot{\hat{\gamma}}(s)\times\ddot{\gamma}(s),\dot{\gamma}(s)\rangle+\lambda_{2}\lambda_{3}\langle\dot{\hat{\gamma}}(s)\times\dot{\gamma}(s),\ddot{\gamma}(s)\rangle\\
&=(\lambda_{2}\lambda_{3}-\lambda_{1}\lambda_{4})\langle\dot{\gamma}(s)\times\ddot{\gamma}(s),\dot{\hat{\gamma}}(s)\rangle\\
\langle\dot{\gamma}(s)\times\dot{\hat{\gamma}}(u),\gamma(u)-\gamma(s)\rangle&=\langle\dot{\gamma}(s)\times\dot{\hat{\gamma}}(u),\lambda_{3}\ddot{\gamma}(s)+\lambda_{4}\dot{\gamma}(s)\rangle\\
&=-\lambda_{3}\langle\dot{\gamma}(s)\times\ddot{\gamma}(s),\dot{\hat{\gamma}}(u)\rangle\\
\langle\dot{\gamma}(s)\times\dot{\gamma}(u),\hat{\gamma}(u)-\hat{\gamma}(s)\rangle&=\langle\dot{\gamma}(s)\times(\lambda_{1}\ddot{\gamma}(s)+\lambda_{2}\dot{\gamma}(s)),\hat{\gamma}(u)-\hat{\gamma}(s)\rangle\\
&=\lambda_{1}\langle\dot{\gamma}(s)\times\ddot{\gamma}(s),\hat{\gamma}(u)-\hat{\gamma}(s)\rangle.
\end{align*}
All in all, with $\bar{\lambda}:=\lambda_{2}\lambda_{3}-\lambda_{1}\lambda_{4}$ we get
\[\frac{\partial f}{\partial\gamma}(\gamma,s,u)\cdot\hat{\gamma}=\langle\dot{\gamma}(s)\times\ddot{\gamma}(s),\bar{\lambda}\dot{\hat{\gamma}}(s)-\lambda_{3}\dot{\hat{\gamma}}(u)+\lambda_{1}(\hat{\gamma}(u)-\hat{\gamma}(s))\rangle.\]
We have $\lambda_{3}\neq0$, otherwise $\gamma(u)-\gamma(s)=\lambda_{4}\dot{\gamma}(s)$, i.e. $(s,u)\in\partial S$. So we can choose $\hat{\gamma}$ such that $\dot{\hat{\gamma}}(u)=\dot{\gamma}(s)\times\ddot{\gamma}(s)$ and $\dot{\hat{\gamma}}(s)=\hat{\gamma}(u)-\hat{\gamma}(s)=0$. With this choice we have $\frac{\partial f}{\partial\gamma}(\gamma,s,u)\cdot\hat{\gamma}\neq0$ and therefore the map $D_{(\gamma,s,u)}f$ is not the zero map. Hence $D_{(\gamma,s,u)}f$ is surjective.\\[.5em]
Thus, 0 is a regular value for $f$ and $f^{-1}(0)$ is a Banach manifold (see \cite{Cie4}, Proposition 3.19(a) (Parametric transversality)). The projection map $pr_{1}:f^{-1}(0)\to C^{k}(S^{1},\mathbb{R}^{3})$, for $k$ big enough, is smooth. According to the Sard-Smale theorem, almost all points in $C^{k}(S^{1},\mathbb{R}^{3})$ are regular values for $pr_{1}$, so the parametrization $\gamma$ of $K$ is generically a regular value. Thus, $pr_{1}^{-1}(\gamma)\subset f^{-1}(0)$ is a Banach submanifold and 0 is a regular value for $f_{\gamma}$ (see \cite{Cie4}, Proposition 3.19(c) (Parametric transversality)). So $M_{1}=f_{\gamma}^{-1}(0)$ is a one-dimensional submanifold and the assertion (i) is shown.\\[.5em]
(ii) Let $(s,u)\in M_{1}$.\\[.5em]
(a) First, we consider the case $\frac{\partial f_{\gamma}}{\partial u}(s,u)=0$, i.e. in a neighborhood of $s$ $M_{1}$ cannot be written as a graph over the $s$-axis, see Figure \ref{dfgammaduzero}.
\begin{figure}[ht]\centering
\includegraphics[scale=0.6]{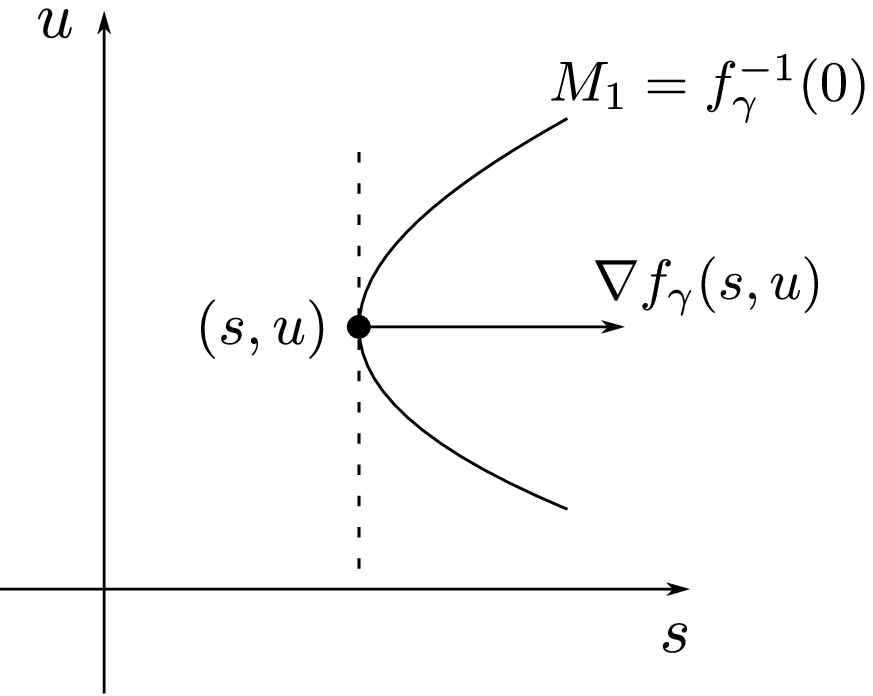}
\caption{A point $(s,u)\in M_{1}$ with $\frac{\partial f_{\gamma}}{\partial u}(s,u)=0$}\label{dfgammaduzero}
\end{figure}%

To simplify the notation we define the function $g_{\gamma}:=\frac{\partial f_{\gamma}}{\partial u}$. So the following holds:
\begingroup
\allowdisplaybreaks
\begin{align*}
&\begin{cases}
f_{\gamma}(s,u)=\langle\dot{\gamma}(s)\times\dot{\gamma}(u),\gamma(u)-\gamma(s)\rangle=0\text{ and}\\
g_{\gamma}(s,u)=\langle\dot{\gamma}(s)\times\ddot{\gamma}(u),\gamma(u)-\gamma(s)\rangle=0
\end{cases}\\
\Leftrightarrow&
\begin{cases}
\dot{\gamma}(s)\in\spann(\dot{\gamma}(u),\gamma(u)-\gamma(s))\text{ and}\\
\ddot{\gamma}(u)\in\spann(\dot{\gamma}(s),\gamma(u)-\gamma(s))\overset{(s,u)\notin\partial S}{=}\spann(\dot{\gamma}(u),\gamma(u)-\gamma(s))
\end{cases}\\
\Leftrightarrow&
\begin{cases}
\dot{\gamma}(s)=\mu_{1}\dot{\gamma}(u)+\mu_{2}(\gamma(u)-\gamma(s)),\ \mu_{1},\mu_{2}\in\mathbb{R},\ \mu_{1}\neq0\text{ and}\\
\ddot{\gamma}(u)=\mu_{3}\dot{\gamma}(u)+\mu_{4}(\gamma(u)-\gamma(s)),\ \mu_{3},\mu_{4}\in\mathbb{R},\ \mu_{4}\neq0.
\end{cases}
\end{align*}
\endgroup
$\mu_{1}$ must not vanish, otherwise we have $(s,u)\in\partial S$, and $\mu_{4}$ must not vanish, since $\dot{\gamma}(u)$ and $\ddot{\gamma}(u)$ are linearly independent.\\
We define the function
\begin{align*}
h_{\gamma}:T^{2}\setminus(\partial S\cup\Delta)&\to\mathbb{R}^{2}\\
(s,u)&\mapsto(f_{\gamma}(s,u),g_{\gamma}(s,u))
\end{align*}
and want to show that $h^{-1}_{\gamma}((0,0))$ is a finite set for generic knots. This holds if and only if $D_{(s,u)}h_{\gamma}$ is surjective for all $(s,u)$ with $h_{\gamma}(s,u)=(0,0)$. Let $(s,u)\in T^{2}\setminus(\partial S\cup\Delta)$ with $h_{\gamma}(s,u)=(0,0)$. Then the following holds:
\begin{align*}
D_{(s,u)}h_{\gamma}&=
\begin{pmatrix}
\langle\ddot{\gamma}(s)\times\dot{\gamma}(u),\gamma(u)-\gamma(s)\rangle&\langle\ddot{\gamma}(s)\times\ddot{\gamma}(u),\gamma(u)-\gamma(s)\rangle\\
\langle\dot{\gamma}(s)\times\ddot{\gamma}(u),\gamma(u)-\gamma(s)\rangle&\langle\dot{\gamma}(s)\times\dddot{\gamma}(u),\gamma(u)-\gamma(s)\rangle+\langle\dot{\gamma}(s)\times\ddot{\gamma}(u),\dot{\gamma}(u)\rangle
\end{pmatrix}\\
&=\begin{pmatrix}
\langle\ddot{\gamma}(s)\times\dot{\gamma}(u),\gamma(u)-\gamma(s)\rangle&\mu_{3}\langle\ddot{\gamma}(s)\times\dot{\gamma}(u),\gamma(u)-\gamma(s)\rangle\\
0&\mu_{1}\langle\dot{\gamma}(u)\times\dddot{\gamma}(u),\gamma(u)-\gamma(s)\rangle
\end{pmatrix}.
\end{align*}
So $D_{(s,u)}h_{\gamma}$ is surjective, i.e. $\det(D_{(s,u)}h_{\gamma})\neq0$, if and only if
\begin{align*}
&\langle\ddot{\gamma}(s)\times\dot{\gamma}(u),\gamma(u)-\gamma(s)\rangle\neq0\text{ and}\\
&\langle\dot{\gamma}(u)\times\dddot{\gamma}(u),\gamma(u)-\gamma(s)\rangle\neq0,
\end{align*}
since $\mu_{1}\neq0$.\\[.5em]
First, we look at $\langle\dot{\gamma}(u)\times\dddot{\gamma}(u),\gamma(u)-\gamma(s)\rangle$. We will need the torsion of the knot at the point $\gamma(u)$. Recall that for all $u\in S^{1}$ the torsion $\tau(u)\in\mathbb{R}$ is the unique real number for which $\dot{b}(u)=\tau(u)\bar{n}(u)$ holds, where $\bar{n}(u):=\frac{\ddot{\gamma}(u)}{\vert\ddot{\gamma}(u)\vert}$ and $b(u):= \dot{\gamma}(u)\times\bar{n}(u)$. We calculate 
\begingroup
\allowdisplaybreaks
\begin{align*}
\dot{b}(u)&=\frac{d}{du}\left(\dot{\gamma}(u)\times\frac{\ddot{\gamma}(u)}{\vert\ddot{\gamma}(u)\vert}\right)\\
&=\underbrace{\ddot{\gamma}(u)\times\frac{\ddot{\gamma}(u)}{\vert\ddot{\gamma}(u)\vert}}_{=0}+\dot{\gamma}(u)\times\frac{\vert\ddot{\gamma}(u)\vert\dddot{\gamma}(u)-\frac{1}{2}\cdot2(\langle\ddot{\gamma}(u),\ddot{\gamma}(u)\rangle)^{-\frac{1}{2}}\langle\ddot{\gamma}(u),\dddot{\gamma}(u)\rangle\ddot{\gamma}(u)}{\vert\ddot{\gamma}(u)\vert^{2}}\\
&=\dot{\gamma}(u)\times\frac{\vert\ddot{\gamma}(u)\vert^{2}\dddot{\gamma}(u)-\langle\ddot{\gamma}(u),\dddot{\gamma}(u)\rangle\ddot{\gamma}(u)}{\vert\ddot{\gamma}(u)\vert^{3}},
\end{align*}
\endgroup
and then we can write
\[\tau(u)\ddot{\gamma}(u)=\dot{\gamma}(u)\times\dddot{\gamma}(u)-\frac{\langle\ddot{\gamma}(u),\dddot{\gamma}(u)\rangle}{\vert\ddot{\gamma}(u)\vert^{2}}\dot{\gamma}(u)\times\ddot{\gamma}(u).\]
Therefore, the following holds:
\begin{align*}
\langle\dot{\gamma}(u)\times\dddot{\gamma}(u),\gamma(u)-\gamma(s)\rangle&=\tau(u)\langle\ddot{\gamma}(u),\gamma(u)-\gamma(s)\rangle+\frac{\langle\ddot{\gamma}(u),\dddot{\gamma}(u)\rangle}{\vert\ddot{\gamma}(u)\vert^{2}}\hspace*{-4pt}\underbrace{\langle\dot{\gamma}(u)\times\ddot{\gamma}(u),\gamma(u)-\gamma(s)\rangle}_{=0\text{ since }\ddot{\gamma}(u)\in\spann(\dot{\gamma}(u),\gamma(u)-\gamma(s))}\\
&=\tau(u)\langle\ddot{\gamma}(u),\gamma(u)-\gamma(s)\rangle.
\end{align*}
According to the above consideration, we have $\ddot{\gamma}(u)\in\spann(\dot{\gamma}(u),\gamma(u)-\gamma(s))$. Besides, $\ddot{\gamma}(u)$ and $\dot{\gamma}(u)$ are linearly independent. Therefore,
\[\gamma(u)-\gamma(s)\in\spann(\dot{\gamma}(u),\ddot{\gamma}(u)),\]
and thus
\[\gamma(u)-\gamma(s)=\mu_{5}\ddot{\gamma}(u)+\mu_{6}\dot{\gamma}(u),\ \mu_{5},\mu_{6}\in\mathbb{R},\mu_{5}\neq0.\]
$\mu_{5}$ must not vanish, otherwise the cord $(s,u)$ would be tangent to $K$. It follows
\begin{align*}
\tau(u)\langle\ddot{\gamma}(u),\gamma(u)-\gamma(s)\rangle&=\tau(u)\langle\ddot{\gamma}(u),\mu_{5}\ddot{\gamma}(u)+\mu_{6}\dot{\gamma}(u)\rangle\\
&=\tau(u)\mu_{5}\vert\ddot{\gamma}(u)\vert^{2}.
\end{align*}
This expression does not vanish if $\tau(u)\neq0$. For a generic knot $\tau(u)=0$ holds for only finitely many $u\in S^{1}$. So let $\lbrace\bar{u}_{i}:i=1,\dots,N_{u}\rbrace\subset S^{1}$, for an $N_{u}\in\mathbb{N}$, be the set of points for which $\tau(\bar{u}_{i})=0$ holds. Consider the set
\begin{align*}
\widetilde{R}_{1}&:=\lbrace(s,u)\in M_{1}:\frac{\partial f_{\gamma}}{\partial u}(s,u)=0,\tau(u)=0\rbrace\\
&=\lbrace(s,u)\in T^{2}\setminus(\partial^{s}S\cup\Delta):h_{\gamma}(s,u)=0,\tau(u)=0\rbrace\\
&=\lbrace(s,\bar{u}_{i})\in T^{2}\setminus(\partial^{s}S\cup\Delta):h_{\gamma}(s,\bar{u}_{i})=0,i\in\lbrace1,\dots,N_{u}\rbrace\rbrace.
\end{align*}
We can choose open neighborhoods of $\partial^{s}S$ and the diagonal of $K\times K$ of which we already know that $F^{s}$ is a submanifold within these neighborhoods. Let $U$ be the union of these neighborhoods. Thus, it suffices to consider the set
\[R_{1}:=\widetilde{R}_{1}\cap(T^{2}\setminus U).\]
$R_{1}$ is finite since: We have $M_{1}=f_{\gamma}^{-1}(0)$. Therefore, the following holds for all $(s,u)\in f_{\gamma}^{-1}(0)$: $\nabla f_{\gamma}(s,u)\perp M_{1}$. If $(s,\bar{u}_{i})\in M_{1}$, we have $\frac{\partial f_{\gamma}}{\partial u}(s,\bar{u}_{i})=0$ according to the assumption. So $M_{1}$ is tangent to the straight line $\lbrace(s,u):s=const\rbrace$ at the point $(s,\bar{u}_{i})\in R_{1}$, see Figure \ref{Ronefinite}.
\begin{figure}[ht]\centering
\includegraphics[scale=0.5]{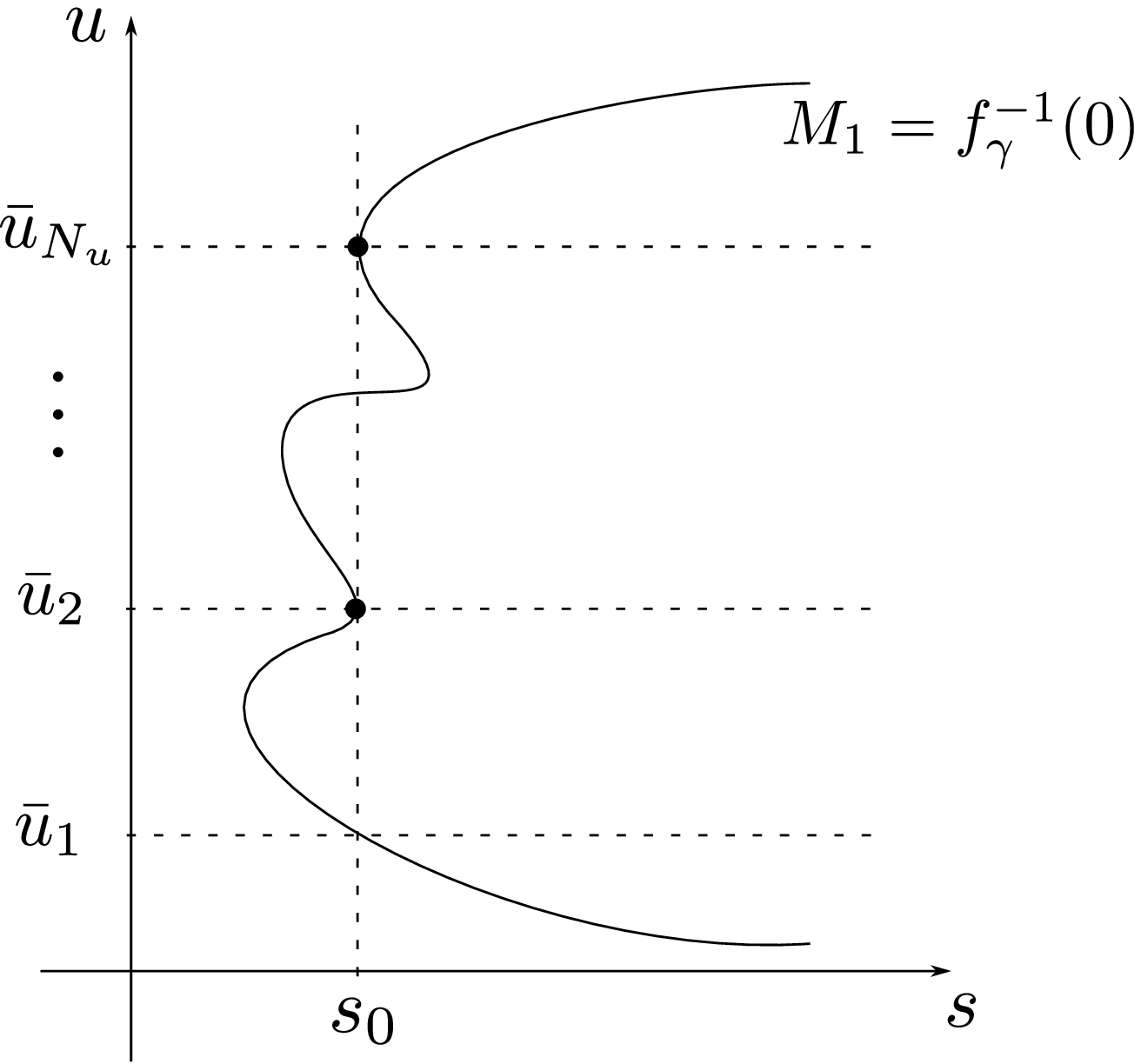}
\caption{The set $R_{1}$ is finite}\label{Ronefinite}
\end{figure}%
Since the set $\lbrace\bar{u}_{i}:i=1,\dots,N_{u}\rbrace\subset S^{1}$ is finite, no further point in $R_{1}$ can be contained in a sufficiently small neighborhood of $(s,\bar{u}_{i})$. Therefore, $R_{1}$ is a discrete set. Since $M_{1}\cap(T^{2}\setminus U)$ is compact, $R_{1}$ is finite. In the example of the Figure \ref{Ronefinite} we have $(s_{0},\bar{u}_{2}),(s_{0},\bar{u}_{N_{u}})\in R_{1}$.\\[.5em]
Now let's look at the entry $\frac{\partial f_{\gamma}}{\partial s}(s,u)=\langle\ddot{\gamma}(s)\times\dot{\gamma}(u),\gamma(u)-\gamma(s)\rangle$ in $D_{(s,u)}h_{\gamma}$. According to the assumption, we have $f_{\gamma}(s,u)=\langle\dot{\gamma}(s)\times\dot{\gamma}(u),\gamma(u)-\gamma(s)\rangle=0$. It follows
\[\frac{\partial f_{\gamma}}{\partial s}(s,u)=\langle(\ddot{\gamma}(s)+\lambda\dot{\gamma}(s))\times\dot{\gamma}(u),\gamma(u)-\gamma(s)\rangle\]
for all $\lambda\in\mathbb{R}$. Assume $\frac{\partial f_{\gamma}}{\partial s}(s,u)=0$. Then we get as in (i):
\[\gamma(u)-\gamma(s)=\lambda_{3}\ddot{\gamma}(s)+\lambda_{4}\dot{\gamma}(s),\ \lambda_{3},\lambda_{4}\in\mathbb{R},\lambda_{3}\neq0.\]
Thus, we can write:
\begin{align*}
0&=f_{\gamma}(s,u)=\lambda_{3}\langle\dot{\gamma}(s)\times\dot{\gamma}(u),\ddot{\gamma}(s)\rangle\text{ and}\\
0&=g_{\gamma}(s,u)=\lambda_{3}\langle\dot{\gamma}(s)\times\ddot{\gamma}(u),\ddot{\gamma}(s)\rangle,
\end{align*}
or, since $\lambda_{3}\neq0$:
\begin{align*}
&\langle\dot{\gamma}(s)\times\dot{\gamma}(u),\ddot{\gamma}(s)\rangle=0\text{ and}\\
&\langle\dot{\gamma}(s)\times\ddot{\gamma}(u),\ddot{\gamma}(s)\rangle=0.
\end{align*}
Therefore,
\[\dot{\gamma}(u),\ddot{\gamma}(u),\gamma(u)-\gamma(s)\in\spann(\dot{\gamma}(s),\ddot{\gamma}(s)).\]
We define the maps (for $k$ big enough)
\begin{align*}
\tilde{f}:C^{k}(S^{1},\mathbb{R}^{3})\times(T^{2}\setminus(\partial S\cup\Delta))&\to\mathbb{R}^{3}\\
(\gamma,s,u)&\mapsto
\begin{pmatrix}
\langle\dot{\gamma}(s)\times\ddot{\gamma}(s),\dot{\gamma}(u)\rangle\\
\langle\dot{\gamma}(s)\times\ddot{\gamma}(s),\ddot{\gamma}(u)\rangle\\
\langle\dot{\gamma}(s)\times\ddot{\gamma}(s),\gamma(u)-\gamma(s)\rangle
\end{pmatrix}
\end{align*}
and
\begin{align*}
\tilde{f}_{\gamma}:T^{2}\setminus(\partial S\cup\Delta)&\to\mathbb{R}^{3}\\
(s,u)&\mapsto f(\gamma,s,u).
\end{align*}
We want to show that 0 is a regular value for $\tilde{f}_{\gamma}$. First, we show that the map
\[D_{(\gamma,s,u)}\tilde{f}:C^{k}(S^{1},\mathbb{R}^{3})\times\mathbb{R}^{2}\to\mathbb{R}^{3}\]
is surjective for all $(\gamma,s,u)$ with $\tilde{f}(\gamma,s,u)=0$:\\[.5em]
Let $(\gamma,s,u)$ be such that $\tilde{f}(\gamma,s,u)=0$. Let $(\hat{\gamma},0,0)\in C^{k}(S^{1},\mathbb{R}^{3})\times\mathbb{R}^{2}$. We determine
\begin{align*}
D_{(\gamma,s,u)}\tilde{f}\cdot(\hat{\gamma},0,0)&=\frac{\partial\tilde{f}}{\partial\gamma}(\gamma,s,u)\cdot\hat{\gamma}\\
&=\begin{pmatrix}
\langle\dot{\hat{\gamma}}(s)\times\ddot{\gamma}(s)+\dot{\gamma}(s)\times\ddot{\hat{\gamma}}(s),\dot{\gamma}(u)\rangle\hspace{1.1cm}+\langle\dot{\gamma}(s)\times\ddot{\gamma}(s),\dot{\hat{\gamma}}(u)\rangle\hspace{1.1cm}\\
\langle\dot{\hat{\gamma}}(s)\times\ddot{\gamma}(s)+\dot{\gamma}(s)\times\ddot{\hat{\gamma}}(s),\ddot{\gamma}(u)\rangle\hspace{1.1cm}+\langle\dot{\gamma}(s)\times\ddot{\gamma}(s),\ddot{\hat{\gamma}}(u)\rangle\hspace{1.1cm}\\
\langle\dot{\hat{\gamma}}(s)\times\ddot{\gamma}(s)+\dot{\gamma}(s)\times\ddot{\hat{\gamma}}(s),\gamma(u)-\gamma(s)\rangle+\langle\dot{\gamma}(s)\times\ddot{\gamma}(s),\hat{\gamma}(u)-\hat{\gamma}(s)\rangle
\end{pmatrix}.
\end{align*}
If we choose $\hat{\gamma}_{1},\hat{\gamma}_{2},\hat{\gamma}_{3}$ such that $\dot{\hat{\gamma}}_{i}(s)=\ddot{\hat{\gamma}}_{i}(s)=0$, for $i=1,2,3$, and
\begin{itemize}\itemsep0pt
\item $\dot{\hat{\gamma}}_{1}(u)=\dot{\gamma}(s)\times\ddot{\gamma}(s),\ddot{\hat{\gamma}}_{1}(u)=\hat{\gamma}_{1}(u)-\hat{\gamma}_{1}(s)=0$, we get
\[D_{(\gamma,s,u)}\tilde{f}\cdot(\hat{\gamma}_{1},0,0)=\begin{pmatrix}\vert\dot{\gamma}(s)\times\ddot{\gamma}(s)\vert^{2}\\0\\0\end{pmatrix}.\]
\item $\ddot{\hat{\gamma}}_{2}(u)=\dot{\gamma}(s)\times\ddot{\gamma}(s),\dot{\hat{\gamma}}_{2}(u)=\hat{\gamma}_{2}(u)-\hat{\gamma}_{2}(s)=0$, we get
\[D_{(\gamma,s,u)}\tilde{f}\cdot(\hat{\gamma}_{2},0,0)=\begin{pmatrix}0\\\vert\dot{\gamma}(s)\times\ddot{\gamma}(s)\vert^{2}\\0\end{pmatrix}.\]
\item $\hat{\gamma}_{3}(u)-\hat{\gamma}_{3}(s)=\dot{\gamma}(s)\times\ddot{\gamma}(s),\dot{\hat{\gamma}}_{3}(u)=\ddot{\hat{\gamma}}_{3}(u)=0$, we get
\[D_{(\gamma,s,u)}\tilde{f}\cdot(\hat{\gamma}_{3},0,0)=\begin{pmatrix}0\\0\\\vert\dot{\gamma}(s)\times\ddot{\gamma}(s)\vert^{2}\end{pmatrix}.\]
\end{itemize}
Since $\vert\dot{\gamma}(s)\times\ddot{\gamma}(s)\vert^{2}\neq0$, the surjectivity results from any linear combination of these three choices. Thus, 0 is a regular value for $\tilde{f}$ and $\tilde{f}^{-1}(0)$ is a Banach manifold (see \cite{Cie4}, Proposition 3.19(a) (Parametric transversality)). The projection map $pr_{1}:\tilde{f}^{-1}(0)\to C^{k}(S^{1},\mathbb{R}^{3})$, for $k$ big enough, is smooth. According to the Sard-Smale theorem, almost all points in $C^{k}(S^{1},\mathbb{R}^{3})$ are regular values for $pr_{1}$, so the parametrization $\gamma$ of $K$ is generically a regular value. Thus, $pr_{1}^{-1}(\gamma)\subset\tilde{f}^{-1}(0)$ is a Banach submanifold and 0 is a regular value for $\tilde{f}_{\gamma}$ (see \cite{Cie4}, Proposition 3.19(c)). It follows $\tilde{f}_{\gamma}^{-1}(0)=\emptyset$, i.e. for a generic knot we have $\langle\ddot{\gamma}(s)\times\dot{\gamma}(u),\gamma(u)-\gamma(s)\rangle\neq0$ in case (a).\\[.5em]
Now we choose some open neighborhoods as follows:
\begin{itemize}\itemsep0pt
\item According to 1), the following holds: Around the finitely many points $\partial^{s}S$, $F^{s}$ is a smooth curve with endpoint in $\partial^{s}S$. So we can choose open neighborhoods of these finitely many points of which we already know that within these neighborhoods $F^{s}$ is a one-dimensional submanifold with boundary. Let $U_{\partial^{s}S}$ be the union of these neighborhoods.
\item According to 2), there exists an open neighborhood $U_{\Delta}$ of the diagonal in $K\times K$ such that $F^{s}\cap U_{\Delta}$ is a one-dimensional submanifold.
\item We can choose the framing such that $F^{s}\cap R_{1}=\emptyset$ since $R_{1}$ depends only on the knot and is finite. So there exist open neighborhoods $U_{(s,u)_{i}}$ of all points $(s,u)_{i}\in R_{1}$ such that for all $i$ the following holds: $F^{s}\cap U_{(s,u)_{i}}=\emptyset$. Let $U_{R_{1}}:=\bigcup\limits_{i}U_{(s,u)_{i}}$.
\end{itemize}
Let $U:=U_{R_{1}}\cup U_{\partial^{s}S}\cup U_{\Delta}$. We define the map
\[\tilde{h}_{\gamma}:=h_{\gamma}\vert_{T^{2}\setminus U}.\]
According to the above, $D_{(s,u)}\tilde{h}_{\gamma}$ is surjective for all $(s,u)$ with $\tilde{h}_{\gamma}(s,u)=(0,0)$. So the set
\[R_{2}:=\tilde{h}_{\gamma}^{-1}((0,0))\]
is finite since $T^{2}\setminus U$ is compact.\\
We choose the framing so that $F^{s}\cap R_{2}=\emptyset$.\\[.5em]
(b) Since $R_{2}$ is finite, i.e. $\frac{\partial f_{\gamma}}{\partial u}(s,u)=0$ for only finitely many $(s,u)\in M_{1}$, $M_{1}\setminus R_{2}$ can be split into finitely many disjoint subsets $M_{1,i},i=1,\dots,N_{M_{1}}$, with $N_{M_{1}}\in\mathbb{N}$, each of which is connected, so that the following holds for all $i=1,\dots,N_{M_{1}}$: $M_{1,i}$ can be written as a graph over the $s$-axis, i.e. the following holds: 
\[M_{1,i}=\lbrace(s,u_{i}(s)):s\in U_{i}\rbrace,\]
where $U_{i}\subset S^{1}$ is open, $u_{i}:U_{i}\to\mathbb{R}$ is smooth, and $\lim\limits_{s\to s_{0}}h_{\gamma}(s_{0},u_{i}(s))=(0,0)$ for $s_{0}\in\partial U_{i}$. We define the set
\[M_{U_{i}}:=\lbrace s\in U_{i}:\hat{a}(s)=g_{1}(s),\dot{\hat{a}}(s)=g_{2}(s)\rbrace\subset S^{1}\cong\mathbb{R}/\mathbb{Z}\]
with
\begin{align*}
g_{1}:S^{1}&\to S^{1}&g_{2}:S^{1}&\to\mathbb{R}\\
s&\mapsto\hat{n}(s,u(s))&s&\mapsto\frac{\partial\hat{n}}{\partial s}(s,u(s)).
\end{align*}
The set of all 1-jets from $S^{1}$ to $S^{1}$ is $J^{1}(S^{1},S^{1})=S^{1}\times S^{1}\times\mathbb{R}$. Furthermore,
\[L_{i}:=\lbrace(s,g_{1}(s),g_{2}(s)):s\in U_{i}\rbrace\subset J^{1}(S^{1},S^{1})\]
is a submanifold, since $L_{i}$ is a graph over the $s$-axis, with $\codim(L_{i}\subset J^{1}(S^{1},S^{1}))=2$. According to Theorem \ref{JetTransvthmadd}, we can perturb $\hat{a}$ in the space $C^{k}(S^{1},S^{1})$, for $k$ big enough, such that the map
\begin{align*}
h:S^{1}&\to J^{1}(S^{1},S^{1})\\
s&\mapsto(s,\hat{a}(s),\dot{\hat{a}}(s))
\end{align*}
is transverse to $\bigcup\limits_{i=1}^{N_{M_{1}}}\hspace*{-.1cm}L_{i}$. So we get $M_{U}:=\bigcup\limits_{i=1}^{N_{M_{1}}}\hspace*{-.1cm}M_{U_{i}}=h^{-1}\left(\bigcup\limits_{i=1}^{N_{M_{1}}}\hspace*{-.1cm}L_{i}\right)=\emptyset$ since $\codim(M_{U}\subset S^{1})=2$.\\[.5em]
Thus, the condition \eqref{transvcond} is satisfied and this proves the lemma.
\end{proof}

\end{document}